\documentclass[a4paper,10pt]{article}
\usepackage{tocloft}
\usepackage{enumerate}
\usepackage{algorithm}
\usepackage{algorithmic}
\usepackage{amsfonts}
\usepackage{amsthm}
\usepackage{amsmath}
\usepackage{amssymb}
\usepackage{bbm}
\usepackage{bm}
\usepackage[mathscr]{eucal}
\usepackage{a4wide}
\usepackage{authblk}
\usepackage{appendix}

\usepackage{graphicx}
\usepackage{tikz}
\usetikzlibrary{shapes,arrows,backgrounds,patterns,decorations.markings,decorations.pathreplacing,intersections}
\usepackage{pgfplots}
\usepgfplotslibrary{fillbetween}

\usepackage{subcaption}
\usepackage{dsfont}

\usepackage{listings}
\usepackage[nottoc,numbib]{tocbibind}
\settocbibname{References}
\usepackage{hyperref}
\hypersetup{
    colorlinks = true,
    urlcolor = {magenta},
    citecolor = {green},
    linkcolor = {blue}
}

\newcommand{\pr}{\mathbb{P}}								
\newcommand{\Pee}{\mathbb{P}}

\newcommand{\Prob}[1]{\pr\left(#1\right)}					
																
\newcommand{\e}{\mathbb{E}}								
\newcommand{\E}{\mathbb{E}}

\newcommand{\Exp}[1]{\e\left[#1\right]}					
															
\newcommand{\CExp}[2]{\e\left[\left.#1\right|#2\right]}	

\newcommand{\bigO}[1]{O\left(#1\right)}				
\newcommand{\smallO}[1]{o\left(#1\right)}			
\newcommand{\bigT}[1]{\Theta\left(#1\right)}		
\newcommand{\smallOp}[1]{o_\pr \left(#1\right)}		

\newcommand{\ind}[1]{\mathbbm{1}_{\left\{#1\right\}}}	
\newcommand{\1}{\mathds{1}}								

\newcommand{\N}{\mathbb{N}}							
\newcommand{\R}{{\ensuremath{\mathbb R}}}			

\newcommand{\Kcal}{\ensuremath{{\mathcal K}}}		
\newcommand{\Bcal}{\ensuremath{{\mathcal B}}}		
\newcommand{\Ncal}{\ensuremath{{\mathcal N}}}		

\newcommand{\eps}{\varepsilon}						
\DeclareMathOperator{\dd}{\mathrm{d}\hspace{-2pt}}			

\renewcommand{\H}{\ensuremath{\mathbb{H}}}			
\newcommand{\Rcal}[0]{\ensuremath{{\mathcal R}}}	
\newcommand{\Pcal}[0]{\ensuremath{{\mathcal P}}}	
\newcommand{\Po}[0]{\ensuremath{{\mathrm{Po}}}}		
\newcommand{\PPP}{\ensuremath{\Pcal}}	

\newcommand{\dom}[1]{\mathcal{D}_{#1}}

\newcommand{\BallHyp}[1]{\Bcal\left(#1\right)}		
\newcommand{\BallPo}[1]{\Bcal_{\infty}\left(#1\right)}		
\newcommand{\BallPon}[1]{\Bcal_{\text{box}}\left(#1\right)}	
\newcommand{\FatBallHyp}[1]{\check{\Bcal}_{\Po}(#1)}		%
\newcommand{\BallSym}[1]{\Bcal_{\Po \bigtriangleup \infty}\left(#1\right)}	
\newcommand{\BallInter}[1]{\Bcal_{\Po \cap \text{box}}\left(#1\right)}

\newcommand{\Mu}[1]{\mu\left(#1\right)}

\newcommand{\expH}{\overline{n}_{\Po}(k_n)}
\newcommand{\expP}{\overline{n}_{\mathrm{box}}(k_n)}

\newcommand{\stripknc}{\mathcal{S}_{k_n,C}}

\newcommand\numberthis{\addtocounter{equation}{1}\tag{\theequation}}	

\newcommand{\MeijerGnew}[7]{G_{#3,#4}^{#1,#2} \left(#7 \bigg|\begin{matrix}#5  \\ #6 \end{matrix}\right)}

\newcommand{\Li}{\operatorname{Li}}

\newcommand{\Nbox}{\ensuremath{N_{\text{box}}}}
\newcommand{\Dbox}{\ensuremath{D_{\Gbox}}}
\newcommand{\Vbox}{\ensuremath{V_{\text{box}}}}
\newcommand{\VPo}{\ensuremath{V_{\Po}}}

\newcommand{\eR}[0]{\ensuremath{{\mathbb R}}}
\newcommand{\pmf}[0]{\ensuremath{{\pi}}}

\newcommand{\Ee}[0]{\ensuremath{{\mathbb E}}}
\newcommand{\Haa}{\ensuremath{{\mathbb H}}}
\newcommand{\GPo}[0]{\ensuremath{G_{\Po}}}
\newcommand{\Gbox}[0]{\ensuremath{G_{\mathrm{box}}}}
\newcommand{\GboxH}[0]{\ensuremath{G_{\mathrm{box}-}}}
\newcommand{\Ginf}[0]{\ensuremath{G_{\infty}}}
\newcommand{\isd}{\stackrel{\text{d}}{=}}

\newcommand{\Fcal}{\ensuremath{\mathcal{F}}}

\newenvironment{proofof}[1]{\vspace{1ex}\noindent{\bf Proof of #1.}}{\hspace*{\fill}$\blacksquare$\vspace{1ex}}
\renewenvironment{proof}{\vspace{1ex}\noindent{\bf Proof.}}{\hspace*{\fill}$\blacksquare$\vspace{1ex}}
\DeclareMathOperator{\Var}{Var}

\definecolor{orange}{RGB}{255,127,0}
\definecolor{pink}{RGB}{255,150,150}

\allowdisplaybreaks

\newtheorem{remark}{Remark}[section]

\newtheorem{theorem}{Theorem}[section]
\newtheorem{lemma}[theorem]{Lemma}
\newtheorem{proposition}[theorem]{Proposition}
\newtheorem{corollary}[theorem]{Corollary}

\title{Clustering in a hyperbolic model of complex networks.}
\author[1]{Nikolaos Fountoulakis\thanks{Research partially supported by the Alan Turing Institute, grant no. EP/N510129/1.}}
\author[2,3]{Pim van der Hoorn\thanks{Research partially supported by ARO Grant W911NF-16-1-0391 and W911NF-17-1-0491}}
\author[4]{Tobias M\"{u}ller\thanks{Research partially supported by NWO grants 639.032.529 and 612.001.409.}}
\author[4]{Markus Schepers\thanks{Research partially supported by NWO grant 639.032.529.}}

\affil[1]{School of Mathematics, University of Birmingham, United Kingdom.}
\affil[2]{Department of Mathematics and Computer Science, Eindhoven University of Technology, The Netherlands}
\affil[3]{Department of Physics, Northeastern University, United States.}
\affil[4]{Bernoulli Institute, University of Groningen, The Netherlands.}

\pgfplotsset{compat=1.14}
\begin{document}

\maketitle

\begin{abstract}
In this paper we consider the clustering coefficient, and clustering function in a random graph model proposed by Krioukov et al.~in 2010. 
In this model, nodes are chosen randomly inside a disk in the hyperbolic plane and two nodes are connected if they are at most at a certain 
hyperbolic distance from each other. It has been previously shown that this model has various properties associated with complex networks, 
including a power-law degree distribution, ``short distances'' and a non-vanishing clustering coefficient. 
The model is specified using three parameters: the number of nodes $n$, which we think of as going to infinity, and $\alpha, \nu > 0$, 
which we think of as constant. Roughly speaking, the parameter $\alpha$ controls the power law exponent of the degree sequence 
and $\nu$ the average degree.

Here we show that the clustering coefficient tends in probability to a constant $\gamma$ that we give explicitly as a closed form 
expression in terms of $\alpha, \nu$ and certain special functions. 
This improves earlier work by Gugelmann et al., who proved that the clustering coefficient remains bounded away from zero 
with high probability, but left open the issue of convergence to a limiting constant. Similarly, we are able to show that 
$c(k)$, the average clustering coefficient over all vertices of degree exactly $k$, tends in probability to a limit 
$\gamma(k)$ which we give explicitly as a closed form expression in terms of $\alpha, \nu$ and certain special functions. 
We are able to extend this last result also to sequences $(k_n)_n$ where $k_n$ grows as a function of $n$. 
Our results show that $\gamma(k)$ scales differently, as $k$ grows, for different ranges of $\alpha$. 
More precisely, there exists constants $c_{\alpha,\nu}$ depending on $\alpha$ and $\nu$, such that as $k \to \infty$, 
$\gamma(k) \sim c_{\alpha,\nu} \cdot k^{2 - 4\alpha}$ if $\frac{1}{2} < \alpha < \frac{3}{4}$, 
$\gamma(k) \sim c_{\alpha,\nu} \cdot \log(k) \cdot k^{-1} $ 
if $\alpha=\frac{3}{4}$ and $\gamma(k) \sim c_{\alpha,\nu} \cdot k^{-1}$ when $\alpha > \frac{3}{4}$. 
These results contradict a claim of Krioukov et al., which stated that the limiting values $\gamma(k)$ should always scale 
with $k^{-1}$ as we let $k$ grow. 
\end{abstract}

\newpage

\tableofcontents

\newpage

\section{Introduction and main results}

In this paper, we will consider clustering in a model of random graphs that involves points taken randomly in the hyperbolic plane. This model was introduced by Krioukov, Papadopoulos, Kitsak, Vahdat and Bogu\~{n}\'a~\cite{krioukov2010hyperbolic} in 
2010 - we abbreviate it as \emph{the KPKVB model}. We should however note that the model also goes by several other names in the literature, including {\em hyperbolic random geometric graphs} and {\em random hyperbolic graphs}. Krioukov et al.~suggested this model as a suitable model for complex networks. It exhibits the three main characteristics usually associated with complex networks: a power-law degree distribution, a non-vanishing clustering coefficient and short graph distances.

\subsection{KPKVB model}
We start with the definition of the model. As mentioned, its nodes are situated in the hyperbolic plane $\Haa$, which is a surface with constant Gaussian curvature $-1$. This surface has several convenient representations (i.e.~coordinate maps), such as the Poincar\'e half-plane model, the Poincar\'e disk model and the Klein disk model. A gentle introduction to Gaussian curvature, hyperbolic geometry and these representations of the hyperbolic plane can be found in~\cite{stillwell2012geometry}. Throughout this paper we will be working with a representation of the hyperbolic plane using {\em hyperbolic polar coordinates}, sometimes called the {\em native representation}. That is, a point $u \in \Haa$ is represented as $(r,\theta)$, where $r$ is the hyperbolic distance between $u$ and the origin $O$ and $\theta$ as the angle between the line segment $Ou$ and the positive $x$-axis. 
Here, when mentioning ``the origin'' and the angle between the line segment and the positive $x$-axis, we think of $\Haa$ embedded as the Poincar\'e disk in the ordinary euclidean plane.

The KPKVB model has three parameters: the number of vertices $n$, which we think of as going to infinity, and $\alpha > \frac{1}{2}$, $\nu > 0$ which we think of as fixed. Given $n, \alpha, \nu$ we define $R = 2\log(n/\nu)$. Then the hyperbolic random graph $G(n;\alpha, \nu)$ is defined as follows:
\begin{itemize}
\item The vertex set is given by $n$ i.i.d.~points $u_1, \dots, u_n$ denoted in polar coordinates $u_i = (r_i, \theta_i)$, where the angular coordinate $\theta$ is chosen uniformly from $(-\pi,\pi]$ while the radial coordinate $r$ is sampled independently according to the cumulative distribution function
\begin{equation}\label{eq:def_hyperbolic_point_distribution}
	F_{\alpha,R}(r) = \begin{cases}
		0 &\mbox{if } r < 0\\
		\frac{\cosh(\alpha r)-1}{\cosh(\alpha R) - 1} &\mbox{if } 0 \le r \le R\\
		1&\mbox{if } r > R
	\end{cases}
\end{equation}
\item Any two vertices $u_i=(r_i,\theta_i)$ and $u_j=(r_j,\theta_j)$ are adjacent if and only if $d_\H(u_i,u_j) \le R$, where $d_\H$ denotes the distance in the hyperbolic plane. We will frequently be using that, by the hyperbolic law of cosines, $d_\H(u_i,u_j) \le R$ is equivalent to
\[
	\cosh(r_i) \cosh(r_j) - \sinh(r_i) \sinh( r_j) \cos(|\theta_i-\theta_j|_{2\pi}) \le \cosh(R),
\]
where $|a|_{b} = \min( |a|, b - |a|)$ for $-b\leq a\leq b$.
\end{itemize}

\begin{figure}[!t]
\centering
\includegraphics[scale=0.3]{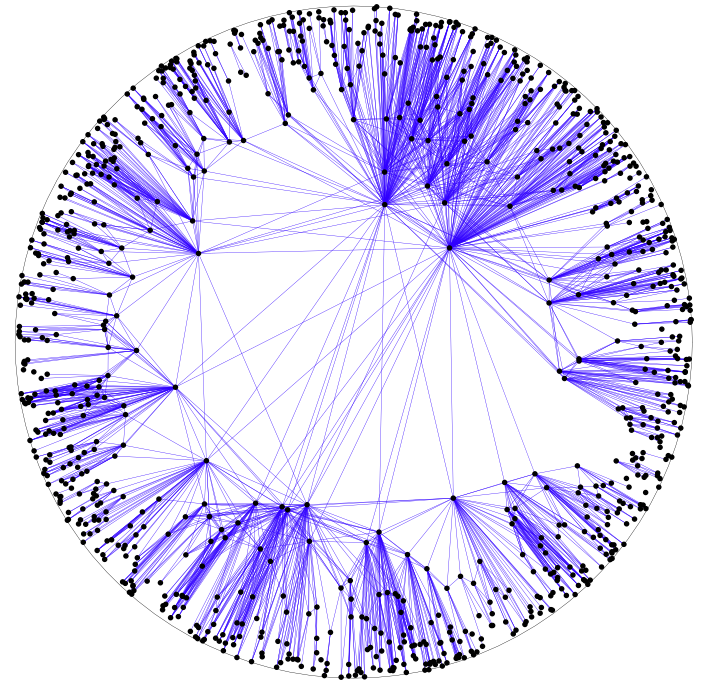}
\caption{Simulation $G(n;\alpha, \nu)$ with $\alpha = 0.9$, $\nu = 0.2$ and $n = 5000$.}
\label{fig:H_graph_example}
\end{figure}

\noindent
Figure~\ref{fig:H_graph_example} shows a computer simulation of $G(n;\alpha, \nu)$.

As observed by Krioukov et al.~\cite{krioukov2010hyperbolic}, and proved rigorously by Gugelmann et al.~\cite{gugelmann2012random}, the 
degree sequence of the KPKVB model follows a power-law with exponent $2\alpha+1$.
Gugelmann et al.~\cite{gugelmann2012random} also showed that the average degree converges in probability to the 
constant $8\nu\alpha^2/ \pi (2\alpha-1)^2$, and they showed that the (local) clustering coefficient is non-vanishing 
in the sense that it is bounded below by a positive constant a.a.s. 
Here, and in the rest of the paper, for a sequence $(E_n)_n$ of events, $E_n$ {\em asymptotically almost surely} (a.a.s.) means 
that $\Prob{E_n} \to 1$ as $n \to \infty$. 

Apart from the degree sequence and clustering, the third main characteristic associated with complex networks, ``short distances'', has 
also been established in the literature. In~\cite{abdullah2017typical} it is shown that for $\alpha < 1$ the largest component 
is what is called an \emph{ultra-small world}: if we randomly sample two vertices of the graph then, a.a.s., conditional on them 
being in the same component, their graph distance is of order $\log\log n$. 
In~\cite{kiwi2015bound} and~\cite{friedrich2018diameter} a.a.s.~polylogarithmic upper and lower bounds on the graph diameter of 
the largest component are shown, and in~\cite{muller2017diameter}, these were sharpened to show that $\log n$ is the correct order of 
the diameter.

Earlier work of the first and third authors with Bode~\cite{bode2015largest} and of the first and third authors~\cite{fountoulakis2018law} has established the ``threshold for a giant component'': if $\alpha < 1$ then there is a unique component of size linear in $n$ no matter how small $\nu$ (i.e. the average degree); if $\alpha > 1$ all components are sublinear no matter the value of $\nu$; and if $\alpha=1$ then there is a critical value $\nu_{\text{c}}$ such that for $\nu < \nu_{\text{c}}$ all components are sublinear and for $\nu > \nu_{\text{c}}$ there is a unique linearly sized component (all of these statements holding a.a.s.). Whether or not there is a giant component if $\alpha=1$ and $\nu=\nu_{\text{c}}$ remains an open problem. In~\cite{kiwi2015bound} and~\cite{kiwi2017second}, Kiwi and Mitsche considered the size of the second largest component and showed that for $\alpha \in (1/2, 1)$, a.a.s., the second largest component has polylogarithmic order with exponent $1/(\alpha -1/2)$.

In another paper of the first and third authors with Bode~\cite{bode2016probability} it was shown that $\alpha=1/2$ is the threshold for connectivity: for $\alpha < 1/2$ the graph is a.a.s.~connected, for $\alpha>1/2$ the graph is a.a.s.~disconnected and when $\alpha=1/2$ the probability of being connected tends to a continuous, non-decreasing function of $\nu$ which is identically one for $\nu \geq \pi$ and strictly less than one for $\nu < \pi$. Friedrich and Krohmer~\cite{blasius2018cliques} studied the size of the largest clique as well as the number of cliques of a given size. Bogu\~{n}a et al.~\cite{boguna2010sustaining} and Bl\"asius et al.~\cite{blasius2018efficient} considered fitting the KPKVB model to data using maximum likelihood estimation. Kiwi and Mitsche~\cite{kiwi2018spectral} studied the spectral gap and related properties, and Bl\"asius et al.\cite{blasius2016hyperbolic} considered the tree-width and related parameters of the KPKVB model. Recently Owada and Yogeshwaran~\cite{owada2018sub} considered subgraph counts, and in particular established a central limit theorem for the number of copies of a fixed tree $T$ in $G(n;\alpha,\nu)$, subject to some restrictions on the parameter $\alpha$.

\subsubsection*{Clustering}

In this work we study the clustering coefficient in the KPKVB model. In the literature there are unfortunately two distinct, rival definitions of the {\em clustering coefficient}. One of those, sometimes called the {\em global} clustering coefficient, is defined as three times the ratio of the number of triangles to the number of paths of length two in the graph. Results for this version of the clustering coefficient in the KPKVB model were obtained by Candellero and the first author~\cite{candellero2016clustering} and for the evolution of graphs on more general spaces with negative curvature by the first author in~\cite{fountoulakis2012evolution}. 

We will study the other notion of clustering, the one which is also considered by Krioukov et al.~\cite{krioukov2010hyperbolic} and Gugelmann et al.~\cite{gugelmann2012random}. It is sometimes called the {\em local} clustering coefficient, although we should point out that Gugelmann et al.~actually call it the global clustering coefficient in their paper. For a graph $G$ and a vertex $v\in V(G)$ we define the clustering coefficient {\em of $v$} as:
\[
	c(v) := \left\{\begin{array}{cl}
		\displaystyle \frac{1}{\binom{\text{deg}(v)}{2}} \sum_{u,w\sim v} 1_{\{uw \in E(G)\}}, 
			& \text{ if $\text{deg}(v) \geq 2$, }\\
		& \\
        0, & \text{ otherwise,}
        \end{array}\right.
\]
where $E(G)$ denotes the edge set of $G$ and $\text{deg}(v)$ is the degree of vertex $v$. That is, provided $v$ has degree at least two, $c(v)$ equals the number of edges that are actually present between the neighbours of $v$ divided by the number of edges that could possibly be present between the neighbours given the degree of $v$.
The clustering coefficient of $G$ is now defined as the average of $c(v)$ over all vertices $v$:
\[
	c(G) := \frac{1}{|V(G)|} \sum_{v\in V(G)} c(v).
\]

As mentioned above, Gugelmann et al.~\cite{gugelmann2012random}, have established that $c(G(n;\alpha,\nu))$ is non-vanishing a.a.s., but they left open the question of convergence. Theorem~\ref{thm:maincc} below establishes that the clustering coefficient indeed converges in probability to a constant $\gamma$ that we give explicitly as a closed form expression involving $\alpha,\nu$ and several classical special functions.

In addition to the clustering coefficient, we shall also be interested in the {\em clustering function}.
This assigns to each non-negative integer $k$ the value
\begin{equation}\label{eq:def_clustering_function}
	c(k; G) := \begin{cases}
		\displaystyle \frac{1}{N(k)} \sum_{\substack{v \in V(G), \\ \text{deg}(v)=k}}  c(v),  &\mbox{ if } N(k) \ge 1,\\
		0, &\mbox{else,}
	\end{cases}
\end{equation}
where $N(k)$ denotes the number of vertices of degree exactly $k$ in $G$. In other words, the clustering function assigns to the integer $k$ the average of the local clustering coefficient over all vertices of degree $k$. We remark that, while it might seem natural to consider $c(k)$ to be ``undefined'' when $N(k)=0$, we prefer to use the above definition for technical 
convenience.  This way $c(k; G(n;\alpha,\nu) )$ is a plain vanilla random variable and we can for instance compute its moments without any issues.

Krioukov et al.~state (\cite{krioukov2010hyperbolic}, last sentence on page 036106-10) that as $k$ tends to infinity, the clustering function decays as $k^{-1}$. This seems to be based on computations that were not included in the paper.
Despite the attention the KPKVB model has generated since then, the behaviour of the clustering function in KPKVB random graphs has not been rigorously determined yet. In particular it has not been established whether it converges as $n\to\infty$ to some suitable limit function.  
Theorems~\ref{thm:mainkfixed} and~\ref{thm:mainktoinfty} below settle this question. Theorem~\ref{thm:mainkfixed} shows that for each fixed $k$, the value $c(k;G(n;\alpha,\nu))$ converges in probability to a constant $\gamma(k)$ that we again give explicitly as a closed form expression involving $\alpha,\nu$ and several classical special functions. Theorem~\ref{thm:mainktoinfty} extends this result to growing sequences satisfying $k \ll n^{1/(2\alpha+1)}$. Proposition~\ref{prop:asymp} clarifies the asymptotic behavior of the limiting function $\gamma(k)$, as $k\to\infty$. This depends on the parameter $\alpha$, and $\gamma(k)$ only scales with $k^{-1}$ when $\alpha > 3/4$, which corresponds to the exponent of the degree distribution exceeding $5/2$. 
So in particular our findings contradict the above-mentioned claim of Krioukov et al.~\cite{krioukov2010hyperbolic}.

\subsubsection*{Notation}

In the statement of our main results, and throughout the rest of the paper, we will use the following notations. 
We set 
$$\xi := \frac{4\alpha\nu}{\pi(2\alpha-1)}. $$

We write $\Gamma(z) := \int_0^\infty t^{z-1} e^{-t}\text{d}t$ for the gamma function, 
$\Gamma^+(a,b) := \int_b^\infty t^{a-1} e^{-t}\text{d}t$ for the upper incomplete gamma function, 
 $B(a,b) := \int_0^1 u^{a-1}(1-u)^{b-1}\text{d}u = \Gamma(a)\Gamma(b) / \Gamma(a+b)$ for the beta function and 
 $B^-(x ; a,b) := \int_0^x u^{a-1}(1-u)^{b-1}\text{d}u$ for the lower incomplete beta function. 
We write $U(a,b,z)$ for the hypergeometric U-function (also called Tricomi's confluent hypergeometric function), which 
has the integral representation 
\[
	U(a,b,z) = \frac{1}{\Gamma(a)} \int_0^\infty e^{-zt} t^{a-1} (1+t)^{b-a-1} dt,
\] 
see~\cite[p.255 Equation (2)]{erdelyi1953higher}, and let $\MeijerGnew{m}{\ell}{p}{q}{{\bf a}}{{\bf b}}{z}$ denote 
Meijer's G-Function~\cite{meijer1946gfunction}, see Appendix~\ref{sec:Meijer_G_functions} for more details.

For a sequence $(X_n)_n$ of random variables, we write $X_n \xrightarrow[n\to\infty]{\Pee} X$ to denote that $X_n$ converges in 
probability to $X$. 



\subsection{Main results}\label{ssec:main_results}

\subsubsection{The clustering coefficient}

Our first main result shows the convergence of the local clustering coefficient in the KPKVB model and 
establishes the limit exactly.

\begin{theorem}\label{thm:clustering_coefficient_hyperbolic}\label{thm:maincc}
Let $\alpha > \frac{1}{2}$, $\nu > 0$ be fixed. Writing $G_n := G(n;\alpha,\nu)$, we have
\[
	c( G_n ) \xrightarrow[n\to\infty]{\Pee} \gamma,
\]
where $\gamma$ is defined for $\alpha \ne 1$ as
\begin{align*}
	\gamma 
	&=\frac{2 + 4 \alpha + 13 \alpha^2 - 34 \alpha^3 - 12\alpha^4 + 24 \alpha^5}{16(\alpha-1)^2 \alpha (\alpha+1) (2\alpha+1)} 
		+  \frac{2^{-1 - 4 \alpha}}{(\alpha - 1)^2} \\
&\hspace{10pt}+ \frac{(\alpha - 1/2) (B(2 \alpha, 2 \alpha + 1) + B^-(1/2; 1 + 2 \alpha, -2 + 2 \alpha))}{2 (\alpha - 1) (3 \alpha - 1)} \\
&\hspace{10pt}+ \frac{\xi^{2\alpha} \left( \Gamma^+( 1 - 2 \alpha, \xi) + \Gamma^+( - 2 \alpha, \xi)\right) }{4(\alpha-1)} \\
&\hspace{10pt}+ \frac{\xi^{2\alpha + 2}\alpha (\alpha - 1/2)^2 \left( \Gamma^+(- 2 \alpha - 1, \xi) + \Gamma^+(- 2 \alpha - 2, \xi)\right)}%
{2(\alpha-1)^2} \\
&\hspace{10pt}- \frac{\xi^{2\alpha + 1}\alpha (2\alpha - 1) \left( \Gamma^+( - 2\alpha,\xi)+\Gamma^+( - 2 \alpha - 1,\xi) \right)}%
{(\alpha-1)} \\
&\hspace{10pt}- \frac{\xi^{6\alpha-2}2^{-4\alpha}(3\alpha - 1)
\left( \Gamma^+( - 6 \alpha + 3, \xi)+\Gamma^+( - 6 \alpha + 2, \xi) \right)}{(\alpha-1)^2}\\
&\hspace{10pt}- \frac{\xi^{6\alpha - 2}(\alpha - 1/2) B^-(1/2; 1 + 2 \alpha, -2 + 2 \alpha)%
\left(\Gamma^+( - 6 \alpha + 3, \xi)+\Gamma^+( - 6 \alpha + 2, \xi)\right)}{(\alpha-1)} \\
&\hspace{10pt}- \frac{e^{-\xi} \Gamma(2\alpha+1) \left(U(2\alpha+1,1-2\alpha,\xi) + U(2\alpha+1,2-2\alpha,\xi)\right)}{4(\alpha-1)} \\
&\hspace{10pt}+ \frac{\xi^{6\alpha - 2} \Gamma(2\alpha+1)\left( \MeijerGnew{3}{0}{2}{3}{1,3-2\alpha}{3-4\alpha,-6\alpha+2,0}{\xi}
 		+ \MeijerGnew{3}{0}{2}{3}{1,3-2\alpha}{3-4\alpha,-6\alpha+3,0}{\xi}\right)}{4(\alpha-1)},
\end{align*}
and for $\alpha = 1$ as the $\alpha\to 1$ limit of the above expression. 
\end{theorem}

\noindent
A plot of $\gamma$ can be found in Figure~\ref{fig:gamma}. The figure also shows the results of computer simulations that appear to be in agreement with our findings.

\begin{figure}[!ht]
    \centering
    \includegraphics[scale=0.6]{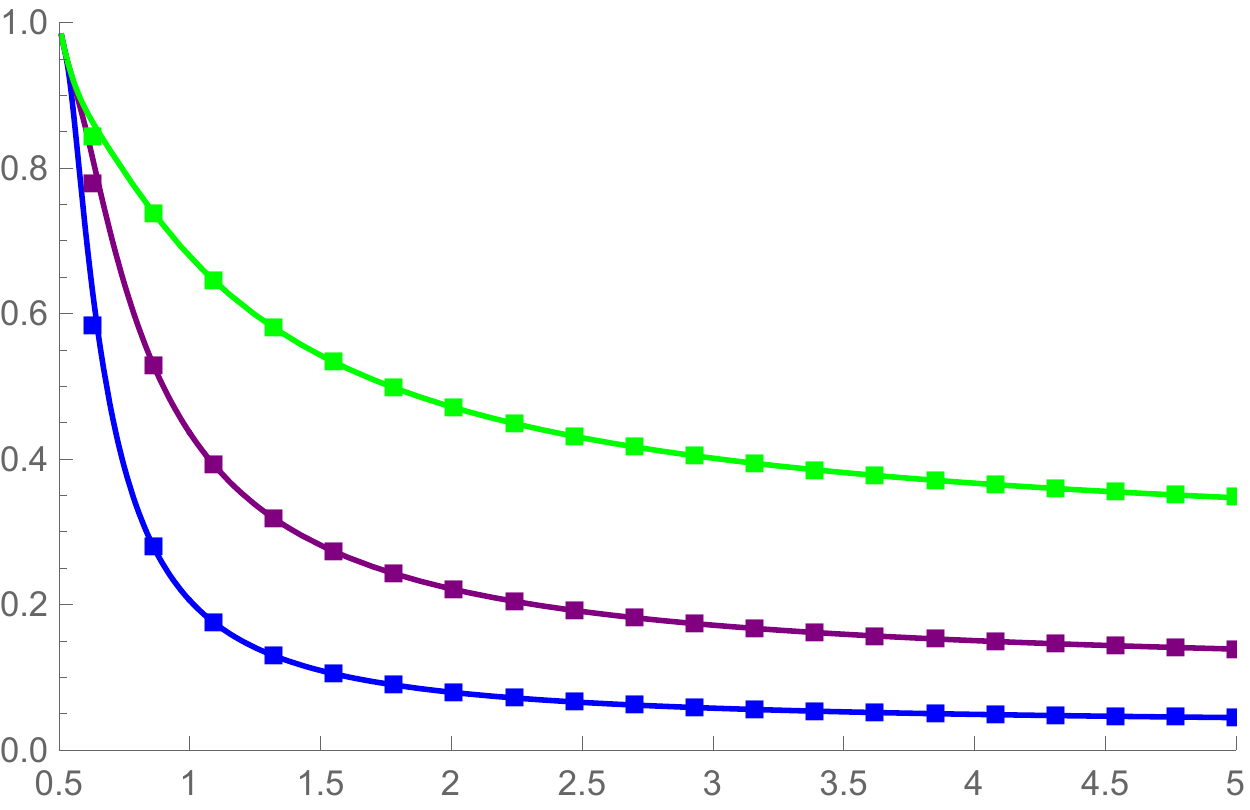}
    \caption{Plot of $\gamma$ for $\alpha$ varying from $0.5$ to $5$ on the horizontal axis and 
    for $\nu=\frac{1}{2}$ (blue), $\nu=1$ (purple), $\nu=2$ (green). Simulations (squares in corresponding colour) with $n=10000$ and $100$ repetitions.\label{fig:gamma}}
\end{figure}

In the above expression for $\gamma$, a factor $\alpha-1$ occurs in the denominator of each term, but we will see that this
corresponds to a removable singularity. We have not been able to find a closed form expression in terms of standard functions in the case 
when $\alpha=1$, but in Section~\ref{ssec:alphais1} we do provide an explicit expression involving integrals.

\subsubsection{The clustering function}

Our second main result is on the clustering function for constant $k$.

\begin{theorem}\label{thm:local_clustering_hyperbolic}\label{thm:mainkfixed}
Let $\alpha > \frac{1}{2}$, $\nu > 0$ and $k\geq2$ be fixed. 
Writing $G_n := G(n;\alpha,\nu)$, we have

\[
	c(k;G_n) \xrightarrow[n\to\infty]{\Pee} \gamma(k),
\]
where $\gamma(k)$ is defined for $\alpha \ne 1$ as 
\begin{align*}
\gamma(k)  =&\frac{1}{8\alpha (\alpha-1)\Gamma^+(k-2\alpha,\xi)} \left( -\Gamma^+(k - 2 \alpha, \xi) - 2\frac{\alpha (\alpha - 1/2)^2 \xi^{2} \Gamma^+(k - 2 \alpha - 2, \xi)}{(\alpha - 1)} \right. \\ 
&\left.+ 8 \alpha (\alpha - 1/2) \xi \Gamma^+(k - 2 \alpha - 1,\xi) \right.\\ 
&\left.+ 4\xi^{4\alpha - 2} \Gamma^+(k - 6 \alpha + 2, 
      \xi) \left( \frac{2^{ - 4\alpha}(3 \alpha - 1)}{(\alpha - 1)} + (\alpha - 1/2) B^-(1/2; 1 + 2 \alpha, -2 + 2 \alpha) \right)  \right.\\ 
&\left.+ \xi^{k-2\alpha} \Gamma(2\alpha+1)e^{-\xi} U(2\alpha+1,1+k-2\alpha,\xi) \right. \\ 
&\left.- \xi^{4\alpha-2} \Gamma(2\alpha+1)\MeijerGnew{3}{0}{2}{3}{1,3-2\alpha}{3-4\alpha,-6\alpha+k+2,0}{\xi}  \right)
\end{align*}
and for $\alpha = 1$ as the $\alpha\to1$ limit of the above expression.
\end{theorem}

\noindent
A plot of $\gamma(k)$,  together with the results of computer experiments, can be found in Figure~\ref{fig:gammak}. %
\begin{figure}[ht]
    \centering
    \includegraphics[scale=0.6]{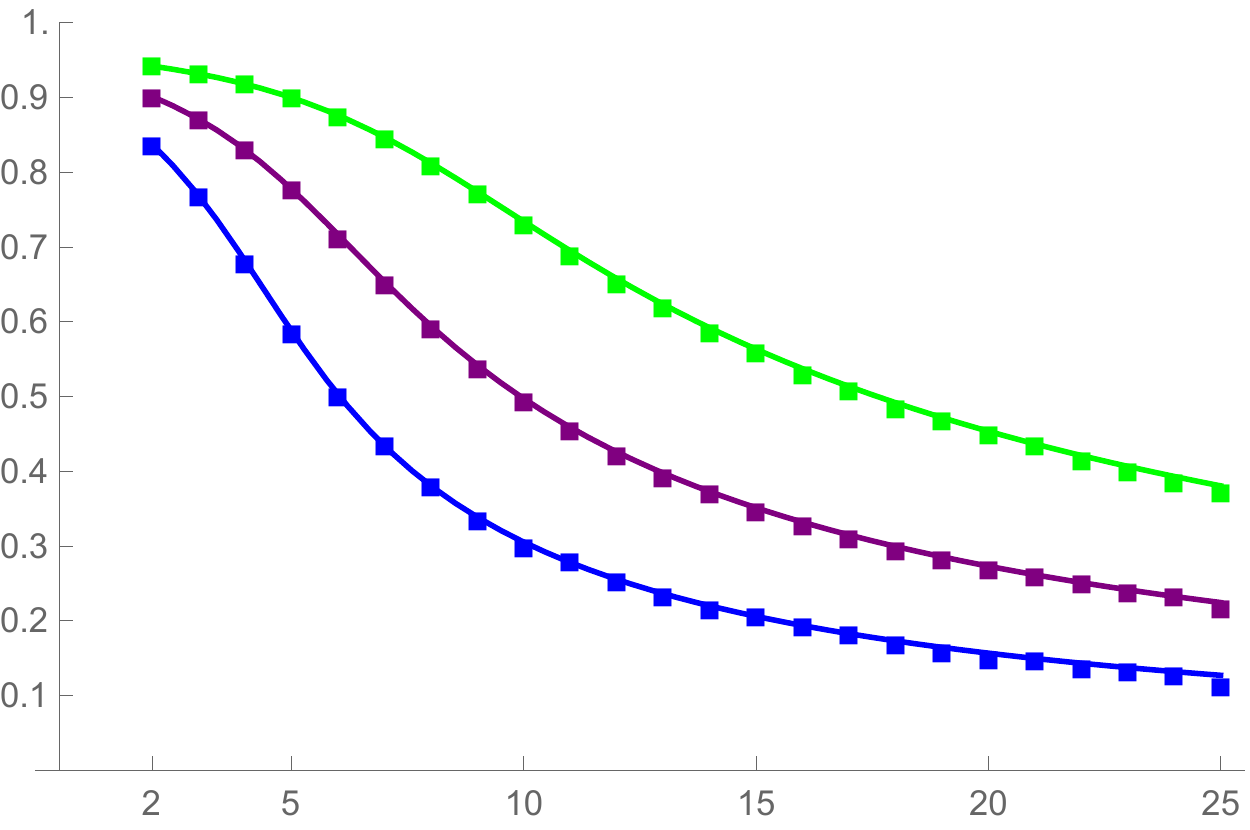}
    \caption{Plot $\gamma(k)$ for $k$ varying from 2 to 25 on the horizontal axis, for $\alpha=0.8$ and $\nu=\frac{1}{2}$ 
    (blue), $\nu=1$ (purple), $\nu=2$ (green). Simulations (squares in corresponding colour) with $n=10000$ and $100$ repetitions.\label{fig:gammak}}
\end{figure}%
Again, we remark that the above expression for $\gamma(k)$ appears to have a singularity at $\alpha=1$, but this will turn out to be a removable singularity. 
Again, we have not been able to find a closed form expression in terms of standard functions in the case when $\alpha=1$, but in 
Section~\ref{ssec:alphais1} we do provide an explicit expression involving integrals.

Theorem~\ref{thm:mainkfixed} in fact generalises to increasing sequences $(k_n)_{n \ge 1}$.

\begin{theorem}\label{thm:mainktoinfty}
Let $\alpha>\frac12, \nu>0$ be fixed and let $k_n$ be a sequence of non-negative integers
satisfying $1 \ll k_n \ll n^{1/(2\alpha+1)}$. Then, writing $G_n := G(n;\alpha,\nu)$, we have

$$ \frac{c(k_n;G_n)}{\gamma(k_n)} \xrightarrow[n\to\infty]{\Pee} 1. $$ 

\end{theorem}

\noindent
The statement of Theorem~\ref{thm:mainktoinfty} is equivalent to $c(k_n;G_n) = (1+o(1)) \gamma(k_n)$ a.a.s., using notation that is common in the random graphs community. 

\subsubsection{Scaling of \texorpdfstring{$\gamma(k)$}{gamma(k)}}

To clarify the scaling behaviour of $\gamma(k)$ we offer the following result.

\begin{proposition}\label{prop:asymp}
As $k\to\infty$, we have

$$ \gamma(k) = 
\left\{ \begin{array}{ll}
(c_{\alpha,\nu}+o(1)) \cdot k^{-1} &\text{ if } \alpha > \frac{3}{4}, \\
(c_{\alpha,\nu}+o(1)) \cdot \frac{\log(k)}{k}& \text{ if } \alpha = \frac{3}{4},\\
(c_{\alpha,\nu}+o(1)) \cdot k^{2-4\alpha} & \text{ if } \frac12 < \alpha < \frac34, 
\end{array} \right.,
$$
where 

$$ c_{\alpha,\nu} := 
\left\{ \begin{array}{cl}
8\alpha \nu / (\pi\left(4\alpha - 3\right)) & \text{ if } \alpha > \frac{3}{4}, \\
6 \nu / \pi & \text{ if } \alpha = \frac{3}{4},\\
 \left( \frac{3\alpha - 1}{2^{4\alpha+1}\alpha(\alpha-1)^2} 
	+ \frac{(\alpha - \frac{1}{2})B^-(\frac{1}{2},2\alpha + 1, 2\alpha - 2)}{2(\alpha - 1)\alpha} 
	- \frac{B(2\alpha, 3\alpha - 4)}{4(\alpha - 1)} \right)  \cdot \xi^{4\alpha-2} 
	& \text{ if } \frac12 < \alpha < \frac34.
\end{array} \right. $$
\end{proposition}

Theorem~\ref{thm:mainktoinfty} states that the clustering function of the KPKVB model scales as $\gamma(k)$ as the number of vertices $n\to\infty$, and
Proposition~\ref{prop:asymp} makes clear how $\gamma(k)$ behaves as $k$ grows.
In particular, these results contradict the scaling claimed in~\cite{krioukov2010hyperbolic} for 
$\alpha \leq \frac{3}{4}$, and confirms it only for $\alpha > \frac{3}{4}$.

We remark that simultaneously and independently Stegehuis, van der Hofstad and van Leeuwaarden~\cite{stegehuis2019scale} used a completely different technique to obtain a similar, though less detailed, result on the $k\to\infty$ scaling of the clustering function in the KPKVB model.

\subsection{Additional observations and results}

There are a few additional remarks we would like to make regarding our results.

\subsubsection{The degree distribution and the range of \texorpdfstring{$k_n$}{kn} in Theorem~\ref{thm:mainktoinfty}}

The reader may already have observed that, with a power law exponent of $2\alpha+1$ for the probability mass function of 
the degree sequence, we would expect $\Theta( n \cdot k^{-(1+2\alpha)} ) = o(1)$ vertices of degree {\em exactly} $k$ whenever
$k \gg n^{1/(1+2\alpha)}$.

This is the reason why in Theorem~\ref{thm:mainktoinfty} we restrict ourselves to 
sequences $k_n$ with $k_n \ll n^{1/(1+2\alpha)}$. 
When  $k_n \gg n^{1/(1+2\alpha)}$ there are no vertices of degree exactly $k_n$ a.a.s., which in particular
implies that the clustering function equals zero a.a.s.~for any such sequence $k_n$.

As mentioned previously, Gugelmann et al.~\cite{gugelmann2012random} gave a mathematically rigorous result on the degree sequence, which can of course be rephrased as a result on the number of nodes with degree exactly $k$. 
Their results allow $k = k_n$ to grow with $n$, but unfortunately require that $k_n \leq n^\delta$ with 
$\delta < \min\left\{\frac{2\alpha-1}{4\alpha(2\alpha +1)}, \frac{2(2\alpha -1)}{5(2\alpha + 1)}\right\} < \frac{1}{2\alpha + 1}$.
For completeness we offer the following result, which extends that of Gugelmann et al.~to the full range $1 \le k_n \le n-1$. 

\begin{theorem}\label{thm:degrees_hyperbolic}
Let $\alpha >\frac{1}{2}, \nu > 0$, denote by $N_n(k)$ the number of vertices with degree $k$ in the KPKVB model $G(n;\alpha,\nu)$ and 
consider a sequence of integers $(k_n)_n$ with $0 \leq k_n \leq n-1$.
\begin{enumerate}
\item If $k_n \ll n^{\frac{1}{2\alpha+1}} $ as $n\rightarrow \infty$, then a.a.s.

$$N_n(k_n) = (1+o(1)) \cdot n \cdot \pmf(k_n),$$

where $\pmf(k_n) = 2\alpha \xi^{2\alpha}\Gamma^+(k_n-2\alpha,\xi)/k_n!$.

\item If $k_n = (1+o(1))c n^{\frac{1}{2\alpha+1}}$ for some fixed $c >0$, then

$$ N_n(k_n) \xrightarrow[n\rightarrow\infty]{d} \operatorname{Po}(2\alpha \xi^{2\alpha} c^{-(2\alpha+1)}). $$

\item If $k_n \gg n^{\frac{1}{2\alpha+1}}$, then a.a.s. $N_n(k_n) = 0$.
\end{enumerate}
\end{theorem}

\subsubsection{Transition in scaling at \texorpdfstring{$\alpha = 3/4$}{alpha equals 3/4}.}

Proposition~\ref{prop:asymp} demonstrates that there is a transition in the scaling of the local clustering 
function at $\alpha = 3/4$. 
This corresponds to an exponent $5/2$ for the probability mass function of the degree distribution. 
This transition is different from those often observed for networks with scale-free degree distributions, where transitions occur at
integer values of the exponent. 
At this point, it is unclear what the underlying reason is for the appearance of the transition at this particular half integer 
exponent. 
Interestingly, a similar transition point has also been observed for both majority vote models~\cite{chen2015critical} and 
flocking dynamics~\cite{miguel2018effects} on networks with scale-free degree degree distributions.



\subsection{Outline of the paper}

In the next section we will recall some useful tools from the literature and define a series of auxiliary random graph models 
that will be used in the proofs. In particular, we relate in a series of steps the KPKVB model to an infinite percolation model 
$\Ginf$ that was used in previous work of the first and third authors~\cite{fountoulakis2018law} on the largest component 
of the KPKVB model. 
The value of the limiting constant $\gamma$, respectively limiting clustering function $\gamma(k)$, correspond to the probability 
that two randomly chosen neighbours of a ``typical point'' in this infinite model are themselves neighbours, respectively the 
probability of this event conditional on the typical point having exactly $k$ neighbours. These probabilities can be expressed as 
certain integrals, which we solve explicitly in Section~\ref{sec:Ginf}. 
In the same section we also prove Proposition~\ref{prop:asymp}, on the asymptotics of $\gamma(k)$. 
We then proceed to prove Theorems~\ref{thm:maincc} and~\ref{thm:mainkfixed} by relating said probabilities for the typical point 
of the infinite model to the corresponding clustering coefficient/function in the original KPKVB random graph, using the Campbell-Mecke 
formula and some other, relatively straightforward considerations.

In Section~\ref{sec:degrees} we prove Theorem~\ref{thm:degrees_hyperbolic}, which also doubles as a warm-up for the proof 
of Theorem~\ref{thm:mainktoinfty} of the clustering function for growing $k$.
The remaining sections are devoted to the proof of Theorem~\ref{thm:mainktoinfty}, which turns out to be technically involved.
The main reason for this is that when we push $k_n$ close to the maximum possible value, a great deal of work is needed 
to properly control the arising error terms.

The Appendix includes some auxiliary results on Meijer's G-function, Chernoff bounds for Poisson and Binomial 
random variables and the code used for simulations.

\section{Preliminaries\label{sec:proof_outline}}

In this section we recall some definitions and tools that we will need in our proofs.

\subsection{The infinite limit model \texorpdfstring{$\Ginf$}{G infinity}\label{ssec:infinite_model}}

We start by recalling the definition of the infinite limit model from~\cite{fountoulakis2018law}.
Let $\mathcal{P}=\mathcal{P}_{\alpha,\nu}$ be a Poisson point process on $\eR^2$ with intensity function $f=f_{\alpha,\nu}$ given by

\begin{equation}\label{eq:def_intensity_function_f}
	f(x,y) = \frac{\alpha \nu}{\pi} e^{-\alpha y} \cdot \1_{\{y>0\}}.
\end{equation} 

The \emph{infinite limit model} $\Ginf = \Ginf(\alpha,\nu)$ has vertex set $\mathcal{P}$ and edge set such that
\[
	pp' \in E(\Ginf) \iff |x - x'| \leq e^{\frac{y + y'}{2}},
\]
for $p=(x,y),p'=(x',y')\in\Pcal$.

For any point $p \in \R \times (0,\infty)$, we write $\BallPo{p}$ to denote the \emph{ball} around $p$, i.e.

\begin{equation}\label{eq:def_ball_P}
	\BallPo{p} = \{p^\prime \in \eR\times(0,\infty) : |x - x^\prime| \leq e^{\frac{y + y^\prime}{2}}\}.
\end{equation}

With this notation we then have that $\BallPo{p} \cap \Pcal$ denotes the set of neighbours of a vertex $p \in \Ginf$.
We will denote the intensity measure of the Poisson process $\mathcal{P}$ by $\mu = \mu_{\alpha, \nu}$, i.e. for every 
Borel-measurable subset $S \subseteq \eR^2$ we have $\mu(S) = \int_S f(x,y) \dd x \dd y$. Using the notation $p = (x,y)$ for a point in $\R \times \R_+$ we shall write $\int_S h(p) \dd \mu(p)$ for the integral of $h$ over $S$ with respect to the intensity measure $\mu$, i.e. $\int_S h(p) \dd \mu(p) = \int_{S} h(x,y) f(x,y) \dd x \dd y$.

\subsection{The finite box model \texorpdfstring{$\Gbox$}{G box}\label{ssec:finite_model}}

Recall that in the definition of the KPKVB model we set $R = 2\log(n/\nu)$.
We consider the box $\Rcal = (-\frac{\pi}{2}e^{R/2}, \frac{\pi}{2}e^{R/2}] \times (0, R]$ in $\eR^2$. 
Then the \emph{finite box model} $\Gbox = \Gbox(n;\alpha, \nu)$ has vertex set $\Vbox := \mathcal{P} \cap \Rcal$ and edge set such that
\[
	pp' \in E(\Gbox) \iff |x - x'|_{\pi e^{R/2}} \leq e^{\frac{y + y'}{2}},
\]
where $|x|_{r} = \min( |x|, r - |x|)$ for $-r\leq x\leq r$. Using $|.|_{\pi e^{R/2}}$ instead of $|.|$ results in the left and right boundaries of the box $\Rcal$ getting identified, which in particular makes the model invariant under horizontal shifts and reflections in vertical lines. 
The graph $\Gbox$ can thus be seen as a subgraph of $\Ginf$ induced on $\Vbox$, with some additional edges caused by the identification of the boundaries.

Similar to the infinite graph, for a point $p \in \Rcal$ we define the ball $\BallPon{p}$ as
\begin{equation}\label{eq:def_ball_P_n}
	\BallPon{p} = 
	\left\{p' \in \Rcal : |x - x'|_{\pi e^{R/2}} \leq e^{\frac{y + y'}{2}}\right\}.
\end{equation}

\subsection{The Poissonized KPKVB model \texorpdfstring{$\GPo$}{G Po}}

Imagine that we have an infinite supply of i.i.d.~points $u_1, u_2, \dots$ in the hyperbolic plane $\Haa$ chosen according to a distribution we'll define shortly, in~\eqref{eq:def_quasi_uniform_density} below. In the standard KPKVB random graph $G(n;\alpha,\nu)$ we take $u_1,\dots, u_n$ as our vertex set and add edges between points at hyperbolic distance at most $R = 2\log(n/\nu)$. In the \emph{Poissonized} KPKVB random graph $\GPo := \GPo(n;\alpha,\nu)$, we instead take $N\isd \Po(n)$, a Poisson random variable 
with mean $n$, independent of our i.i.d.~sequence of points and let the vertex set be $u_1,\dots, u_N$ and add edges 
according to the same rule as before. Equivalently, we could say that the vertex set consists of the points of a Poisson point process with intensity function $n g$, where $g$ denotes the probability density of the $(\alpha,R)$-quasi uniform distribution. That is,
\begin{equation}\label{eq:def_quasi_uniform_density}
	g(r,\theta) = \frac{\alpha\sinh(\alpha r)}{2\pi(\cosh(\alpha R) - 1)} \cdot 1_{\{0\leq r\leq R, -\pi<\theta\leq \pi\}}.
\end{equation}

Working with the Poissonized model has the advantage that when we take two disjoint regions $A, B$ then the number of points in $A$ and the number of points  in $B$ are independent Poisson-distributed random variables. As we will see, and as is to be expected, switching to the Poissonized model does not significantly alter the limiting behaviour of the clustering coefficient and function.

\subsection{Coupling \texorpdfstring{$\GPo$}{G Po} and \texorpdfstring{$\Gbox$}{G box}\label{ssec:coupling_H_P}}

The following lemmas from \cite{fountoulakis2018law} establish a useful coupling between the Poissonized KPKVB random graph
and the finite box model and relate the edge  sets of the two graphs. 

\begin{lemma}[{\cite[Lemma 27]{fountoulakis2018law}}]\label{lem:coupling_hyperbolic_poisson}
Let $\VPo$ denote the vertex set of $\GPo(n;\alpha, \nu)$ and $\Vbox$ the vertex set of $\Gbox(n;\alpha, \nu)$. 
Define the map $\Psi : [0,R] \times (-\pi, \pi] \to \Rcal$ by
\begin{equation}\label{eq:def_Psi}
	\Psi(r,\theta) = \left(\theta \frac{e^{R/2}}{2}, R - r\right).
\end{equation}
Then there exists a coupling such that, a.a.s., $\Vbox = \Psi[\VPo]$. 
\end{lemma}

In the remainder of this paper we will write $\BallHyp{p}$ to denote the image under $\Psi$ of the ball of hyperbolic radius $R$ around the point 
$\Psi^{-1}(p)$ for $p \in \Rcal$, i.e. 
\[
	\BallHyp{p} := 
	\Psi\left[ \left\{ u \in \Haa : 
	d_\H(\Psi^{-1}(p),u), d_\H(O,u) \le R \right\}\right] \subset \Rcal.
\]

Under the map $\Psi$, a point $p = (x,y) \in \Rcal$ corresponds to $u := \Psi^{-1}(p) = (2 e^{-R/2} x, R - y)$. 

By the hyperbolic rule of cosines, for two points $p = (x,y) = \Psi( (r,\theta) ), p' = (x',y') = \Psi( (r',\theta') ) \in \Rcal$ we have that $p' \in \BallHyp{p}$ iff.~either $r+r'\leq R$ or $r+r'>R$ and
\[
	\cosh r \cosh r' - \sinh r \sinh r'\cos\left( |\theta-\theta'|_{2\pi} \right) \le \cosh(R),
\]
This can be rephrased as $p'\in \BallHyp{p}$ iff.~either $y+y'\geq R$ or $y+y'<R$ and

\begin{equation}\label{eq:def_Omega_hyperbolic}
	|x-x'|_{\pi e^{R/2}} \leq \Phi(y,y^\prime) := \frac{1}{2}e^{R/2} \arccos\left( \frac{\cosh(R-y) \cosh(R-y^\prime) - \cosh R}{\sinh(R-y) \sinh(R-y^\prime)} \right).
\end{equation}

The following lemma provides useful bounds on the function $\Phi(r,r^\prime)$. Note that in~\cite{fountoulakis2018law} the function $\Phi$ is written in terms of $r := R - y, r^\prime := R - y^\prime$. 

\begin{lemma}[{\cite[Lemma 28]{fountoulakis2018law}}]\label{lem:asymptotics_Omega_hyperbolic}
There exists a constant $K>0$ such that, for every $\varepsilon > 0$ and for $R$ sufficiently large, the following holds.
For every $r,r^\prime \in [\varepsilon R,R]$ with $y + y^\prime < R$ we have that 
\begin{equation}\label{eq:asymp1}
	e^{\frac{1}{2}(y+y^\prime)} - K e^{\frac{3}{2}(y+y^\prime) - R} \leq \Phi(y, y^\prime) 
	\leq  e^{\frac{1}{2}(y+y^\prime)} + K e^{\frac{3}{2}(y+y^\prime) - R},
\end{equation}
Moreover:
\begin{equation}\label{eq:asymp2} 
\Phi(y,y^\prime) \geq e^{\frac12(y+y^\prime)} \quad \text{if \quad $y, y^\prime > K$.} 
\end{equation}
\end{lemma}

A key consequence of Lemma~\ref{lem:asymptotics_Omega_hyperbolic} is that the coupling from Lemma~\ref{lem:coupling_hyperbolic_poisson} preserves edges between points whose heights are not too large.  

\begin{lemma}[{\cite[Lemma 30]{fountoulakis2018law}}]\label{lem:coupling_edges}
On the coupling space of Lemma~\ref{lem:coupling_hyperbolic_poisson} the following holds a.a.s.:
\begin{enumerate}
\item for any two points $p, p' \in \Vbox$ with $y, y'\le R/2$, we have 
\[
	pp' \in E(\Gbox) \Rightarrow \Psi^{-1}(p)\Psi^{-1}(p') \in E(\GPo),
\]
\item for any two points $p, p' \in \Vbox$ with $y, y' \le R/4$, we have that 
\[
	pp' \in E(\Gbox) \iff \Psi^{-1}(p)\Psi^{-1}(p') \in E(\GPo).
\]

\end{enumerate}
\end{lemma}

\begin{remark}[Notational convention for points]
We will often be working with the finite box graph $\GPo$ or the infinite graph $\Ginf$, whose nodes are points 
in $\R \times \R_+$. For any point $p \in \R \times \R_+$ we will always use $p = (x,y)$. 
When considering different points $p,p^\prime \in \R \times \R_+$, we will use primed coordinates to refer 
to $p^\prime$, i.e. $p^\prime = (x^\prime ,y^\prime)$, and similar with subscripts, i.e. $p_i = (x_i,y_i)$.
\end{remark}

\subsection{The Campbell-Mecke formula}

A useful tool for analyzing subgraph counts, and their generalizations, in the 
setting of Poissonized random geometric graphs, and in particular the Poissonized KPKVB model and the box 
model is the \emph{Campbell-Mecke formula}. 
We use a specific incarnation, which follows from the Palm theory of Poisson point processes on metric 
spaces, see~\cite{last2017lectures}. For this consider a Poisson point process $\Pcal$ on some metric 
space $\mathcal{M}$ with density $\mu$ and let $\mathcal{N}$ denote the set of all possible point configurations 
in $\mathcal{M}$, equipped with the sigma algebra of the process $\Pcal$. 
Then, for any natural number $k$ and measurable function $h : \R^k \times \mathcal{N} \to \R$,

\begin{equation}\label{eq:def_campbell_mecke}
\begin{array}{c}
\displaystyle	\Exp{\sum_{\substack{p_1, \dots, p_k \in \Pcal, \\ \text{distinct}}} h(p_1, \dots, p_k, \Pcal)} \\
= \\
\displaystyle \int_{\mathcal{M}} \dots \int_{\mathcal{M}} \Exp{h(x_1,\dots, x_k, \Pcal\cup\{x_1,\dots,x_k\})} \mu(dx_1)\dots\mu(dx_k).
\end{array}
\end{equation}

\subsection{Concentration of heights}

When analyzing degrees and clustering in the Poissonized KPKVB and related models we often encounter expressions of the form
\begin{equation}\label{eq:example_poisson_integral}
	\int_{0}^R \Prob{\Po(\hat{\mu}(y)) = k_n} h(y) e^{-\alpha y} \dd y,
\end{equation}
where $h(y)$ is some function and $\hat{\mu}(y)$ is $\Mu{\BallHyp{y}}$, $\Mu{\BallPon{y}}$ or $\Mu{\BallPo{y}}$. We will often have to either bound the behavior of such integrals as $k_n \to \infty$ or establish their asymptotic behavior. For this we will utilize that Poisson random variables are well concentrated around their mean. 

Let $\Po(\lambda)$ denote a Poisson random variable with mean $\lambda$. Then we have the following Chernoff bound (c.f. \cite[Lemma 1.2]{penrose2003random})
\begin{equation}\label{eq:def_chernoff_bound_poisson}
	\Prob{\left|\mathrm{Po}(\lambda) - \lambda\right| \ge x} \le 2e^{-\frac{x^2}{2(\lambda + x)}}.
\end{equation}
In particular, if $\lambda = \lambda_n \to \infty$, then for any $C>0$,
\begin{equation}\label{eq:def_chernoff_bound_poisson_C}
	\Prob{\left|\mathrm{Po}(\lambda_n) - \lambda_n\right| \ge C \sqrt{\lambda_n\log(\lambda_n)}} \le \bigO{\lambda_n^{-\frac{C^2}{2}}}.
\end{equation}

For our application these Chernoff bounds imply that if $y$ is such that $\hat{\mu}(y)$ is far from $k_n$ then $\Prob{\Po(\hat{\mu}(y)) = k_n}$ becomes very small. To be more specific, we define for any $k \ge 0$ and $C > 0$,
\begin{equation}\label{eq:def_y_k_C}
	y_{k,C}^\pm = 2 \log\left(\frac{k \pm C \sqrt{k \log(k)}}{\xi}\right),
\end{equation}
where we set $y_{k,C}^- = 0$ if $k - C \sqrt{k \log(k)} < \xi$ (recall that we define $\xi := 4\alpha\nu/\pi(2\alpha-1)$ throughout the paper) 
and likewise if $k + C \sqrt{k \log(k)} < \xi$ 
we set $y_{k,C}^+=0$, but  note that as we consider $k\rightarrow\infty$, we can assume that this case does not occur. For convenience we write $\Kcal_C(k_n) := [y_{k_n,C}^-, y_{k_n,C}^+]$. Then we can show that for all $y$ outside $\Kcal_C(k_n)$
\begin{equation}\label{eq:chernoff_bound_degrees}
	\Prob{\Po(\hat{\mu}(y)) = k_n} \le \bigO{k_n^{-\frac{C^2}{2}}}.
\end{equation}
Since we can select $C$ to be as big as we want we can make this error as small as needed. This implies that then the main contribution to the integral~\eqref{eq:example_poisson_integral} comes from those `'heights" $y$ that are in the interval $\Kcal_C(k_n)$. In other words, the main contribution is concentrated around the heights $y$ for which $\mu(y) = k_n$. We thus refer to this as the concentration of heights result. More precisely, we prove the following.

\begin{proposition}[Concentration of heights]\label{prop:concentration_height_general}
Let $\alpha > \frac{1}{2}$, $\nu > 0$, $(k_n)_{n \ge 1}$ be any positive sequence such that $k_n \to \infty$ and $k_n = \smallO{n}$. Furthermore, let $\hat{\mu}(y)$ denoting either $\Mu{\BallHyp{y}}$, $\Mu{\BallPon{y}}$ or $\Mu{\BallPo{y}}$. Then for any continuous function $h : \R_+ \rightarrow  \R$, such that $h(y) = \bigO{e^{\beta y}}$ as $y \to \infty$ for some $\beta < \alpha$, it holds that
\begin{equation*}
	\int_0^\infty \hspace{-3pt} h(y) \Prob{\Po(\hat{\mu}(y)) = k_n} \alpha e^{-\alpha y} \dd y
	\sim \int_{\Kcal_C(k_n)} \hspace{-8pt} h(y) \Prob{\Po(\hat{\mu}(y)) = k_n} \alpha e^{-\alpha y} \dd y,
\end{equation*}
as $n \to \infty$.
\end{proposition}

The key implication of Proposition~\ref{prop:concentration_height_general} is that if the function $h(y)$ does not increase too fast, then we can restrict integration to the interval $\Kcal_C(k_n)$. The full details associated with these concentration of heights and the proof of Proposition~\ref{prop:concentration_height_general} can be found in the Section~\ref{sec:concentration_argument} of the Appendix.

\section{Clustering and the degree of a typical point in \texorpdfstring{$\Ginf$}{G infinity}\label{sec:Ginf}\label{sec:asymptotics_average_clustering_ast_P}}

As alluded to earlier, we plan to make use of the Campbell-Mecke formula for comparing the clustering coefficient and function of the (Poissonized) KPKVB random graph 
with certain quantities associated with $\Ginf$.
We will be considering the Poisson process $\Pcal$ to which we add one additional point $(0,y)$ on the $y$-axis.
In some computations the height $y$ will be fixed, but eventually we shall take it exponentially distributed with parameter $\alpha$, and
independent of $\Pcal$.
We refer to $(0,y)$ as ``the typical point''.

To provide some intuition for this definition and name, note that we can alternatively view $\Pcal$ as follows. 
We take a constant intensity Poisson process on $\eR$ corresponding to the $x$-coordinates, and to each point
we attach a random ``mark'', corresponding to the $y$-coordinate, where the marks are i.i.d.~exponentially distributed with parameter $\alpha$.

Since $c(G)$ is defined as an average over all vertices of the graph, it is not immediately obvious how to meaningfully define 
a corresponding notion for infinite graphs, and similarly for the clustering function, the degree sequence, etc.
We can however without any issues speak of the (expected) clustering coefficient of the typical point, or the expected clustering
coefficient given that it has degree k, or the distribution of the degree of the typical point.
(All considered in the graph obtained from $\Ginf$ by adding the typical point to its vertex set.)

If $p = (x,y) \in \eR\times[0,\infty)$ is a point, not necessarily part of the Poisson process, then we will write

$$ \mu(y) = \mu(p) := \mu(\BallPo{p}). $$

Integrating the intensity function of $\Pcal$ over $\BallPo{p}$ gives

$$ \begin{array}{rcl} 
\mu(y) & = & \int_{\BallPo{p}} f(x',y') \, dx' \,d y'  
 = \int_0^\infty \int_{-e^{(y+y')/2}}^{e^{(y+y')/2}} \frac{\alpha \nu}{\pi} e^{-\alpha y'} \, dx' \, dy' \\
& = & \int_0^\infty 2 e^{(y+y')/2} \frac{\alpha \nu}{\pi} e^{-\alpha y'} \, dy' 
 = \frac{2\alpha \nu e^{y/2}}{\pi} \int_0^\infty e^{(\frac12-\alpha)y'} \, dy' \\
& = & \frac{2\alpha \nu e^{y/2}}{\pi(\alpha-\frac12)} = \xi e^{y/2}.
\end{array} $$

\subsection{The degree of the typical point\label{ssec:degrees_infinite_model}}

Before considering clustering we briefly investigate the distribution of the degree of 
the typical point. 
For $y \geq 0$ we define
\begin{equation}
	\rho(y,k) := \Prob{\Po(\mu(y)) = k},
\end{equation}

where $\Po(\lambda)$ denotes a Poisson random variable with mean $\lambda$. 

Let the random variable $D$ denote the degree of the typical point.
Since the typical point has a height that is independent of the Poisson process and exponential($\alpha$)-distributed:

$$ \pmf(k) := \Pee( D=k ) = \int_0^\infty \rho(y,k) \alpha e^{-\alpha y} \, dy. $$

\noindent
(Note that here we {\em define} $\pi(k)$ as the probability that the degree of the typical point equals $k$.)
Using the transformation of variables $z = \xi e^{\frac{y}{2}}$ (so $dy = \frac{2}{z}dz$), we compute
\begin{align*}
	\pmf(k)
    &= \frac{1}{k!} \int_0^\infty \left(\xi e^{\frac{y}{2}}\right)^k 
    	e^{-\xi e^{\frac{y}{2}}} \alpha e^{-\alpha y} \, dy\\
    &= \frac{\alpha \xi^{2\alpha}}{k!} \int_0^\infty 
    	\left(\xi e^{\frac{y}{2}}\right)^{k - 2\alpha} e^{-\xi e^{\frac{y}{2}}}
        \, dy\\
    &= \frac{2\alpha \xi^{2\alpha}}{k!} \int_{\xi}^{\infty} 
    	z^{k -2\alpha-1} e^{-z} \, dz\\
    &= \frac{2\alpha \xi^{2\alpha}\Gamma^+(k - 2\alpha, \xi)}{k!}, \numberthis \label{eq:def_pk}
\end{align*}
where we recall that $\Gamma$ denotes the gamma-function and $\Gamma^{+}$ the upper incomplete gamma-function.
Note that, unsurprisingly, this is identical to the expression Gugelmann et al.~\cite{gugelmann2012random} gave for the limiting 
degree distribution of $G(n;\alpha,\nu)$.
Using Stirling's approximation to the gamma function, we find that 

\begin{equation}\label{eq:degree_distribution_P_asymptotics}
	\pmf(k) \sim 2\alpha\xi^{2\alpha} k^{-(2\alpha + 1)}
	\quad \text{as } k \to \infty.
\end{equation}

\noindent
(In a bit more detail: we use that 
$\Gamma(a,b) = (1+o_a(1))\cdot \Gamma(a)$ if $a$ tends to infinity and $b$ remains constant, and that 
$\Gamma(a+1) = (1+o_a(1))\cdot \sqrt{2\pi a} \cdot (a/e)^{a}$ as $a$ tends to infinity by Stirling's approximation. For a proof of Stirling's approximation to the $\Gamma$ function, also for non-integer values of the argument, see for instance~\cite{DiaconisFreedman}.)

By a similar computation we have the following result, which will be useful later on. For any $\beta > 0$, as $k \to \infty$

\begin{equation}\label{eq:general_integral_rho_y_k}
	\int_0^\infty e^{-\beta y} \rho(y, k) \alpha e^{-\alpha y} \, dy
    \sim 2\alpha \xi^{2(\beta + \alpha)} k^{-2(\beta + \alpha)-1}.
\end{equation}

\subsection{The expected clustering coefficient of the typical point\label{sec:42}}

Let the random variable $C$ denote the clustering coefficient of the typical point $(0,y)$, in the graph obtained from $\Ginf$ by adding $(0,y)$. We now {\em define}
\[
	\gamma := \Exp{C}, \quad \gamma(k) := \CExp{C}{D = k}.
\]
(Where we take the expectation over both the Poisson point process $\Pcal$ and $y \isd \text{exp}(\alpha)$, independently of the Poisson process $\Pcal$.) We shall show shortly that these take on the values stated in Theorem~\ref{thm:maincc} and~\ref{thm:mainkfixed}.

For any fixed value $y_0>0$, the set of points inside $\BallPo{y_0}$ is a Poisson process with intensity $f \cdot 1_{\BallPo{y_0}}$. As $\mu(\BallPo{y_0}) = \mu(y_0) = \xi e^{y_0/2} < \infty$, this can be described alternatively by first picking $N \isd \Po( \mu(y_0) )$ and then taking $N$ i.i.d.~points in $\BallPo{y_0}$ according to the probability density $f \cdot 1_{\BallPo{y_0}} / \mu(y_0)$. (That is, the intensity function of the Poisson point process, but set to zero outside of $\BallPo{y_0}$ and re-normalized in such a way that it integrates to one.) Hence, if we condition on the event that $y$ takes on some fixed value $y_0$ and that there are exactly $k$ points of $\Pcal$ inside $\BallPo{y_0}$, then those $k$ points behave like $k$ i.i.d.~points in $\BallPo{y_0}$ chosen according to the mentioned re-normalized probability density function. This shows that, for every $k\geq 2$:
\[
	\CExp{C}{D=k, y=y_0} 
	= \frac{1}{\binom{k}{2}} \Ee\left( \sum_{1\leq i < j \leq k} \1_{\{u_i\in \BallPo{u_j}\}} \right)
	= \Exp{\1_{\{u_1\in\BallPo{u_2}\}}},
\]
where $u_1,\dots, u_k$ are i.i.d.~points in $\BallPo{y_0}$ with the above mentioned density.
Note that this does not depend on the value of $k$. For notational convenience, we will write 
\begin{equation}\label{eq:def_Py}
	P(y_0) := \Exp{\1_{\{u_1\in\BallPo{u_2}\}}},
\end{equation}
with $u_1, u_2$ as above.

We now observe that 
\[
	\gamma(k) = \CExp{C}{D=k} = \int_0^\infty \CExp{C}{D=k, y=y_0} g_k(y_0) \, d y_0,
\]
where $g_k$ denotes the density of $y$ {\em conditional on} $D=k$. That is,
\[
	g_k(y_0)  = \frac{\rho(y_0, k) \alpha e^{-\alpha y_0} }{\int_0^\infty \rho(t, k) \alpha e^{-\alpha t} \, d t}
	= \frac{1}{\pmf(k)} \cdot \rho(y_0, k) \alpha e^{-\alpha y_0},
\]
where we recall that $\rho(y,k) = \Prob{\Po(\mu(y)) = k}$ denotes the probability that a Poisson random variable with mean $\mu(y)$ is $k$. Hence, 
\begin{equation}\label{eq:gammakint}
\gamma(k) 
= \frac{1}{\pmf(k)} \cdot \int_0^\infty P(y_0) \rho(y_0,k) \alpha e^{-\alpha y_0} \, d y_0. 
\end{equation}
This also gives
\begin{equation}\label{eq:gammaint}
\begin{array}{rcl} 
\gamma 
& = & \displaystyle \Exp{C} = \sum_{k\geq 2} \CExp{ C }{ D=k }\Prob{ D=k } \\
& = & \displaystyle \int_0^\infty P(y_0) \left( \sum_{k=2}^\infty \rho(y_0,k) \right) \alpha e^{-\alpha y_0} \, d y_0 \\
& = & \displaystyle \int_0^\infty P(y_0) \left(1 - \rho(y_0,0)-\rho(y_0,1)\right) \alpha e^{-\alpha y_0} \, d y_0.
\end{array}
\end{equation}

A key step is to derive the following explicit expression for $P(y)$.

\begin{lemma}\label{lem:Paneq1}
	If $\alpha \not = 1$, then
	\begin{align*}
	 P(y) &=-\frac{1}{8 (\alpha - 1) \alpha} + \frac{(\alpha - 1/2) e^{-\frac{1}{2}y}}{\alpha - 1} - \frac{(\alpha - 1/2)^2 e^{-y}}{
		4 (\alpha - 1)^2} \\
	&+ 
	(e^{-\frac{1}{2}y})^{4\alpha -2} \left(\frac{2^{-4 \alpha-1} (3 \alpha - 1)}{\alpha (\alpha - 1)^2} 
		+ \frac{(\alpha - 1/2 ) B^-(1/2; 1 + 2 \alpha, -2 + 2 \alpha)}{2(\alpha - 1) \alpha} \right) \\
	&+ \frac{(1 - 
		e^{-\frac{1}{2}y})^{2 \alpha}}{8 (\alpha - 1) \alpha} - \frac{  
		(e^{-\frac{1}{2}y})^{4 \alpha - 2} B^-(1 - e^{-\frac{1}{2}y}; 2 \alpha, 3 - 4 \alpha)}{4 (\alpha - 1)}
	\end{align*}
\end{lemma}

We will prove this lemma in a sequence of steps.

Recall that $P(y_0)$ is the probability that $u_1 = (x_1,y_1), u_2 = (x_2,y_2)$ are neighbours in $\Ginf$, where
$u_1, u_2$ are i.i.d.~with probability density $f \cdot \1_{\BallPo{y_0}} / \mu(y_0)$.
In particular
\begin{align*}
	\Prob{y_i > t} &= \frac{\nu\alpha}{\pi \mu(y_0)} \int_t^\infty \int_{-e^{(y+y_0)/2}}^{e^{(y+y_0)/2}} e^{-\alpha y}
		\dd x \dd y
		= \frac{\nu\alpha}{\pi \mu(y_0)} \int_t^\infty 2 e^{(y+y_0)/2} \cdot e^{-\alpha y} \dd y \\
	&= \frac{2 \nu\alpha e^{y_0/2} }{\pi \xi e^{y_0/2} (\alpha-\frac{1}{2}) }  \cdot e^{(\frac{1}{2}-\alpha)t}
		= e^{(\frac{1}{2}-\alpha)t},
\end{align*}
using that $\mu(y_0) = \xi e^{y_0/2} = \left(\frac{2\alpha\nu}{\pi(\alpha-\frac12)}\right) e^{y_0/2}$.
Thus, $y_1, y_2$ are exponentially distributed with parameter $\alpha-\frac12$.
Now note that, for each $t>0$, the probability density $f \cdot 1_{\BallPo{y_0}} / \mu(y_0)$
is constant on $[-e^{(t+y_0)/2}, e^{(t+y_0)/2}] \times \{t\}$ and it is 
vanishes on $(-\infty, -e^{(t+y_0)/2}) \times \{t\} \cup (e^{(t+y_0)/2},\infty)\times\{t\}$.

Hence, given the height $y_i$ of $u_i$, the $x$-coordinate of $u_i$ is uniform in $[-e^{\frac{1}{2}(y+y_i)},e^{\frac{1}{2}(y+y_i)}]$. 
With this in mind we define $P(y_0,y_1,y_2)$ to be the probability that $y_0, (x_1,y_1), (x_2,y_2)$ 
satisfy \\ $|x_1-x_2| \leq e^{(y_1+y_2)/2}$, where $x_1$ and $x_2$ are independent uniform random variables in the intervals
$[-e^{\frac{1}{2}(y_0+y_1)},e^{\frac{1}{2}(y_0+y_1)}]$ and  $[-e^{\frac{1}{2}(y_0+y_2)},e^{\frac{1}{2}(y_0+y_2)}]$, respectively. We then have that

\begin{equation}\label{eq:delta_P}
 P(y_0) = (\alpha-1/2)^2 \int_0^\infty \int_0^\infty P(y_0, y_1, y_2) e^{-(\alpha-1/2)(y_1+y_2)} 
 \dd y_2 \dd y_1.
\end{equation}

\subsubsection{Determining \texorpdfstring{$P(y_0,y_1,y_2)$}{P(y0,y1,y2)}}

To compute the integral~\eqref{eq:delta_P} it will be convenient to use the 
change of variable $z_i = e^{-y_i/2}$, for $i= 0, 1, 2$. 
The following result will turn out to be
all we need to compute the integral~\eqref{eq:delta_P}.

\begin{lemma}\label{lem:triangle_prob_y_coordinates}
Set $z_i = e^{-y_i/2}$, $i=0,1,2$. We have
\begin{align*}
P(y_0,y_1,y_2) = \begin{cases}
	1, &\text{ if } z_0 \geq z_1+z_2, z_0 > z_1 > z_2, \\
	1-G(z_0,z_1,z_2), &\text{ if } z_0 < z_1+z_2, z_0 > z_1 > z_2, \\
	\frac{z_0}{z_1}, &\text{ if } z_1 \geq z_0+z_2, z_1 > \max(z_0,z_2), \\
	\frac{z_0}{z_1}\left(1-G(z_1,z_0,z_2)\right), &\text{ if } z_1 < z_0+z_2, z_1 > \max(z_0,z_2),
\end{cases}
\end{align*}
where 
\begin{align*}
G(a,b,c) = \frac{1}{4}
\left( b^{-1}c + bc^{-1} + a^2b^{-1}c^{-1} + 2 - 2ab^{-1}-2ac^{-1}\right).
\end{align*}
\end{lemma}


We split the proof of this lemma into a couple of smaller pieces. We begin with the following lemma.

\begin{lemma}\label{lem:ordered}
Write $z_i = e^{-y_i/2}$, $i=0,1,2$. If $y_0<y_1<y_2$ (or equivalently $z_0 > z_1 > z_2$), then
\begin{align*}
P(y_0,y_1,y_2) = \begin{cases}
1, &\text{ if } z_0 \geq z_1+z_2,  \\
1-G(z_0,z_1,z_2), &\text{ if } z_0 < z_1+z_2
\end{cases}
\end{align*}
\end{lemma}

\begin{figure}[!t]
\centering
\begin{tikzpicture}
	\pgfmathsetmacro{\u}{0} 
	\pgfmathsetmacro{\v}{0.5} 
	\pgfmathsetmacro{\vv}{1.2} 
	\pgfmathsetmacro{\vvv}{2} 
	\pgfmathsetmacro{\uuu}{-3} 
	\pgfmathsetmacro{\r}{6.5}
	\pgfmathsetmacro{\t}{4}
	
	\pgfmathsetmacro{\ubound}{exp((\vv+\vvv)/2)-exp((\v+\vvv)/2)}
	
	\pgfmathsetmacro{\uu}{3*\ubound/4}
	
	\pgfmathsetmacro{\leftintvandvvv}{-exp((\v + \vvv)/2)}
	\pgfmathsetmacro{\rightintvandvvv}{exp((\v + \vvv)/2)}
	\pgfmathsetmacro{\leftintvvandvvv}{\uu-exp((\vv + \vvv)/2)}
	\pgfmathsetmacro{\rightintvvandvvv}{\uu+exp((\vv + \vvv)/2)}
		
	\draw[line width=1pt,dashed] (-\r,0) -- (\r,0) -- (\r,\t) -- (-\r,\t) -- (-\r,0);

    \draw node[fill, circle, inner sep=0pt, minimum size=5pt] (p1) at (\u,\v) {};
    \path (p1)+(-0.2,0.2) node {$p_0$};
    \draw node[fill,blue, circle, inner sep=0pt, minimum size=5pt] (p2) at (\uu,\vv) {};
    \path (p2)+(0.2,0.2) node {\color{blue}$p_1$};	
	
	
	\pgfmathsetmacro{\rightbounduv}{\u+exp((\v)/2)}
	\draw[domain=\rightbounduv:\r,smooth,variable=\x,black,line width=1pt] plot (\x, {2*ln(\x)-\v});
    \pgfmathsetmacro{\leftbounduv}{\u-exp((\v)/2)}
    \draw[domain=\leftbounduv:-\r,smooth,variable=\x,black,line width=1pt] plot (\x, {2*ln(-\x)-\v});
    
    
    \pgfmathsetmacro{\rightbounduuvv}{\uu+exp((\vv)/2)}
    \draw[domain=\rightbounduuvv:\r,smooth,variable=\x,blue,line width=1pt] plot (\x, {2*ln(\x-\uu)-\vv});
    \pgfmathsetmacro{\leftbounduuvv}{\uu-exp((\vv)/2)}
    \draw[domain=\leftbounduuvv:-\r,smooth,variable=\x,blue,line width=1pt] plot (\x, {2*ln(\uu-\x)-\vv});

	\draw [red,dotted,line width=1pt] (\leftintvandvvv,\vvv) -- (\rightintvandvvv,\vvv);
	
	\draw [red,dotted,line width=1pt] (\leftintvvandvvv,\vvv) -- (\rightintvvandvvv,\vvv);

	\draw [red,line width=1pt] (\leftintvandvvv,\vvv) -- (\rightintvandvvv,\vvv);

    \draw node[fill,black!20!green, circle, inner sep=0pt, minimum size=5pt] (p2) at (\uuu,\vvv) {};
    \path (p2)+(0.25,-0.2) node {\color{black!20!green}$p_2$};
    
	
	\pgfmathsetmacro{\rightbounduuuvvv}{\uuu+exp((\vvv)/2)}
	\draw[domain=\rightbounduuuvvv:\r,smooth,variable=\x,black!20!green,line width=1pt] plot (\x, {2*ln(\x-\uuu)-\vvv});
    \pgfmathsetmacro{\leftbounduuuvvv}{\uuu-exp((\vvv)/2)}
    \draw[domain=\leftbounduuuvvv:-\r,smooth,variable=\x,black!20!green,line width=1pt] plot (\x, {2*ln(\uuu-\x)-\vvv});

	\draw node[fill,red, circle, inner sep=0pt, minimum size=3pt] at (\leftintvandvvv,\vvv) {};
	\path (\leftintvandvvv,\vvv)+(-0.5,-0.5) node[red] {$-e^{\frac{y_0 + y_2}{2}}$};
	\draw node[fill,red, circle, inner sep=0pt, minimum size=3pt] at (\rightintvandvvv,\vvv) {};
	\path (\rightintvandvvv,\vvv)+(0,-0.5) node[red] {$e^{\frac{y_0 + y_2}{2}}$};
	
	\draw node[fill,red, circle, inner sep=0pt, minimum size=3pt] at (\leftintvvandvvv,\vvv) {};
	\path (\leftintvvandvvv,\vvv)+(0.75,0.5) node[red] {$x_1 - e^{\frac{y_0 + y_2}{2}}$};
	\draw node[fill,red, circle, inner sep=0pt, minimum size=3pt] at (\rightintvvandvvv,\vvv) {};
	\path (\rightintvvandvvv,\vvv)+(-0.5,0.5) node[red] {$x_1 + e^{\frac{y_0 + y_2}{2}}$};

\end{tikzpicture}\\
\vspace{0.5cm}
\begin{tikzpicture}
	\pgfmathsetmacro{\u}{0} 
	\pgfmathsetmacro{\v}{0.5} 
	\pgfmathsetmacro{\vv}{1.2} 
	\pgfmathsetmacro{\vvv}{1.8} 
	\pgfmathsetmacro{\uuu}{2.5} 
	\pgfmathsetmacro{\r}{6.5}
	\pgfmathsetmacro{\t}{4}
	
	\pgfmathsetmacro{\ubound}{exp((\vv+\vvv)/2)-exp((\v+\vvv)/2)}
	
	\pgfmathsetmacro{\uu}{5*\ubound/4}
	
	\pgfmathsetmacro{\leftintvandvvv}{-exp((\v + \vvv)/2)}
	\pgfmathsetmacro{\rightintvandvvv}{exp((\v + \vvv)/2)}
	\pgfmathsetmacro{\leftintvvandvvv}{\uu-exp((\vv + \vvv)/2)}
	\pgfmathsetmacro{\rightintvvandvvv}{\uu+exp((\vv + \vvv)/2)}
		
	\draw[line width=1pt,dashed] (-\r,0) -- (\r,0) -- (\r,\t) -- (-\r,\t) -- (-\r,0);

    \draw node[fill, circle, inner sep=0pt, minimum size=5pt] (p1) at (\u,\v) {};
    \path (p1)+(-0.2,0.2) node {$p_0$};
    \draw node[fill,blue, circle, inner sep=0pt, minimum size=5pt] (p2) at (\uu,\vv) {};
    \path (p2)+(0.2,0.2) node {\color{blue}$p_1$};	
	
	
	\pgfmathsetmacro{\rightbounduv}{\u+exp((\v)/2)}
	\draw[domain=\rightbounduv:\r,smooth,variable=\x,black,line width=1pt] plot (\x, {2*ln(\x)-\v});
    \pgfmathsetmacro{\leftbounduv}{\u-exp((\v)/2)}
    \draw[domain=\leftbounduv:-\r,smooth,variable=\x,black,line width=1pt] plot (\x, {2*ln(-\x)-\v});
    
    
    \pgfmathsetmacro{\rightbounduuvv}{\uu+exp((\vv)/2)}
    \draw[domain=\rightbounduuvv:\r,smooth,variable=\x,blue,line width=1pt] plot (\x, {2*ln(\x-\uu)-\vv});
    \pgfmathsetmacro{\leftbounduuvv}{\uu-exp((\vv)/2)}
    \draw[domain=\leftbounduuvv:-\r,smooth,variable=\x,blue,line width=1pt] plot (\x, {2*ln(\uu-\x)-\vv});

	\draw [red,dotted,line width=1pt] (\leftintvandvvv,\vvv) -- (\rightintvandvvv,\vvv);
	
	\draw [red,dotted,line width=1pt] (\leftintvvandvvv,\vvv) -- (\rightintvvandvvv,\vvv);

	\draw [red,line width=1pt] (\leftintvvandvvv,\vvv) -- (\rightintvandvvv,\vvv);

    \draw node[fill,black!20!green, circle, inner sep=0pt, minimum size=5pt] (p2) at (\uuu,\vvv) {};
    \path (p2)+(-0.2,0.2) node {\color{black!20!green}$p_2$};
    
	
	\pgfmathsetmacro{\rightbounduuuvvv}{\uuu+exp((\vvv)/2)}
	\draw[domain=\rightbounduuuvvv:\r,smooth,variable=\x,black!20!green,line width=1pt] plot (\x, {2*ln(\x-\uuu)-\vvv});
    \pgfmathsetmacro{\leftbounduuuvvv}{\uuu-exp((\vvv)/2)}
    \draw[domain=\leftbounduuuvvv:-\r,smooth,variable=\x,black!20!green,line width=1pt] plot (\x, {2*ln(\uuu-\x)-\vvv});

	\draw node[fill,red, circle, inner sep=0pt, minimum size=3pt] at (\leftintvandvvv,\vvv) {};
	\path (\leftintvandvvv,\vvv)+(-0.5,-0.55) node[red] {$-e^{\frac{y_0 + y_2}{2}}$};
	\draw node[fill,red, circle, inner sep=0pt, minimum size=3pt] at (\rightintvandvvv,\vvv) {};
	\path (\rightintvandvvv,\vvv)+(0,-0.55) node[red] {$e^{\frac{y_0 + y_2}{2}}$};
	
	\draw node[fill,red, circle, inner sep=0pt, minimum size=3pt] at (\leftintvvandvvv,\vvv) {};
	\path (\leftintvvandvvv,\vvv)+(0.75,0.5) node[red] {$x_1 - e^{\frac{y_0 + y_2}{2}}$};
	\draw node[fill,red, circle, inner sep=0pt, minimum size=3pt] at (\rightintvvandvvv,\vvv) {};
	\path (\rightintvvandvvv,\vvv)+(-0.5,0.5) node[red] {$x_1 + e^{\frac{y_0 + y_2}{2}}$};

\end{tikzpicture}
\caption{Situation for the intersections of the connection intervals considered in Lemma~\ref{lem:ordered}, with $y_0 < y_1 <y_2$ fixed and for different cases of $0 \le x_1 \le e^{(y_0 + y_1)/2}$. The top figure shows the case where $0 \le x_1 \le e^{(y_1 + y_2)/2} - e^{(y_0 + y_2)/2}$, while the bottom one shows the case $x_1 > e^{(y_1 + y_2)/2} - e^{(y_0 + y_2)/2}$. The solid red line indicates the range for $x_2$ such that the points $p_0$, $p_1$ and $p_2$ form a triangle. The boundaries of their neighbourhoods are shown in, respectively, black, blue and green.}
\label{fig:triangle_prob_lemma}
\end{figure}
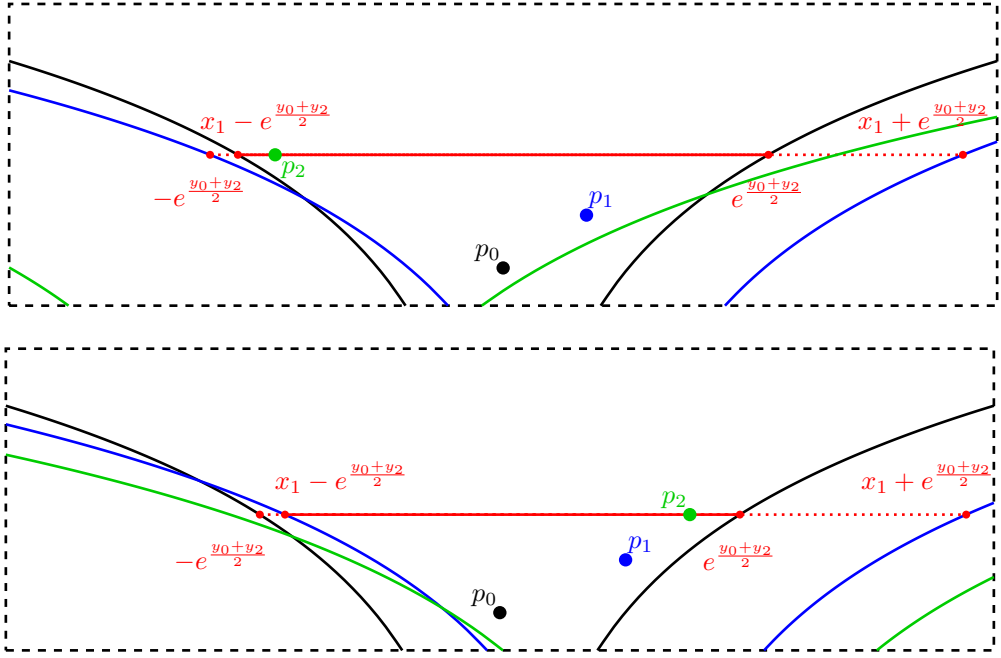

\begin{proof}
Note that $P(y_0,y_1,y_2)$ is the probability that $x_2$ falls into the interval $[x_1-e^{(y_1+y_2)/2},x_1+e^{(y_1+y_2)/2}]$, as well as into the interval $[-e^{(y_0+y_2)/2},e^{(y_0+y_2)/2}]$. By symmetry considerations, we can take $x_1$ uniformly at random from $[0,e^{y_0/2+y_1/2}]$ as opposed to $[-e^{y_0/2+y_1/2}, e^{y_0/2+y_1/2}]$. Figure~\ref{fig:triangle_prob_lemma} shows the intersection of the intervals (red line) for two different cases for $x_1 \le e^{(y_0 + y_1)/2}$. 

Since $y_0 < y_1 < y_2$ we have that $e^{(y_1+y_2)/2} > e^{(y_0+y_2)/2}$ and so, when $x_1 \geq 0$, the ``right half'' of the 
interval $[-e^{(y_0+y_2)/2}, e^{(y_0+y_2)/2}]$ is always covered by the interval $[x_1-e^{(y_1+y_2)/2}, x_1+e^{(y_1+y_2)/2}]$.
If $e^{(y_1+y_2)/2} - e^{(y_0+y_1)/2} \geq e^{(y_0+y_2)/2}$ then the ``left half'' is always covered as well.
In other words:
\[
	e^{(y_1+y_2)/2} - e^{(y_0+y_1)/2} \geq e^{(y_0+y_2)/2} \Rightarrow P(y_0,y_1,y_2) = 1.
\]

Now consider the case where $e^{(y_1+y_2)/2} - e^{(y_0+y_1)/2} < e^{(y_0+y_2)/2}$. 
Then, if $x_1 \in [0, e^{(y_1+y_2)/2} - e^{(y_0+y_2)/2}]$ the whole interval $[-e^{(y_0+y_2)/2}, e^{(y_0+y_2)/2}]$ is still covered 
so that $p_0, p_1$ and $p_2$ form a triangle. If, on the other hand $e^{(y_1+y_2)/2} - e^{(y_0+y_2)/2} < x_1 \leq e^{(y_0+y_1)/2}$ then
the probability that $|x_2-x_1| \leq e^{(y_1+y_2)/2}$ equals
\[ 
	1 - \frac{x_1 - (e^{(y_1+y_2)/2} - e^{(y_0+y_2)/2)}) }{ 2e^{(y_0+y_2)/2} }. 
\]

Hence, when $e^{(y_1+y_2)/2} - e^{(y_0+y_1)/2} < e^{(y_0+y_2)/2}$ we have
\begin{align*}
	P(y_0,y_1,y_2) &= \frac{e^{(y_1+y_2)/2} - e^{(y_0+y_2)/2} }{ e^{(y_0+y_1)/2} }  \\
	&\hspace{10pt}+ \int_{ e^{(y_1+y_2)/2} - e^{(y_0+y_2)/2} }^{ e^{(y_0+y_1)/2} } 
	    \left(1 - \frac{x_1 - (e^{(y_1+y_2)/2} - e^{(y_0+y_2)/2)}) }{ 2e^{(y_0+y_2)/2} }\right)
	    \cdot \frac{1}{e^{(y_0+y_1)/2}} \dd x_1 \\
	&= 1 - \frac{1}{2e^{y_0+y_1/2+y_2/2} } \int_0^{ e^{(y_0+y_1)/2}+e^{(y_0+y_2)/2}-e^{(y_1+y_2)/2} } x_1 \dd x_1 \\
	&= 1 - \frac{ \left( e^{(y_0+y_1)/2}+e^{(y_0+y_2)/2}-e^{(y_1+y_2)/2} \right)^2 }{ 4 e^{y_0+y_1/2+y_2/2} }.
\end{align*}

At this point it is convenient to rewrite everything in terms of $z_i := e^{-y_i/2}$.
Note that $y_0 < y_1 < y_2$ if and only if $z_0 > z_1 > z_2$ while the condition $e^{(y_1+y_2)/2} - e^{(y_0+y_1)/2} < e^{(y_0+y_2)/2}$ becomes
\[ e^{(y_1+y_2)/2} - e^{(y_0+y_1)/2} < e^{(y_0+y_2)/2} \Leftrightarrow 
z_1^{-1} z_2^{-1} < z_0^{-1} z_1^{-1} + z_0^{-1}z_2^{-1} 
\Leftrightarrow
z_0 < z_1+z_2. 
\]

We now conclude that
\[
	P(y_0(z_0), y_1(z_1), y_2(z_2)) = 1 \quad \text{if} \quad z_0 > z_1 > z_2 \text{ and } z_0 \geq z_1 + z_2
\]
while for $z_0 > z_1 > z_2$ and $z_0 < z_1 + z_2$
\begin{align*}
	P(y_0(z_0), y_1(z_1), y_2(z_2)) 
	&= 1 - \frac{z_0^2z_1z_2}{4} \cdot \left( z_0^{-1}z_1^{-1}+z_0^{-1}z_2^{-1}-z_1^{-1}z_2^{-1} \right)^2 \\
	&= 1 - \frac{1}{4} \left( z_1^{-1}z_2 + z_1z_2^{-1} + z_0^2z_1^{-1}z_2^{-1} + 2 - 2z_0z_1^{-1}-2z_0z_2^{-1}\right),
\end{align*}
which finishes the proof.
\end{proof}

The previous lemma covers the case when $y_0<y_1<y_2$. We now leverage it to take care of the other cases as well. 

\begin{proofof}{Lemma~\ref{lem:triangle_prob_y_coordinates}}
Let $y_i >0$ and $z_i = e^{-y_i/2}$, $i=0,1,2$. Lemma~\ref{lem:ordered} gives the expression for $P(y_0(z_0),y_1(z_1),y_2(z_2))$ in the case $y_0<y_1<y_2$, or equivalently $z_0>z_1>z_2$, i.e. the first two lines in the claim of Lemma~\ref{lem:triangle_prob_y_coordinates}. To analyze the other cases we shall express $P(y_1,y_0,y_2)$ and $P(y_1,y_2,y_0)$ in terms of $P(y_0,y_1,y_2)$ and $z_i$. For this we note that we can view $P(y_0,y_1,y_2)$ as a 2-fold integral of the indicator function
\[ 
	h(x_0, x_1, x_2) := \ind{ |x_0 - x_1| < e^{(y_0+y_1)/2}, |x_0 - x_2| < e^{(y_0+y_2)/2}, |x_1-x_2| < e^{(y_1+y_2)/2}}, 
\]
where $x_0$ was set to zero, without loss of generality, and the other two $x_i$ are uniform random variables on $[-e^{(y_0+y_i)/2}, e^{(y_0+y_i)/2}]$. When we consider the probability $P(y_1,y_0,y_2)$, this is the 2-fold integral of $h(x_0,0,x_2)$ so that
\begin{align*}
	P(y_1,y_0,y_2) &= \frac{1}{2e^{(y_1+y_0)/2}} \cdot \frac{1}{2e^{(y_1+y_2)/2}} 
		\iint_{\R} h(x_0,0,x_2) \dd x_0 \dd x_2\\
	&= \frac{e^{y_0/2}}{e^{y_1/2}} \frac{1}{2e^{(y_0+y_1)/2}} \frac{1}{2e^{(y_0+y_2)/2}} 
		\iint_{\R} h(0,x_1,x_2) \dd x_1 \dd x_2\\
	&= \frac{e^{y_0/2}}{e^{y_1/2}} P(y_0,y_1,y_2) = \frac{z_1}{z_0}  P(y_0,y_1,y_2).
\end{align*}
Finally we note that $h(x_0,0,x_2) = h(x_2,0,x_0)$ from which we conclude that
\begin{equation}\label{eq:symmetry_relation_triangle_prob}
	P(y_0, y_1, y_2) = \left(z_0/z_1\right) P(y_1,y_0,y_2) = \left(z_0/z_1\right) P(y_1,y_2,y_0).
\end{equation}

To complete the proof for the other cases we note that since $P(y_0,y_1,y_2)$ is symmetric in $y_1$ and $y_2$, we can assume, without loss of generality, that $y_1 < y_2$. Then, there are two more orderings of $y_0, y_1, y_2$, namely $y_1< y_0< y_2$ and $y_1<y_2<y_0$, which can be summarized as $y_1 < \min (y_0,y_2)$, or equivalently $z_1 > \max(z_0,z_2)$. For $y_1 < y_0 < y_2$ and $y_1 < y_2<y_0$ we can apply Lemma~\ref{lem:ordered} to obtain $P(y_1,y_0,y_2) = P(y_1,y_2,y_0)$ which happen to agree due to the symmetry in the last two arguments of the expression found in Lemma~\ref{lem:ordered}. The expression for $P(y_0,y_1,y_2)$ then follows from~\eqref{eq:symmetry_relation_triangle_prob}.
\end{proofof}

\subsubsection{Integrating over \texorpdfstring{$y_1, y_2$}{y1, y2}}

Now that we have established the expression for $P(y_0,y_1,y_2)$ we can proceed to compute $P(y_0)$ by integrating over $y_1, y_2$.
We however start with the following observation.

\begin{lemma}\label{lem:continuity_Delta_function}
The function $\alpha \mapsto P_\alpha(y_0)$ is continuous for all $\alpha > \frac{1}{2}$.
\end{lemma}

\begin{proof}
This follows from the theorem of dominated convergence:
Let $\alpha > \frac{1}{2}$ and $(\alpha_n)_{n\in \mathbb{N}}$ a sequence of real numbers converging to $\alpha$, so we can 
assume $|\alpha_n - \alpha| < \epsilon := \frac{\alpha-1/2}{2}$. 
This means that $-\epsilon < \alpha_n - \alpha < \epsilon$, i.e. $\frac{\alpha-1/2}{2} < \alpha_n - 1/2 < \frac{3\alpha-3/2}{2}$. Define 

$$f_n(y_1,y_2) = P(y_0,y_1,y_2) (\alpha_n - 1/2)^2 e^{-(\alpha_n-1/2)(y_1+y_2)}.$$ 

As the function $x \mapsto x^2$ is increasing in $x$ for $x>0$ and the function $x \mapsto e^{-(y_1+y_2)x}$ is decreasing 
in $x$ and $P(y_0,y_1,y_2) \in [0,1]$, it holds that 

$$|f_n(y_1,y_2)| \leq \left(\frac{3\alpha-3/2}{2}\right)^2e^{-(y_1+y_2)\frac{\alpha-1/2}{2}}$$

which is integrable over $\R_{\geq 0} \times \R_{\geq 0}$ (with integral equalling $(6\alpha-3)^2/(2\alpha-1)^2$). 
Application of the theorem of dominated convergence yields that 
$P_{\alpha_n}(y_0) \rightarrow P_\alpha(y_0)$ which gives the claim as the 
sequence $(\alpha_n)_n$ was arbitrary.
\end{proof}

Due to this lemma we can first assume $\alpha \notin \{ \frac{3}{4},1 \}$, compute $P(y_0)$ and then obtain the values of $P(y_0)$ at 
the remaining two points by taking the corresponding limit in $\alpha$. 
This strategy is executed below. 
It involves the computation of several integrals which are involved and will take up a few pages. 
The proof is structured using headers, to aid the reader. 

Note that when writing $P(y_0)$ as an integral, see equation~\eqref{eq:delta_P}, by symmetry in the integration 
variables $y_1$ and $y_2$, we can assume that $y_1<y_2$ in which case either $y_0$ or $y_1$ is the smallest height. 
This gives half the value of $P(y_0)$ and hence
\[ 
	P(y_0) = 2(I_1(y_0) +I_2(y_0)), 
\] 
where $I_1$ and $I_2$ are given by:
\begin{align*}
	I_1(y_0) &:= \int_{0<y_0<y_1<y_2} P(y_0,y_1,y_2) \cdot (\alpha-1/2)^2 e^{-(\alpha-1/2)(y_1+y_2)}  \dd y_2 \dd y_1 \\ 
	I_2(y_0) &:= \int_{0<y_1<\min(y_0,y_2)} P(y_0,y_1,y_2) \cdot (\alpha-1/2)^2 e^{-(\alpha-1/2)(y_1+y_2)} \dd y_2 \dd y_1 \\ 
\end{align*}

We proceed with computing these two integrals, each of which is split into two parts.
The final expressions of those four integrals can be found 
in~\eqref{eq:Delta_P_computation_I11}, \eqref{eq:Delta_P_computation_I12}, \eqref{eq:Delta_P_computation_I21} 
and~\eqref{eq:Delta_P_computation_I22}.

\paragraph{Computing $\bm{I_1(y_0)}$}

Applying the change of variables $z_i := e^{-y_i/2}$, $i=1,2$, and Lemma~\ref{lem:triangle_prob_y_coordinates} gives 
\begin{align*}
	I_1(y_0) &=	4 (\alpha-1/2)^2 \cdot \int_{z_0>z_1>z_2>0} P(y_0,y_1(z),y_2(z)) z_1^{2\alpha-2} z_2^{2\alpha-2} 
		\dd z_2\dd z_1 \\
	&= 4 (\alpha-1/2)^2 \cdot \left( \int_{z_0>z_1>z_2>0} 1 \cdot z_1^{2\alpha-2} z_2^{2\alpha-2} 
		\dd z_2 \dd z_1 \right. \\
	&\hspace{10pt} \left. - \int_{\substack{{z_0>z_1>z_2>0,}\\{z_0 < z_1+z_2}}} G(z_0,z_1,z_2) \cdot z_1^{2\alpha-2} z_2^{2\alpha-2} 
		\dd z_2 \dd z_1 \right) \\
	&=: 4 (\alpha-1/2)^2 ( I_{11}(y_0) - I_{12}(y_0)). 
\end{align*}

The integral $I_{11}(y_0)$ is easily obtained:
\begin{align*}
	I_{11}(y_0) &= \int_0^{z_0} \int_0^{z_1} z_1^{2\alpha-2} z_2^{2\alpha-2} \dd z_2 \dd z_1
		= \int_0^{z_0} z_1^{2\alpha-2} \left[ \frac{z_2^{2\alpha-1}}{2\alpha-1} \right]_0^{z_1} \dd z_1\\
	&= \frac{1}{2\alpha-1} \cdot \int_0^{z_0} z_1^{4\alpha-3} \dd z_1
		= \frac{1}{2(2\alpha-1)^2} \cdot z_0^{4\alpha-2}. \numberthis \label{eq:Delta_P_computation_I11}
\end{align*}

To deal with $I_{12}$ we note that $G(z_0,z_1,z_2)$ is a linear combination of monomials of the form $z_0^az_1^bz_2^c$ with 
$a,b,c \in \{-1,0,1,2\}$ and $a+b+c=0$. Let us consider the integral $J_{(a,b,c)}(z_0)$ defined by 

\begin{equation}\label{eq:def_Delta_P_computation_int_J_1}
	J_{a,b,c}(z_0) := z_0^a \int_{\substack{{z_0>z_1>z_2>0,}\\{z_0 < z_1+z_2}}} z_1^{b+2\alpha-2} z_2^{c+2\alpha-2} \dd z_2\dd z_1.
\end{equation}
and note that
\begin{equation}\label{eq:Delta_P_computation_I12_with_J}
	I_{1,2}(y_0) = \frac{1}{4} (J_{0,-1,1}(z_0)+J_{0,1,-1}(z_0)+J_{2,-1,-1}(z_0)+2J_{0,0,0}(z_0)-2J_{1,-1,0}(z_0)-2J_{1,0,-1}(z_0)).
\end{equation}

Next we compute $J_{a,b,c}(z_0)$.
\begin{align*}
	J_{a,b,c}(z_0) 
	&= z_0^a \int_{z_0/2}^{z_0}\int_{z_0-z_1}^{z_1} z_1^{b+2\alpha-2} z_2^{c+2\alpha-2} \dd z_2 \dd z_1
		= z_0^a \int_{z_0/2}^{z_0} z_1^{b+2\alpha-2} \left[ \frac{ z_2^{c+2\alpha-1} }{ c+2\alpha-1 } \right]_{z_0-z_1}^{z_1} \dd z_1\\
	&= \frac{z_0^a}{c+2\alpha-1} \cdot \left( \int_{z_0/2}^{z_0} z_1^{b+c+4\alpha-3} \dd z_1
	   - \int_{z_0/2}^{z_0} z_1^{b+2\alpha-2} (z_0-z_1)^{c+2\alpha-1} \dd z_1 \right) \\
	&= \frac{z_0^{a+b+c+4\alpha-2}(1-(1/2)^{b+c+4\alpha-2})}{(c+2\alpha-1)(b+c+4\alpha-2)} \\
	&\hspace{10pt}- \frac{z_0^{a+b+c+4\alpha-3}}{c+2\alpha-1} \int_{z_0/2}^{z_0}  \left(z_1/z_0\right)^{b+2\alpha-2} 
	    \left(1-(z_1/z_0)\right)^{c+2\alpha-1} \dd z_1\\
	&= \frac{z_0^{4\alpha-2}(1-(1/2)^{b+c+4\alpha-2})}{(c+2\alpha-1)(b+c+4\alpha-2)} 
		- \frac{z_0^{4\alpha-2}}{c+2\alpha-1}
	   \int_{1/2}^1  u^{b+2\alpha-2}(1-u)^{c+2\alpha-1} \dd u \\
	&= \frac{z_0^{4\alpha-2}(1-(1/2)^{b+c+4\alpha-2})}{(c+2\alpha-1)(b+c+4\alpha-2)} 
		- \frac{z_0^{4\alpha-2}}{c+2\alpha-1} B^-(1/2;c+2\alpha, b+2\alpha-1),
\end{align*}
where we have used the substitution $u := z_1/z_0$ giving $z_0 \dd u = \dd z_1$ in the penultimate line and
$B^-$ denotes the (lower) incomplete beta function. Note that since $c \geq -1$, $-a \in \{0,-1,-2\}$ and by our assumption $\alpha \not \in \{\frac{3}{4},1\}$, the denominators that occur during the integration are all non-zero.

Plugging this back into~\eqref{eq:Delta_P_computation_I12_with_J} gives
\begin{align*}
	I_{1,2}(y_0)
	&= \frac{z_0^{4\alpha-2}(1-(1/2)^{4\alpha-2})}{32\alpha(\alpha-1/2)} 
   		- \frac{z_0^{4\alpha-2}}{8\alpha} B^-(1/2;1+2\alpha, 2\alpha-2)\\
	&\hspace{10pt}+ \frac{z_0^{4\alpha-2}(1-(1/2)^{4\alpha-2})}{32(\alpha-1)(\alpha-1/2)} 
   		-  \frac{z_0^{4\alpha-2}}{4(2\alpha-2)} B^-(1/2;2\alpha-1,2\alpha)\\
	&\hspace{10pt}+ \frac{z_0^{4\alpha-2}(1-(1/2)^{4\alpha-4})}{32(\alpha-1)^2} 
   		- \frac{z_0^{4\alpha-2}}{4(2\alpha-2)} B^-(1/2;-1+2\alpha, 2\alpha-2)\\
	&\hspace{10pt}+ \frac{z_0^{4\alpha-2}(1-(1/2)^{4\alpha-2})}{16(\alpha-1/2)^2} 
   		- \frac{z_0^{4\alpha-2}}{2(2\alpha-1)} B^-(1/2;2\alpha,2\alpha-1)\\
	&\hspace{10pt}- \frac{z_0^{4\alpha-2}(1-(1/2)^{4\alpha-3})}{16(\alpha-1/2)(\alpha-3/4)} 
   		+ \frac{z_0^{4\alpha-2}}{2(2\alpha-1)} B^-(1/2;2\alpha, 2\alpha-2)\\
	&\hspace{10pt}- \frac{z_0^{4\alpha-2}(1-(1/2)^{4\alpha-3})}{16(\alpha-1)(\alpha-3/4)} 
   		+ \frac{z_0^{4\alpha-2}}{2(2\alpha-2)} B^-(1/2;-1+2\alpha, 2\alpha-1) \\
	&=\frac{\left(\frac{3}{64}- \frac{3}{16} 2^{-4\alpha}+ 
   		\alpha (-\frac{41}{128} + \frac{13}{16}  2^{-4\alpha}) 
   		+ \alpha^2 (\frac{5}{8} - \frac{3}{4} 2^{-4\alpha}) - \frac{15}{32}\alpha^3 +\frac{1}{8} \alpha^4\right) 
   		z_0^{4 \alpha-2} }{4(\alpha-1/2)^2 (\alpha-1)^2 (\alpha-3/4) \alpha} \\
	&\hspace{10pt}+ \frac{z_0^{4\alpha-2}}{8 (\alpha-1) \alpha (2\alpha-1)}(4 (\alpha-1) \alpha 
		(B^-(1/2; 2\alpha, 2\alpha-2) - B^-(1/2;2 \alpha,2\alpha-1) ) \\
    &\hspace{10pt}- (2 \alpha-1)\alpha ( B^-(1/2; 2\alpha-1, 2\alpha-2) + 
    	B^-(1/2; 2\alpha-1, 2\alpha) - 
    	2 B^-(1/2;2\alpha -1, 2\alpha-1) ) \\
    &\hspace{10pt}- (2\alpha-1)(\alpha-1) B^-(1/2; 1 + 2\alpha, 2\alpha-2)) \\
	&=\frac{\left(\frac{3}{64}- \frac{3}{16} 2^{-4\alpha} 
		+ \alpha (-\frac{41}{128} + \frac{13}{16}  2^{-4\alpha})  
		+ \alpha^2 (\frac{5}{8} - \frac{3}{4} 2^{-4\alpha}) - \frac{15}{32}\alpha^3 +\frac{1}{8} \alpha^4\right) 
		z_0^{4 \alpha-2} }{4(\alpha-1/2)^2 (\alpha-1)^2 (\alpha-3/4) \alpha} \\
 	&\hspace{10pt}+ \frac{z_0^{4\alpha-2}}{8 (\alpha-1) \alpha (2\alpha-1)}(4 (\alpha-1) \alpha 
 		B^-(1/2; 2\alpha+1, 2\alpha-2) \\
  	&\hspace{10pt}- (2 \alpha-1)\alpha B^-(1/2;2\alpha+1,2\alpha-2) \\
    &\hspace{10pt}- (2\alpha-1)(\alpha-1) B^-(1/2; 2\alpha+1, 2\alpha-2)). \\
\end{align*}
For the last step we use the identities 
\begin{align}
	B^-(z;a,b)-B^-(z;a,b+1) &= B^-(z; a+1,b), \label{eq:Delta_P_computation_beta_id_1}\\
	B^-(z;a,b)+B^-(z;a,b+2)-2B^-(z;a,b+1) &= B^-(z;a+2,b). \label{eq:Delta_P_computation_beta_id_2}
\end{align}
to obtain
\begin{equation}
\begin{aligned}
	I_{1,2}(y_0) &=\frac{\left(\frac{3}{64}- \frac{3}{16} 2^{-4\alpha}
		+ \alpha (-\frac{41}{128} + \frac{13}{16}  2^{-4\alpha})
		+ \alpha^2 (\frac{5}{8} - \frac{3}{4} 2^{-4\alpha}) - \frac{15}{32}\alpha^3 +\frac{1}{8} \alpha^4\right) 
		z_0^{4 \alpha-2} }{4(\alpha-1/2)^2 (\alpha-1)^2 (\alpha-3/4) \alpha} \\
 	&\hspace{10pt}- \frac{z_0^{4\alpha-2}B^-(1/2; 2\alpha+1, 2\alpha-2) }{8 (\alpha-1) \alpha (2\alpha-1)}	\label{eq:Delta_P_computation_I12}
\end{aligned}
\end{equation}

\paragraph{Computing $\bm{I_2(y_0)}$}

We will follow a similar strategy as for $I_1(y_0)$. First, using the change of variables $z_i := e^{-y_i/2}$, $i=1,2$,
we get
\begin{align*}
	I_2(y_0) &= 4 (\alpha-1/2)^2 \cdot \int_{1>z_1>\max(z_2,z_0), \atop z_0,z_2>0} P(y_0,y_1(z_1),y_2(z_2)) z_1^{2\alpha-2} z_2^{2\alpha-2} 
		\dd z_2 \dd z_1 \\
	&= 4 (\alpha-1/2)^2 \cdot \left( \int_{1>z_1>\max(z_0,z_2)>0, \atop z_0,z_2>0}  z_0 z_1^{2\alpha-3} z_2^{2\alpha-2} 
		 \dd z_2 \dd z_1 \right. \\
	&\hspace{10pt}\left. - \int_{{1>z_1>\max(z_0,z_2),}\atop{z_0,z_2>0,\atop{z_1 < z_0+z_2}}} G(z_1,z_0,z_2) z_0 z_1^{2\alpha-3} 	
		z_2^{2\alpha-2} \dd z_2 \dd z_1 \right) \\
	&=: 4 (\alpha-1/2)^2 (I_{21}(y_0) - I_{22}(y_0)). 
\end{align*}

We start with the easy integral:
\begin{align*}
	I_{21}(y_0) &= z_0 \int_{1>z_1>\max(z_2,z_0),\atop z_0,z_2>0} z_1^{2\alpha-3} z_2^{2\alpha-2}  \dd z_2 \dd z_1   
		= z_0 \int_{z_0}^1 \int_{0}^{z_1} z_1^{2\alpha-3} z_2^{2\alpha-2}  \dd z_2 \dd z_1\\
	&= z_0 \int_{z_0}^1 \left[ \frac{z_2^{2\alpha-1}}{2\alpha-1}\right]_{0}^{z_1} z_1^{2\alpha-3} \dd z_1 
		= \frac{z_0}{2\alpha-1}  \int_{z_0}^1 z_1^{4\alpha-4} \dd z_1
		= \frac{z_0 - z_0^{4\alpha-2}}{(4\alpha-3)(2\alpha-1)}. \numberthis \label{eq:Delta_P_computation_I21}
\end{align*}
We note that the denominators above are non-zero as $\alpha > \frac{1}{2}$ and $\alpha \not =\frac{3}{4}$.

To deal with $I_{22}(y_0)$ we consider the function
\[
	J_{a,b,c}^\prime(z_0) := z_0^a \int_{ {1>z_1>\max(z_0,z_2),} \atop{{z_0,z_2>0,}\atop{z_1 < z_0+z_2}} } z_1^{b+2\alpha-2} z_2^{c+2\alpha-2} \dd z_2 \dd z_1
\]
and write
\begin{equation}\label{eq:Delta_P_computation_I22_J}
\begin{aligned}
		I_{2,2}(y_0) &= \frac{1}{4}\left(J_{0,-1,1}^\prime(z_0) + J_{2,-1,-1}^\prime(z_0)
		+ J_{0,1,-1}^\prime(z_0)\right) \\
		&\hspace{10pt}+ \frac{1}{2}\left( J_{1,-1,0}^\prime(z_0) - J_{0,0,0}^\prime(z_0) - J_{1,0,-1}^\prime(z_0)\right).
\end{aligned}
\end{equation}

We now compute $J_{a,b,c}^\prime(z_0)$
\begin{align*}
	J_{a,b,c}^\prime(z_0) 
	&= z_0^a \int_{z_0}^1 \int_{z_1-z_0}^{z_1}  z_1^{b+2\alpha-2} z_2^{c+2\alpha-2} \dd z_2 \dd z_1  \\
 	&= z_0^a \int_{z_0}^1 \frac{1}{c+2\alpha-1}  z_1^{b+2\alpha-2}( z_1^{c+2\alpha-1}-(z_1-z_0)^{c+2\alpha-1})  \dd z_1  \\
 	&=z_0^a \int_{z_0}^1 \frac{1}{c+2\alpha-1} z_1^{b+c+4\alpha-3} {\dd } z_1 -z_0^{a} \int_{z_0}^1 	
 		\frac{1}{c+2\alpha-1}z_1^{b+2\alpha-2}(z_1-z_0)^{c+2\alpha-1}  \dd z_1  \\
 	&= z_0^a  \frac{1}{(c+2\alpha-1)(b+c+4\alpha-2)}(1-z_0^{b+c+4\alpha-2}) \\
 	&\hspace{10pt}- \frac{z_0^a}{c+2\alpha-1}z_0^{b+c+4\alpha-2}B^-(1-z_0;c+2\alpha,-b-c-4\alpha+2) \\
	&= \frac{z_0^a -z_0^{4\alpha-2}}{(c+2\alpha-1)(b+c+4\alpha-2)}  
		- \frac{z_0^{4\alpha -2}B^-(1-z_0;c+2\alpha,-b-c-4\alpha+2)}{c+2\alpha-1}.  
\end{align*}
Here we used that for $x  \in \R,y>-1$ (note that as $c\geq -1$, it holds that $c+2\alpha-1 >-1$):
\begin{align*}
 \int_{z_0}^{1} z_1^x (z_1-z_0)^y \dd z_1 
 &= \int_0^{1-z_0} (s+z_0)^x s^y \dd s \\
 &= z_0^{x+y} \int_0^{1-z_0} \left( (s/z_0) + 1 \right)^x (s/z_0)^y \dd s \\
 &= z_0^{x+y+1} \int_{0}^{1/z_0 -1 } (t+1)^x t^y \dd t \\
 &= z_0^{x+y+1} \int_0^{1-z_0} u^y (1-u)^{-(x+y+2)} \dd u \\
 &= z_0^{x+y+1} B^-(1-z_0; y+1,-x-y-1 ).
\end{align*}
As $c \geq -1$ and $-a \in \{0,-1,-2\}$ and by our assumption $\alpha \not \in \{\frac{3}{4}\}$, the denominators 
that occur during the computations above are non-zero. 

Plugging the expression for $J_{a,b,c}^\prime(z_0)$ back into~\eqref{eq:Delta_P_computation_I22_J} we get,
\begin{align*}
	I_{2,2}(y_0) 
	&= \frac{1 -z_0^{4\alpha-2}}{32\alpha(\alpha-1/2)} 
		-  \frac{z_0^{4\alpha -2}B^-(1-z_0;1+2\alpha,-4\alpha+2)}{8\alpha} \\
	&\hspace{10pt}+ \frac{z_0^2 -z_0^{4\alpha-2}}{32(\alpha-1)^2}  
		- \frac{z_0^{4\alpha -2}B^-(1-z_0;-1+2\alpha,-4\alpha+4)}{8(\alpha-1)} \\
	&\hspace{10pt}+ \frac{1 -z_0^{4\alpha-2}}{32(\alpha-1)(\alpha-1/2)}  
		- \frac{z_0^{4\alpha -2}B^-(1-z_0;-1+2\alpha,-4\alpha+2)}{8(\alpha-1)} \\
	&\hspace{10pt}+ \frac{z_0 -z_0^{4\alpha-2}}{16(\alpha-1/2)(\alpha-3/4)}  
		- \frac{z_0^{4\alpha -2}B^-(1-z_0;2\alpha,-4\alpha+3)}{4(\alpha-1/2)} \\
	&\hspace{10pt}- \frac{1 -z_0^{4\alpha-2}}{16(\alpha-1/2)^2}  
		+ \frac{z_0^{4\alpha -2}B^-(1-z_0;2\alpha,-4\alpha+2)}{4(\alpha-1/2)}\\
	&\hspace{10pt}-\frac{z_0 -z_0^{4\alpha-2}}{16(\alpha-1)(\alpha-3/4)} 
		+ \frac{z_0^{4\alpha -2}B^-(1-z_0;-1+2\alpha,-4\alpha+3)}{4(\alpha-1)}.
\end{align*}
Using some algebra and the identities~\eqref{eq:Delta_P_computation_beta_id_1} and~\eqref{eq:Delta_P_computation_beta_id_2}  
this can be reduced to
\begin{equation}\label{eq:Delta_P_computation_I22}
\begin{aligned}
	I_{2,2}(y_0)
	&=\frac{1}{64\alpha(\alpha-1/2)^2(\alpha-1)} -\frac{(1 - z_0)^{2\alpha}}{64\alpha(\alpha-1/2)^2 (\alpha-1)} 
		- \frac{z_0}{8(\alpha-1/2)(\alpha-1)(4\alpha-3)}\\ 
	&\hspace{10pt}+ \frac{z_0^2}{32(\alpha-1)^2} + \frac{(-6 + 25\alpha - 48\alpha^2 + 44\alpha^3 -16\alpha^4) 
   		z_0^{4\alpha-2}}{512\alpha(\alpha-1/2)^2(\alpha-1)^2(\alpha-3/4)} \\
	&\hspace{10pt}+ \frac{z_0^{4\alpha-2}B^-(1 - z_0; 2\alpha, 3 - 4\alpha)}{32(\alpha-1)(\alpha-1/2)^2}.
\end{aligned}
\end{equation}

\paragraph{Combining the results for $\bm{I_1(y_0)}$ and $\bm{I_2(y_0)}$}

Combining the results for $I_{11}(y_0), I_{12}(y_0), I_{21}(y_0)$ and $I_{22}(y_0)$ we get, after some algebra, an explicit expression 
for $P(y_0)$ as a linear combination of terms of the form $z_0^u$, $(1-z_0)^u$ and $z_0^u B^-(1-z_0;a,b)$: 
\begin{align*}
P(y_0)=& 2(I_1+I_2) = 8(\alpha- 1/2)^2(I_{1,1}-I_{1,2}+I_{2,1}-I_{2,2}) \\
=&8(\alpha-1/2)^2 \left(\frac{1}{2(2\alpha-1)^2 }z_0^{4\alpha -2} \right. \\ 
&\left.-\frac{\left(\frac{3}{64}- \frac{3}{16} 2^{-4\alpha}+ 
   \alpha (-\frac{41}{128} + \frac{13}{16}  2^{-4\alpha})  + 
   \alpha^2 (\frac{5}{8} - \frac{3}{4} 2^{-4\alpha}) - \frac{15}{32}\alpha^3 +\frac{1}{8} a^4   \right) z_0^{4 \alpha-2} }{4(\alpha-1/2)^2 (\alpha-1)^2 (\alpha-3/4) \alpha} \right. \\ 
&\left.+\frac{z_0^{4\alpha-2}B^-(1/2; 2\alpha+1, 2\alpha-2) }{8 (\alpha-1) \alpha (2\alpha-1)} +\frac{z_0 - z_0^{4\alpha -2}}{(4\alpha-3)(2\alpha-1)} \right.\\
&\left.-\frac{1}{64\alpha(\alpha-1/2)^2(\alpha-1)} +\frac{(1 - z_0)^{2\alpha}}{64\alpha(\alpha-1/2)^2 (\alpha-1)} +\frac{z_0}{8(\alpha-1/2)(\alpha-1)(4\alpha-3)} \right.\\
   &\left. - \frac{z_0^2}{32(\alpha-1)^2}
   - \frac{(-6 + 25\alpha - 48\alpha^2 + 44\alpha^3 -16\alpha^4) z_0^{4\alpha-2}}{512\alpha(\alpha-1/2)^2(\alpha-1)^2(\alpha-3/4)} \right. \\
   & \left. - \frac{z_0^{4\alpha-2}B^-(1 - z_0; 2\alpha, 3 - 4\alpha)}{32(\alpha-1)(\alpha-1/2)^2} \right) \\
=&-\frac{1}{8 (\alpha - 1) \alpha} + \frac{(\alpha - 1/2) z_0}{\alpha - 1} - \frac{(\alpha - 1/2)^2 z_0^2}{
	4 (\alpha - 1)^2} \\
&+ 
z_0^{-2 + 4 \alpha} \left(\frac{2^{-4 \alpha-1} (3 \alpha - 1)}{\alpha (\alpha - 1)^2} + \frac{(\alpha - 
	1/2 ) B^-(1/2; 1 + 2 \alpha, 
	-2 + 2 \alpha)}{2(\alpha - 1) \alpha} \right) \\
&+ \frac{(1 - 
	z_0)^{2 \alpha}}{8 (\alpha - 1) \alpha} - \frac{  
	z_0^{4 \alpha - 2} B^-(1 - z_0; 2 \alpha, 3 - 4 \alpha)}{4 (\alpha - 1)}
\end{align*}

Observe that the above expression only contains terms of the form $\alpha - 1$ in the denominator. 
The only expression of the form $\alpha - 3/4$ is in the lower incomplete beta-function $B^-(1 - z_0; 2 \alpha, 3 - 4 \alpha)$ which 
appears twice in the expression for $P(y_0)$.

\paragraph{The case of $\bm{\alpha = 3/4}$}\hfill\\

Note that the factor $\alpha-\frac{3}{4}$ does not occur in any denominator of the previously obtained expression. 
For the lower incomplete beta function, the last argument $3-4\alpha$ is zero for $\alpha=\frac{3}{4}$, however as $z_0 < 1$ the 
integration domain of the lower incomplete beta function does not touch the singularity at $t=1$ 
(note $B^-(1-z_0;2\alpha,3-4\alpha) = \int_0^{1-z_0} t^{2\alpha-1} (1-t)^{2-4\alpha}dt$). 
Therefore, the previous expression holds for this case as well.


\subsubsection{Computing \texorpdfstring{$\gamma$}{gamma} and \texorpdfstring{$\gamma(k)$}{gamma(k)}\label{ssec:exact_expressions_clustering_P}}

Now that we have an expression for $P(y_0)$ we can compute $\gamma, \gamma(k)$ by integrating 
over $y_0$ and prove that they equal the expressions given in, respectively, Theorem~\ref{thm:clustering_coefficient_hyperbolic} and 
Theorem~\ref{thm:local_clustering_hyperbolic}.

We define
\[
	I^{(k)} := 
	\int_0^{\infty} P(y) \alpha e^{-\alpha y}\rho(y,k) \dd y = 
	\int_0^{\infty} P(y) \alpha e^{-\alpha y} \frac{\left(\xi e^{y/2}\right)^k}{k!} e^{-\xi e^{y/2}} \dd y
\]
and
\[
	J := \int_0^\infty P(y) \alpha e^{-\alpha y} \dd y.
\]

Then, recalling~\eqref{eq:gammaint} and~\eqref{eq:gammakint}, we have
\[   
	\gamma = J - I^{(1)} - I^{(2)} \quad \text{and} \quad
	\gamma(k) = \frac{I^{(k)}}{\pmf(k)}.
\]

We will thus compute $J$ and $I^{(k)}$. It will be helpful to change coordinates to $z := e^{-y/2}$. This yields 
\[ 
	J = 2 \alpha \int_0^1 P(y) z^{2\alpha-1} \dd z, 
\]
and 
\[ 
	I^{(k)} = \frac{2 \alpha \xi^k}{k!} \cdot \int_0^1 P(y(z)) \cdot z^{2\alpha-(k+1)} e^{-\xi z^{-1}} \dd z. 
\]

We shall be assuming $\alpha \not = 1$. 
We observe from Lemma~\ref{lem:Paneq1} that for $\alpha \not =1$, $P(y(z))$ is in fact a linear combination 
of terms of the form $z^u$, $(1-z)^u$ and $z^u B^-(1-z;v,w)$.

To compute $J$ we observe that, by integration by parts, 
\begin{align*}
	\int_0^1 z^{u+2\alpha-1} B^-(1-z;v,w) \dd z 
	&= \left[ \frac{z^{u+2\alpha}}{u+2\alpha} B^-(1-z;v,w) \right]_0^1 
		+ \frac{1}{u+2\alpha} \int_0^1 z^{u+2\alpha+w-1} (1-z)^{v-1} \dd z \\
	&= \frac{1}{u+2\alpha} B(u+w+2\alpha,v)
\end{align*}
where we used that $\frac{\partial}{\partial z} B^-(1-z;v,w) = - z^{w-1} (1-z)^{v-1}$. This takes care of the two integrands involving the beta function in $P(y)$. The other integrals are easily computed and yield the following expression for $J$ (note that it only depends on $\alpha$ but not on $\nu$)
\begin{align*}
J&=\frac{2 + 4 \alpha + 13 \alpha^2 - 34 \alpha^3 - 12\alpha^4 + 
	24 \alpha^5}{16(\alpha-1)^2 \alpha (\alpha+1) (2\alpha+1)} +  \frac{2^{-1 - 
		4 \alpha}}{(\alpha - 1)^2} \\
	&\qquad+ \frac{(\alpha - 1/2) (B(2 \alpha, 2 \alpha + 1) + 
	B^-(1/2; 1 + 2 \alpha, -2 + 2 \alpha))}{2 (\alpha - 1) (3 \alpha - 1)}.
\end{align*}

We proceed to work out $I^{(k)}$. For this we will compute the integrals involving terms in $P(y(z))$ of the form $z^u$, $(1-z)^u$ and $B(1-z,v,w)$ separately. We first point out that for any $0 \le a < b \le 1$
\begin{align*}
	\int_a^b z^{u+2\alpha-(k+1)} e^{-\xi z^{-1}} \dd z
	&= \xi^{u+2\alpha-k} \int_{\xi/b}^{\xi/a} t^{k-1-2\alpha-u} e^{-t} \dd t \\
	&= \xi^{u+2\alpha-k} \left( \Gamma^+( k-2\alpha-u,\xi/b) - \Gamma^+( k-2\alpha-u, \xi/a) \right).
\end{align*}
In particular
\begin{equation}\label{eq:integral_Delta_P_z}
	\int_0^1 z^{u+2\alpha-k-1} e^{-\xi z^{-1}} \dd z = \xi^{u+2\alpha-k} \Gamma^+(k-2\alpha-u,\xi).
\end{equation}
where $\Gamma^+$ denotes the (upper) incomplete gamma function, and we have used the substitution
$t = \xi / z$ which gives $\dd z = -\xi t^{-2} \dd t$. (And of course it is understood that 
$\xi/0 = \infty$). This takes care of the integrals of all terms in $P(y(z))$ of the form $z^{u}$.

Next we will consider the integrals over the terms in $P(y(z))$ of the form $(1-z)^u$. For this we need the 
hypergeometric U-function (also called Tricomi's confluent hypergeometric function), which has the integral representation 
\[
	U(a,b,z) = \frac{1}{\Gamma(a)} \int_0^\infty e^{-zt} t^{a-1} (1+t)^{b-a-1} \dd t.
\] 
which holds for $a,b,z\in \mathbb{C}$, $b \not \in \mathbb{Z}_{\leq 0}$, $Re(a), Re(z) >0$, see~\cite[p.255]{erdelyi1953higher}. 
Applying the change of variables $t=\frac{1-s}{s}$ (i.e. $\dd t = -s^{-2} \dd s$ and $s = \frac{1}{t+1}$) yields
\begin{align*}
	U(a,b,z) = \frac{e^z}{\Gamma(a)} \int_0^1 s^{-b} (1-s)^{a-1} e^{-z/s} ds
\end{align*}
Setting $a=2\alpha+1 >0$, $b=-2\alpha+k+1$, $z=\xi>0$, then gives
\begin{equation}\label{eq:integral_Delta_P_1_z}
	\int_0^1 z_0^{2\alpha-k-1} e^{-\xi/z_0} (1-z_0)^{2\alpha} dz_0 = \Gamma(2\alpha+1)e^{-\xi} U(2\alpha+1,1+k-2\alpha,\xi).
\end{equation}

Finally we need to deal with the terms in $P(y(z))$ that involve the incomplete beta function. Let $a, c \in \R$, $\xi, b >0$ 
positive real numbers. Using the integral definition of the incomplete beta function, the change of variables $s=1-t$ gives:
\begin{align*}
	\int_0^1 z^a e^{-\xi/z} B^-(1-z;b,c) \dd z 
	&=\int_0^1 z^a e^{-\xi/z} \int_0^{1-z} t^{b-1} (1-t)^{c-1} \dd t \dd z \\
	&=\int_0^1 z^a e^{-\xi/z} \int_z^1 s^{c-1} (1-s)^{b-1} \dd s \dd z.
\end{align*}
Then changing the order of integration and using the substitution $u = \xi/z$  and recognizing the upper incomplete gamma function yields
\begin{align*}
	&\hspace{-30pt}\int_0^1 z^a e^{-\xi/z} \int_z^1 s^{c-1} (1-s)^{b-1} \dd s \dd z\\
	&=\int_0^1 \int_0^s z^a e^{-\xi/z} \dd z \, s^{c-1} (1-s)^{b-1} \dd s \\
	&=\int_0^1 \int_{\xi/s}^\infty \xi^{a+1} u^{-a-2} e^{-u} \dd u \, s^{c-1} (1-s)^{b-1} \dd s \\
	&= \xi^{a+1} \int_0^1 \Gamma^+(-a-1,\xi/s) s^{c-1} (1-s)^{b-1} \dd s.
		\numberthis \label{eq:integral_gamma_function}
\end{align*}
To compute this last integral we make use of the fact that the incomplete $\Gamma$-function has a representation in terms of 
Meijer's $G$-function (see Lemma~\ref{lem:gamma_meijer_G} in Appendix~\ref{sec:Meijer_G_functions})
\[
	\Gamma^+(-a-1,\xi/s) = \MeijerGnew{2}{0}{1}{2}{1}{-a-1,0}{\frac{\xi}{s}},
\] 
which holds for any $a\in \R$ and $s>0$ (that for a fixed second argument, the upper incomplete gamma function is entire 
in the first argument, see~\cite[pp. 899, 1032ff.]{gradshteyn2015table}). 
We can now evaluate the integral in~\eqref{eq:integral_gamma_function} using several identities for Meijer's $G$-function. 
First, inserting the expression for the incomplete Gamma-function into~\eqref{eq:integral_gamma_function} gives
\begin{align*}
	\xi^{a+1} \int_0^1 s^{c-1} (1-s)^{b-1} \MeijerGnew{2}{0}{1}{2}{1}{-a-1,0}{\frac{\xi}{s}} \dd s.
\end{align*}
Next we apply the inversion identity for Meijer's $G$-function (see \cite[p. 209, 5.3.1.(9))]{erdelyi1953higher}) to get
\begin{align*}
	\xi^{a+1} \int_0^1 s^{c-1} (1-s)^{b-1} \MeijerGnew{0}{2}{2}{1}{2+a,1}{0}{\frac{s}{\xi}} \dd s.
\end{align*}
This expression is actually the Euler transform of Meijer's $G$-function (see \cite[p. 214, 5.5.2.(5)]{erdelyi1953higher}) and (as the conditions $2+1<2(0+2)$ and $|\arg(\xi^{-1})| < \frac{\pi}{2}$ (as $\xi>0$) and $1-c-b<1-c$ (as $b>0$)  are satisfied) it equals
\begin{align*}
	\xi^{a+1} \Gamma(b) \MeijerGnew{0}{3}{3}{2}{1-c,2+a,1}{0,1-c-b}{\xi^{-1}}.
\end{align*}
Using again the inversion identity for Meijer's $G$-function we now get
\begin{align*}
	\xi^{a+1} \Gamma(b) \MeijerGnew{3}{0}{2}{3}{1,b+c}{c,-1-a,0}{\xi}.
\end{align*}
Finally, plugging in $a= 6\alpha-k-3$, $b=2\alpha$, $c=3-4\alpha$ we obtain
\begin{equation}\label{eq:integral_Delta_P_Beta}
	\int_0^1 z^a e^{-\xi/z} B^-(1-z;b,c) \dd z =
	\xi^{6a-k-2} \Gamma(2\alpha) \MeijerGnew{3}{0}{2}{3}{1,3-2\alpha}{3-4\alpha,-6\alpha+k+2,0}{\xi}.
\end{equation}

Using equation~\eqref{eq:integral_Delta_P_z}, \eqref{eq:integral_Delta_P_1_z} and~\eqref{eq:integral_Delta_P_Beta} we get
\begin{align*}
	I^{(k)}
	&=\frac{\xi^{2\alpha}}{4k!(\alpha-1)} \left( -\Gamma^+(k - 2 \alpha, \xi) 
		- 2\frac{\alpha (\alpha - 1/2)^2 \xi^{2} \Gamma^+(k - 2 \alpha - 2, \xi)}{(\alpha - 1)} \right. \\ 
	&\hspace{15pt}\left.+ 8 \alpha (\alpha - 1/2) \xi \Gamma^+(k - 2 \alpha - 1,\xi) \right.\\ 
	&\hspace{15pt}\left.+ 4\xi^{4\alpha - 2} 
		\Gamma^+(k - 6 \alpha + 2, \xi) \left( 
		\frac{2^{ - 4\alpha}(3 \alpha - 1)}{(\alpha - 1)} + (\alpha - 1/2) B^-(1/2; 1 + 2 \alpha, -2 + 2 \alpha) \right)  \right.\\ 
	&\hspace{15pt}\left.+ \xi^{k-2\alpha} \Gamma(2\alpha+1)e^{-\xi} 	
		U(2\alpha+1,1+k-2\alpha,\xi) \right. \\ 
	&\hspace{15pt}\left.- \xi^{4\alpha-2} 
		\Gamma(2\alpha+1)\MeijerGnew{3}{0}{2}{3}{1,3-2\alpha}{3-4\alpha,-6\alpha+k+2,0}{\xi}  \right)\\
\end{align*}

With the expressions for $J$ and $I^{(k)}$ and using $\Gamma^\ast(q,z) = \Gamma^+(q+1,z) + \Gamma^+(q,z)$ we now obtain, after 
some algebra, the expression for $\gamma$
\begin{align*}
	\gamma &= J-I^{(0)}-I^{(1)} \\
	&=\frac{2 + 4 \alpha + 13 \alpha^2 - 34 \alpha^3 - 12\alpha^4 + 
	24 \alpha^5}{16(\alpha-1)^2 \alpha (\alpha+1) (2\alpha+1)} +  \frac{2^{-1 - 
		4 \alpha}}{(\alpha - 1)^2} \\
	&\hspace{10pt}+ \frac{(\alpha - 1/2) (B(2 \alpha, 2 \alpha + 1) + 
	B^-(1/2; 1 + 2 \alpha, -2 + 2 \alpha))}{2 (\alpha - 1) (3 \alpha - 1)} \\
	&\hspace{10pt}-\frac{\xi^{2\alpha}}{4(\alpha-1)} \left( -\Gamma^+( - 2 \alpha, \xi) - 2\frac{\alpha (\alpha - 1/2)^2 \xi^{2} 
	\Gamma^+(- 2 \alpha - 2, \xi)}{(\alpha - 1)} \right. \\ 
	&\hspace{10pt}\hspace{10pt}\left.+ 8 \alpha (\alpha - 1/2) \xi \Gamma^+( - 2 \alpha - 1,\xi) \right.\\ 
	&\hspace{10pt}\hspace{10pt}\left.+ 4\xi^{4\alpha - 2} \Gamma^+( - 6 \alpha + 2, 
      \xi) \left( \frac{2^{ - 4\alpha}(3 \alpha - 1)}{(\alpha - 1)} + (\alpha - 1/2) B^-(1/2; 1 + 2 \alpha, -2 + 2 \alpha) \right)  \right.\\ 
	&\hspace{10pt}\hspace{10pt}\left.+ \xi^{-2\alpha} \Gamma(2\alpha+1)e^{-\xi} 
		U(2\alpha+1,1-2\alpha,\xi) \right. \\ 
	&\hspace{10pt}\hspace{10pt}\left.- \xi^{4\alpha-2} 
		\Gamma(2\alpha+1)\MeijerGnew{3}{0}{2}{3}{1,3-2\alpha}{3-4\alpha,-6\alpha+2,0}{\xi}  \right) \\
	&\hspace{10pt}-\frac{\xi^{2\alpha}}{4(\alpha-1)} \left( -\Gamma^+(1 - 2 \alpha, \xi) - 
		2\frac{\alpha (\alpha - 1/2)^2 \xi^{2} \Gamma^+( - 2 \alpha - 1, \xi)}{(\alpha - 1)} \right. \\ 
	&\hspace{10pt}\hspace{10pt}\left.+ 8 \alpha (\alpha - 1/2) \xi \Gamma^+(1 - 2 \alpha - 1,\xi) 
		\right.\\ 
	&\hspace{10pt}\hspace{10pt}\left.+ 4\xi^{4\alpha - 2} \Gamma^+(1 - 6 \alpha + 2, 
    	\xi) \left( \frac{2^{ - 4\alpha}(3 \alpha - 1)}{(\alpha - 1)} + (\alpha - 1/2) B^-(1/2; 1 + 2 \alpha, -2 + 2 \alpha) \right)  \right.\\ 
	&\hspace{10pt}\hspace{10pt}\left.+ \xi^{1-2\alpha} \Gamma(2\alpha+1)e^{-\xi} 
		U(2\alpha+1,2-2\alpha,\xi) \right. \\ 
	&\hspace{10pt}\hspace{10pt}\left.- \xi^{4\alpha-2} 
		\Gamma(2\alpha+1)\MeijerGnew{3}{0}{2}{3}{1,3-2\alpha}{3-4\alpha,-6\alpha+3,0}{\xi}  \right)\\
	&=\frac{2 + 4 \alpha + 13 \alpha^2 - 34 \alpha^3 - 12\alpha^4 + 24 \alpha^5}
		{16(\alpha-1)^2 \alpha (\alpha+1) (2\alpha+1)} 
		+  \frac{2^{-1 - 4 \alpha}}{(\alpha - 1)^2} \\
	&\hspace{10pt}+ \frac{(\alpha - 1/2) (B(2 \alpha, 2 \alpha + 1) + B^-(1/2; 1 + 2 \alpha, -2 + 2 \alpha))}
		{2 (\alpha - 1) (3 \alpha - 1)} \\
	&\hspace{10pt}+ \frac{\xi^{2\alpha} \Gamma^\ast( - 2 \alpha, \xi)}{4(\alpha-1)}
		+ \frac{\xi^{2\alpha + 2}\alpha (\alpha - 1/2)^2 \Gamma^\ast(- 2 \alpha - 2, \xi)}
		{2(\alpha-1)^2} \\
	&\hspace{10pt}- \frac{\xi^{2\alpha + 1}\alpha (2\alpha - 1) \Gamma^\ast( - 2 \alpha - 1,\xi)}{(\alpha-1)}
		- \frac{\xi^{6\alpha-2}2^{-4\alpha}(3\alpha - 1)\Gamma^\ast( - 6 \alpha + 2, \xi)}{(\alpha-1)^2}\\
	&\hspace{10pt}-\frac{\xi^{6\alpha - 2}(\alpha - 1/2) B^-(1/2; 1 + 2 \alpha, -2 + 2 \alpha)\Gamma^\ast( - 6 \alpha + 2, \xi)}{(\alpha-1)} \\
	&\hspace{10pt}- \frac{e^{-\xi} \Gamma(2\alpha+1) 
		\left(U(2\alpha+1,1-2\alpha,\xi) + U(2\alpha+1,2-2\alpha,\xi)\right)}{4(\alpha-1)} \\
	&\hspace{10pt}+ \frac{\xi^{6\alpha - 2} \Gamma(2\alpha+1)\left( 	
		\MeijerGnew{3}{0}{2}{3}{1,3-2\alpha}{3-4\alpha,-6\alpha+2,0}{\xi}
		+ \MeijerGnew{3}{0}{2}{3}{1,3-2\alpha}{3-4\alpha,-6\alpha+3,0}{\xi}\right)}{4(\alpha-1)}.
\end{align*}
which is the expression in Theorem~\ref{thm:clustering_coefficient_hyperbolic}.

Similarly, we get
\begin{align*}
	\gamma(k) &= \frac{I^{(k)}}{\pmf(k)} \\
	&=\frac{1}{8\alpha (\alpha-1)\Gamma^+(k-2\alpha,\xi)} \left( -\Gamma^+(k - 2 \alpha, \xi) 
		- 2\frac{\alpha (\alpha - 1/2)^2 \xi^{2} \Gamma^+(k - 2 \alpha - 2, \xi)}{(\alpha - 1)} \right. \\ 
	&\hspace{10pt}\left.+ 8 \alpha (\alpha - 1/2) \xi \Gamma^+(k - 2 \alpha - 1,\xi) \right.\\ 
	&\hspace{10pt}\left.+ 4\xi^{4\alpha - 2} \Gamma^+(k - 6 \alpha + 2, 
      \xi) \left( \frac{2^{ - 4\alpha}(3 \alpha - 1)}{(\alpha - 1)} + (\alpha - 1/2) B^-(1/2; 1 + 2 \alpha, -2 + 2 \alpha) \right)  \right.\\ 
	&\hspace{10pt}\left.+ \xi^{k-2\alpha} \Gamma(2\alpha+1)e^{-\xi} U(2\alpha+1,1+k-2\alpha,\xi) \right. \\ 
	&\hspace{10pt}\left.- \xi^{4\alpha-2} \Gamma(2\alpha+1)\MeijerGnew{3}{0}{2}{3}{1,3-2\alpha}{3-4\alpha,-6\alpha+k+2,0}{\xi}  \right),
\end{align*}
which equals the expression in Theorem~\ref{thm:local_clustering_hyperbolic}. 

\subsection{The proof of Proposition~\ref{prop:asymp}\label{ssec:asymptotics_local_clustering_P}}

Instead of extracting the scaling of $\gamma(k)$ from its explicit expression, it turns out to be more convenient to derive it using $P(y)$. Recall that
\[
	\gamma(k) = \frac{\int_0^\infty \rho(y,k) P(y) \alpha e^{-\alpha y} \dd y}{\int_0^\infty \rho(y,k) \alpha e^{-\alpha y} \dd y}.
\]
The asymptotic behavior for the denominator is given by~\eqref{eq:degree_distribution_P_asymptotics}. Hence, the main term to consider is the numerator
\[
	\int_0^{\infty} P(y) \, \rho(y,k) \alpha e^{-\alpha y} \, dy,
\]
and in particular the function $P(y)$. We therefore start with establishing the asymptotic behavior of the latter. First we combine~\eqref{eq:degree_distribution_P_asymptotics} and~\eqref{eq:general_integral_rho_y_k} to obtain the following scaling result
\begin{equation}\label{eq:general_clustering_integral_scaling}
	\frac{\int_0^\infty e^{-\beta y} \rho(y,k) \alpha e^{-\alpha y} \dd y}
	{\int_0^\infty \rho(y,k) \alpha e^{-\alpha y} \dd y}
	\sim \xi^{2\beta} k^{-2\beta}.
\end{equation}


\begin{proposition}[Asymptotic behavior of $P(y)$]\label{prop:asymptotics_P}
Let $\alpha > \frac{1}{2}$, $\nu > 0$ and $c_{\alpha,\nu}$ as defined in Proposition~\ref{prop:asymp} Then, as $y \to \infty$, 
\begin{enumerate}
\item for $\frac{1}{2} < \alpha < \frac{3}{4}$,
\[
	P(y) \sim e^{-\frac{y}{2}(4\alpha - 2)} c_{\alpha,\nu} \xi^{4\alpha - 2},
\]
\item for $\alpha = \frac{3}{4}$,
\[
	P(y) \sim \frac{y}{2} e^{-\frac{y}{2}},
\]
\item and for $\alpha > \frac{3}{4}$,
\[
	P(y) \sim e^{-\frac{y}{2}} \frac{\alpha - \frac{1}{2}}{\alpha - \frac{3}{4}}.
\]
\end{enumerate}
\end{proposition}

\begin{proof}
We shall deal with each of the three cases for $\alpha$ separately.

\paragraph{\bm{$1/2 < \alpha < 3/4$}}
By Lemma~\ref{lem:Paneq1} we get that
\begin{align*}
	e^{(4\alpha - 2)\frac{y}{2}}P(y) &= \frac{2^{-4 \alpha-1} (3 \alpha - 1)}{\alpha (\alpha - 1)^2} 
		+ \frac{(\alpha - \frac{1}{2} ) B^-(\frac{1}{2}; 1 + 2 \alpha, -2 + 2 \alpha)}{2(\alpha - 1) \alpha}
		- \frac{B^-(1-e^{-\frac{y}{2}}; 2\alpha, 3-4\alpha)}{4(\alpha - 1)} \\
	&\hspace{10pt}+ \frac{e^{(4\alpha - 2)\frac{y}{2}}}{8(\alpha-1)\alpha}\left((1 - e^{-\frac{y}{2}})^{2\alpha} - 1\right)
		+ \frac{\alpha-\frac{1}{2}}{\alpha-1} e^{(4\alpha-3)\frac{y}{2}}
		- \frac{(\alpha - \frac{1}{2})^2}{4(\alpha-1)^2} e^{4(\alpha-1)\frac{y}{2}}.
\end{align*}
Now consider again variable $z = e^{-y/2}$ and not that $z \to 0$ as $y \to \infty$. Because for any $b < 1$, $B^-(1-z: a,b)$ converges to $B(a,b) < \infty$ as $z \to 0$, we get that as $y \to \infty$, the first line is asymptotically equivalent to
\[
	\frac{3 \alpha - 1}{2^{4 \alpha+1} \alpha (\alpha - 1)^2} 
			+ \frac{(\alpha - 1/2 ) B^-(1/2; 1 + 2 \alpha, -2 + 2 \alpha)}{2(\alpha - 1) \alpha}
			- \frac{B(2\alpha, 3-4\alpha)}{4(\alpha - 1)} = c_{\alpha,\nu} \xi^{-(4\alpha - 2)},
\]
with $c_{\alpha,\nu}$ as defined in Proposition~\eqref{prop:asymp}. The proof now follows since for $1/2 < \alpha < 3/4$, the remaining three terms go to zero as $y \to \infty$. For the first of these terms this is true since 
\[
	e^{(4\alpha - 2)\frac{y}{2}}\left((1 - e^{-\frac{y}{2}})^{2\alpha} - 1\right) = \bigO{e^{(4\alpha - 2)\frac{y}{2}}e^{-\frac{y}{2}}}
	= \bigO{e^{(4\alpha - 3) \frac{y}{2}}} = \smallO{1}
\]
as $y \to \infty$ and $1/2 < \alpha < 3/4$.

\paragraph{\bm{$\alpha = 3/4$}}

Similar to the previous case we use Lemma~\ref{lem:Paneq1} to obtain (evaluating the expressions for $\alpha = 3/4$)
\begin{align*}
	\frac{2}{y} e^{\frac{y}{2}}P(y) 
	&= \frac{2}{y}B^-(1-e^{-\frac{y}{2}}; 3/2, 0)
		- \frac{4}{y}\frac{e^{\frac{y}{2}}\left((1 - e^{-\frac{y}{2}})^{3/2} - 1\right)}{3}
		- \frac{1}{y} - \frac{ e^{-\frac{y}{2}}}{4 y} \\
	&\hspace{10pt}+ \frac{2}{y} 
		\left(\frac{5}{3} 
		- \frac{2 B^-(\frac{1}{2}; 5/2, -1/2)}{3}\right)
\end{align*}
First we note that as $y \to \infty$,
\begin{equation}\label{eq:asymptotics_Delta_P_auxiliary}
	e^{\frac{y}{2}}\left(\left(1 - e^{-\frac{y}{2}}\right)^{3/2}-1\right) 
	\sim -\frac{3}{2},
\end{equation}
which implies that
\[
	\lim_{y \to \infty} \frac{4}{y}\frac{e^{\frac{y}{2}}\left((1 - e^{-\frac{y}{2}})^{3/2} - 1\right)}{3} = 0.
\]
We can now conclude that all terms in $\frac{2}{y} e^{\frac{y}{2}}P(y)$ except the first one are $\smallO{1}$ as $y \to \infty$. By writing $z = e^{-\frac{y}{2}}$ we can rewrite the first term as
\[
	\frac{2}{y}B^-(1-e^{-\frac{y}{2}}; 3/2, 0) 
	= -\frac{B^-(1-z; 3/2, 0)}{\log(z)} .
\]
Since $B^-(1-z,3/2,0) \sim - \log(z)$ as $z \to 0$, see Lemma~\ref{lem:asymptotics_incomplete_beta}, it now follows
that for $\alpha = 3/4$,
\[
	\lim_{y \to \infty} \frac{2}{y} B^-(1-e^{-\frac{y}{2}}; 3/2, 0)
	= \lim_{z \to 0} -\frac{1}{\log(z)} B^-(1-z; 3/2, 0) = 1.
\]
We therefore conclude that
\[
	P(y) \sim \frac{y}{2} e^{-\frac{y}{2}},
\]
as $y \to \infty$.

\paragraph{\bm{$\alpha > 3/4$}}
We first deal with the case $\alpha = 1$. Here it follows from Lemma~\ref{lem:Pais1} that
\begin{align*}
	e^{y/2}P(y) &= \frac{9}{4} + \frac{e^{y/2}\log(1-e^{-y/2})}{4} \\
	&\hspace{10pt}-\log(1-e^{-y/2}) + e^{-y/2}\left(\frac{3}{4}\log(1-e^{-y/2}) - \frac{7 + \pi^2}{8} + 	
		\frac{1}{2}\mathrm{Li}_2(e^{-y})\right) \\
	&= 2 + \frac{1}{4}\left(e^{y/2}\log(1-e^{-y/2}) + 1\right)\\
	&\hspace{10pt}-\log(1-e^{-y/2}) + e^{-y/2}\left(\frac{3}{4}\log(1-e^{-y/2}) - \frac{7 + \pi^2}{8} + 	
		\frac{1}{2}\mathrm{Li}_2(e^{-y})\right)
\end{align*}
The last three terms are $\smallO{1}$ as $y \to \infty$, while $2 = (\alpha - 1/2)/(\alpha - 3/4)$ for $\alpha = 1$.

Now we will deal with the case $\alpha > 3/4$ and $\alpha \not = 1$. For simplicity we write
\[
	Q_\alpha := \frac{2^{-4 \alpha-1} (3 \alpha - 1)}{\alpha (\alpha - 1)^2} 
		+ \frac{(\alpha - 1/2 ) B^-(1/2; 1 + 2 \alpha, -2 + 2 \alpha)}{2(\alpha - 1) \alpha}.
\]
Then, by Lemma~\ref{lem:Paneq1} we get
\begin{align*}
	e^{y/2} P(y) &= \frac{\alpha - \frac{1}{2}}{\alpha - 1} 
		+ \frac{e^{\frac{y}{2}}}{8(\alpha - 1)\alpha}\left(\left(1 - e^{-\frac{y}{2}}\right)^{2\alpha}-1\right)\\
	&\hspace{10pt}- e^{-(4\alpha - 3)\frac{y}{2}}\frac{B^-(1 - e^{-\frac{1}{2}y}; 2 \alpha, 3 - 4 \alpha)}{4 (\alpha - 1)}\\
	&\hspace{10pt}+ e^{-(4\alpha - 3)\frac{y}{2}}Q_\alpha + \frac{(\alpha-\frac{1}{2})^2}{4(\alpha-1)^2} e^{-\frac{y}{2}}.
\end{align*}
The first term is constant while the last two terms go to zero as $y \to \infty$. We will therefore focus on the remaining two terms. For the first we have, see~\eqref{eq:asymptotics_Delta_P_auxiliary} 
\[
	\frac{e^{\frac{y}{2}}}{8(\alpha - 1)\alpha}\left(\left(1 - e^{-\frac{y}{2}}\right)^{2\alpha}-1\right) 
	\sim \frac{-2\alpha}{8(\alpha - 1)\alpha} = -\frac{1}{4(\alpha -1)},
\]
as $y \to \infty$. Finally, writing $z = e^{-\frac{y}{2}}$ we get that
\[
	e^{-(4\alpha - 3)\frac{y}{2}} B^-(1 - e^{-\frac{1}{2}y}; 2 \alpha, 3 - 4 \alpha)
	= z^{4\alpha - 3} B^-(1 - z, 2 \alpha, 3 - 4 \alpha).
\]
Therefore it follows, see Lemma~\ref{lem:asymptotics_incomplete_beta}, that
\begin{align*}
	\lim_{y \to \infty} - e^{-(4\alpha - 3)\frac{y}{2}}
		\frac{B^-(1 - e^{-\frac{1}{2}y}; 2 \alpha, 3 - 4 \alpha)}{4 (\alpha - 1)}
	&= \lim_{z \to 0} z^{4\alpha - 3} \frac{B^-(1 - z; 2 \alpha, 3 - 4 \alpha)}{4(\alpha - 1)}\\
	&= \frac{1}{4(\alpha - 1)(4\alpha - 3)}.
\end{align*}
We conclude that as $y \to \infty$
\begin{align*}
	e^{y/2} P(y) 
	&\sim \frac{\alpha - \frac{1}{2}}{\alpha - 1} -\frac{1}{4(\alpha -1)} - \frac{1}{4(\alpha - 1)(4\alpha - 3)}
	= \frac{1 - 3\alpha + 2 \alpha^2}{(\alpha - 1)(\alpha - \frac{3}{4})} 
	= \frac{\alpha - \frac{1}{2}}{\alpha - \frac{3}{4}},
\end{align*}
which finishes the proof.
\end{proof}

With the asymptotic behavior of $P(y)$ we are ready to prove Proposition~\ref{prop:asymp}. Recall that for any $C > 0$ we defined
\[
	y_{k,C}^\pm = 2 \log\left(\frac{k \pm C \sqrt{k \log(k)}}{\xi} \vee 1\right).
\] 
and $\Kcal_C(k) = [y_{k,C}^-, y_{k,C}^+]$. Since $P(y) \le 1$ by the concentration of heights results (Proposition~\ref{prop:concentration_height_general}) we have that, as $k \to \infty$,
\begin{equation}\label{eq:asymptotics_average_clustering_concentration}
	\int_0^\infty P(y) \rho(y,k) \alpha e^{-\alpha y} \dd y
	= (1+\smallO{1}) \int_{y_{k,C}^-}^{y_{k,C}^+} P(y) \rho(y,k) \alpha e^{-\alpha y} \dd y.
\end{equation}

Note that this implies that if $P(y) = h(y)(1 + \smallO{1})$ as $y \to \infty$, then 
\begin{equation}\label{eq:asymptotics_average_clustering_error_term}
	\int_0^\infty P(y) \rho(y,k) \alpha e^{-\alpha y} \dd y
	\sim \int_0^\infty h(y) \rho(y,k) \alpha e^{-\alpha y} \dd y,
\end{equation}
as $y \to \infty$.

\begin{proofof}{Proposition~\ref{prop:asymp}}
We split the proof over the different cases for $\alpha$.

\paragraph{\bm{$1/2 < \alpha < 3/4$}}
By Proposition~\ref{prop:asymptotics_P} and~\eqref{eq:asymptotics_average_clustering_error_term}  it follows that as $k \to \infty$,
\begin{align*}
	\gamma(k) \sim c_{\alpha,\nu} \xi^{-(4\alpha - 2)} \frac{\int_0^{\infty} e^{-(4\alpha - 2)y/2} \rho(y,k) \alpha  e^{-\alpha y} \dd y}
		{\int_0^\infty \rho_{y}(k) \alpha e^{-\alpha y} \dd y} 
	\sim c_{\alpha,\nu} k^{-4\alpha + 2}.
\end{align*}
where the last line is due to~\eqref{eq:general_clustering_integral_scaling} with $\beta = 2\alpha - 1$.

\paragraph{\bm{$\alpha = 3/4$}}
Similar to the previous case Proposition~\ref{prop:asymptotics_P} and~\eqref{eq:asymptotics_average_clustering_error_term} imply that as $k \to \infty$
\[
	\gamma(k) = \frac{\int_0^{\infty} P(y) \rho(y,k) \alpha  e^{-\alpha y} \dd y}
	{\int_0^\infty \rho_{y}(k) \alpha e^{-\alpha y} \dd y}
	\sim  \frac{\int_0^{\infty} \frac{y}{2} e^{-y/2} \rho(y,k) \alpha  e^{-\alpha y} \dd y}
		{\int_0^\infty \rho_{y}(k) \alpha e^{-\alpha y} \dd y}.
\]
However, the final step does not follow immediately from~\eqref{eq:general_clustering_integral_scaling} because of the additional logarithmic term. 
To deal with this we observe the following upper and lower bound
\begin{align*}
	\int_{\Kcal_C(k)} \frac{y}{2} e^{-y/2} \rho(y,k) \alpha e^{-\alpha y} \dd y
	&\le \frac{y_{k,C}^+}{2} \int_{\Kcal_C(k)} e^{-y/2} \rho(y,k) \alpha e^{-\alpha y} \dd y
\end{align*}
and similarly, a lower bound
\begin{align*}
	\int_{\Kcal_C(k)} \frac{y}{2} e^{-y/2} \rho(y,k) \alpha e^{-\alpha y} \dd y
	&\ge \frac{y_{k,C}^-}{2} \int_{\Kcal_C(k)} e^{-y/2} \rho(y,k) \alpha e^{-\alpha y} \dd y
\end{align*}
Now observe that as $k \to \infty$,
\[
	\frac{y_{k,C}^\pm}{2} = \log\left(\frac{k \pm \sqrt{k\log(k)}}{\xi}\right) \sim \log(k)
\]
It follows that
\[
	\limsup_{k \to \infty} \frac{\int_{\Kcal_C(k)} \frac{y}{2} e^{-y/2} \rho(y,k) \alpha e^{-\alpha y} \dd y}
	{\log(k) \int_{\Kcal_C(k)} e^{-y/2} \rho(y,k) \alpha e^{-\alpha y} \dd y} \le 1.
\]
and
\[
	\liminf_{k \to \infty} \frac{\int_{\Kcal_C(k)} \frac{y}{2} e^{-y/2} \rho(y,k) \alpha e^{-\alpha y} \dd y}
	{\log(k) \int_{\Kcal_C(k)} e^{-y/2} \rho(y,k) \alpha e^{-\alpha y} \dd y} \ge 1,
\]
which imply
\begin{equation}\label{eq:average_clustering_proof_34_main}
	\int_{\Kcal_C(k)} \frac{y}{2} e^{-y/2} \rho(y,k) \alpha e^{-\alpha y} \dd y
	\sim \log(k) \int_{\Kcal_C(k)} e^{-y/2} \rho(y,k) \alpha e^{-\alpha y} \dd y,
\end{equation}
as $k \to \infty$.

Since~\eqref{eq:general_clustering_integral_scaling} with $\beta = 1/2$ implies
\[
	\frac{\int_0^\infty e^{-y/2} \rho(y,k) \alpha e^{-\alpha y} \dd y}{\int_0^\infty \rho(y,k) \alpha e^{-\alpha y} \dd y}
	\sim \xi k^{-1},
\]
it follows from~\eqref{eq:average_clustering_proof_34_main} that as $k \to \infty$,
\begin{align*}
	\gamma(k) &\sim \frac{\int_0^\infty \frac{y}{2} e^{-y/2} \rho(y,k) \alpha e^{-\alpha y} \dd y}
		{\int_0^\infty \rho(y,k) \alpha e^{-\alpha y} \dd y} \\
	&\sim \log(k) \frac{\int_0^\infty e^{-y/2} \rho(y,k) \alpha e^{-\alpha y} \dd y}
	{\int_0^\infty \rho(y,k) \alpha e^{-\alpha y} \dd y}
	\sim \xi \log(k)k^{-1} = \frac{6\nu}{\pi} \log(k) k^{-1},
\end{align*}
when $\alpha = 3/4$.

\paragraph{\bm{$\alpha > 3/4$}}

Again, by Proposition~\ref{prop:asymptotics_P}, equation~\eqref{eq:asymptotics_average_clustering_error_term} and~\eqref{eq:general_clustering_integral_scaling} with $\beta = 1/2$, it follows that as $k \to \infty$,
\[
	\gamma(k) \sim \frac{\alpha - \frac{1}{2}}{\alpha - \frac{3}{4}} \, \frac{\int_0^{\infty} e^{-y/2} \rho(y,k) \alpha  e^{-\alpha y} \dd y} {\int_0^\infty \rho(y,k) \alpha e^{-\alpha y} \dd y} 
	\sim \frac{\alpha - \frac{1}{2}}{\alpha - \frac{3}{4}} \, \xi k^{-1} = \frac{8\alpha \nu}{\pi(4\alpha - 3)}.
\]
\end{proofof}

\section{Proofs of Theorem~\ref{thm:maincc} and Theorem~\ref{thm:mainkfixed}\label{sec:proofs_fixed_k}}

We will first derive Theorem~\ref{thm:mainkfixed}. It will turn out that Theorem~\ref{thm:maincc} has a quick derivation
assuming Theorem~\ref{thm:mainkfixed}.

\subsection{Clustering function for fixed \texorpdfstring{$k$}{k}, proving Theorem~\ref{thm:mainkfixed}}

We will now show that the clustering function of the KPKVB model $c(k;G_n) \xrightarrow{\Pee} \gamma(k)$ for a fixed $k$. 
The key ideas are that the coupling of the Poissonized KPKVB model with the box model is guaranteed to be exact 
(in the sense that it also preserves edges) for all vertices up to height $R/4$; and that when computing 
the expected value clustering function $c(k;G)$ in the subgraph of the box model induced by all vertices up to height $R/4$
using the Campbell-Mecke formula we obtain integrals that are very similar to the expressions we found earlier for $\gamma(k)$.


We will repeatedly rely on the following observation.
When we are simultaneously considering two graphs $G,H$
we will use the notations $N_G(k), N_H(k)$ to 
denote the number of vertices of degree $k$ in $G$, respectively $H$.

\begin{lemma}\label{lem:ckGckH}
Let $k\geq 2$ and let $G, H$ be graphs such that $G$ is an induced subgraph of $H$, or vice versa.
Then

$$ \left| c(k;G) - c(k;H) \right| \leq \frac{6|E(G)\Delta E(H)|}{N_G(k) - 2|E(G)\Delta E(H)|}, $$

\noindent
provided $N_G(k) > 2|E(G)\Delta E(H)|$.
\end{lemma}

\begin{proof} 
We observe that
 
$$ \begin{array}{rcl} 
|c(k; G)-c(k;H)| 
& = & \displaystyle
\left| \sum_{\substack{v\in V(G), \\ \text{deg}_G(v)=k}} \frac{c_G(v)}{N_G(k)} - 
\sum_{\substack{v\in V(H), \\ \text{deg}_{H}(v)=k}} \frac{c_{H}(v)}{N_{H}(k)} \right| \\[8ex]
& = & \displaystyle 
\left| \frac{1}{N_G(k)} \left( \sum_{\substack{v\in V(G)\setminus V(H), \\ \text{deg}_G(v)=k} } c_G(v)
+ \sum_{\substack{v\in V(G)\cap V(H), \\ \text{deg}_G(v)=k, \\ \text{deg}_{H}(v)\neq k}} c_G(v) \right) \right. \\
& & \displaystyle
- \frac{1}{N_{H}(k)} \left(\sum_{\substack{ v\in V(H)\setminus V(G), \\ \text{deg}_{H}(v)=k} } c_{H}(v)
+ \sum_{\substack{v\in V(G)\cap V(H), \\ \text{deg}_G(v)\neq k, \\ \text{deg}_{H}(v)=k}} c_{H}(v) \right) \\
& & \displaystyle  
\left. + \left(\frac{1}{N_G(k)}-\frac{1}{N_{H}(k)} \right) \cdot 
\left( \sum_{\substack{v \in V(G)\cap V(H), \\ \text{deg}_G(v)=\text{deg}_{H}(v)=k}}
c_G(v) \right) \right| \\[8ex]
 & \leq & \displaystyle
 \frac{2|E(G)\Delta E(H)|}{N_G(k)} + \frac{2|E(G)\Delta E(H)|}{N_{H}(k)} + 
 \frac{|N_{H}(k) - N_G(k)|}{N_G(k)\cdot N_{H}(k)} \cdot N_{H}(k) \\[5ex]
 & = & \displaystyle 
 \frac{2|E(G)\Delta E(H)|}{N_G(k)} + \frac{2|E(G)\Delta E(H)|}{N_{H}(k)} + \frac{|N_{H}(k) - N_G(k)|}{N_G(k)} \\[5ex]
 & \leq & \displaystyle
 \frac{2|E(G)\Delta E(H)|}{N_G(k)} + \frac{2|E(G)\Delta E(H)|}{N_{H}(k)} + \frac{2|E(G)\Delta E(H)|}{N_G(k)} \\[5ex]
 & \leq & \displaystyle
\frac{6 |E(G)\Delta E(H)|}{N_G(k)-2|E(G)\Delta E(H)|}.
\end{array} $$

\noindent
(In the second line we use that $\text{deg}_G(v) = \text{deg}_{H}(v)$ in fact 
implies that $c_G(v) = c_{H}(v)$ since one of $G,H$ is an induced subgraph of the other. 
In the third line we use that clustering coefficients $c_G(v), c_H(v)$ take values in $[0,1]$, and if either 
$\text{deg}_{G}(v) \neq \text{deg}_{H}(v)$ or $v \in V(G)\Delta V(H)$ and $v$ has degree $K$ in whichever of $G, H$ it belongs to 
then at least one edge of $E(G)\Delta E(H)$ is incident with $v$, and that every edge in 
$E(G)\Delta E(H)$ only affects the status of its two incident vertices.
For the fifth line we used that $|N_G(k)-N_H(k)| \leq |\{ v\in V(G) : \text{deg}_G(v) = k \} \Delta 
\{ v\in V(H) : \text{deg}_H(v) = k \}| \leq 2|E(G)\Delta E(H)|$ for similar reasons.
In the last line we used that $N_H(k) \geq N_G(k) - |N_G(k)-N_H(k)| \geq N_G(k) - 2|E(G)\Delta E(H)|$.)
\end{proof}

\begin{lemma}\label{lem:NDGnGPo}
$|E(G_n) \Delta E(\GPo) | = o(n)$ a.a.s. 
\end{lemma}

\begin{proof}
Let us fix some $\eps>0$ and write 

$$G_- := G((1-\eps)n, \alpha, (1-\eps)\nu ), \quad G_+ := G((1+\eps)n, \alpha, (1+\eps)\nu ). $$

\noindent
(We ignore rounding issues, i.e.~the issue that $(1-\eps)n, (1+\eps)n$ may not be integers, 
to avoid notational burden. We leave the straightforward details of adapting our arguments below to deal with it to the reader.)

Observe that the vertices of $G_-, G_+, G_n, \GPo$ all live on the same hyperbolic disk, of radius 
$R = 2\ln(n/\nu)$. 
We consider the standard coupling where we have an infinite supply of i.i.d.~points $u_1, u_2, \dots$ chosen according
to the $(\alpha,R)$-quasi uniform distribution, the vertices of $G_n = G(n;\alpha,\nu)$ are $u_1,\dots, u_n$, 
the vertices of $G_-$ are $u_1,\dots, u_{(1-\eps)n}$, the vertices of $G_+$ are $u_1,\dots, u_{(1+\eps)n}$ and 
the vertices of $\GPo$ are $u_1,\dots, u_N$ with $N\isd \Po(n)$ independently of
$u_1, u_2, \dots$.

As $N\isd\Po(n)$, by Chebychev's inequality, we have that $|N-n| < \eps n$ a.a.s.
So in particular, under our coupling we have $G_- \subseteq G_n, \GPo \subseteq G_+$ a.a.s. 
We now point out that, by the results of Gugelmann et al.~on the average degree (\cite{gugelmann2012random}, Theorem 2.3), 
we have 

$$|E(G_-)| = (1-\eps)^2\cdot \frac{4\nu\alpha^2}{\pi(2\alpha-1)^2} \cdot n + o(n), \quad 
|E(G_+)| = (1+\eps)^2 \cdot \frac{4\nu\alpha^2}{\pi(2\alpha-1)^2} \cdot n + o(n) \quad \text{ a.a.s. } $$

It follows that 

$$ | E(G_n) \Delta E(\GPo) | \leq |E(G_+)\setminus E(G_-)|
= \eps \cdot \frac{16\nu\alpha^2}{\pi(2\alpha-1)^2}\cdot n + o(n) \text{ a.a.s. }
$$

This holds for every fixed $\eps>0$. Sending $\eps\searrow 0$, concludes the proof of the lemma.
\end{proof}

Next, let us recall that by the results of Gugelmann et al.~on the degree sequence
(\cite{gugelmann2012random}, Theorem 2.2) we have that 

\begin{equation}\label{eq:gugeldegseqaap} 
\frac{N_{G_n}(k)}{n} \xrightarrow[n\to\infty]{\Pee} \pmf(k), 
\end{equation}

\noindent
for every fixed $k$.
In particular $N_{G_n}(k) = \Omega(n)$ a.a.s. Combining this with lemmas~\ref{lem:ckGckH} and~\ref{lem:NDGnGPo} we obtain:

\begin{corollary}\label{cor:pok} 
For every fixed $k\geq 2$, we have

$$ c(k; G_n)-c(k; \GPo) \xrightarrow[n\to\infty]{\Pee} 0, \quad\text{ and }\quad
\frac{N_{\GPo}(k)}{n} \xrightarrow[n\to\infty]{\Pee} \pmf(k). $$

\end{corollary}

\noindent
(For the second statement we use that $N_{G_n}(k) - 2|E(G_n)\Delta E(\GPo)| \leq N_{\GPo}(k) \leq N_{G_n}(k) + 2|E(G_n)\Delta E(\GPo)|$.)

In the remainder of this section, we will denote by $\GboxH$ the subgraph of $\Gbox$ induced by all 
vertices $(x,y) \in \Vbox = \Pcal \cap \Rcal$ of height at most $R/4$.

\begin{lemma}
Under the coupling provided by Lemma~\ref{lem:coupling_hyperbolic_poisson}, a.a.s., 
$\GboxH$ is an induced subgraph of $\GPo$ and $|E(\GPo) \setminus E(\GboxH)| = o(n)$.
\end{lemma}

\begin{proof}
We remind the reader that under the coupling of Lemma~\ref{lem:coupling_hyperbolic_poisson}, 
we can view $\Gbox$ and $\GPo$ as having the same vertex set $\Vbox = \Pcal \cap \Rcal$; and 
two points $p = (x,y), p' = (x',y') \in \Vbox$ are joined by an edge in $\Gbox$ if 
$|x-x'|_{\pi e^{R/2}} \leq e^{(y+y')/2}$, while $p,p'$ are joined by an edge in $\GPo$ 
if either $y+y'\geq R$ or $y+y'<R$ and $|x-x'|_{\pi e^{R/2}} \leq \Phi(y,y')$ with $\Phi$ as provided by~\eqref{eq:def_Omega_hyperbolic}.
It follows immediately from Lemma~\ref{lem:coupling_edges} that $\GboxH$ is an induced subgraph of $\GPo$, a.a.s., as claimed.

Fix $\eps > 0$, and let $X$ denote the number points of $\Vbox$ with $y$-coordinate $\geq (1-\eps) R$. 
Then $X$ is a Poisson random variable with mean

$$ \begin{array}{rcl} 
\Ee X 
& = & \displaystyle \mu\left( (-\frac{\pi}{2} e^{R/2},\frac{\pi}{2} e^{R/2}] \times [(1-\eps) R,R] \right)
= \int_{-\frac{\pi}{2} e^{R/2}}^{\frac{\pi}{2} e^{R/2}}\int_{(1-\eps)R}^R \left(\frac{\nu\alpha}{\pi}\right) e^{-\alpha y}\;dy\;dx \\[4ex]
& = & O( e^{R/2-(1-\eps)\alpha R} ) = o(1), 
\end{array} $$

\noindent
the last equality holding provided $\eps$ was chosen sufficiently small (using that $\alpha > 1/2$).
We conclude that, a.a.s., there are no vertices of height at least 
$(1-\eps) R$.

Let $Y$ denote the number of pairs of vertices $p = (x,y), p'=(x',y') \in \Vbox$ with $y+y' \geq R$.
Then, by the Campbell-Mecke formula

$$ \begin{array}{rcl} \Ee Y 
& = & \displaystyle
\int_\Rcal\int_\Rcal 1_{\{y+y'\geq R\}} \mu(dp')\mu(dp) \\[4ex]
& = & \displaystyle
\int_{-\frac{\pi}{2}e^{R/2}}^{\frac{\pi}{2}e^{R/2}}\int_0^R
\int_{-\frac{\pi}{2}e^{R/2}}^{\frac{\pi}{2}e^{R/2}}\int_{R-y}^R \left(\frac{\nu\alpha}{\pi}\right)^2 e^{-\alpha(y+y')}dy'dx'dydx \\[4ex]
& = & O( R e^{(1-\alpha)R} ) = o( n ),
\end{array} $$

\noindent
the last equality holding because $\alpha > 1/2$ and $n = \nu e^{R/2}$.
In particular, by Markov's inequality, $Y = o(n)$ a.a.s.

Now let $Z$ denote the number of pairs of vertices $p = (x,y), p'=(x',y')$
for which $y+y'<R, y < (1-\eps)R, R/4 \leq y' < (1-\eps)R$ and $|x-x'|_{\pi e^{R/2}} < \Phi(y,y')$.
By Lemma~\ref{lem:asymptotics_Omega_hyperbolic} we have that $\Phi(y,y') = O( e^{(y+y')/2} )$ 
for all such $y,y'$.
By Campbell-Mecke we can write

$$ \begin{array}{rcl} 
\Ee Z 
& = & \displaystyle
\int_{-\frac{\pi}{2}e^{R/2}}^{\frac{\pi}{2}e^{R/2}}\int_0^{(1-\eps)R}
\int_{-\frac{\pi}{2}e^{R/2}}^{\frac{\pi}{2}e^{R/2}}\int_{R/4}^{(1-\eps)R} 
1_{\{|x-x'|_{\pi e^{R/2}}<\Phi(y,y'), y+y'<R\}}\left(\frac{\nu\alpha}{\pi}\right)^2 e^{-\alpha(y+y')}dy'dxdydx' \\[4ex]
& = & \displaystyle
O\left( e^{R/2} \int_0^{(1-\eps)R} \int_{R/4}^{(1-\eps)R} e^{(1/2-\alpha)(y+y')} dy'dy \right) \\[4ex]
& = & \displaystyle
O\left( e^{R/2 + (1/2-\alpha) R/4} \right) = o(n).
\end{array} $$ 

\noindent
Hence also $Z = o(n)$ a.a.s.

This concludes the proof as we've now shown that under the stated coupling, a.a.s., $\GboxH$ and $\GPo$ differ by only $o(n)$ edges.
\end{proof}

Analogously to Corollary~\ref{cor:pok} we obtain:

\begin{corollary}\label{cor:GPoGboxH}
For every fixed $k\geq 2$ we have 

$$c(k;\GPo) - c(k;\GboxH) \xrightarrow[n \rightarrow \infty]{\Pee} 0, \quad\text{ and }\quad  
\frac{N_{\GboxH}(k)}{n} \xrightarrow[n\to\infty]{\Pee} \pmf(k). $$

\end{corollary}

\begin{lemma}\label{lem:ckGboxH} 
For every fixed $k\geq 2$ we have

 $$ c(k;\GboxH) \xrightarrow[n \rightarrow \infty]{\Pee}  \gamma(k). $$

 \end{lemma}

 \begin{proof}
 We write $\Rcal_- := \Rcal \cap (\eR\times[0,R/4]) = (-\frac{\pi}{2}e^{R/2}, \frac{\pi}{2} e^{R/2}] \times [0,R/4]$ and set 
 
 $$ 
 X := \sum_{v \in N_{\GboxH}(k)} c(v) = \sum_{v \in \Pcal \cap \Rcal_-} c_{\GboxH}(v) \cdot 1_{\{\text{deg}_{\GboxH}(v)=k\}}.
 $$
 
 \noindent
 By the Campbell-Mecke formula
 
 $$ \Ee X = \int_{\Rcal_-} 
 \Ee_\Pcal\left[ c_{\GboxH^z}(z) \cdot 1_{\{\text{deg}_{\GboxH^z}(z)=k\}}\right] \mu(dz), $$
 
 \noindent
 where $\GboxH^z$ denotes the graph we get by adding $z$ as an additional vertex to $\GboxH$, and adding edges between $z$ and the other 
 vertices as per the connection rule (for $\Gbox$).
 Spelling out the intensity measure $\mu$, plus symmetry considerations, gives
 
 $$ \begin{array}{rcl}
     \Ee X 
     & = & \displaystyle 
     \int_{-\frac{\pi}{2}e^{R/2}}^{\frac{\pi}{2}e^{R/2}} \int_0^{R/4}
     \Ee_\Pcal\left[ c_{\GboxH^{(x,y)}}((x,y)) \cdot 1_{\{\text{deg}_{\GboxH^{(x,y)}}((x,y))=k\}}\right] 
     \left(\frac{\nu\alpha}{\pi}\right) e^{-\alpha y}\;dy\;dx \\[6ex]
     & = & \displaystyle
     n \dot \int_0^{R/4}
     \Ee_\Pcal\left[ c_{\GboxH^{(0,y)}}((0,y)) \cdot 1_{\{\text{deg}_{\GboxH^{(0,y)}}((0,y))=k\}}\right] 
     \alpha e^{-\alpha y}\;d y \\[6ex]
     & = & \displaystyle
     n \cdot \int_0^{R/4} \Ee_\Pcal\left[ c_{\GboxH^{(0,y_0)}}((0,y_0)) {\Big|} \text{deg}_{\GboxH^{(0,y_0)}}((0,y_0))=k \right] \cdot \\[6ex]
     & & \displaystyle \hspace{6ex} 
     \Pee\left[ \text{deg}_{\GboxH^{(0,y_0)}}((0,y_0))=k  \right] \alpha e^{-\alpha y_0} \;dy_0.
    \end{array}
$$

\noindent
The random variable $\text{deg}_{\GboxH^{(0,y_0)}}((0,y_0))$ follows a Poisson distribution with mean 

$$ \begin{array}{rcl} 
\Ee\left[ \text{deg}_{\GboxH^{(0,y_0)}}((0,y_0))\right] 
& = & \displaystyle  
\mu(\BallPo{(0,y_0)}\cap\Rcal_-) 
= \int_0^{R/4} \int_{-e^{(y+y_0)/2}}^{e^{(y+y_0)/2}} \left(\frac{\nu\alpha}{\pi}\right) e^{-\alpha y}\;dx\;dy \\
& = & \xi e^{y_0/2} \cdot (1-e^{(1/2-\alpha)R/4}). 
\end{array} $$

\noindent
Hence, for every fixed $y_0$ and $k$, we have that 

$$ \Pee\left[ \text{deg}_{\GboxH^{(0,y_0)}}((0,y_0))=k\right] 
\xrightarrow[n\to\infty]{} 
\Pee( \Po(\xi e^{y_0/2})=k ) = \rho(y_0,k).$$


Next we remark that, analogously to the argument given in the beginning of Section~\ref{sec:42}, we have

$$ 
\Ee\left[ c_{\GboxH^{(0,y_0)}}((0,y_0)) {\Big|} \text{deg}_{\GboxH^{(0,y_0)}}((0,y_0))=k \right] 
= \Pee( w_1 \in \BallPo{w_2} ) =: P_n(y_0), $$

\noindent
with $w_1 = (x_1, y_1), w_2 = (x_2, y_2)$ chosen independently from $\BallPo{(0,y_0)} \cap \Rcal_-$ according to the probability 
measure we get by renormalizing $\mu$, i.e.~with pdf $f_\mu \cdot 1_{\BallPo{(0,y_0)} \cap \Rcal_-} / \mu(\BallPo{(0,y_0)}\cap \Rcal_-)$.
By considerations completely analogous to those following Lemma~\ref{lem:Paneq1}, the random variables $y_1, y_2$ both follow
a truncated exponential distribution with parameter $\alpha - 1/2$ truncated at height $R/4$
(i.e.~with density $1_{\{y_i \leq R/4\}} \cdot (\alpha-1/2) e^{(1/2-\alpha)y_i} / (1-e^{(1/2-\alpha)R/4})$) and, given the 
values of $y_0, y_1, y_2$, each $x_i$ is chosen uniformly on the interval $[-e^{(y_0+y_i)/2}, e^{(y_0,y_i)/2}]$.
In particular

$$ P_n(y_0) = 
\left(\frac{\alpha-1/2}{1-e^{(1/2-\alpha)R/4}}\right)^2 \int_0^{R/4}\int_0^{R/4} P(y_0, y_1, y_2) e^{(1/2-\alpha)(y_1+y_2)}\;dy_1\;dy_2,
$$

\noindent
with $P(.,.,y.)$ as defined in the paragraph following Lemma~\ref{lem:Paneq1}.
(That is, $P(y_0, y_1, y_2)$ is the probability that $|x_1-x_2| < e^{(y_1+y_2)/2}$, where $x_1, x_2$ are independent with $x_i$ uniform 
on the interval $[-e^{(y_0+y_i)/2}, e^{(y_0+y_i)/2}]$). 
It follows that, for any fixed $y_0$, we have

$$ P_n(y_0) \xrightarrow[n\to\infty]{} 
(\alpha-1/2)^2\int_0^\infty\int_0^\infty P(y_0, y_1, y_2) e^{(1/2-\alpha)(y_1+y_2)}\;dy_1\;dy_2 = P(y_0). $$

\noindent
(Applying monotone convergence to justify the convergence of the integral as $n \to \infty$.) 

Since (expected) clustering coefficients and probabilities are between zero and one and $\alpha e^{\alpha y_0}$ is integrable, we can now apply the dominated 
convergence theorem to obtain that 

\begin{equation}\label{eq:EXn} 
\frac{\Ee X}{n} \xrightarrow[n\to\infty]{} \int_0^\infty P(y_0) \rho(y_0,k)\alpha e^{-\alpha y_0}\;dy_0 = \pmf(k) \cdot \gamma(k).  
\end{equation}

\noindent
(Applying~\eqref{eq:gammakint} for the last equality.)

Next, we turn attention our to $X(X-1) = \sum_{u\neq v \in N_{\GboxH}(k)} c(v)c(u)$.
Another application of Campbell-Mecke shows that 

$$
\Ee X(X-1) 
= 
\int_{\Rcal_-}\int_{\Rcal_-} \Ee\left[ c_{\GboxH^{z,z'}}(z)c_{\GboxH^{z,z'}}(z') \cdot 
1_{\{\text{deg}_{\GboxH^{z,z'}}(z)=\text{deg}_{\GboxH^{z,z'}}(z')=k\}}\right]\mu(dz)\mu(dz'),
$$

\noindent
with $\GboxH^{z,z'}$ denoting the graph we get by adding $z, z'$ as additional vertices to $\GboxH$.
Now note that if $z=(x,y)$ and $z'=(x',y')$ satisfy $|x-x'|_{\pi e^{R/2}} > 2 e^{R/4}$ then 
the neighbourhoods of $z,z'$ are determined by the points of the Poisson process $\Pcal$ in disjoint areas
of the plane.
This implies that, provided $|x-x'|_{\pi e^{R/2}} > 2 e^{R/4}$:

\begin{equation}\label{eq:noot1} 
\begin{array}{c} 
\Ee\left[ c_{\GboxH^{z,z'}}(z)c_{\GboxH^{z,z'}}(z') \cdot 
1_{\{\text{deg}_{\GboxH^{z,z'}}(z)=\text{deg}_{\GboxH^{z,z'}}(z')=k\}}\right] \\
= \\
\Ee\left[ c_{\GboxH^{z}}(z) \cdot 
1_{\{\text{deg}_{\GboxH^{z}}(z)=k\}}\right] \cdot 
\Ee\left[ c_{\GboxH^{z'}}(z') \cdot 
1_{\{\text{deg}_{\GboxH^{z'}}(z')=k\}}\right].
\end{array} 
\end{equation}

\noindent
On the other hand, the LHS of~\eqref{eq:noot1} is always between zero and one, also if $|x-x'|_{\pi e^{R/2}} \leq 2 e^{R/4}$. 
We may conclude that 

$$ \begin{array}{rcl} 
\Ee X(X-1) 
& \leq & \displaystyle  
\int_{\Rcal_-}\int_{\Rcal_-} \Ee\left[ c_{\GboxH^{z}}(z) \cdot 
1_{\{\text{deg}_{\GboxH^{z}}(z)=k\}}\right] 
\cdot 
\Ee\left[ c_{\GboxH^{z'}}(z') \cdot 
1_{\{\text{deg}_{\GboxH^{z'}}(z')=k\}}\right] \mu(dz)\mu(dz') \\[8ex]
& & \displaystyle 
+ \int_{\Rcal_-}\int_{\Rcal_-} 1_{\{|x-x'|\leq 2e^{R/4}\}} \mu(dz)\mu(dz') \\[8ex]
& = & \displaystyle
\left( \int_{\Rcal_-} \Ee\left[ c_{\GboxH^z}(z) \cdot 1_{\{\text{deg}_{\GboxH^z}(z)=k\}}\right] \mu(dz) \right)^2
+ O( e^{3R/4} ) \\[8ex] 
& = & \displaystyle
\left(\Ee X\right)^2 + O( n^{3/2} ).
\end{array} $$

\noindent
Combining this with~\eqref{eq:EXn}, it follows that 
$\Var X = \Ee X(X-1) + \Ee X - \left(\Ee X\right)^2 = o\left( \left(\Ee X\right)^2 \right)$.
 By Chebychev's inequality, we therefore have 
 
 $$ X = n \cdot \gamma(k) \cdot \pmf(k) + o(n) \text{ a.a.s. } $$
 
In combination with Corollary~\ref{cor:GPoGboxH} (second limit) we can conclude that

$$ c(k;\GboxH) = \frac{X}{N_{\GboxH(k)}} \xrightarrow[n\to\infty]{\Pee} \gamma(k), $$

\noindent
as desired.
 \end{proof}

\begin{proofof}{Theorem~\ref{thm:mainkfixed}}
For completeness, we point out that Theorem~\ref{thm:mainkfixed} follows immediately from 
Corollaries~\ref{cor:pok},~\ref{cor:GPoGboxH} and Lemma~\ref{lem:ckGboxH}.
\end{proofof}

\subsection{Overall clustering coefficient, proving Theorem~\ref{thm:maincc}}

%
%
%

\begin{proofof}{Theorem~\ref{thm:maincc}}
Recall in Section~\ref{sec:Ginf},  we {\em defined} $\pmf(k) := \Pee(D=k), \gamma := \Ee C, \gamma(k) := \Ee(C|D=k)$ with $D$ the degree
and $C$ the clustering coefficient of the ``typical point'' in the infinite limit model $\Ginf$. We can write

$$ \gamma = \Ee C = \sum_{k\geq 2} \Ee\left( C | D=k \right) \Pee( D=k ) = \sum_{k\geq 2} \gamma(k) \cdot \pmf(k). $$

For the KPKVB random graph, or any graph for that matter, we have the similar relation

$$ c(G_n) = \sum_{k\geq 2} c(k;G_n) \cdot (N_{G_n}(k)/n). $$

By Theorem~\ref{thm:mainkfixed} and~\eqref{eq:gugeldegseqaap} we have, for any fixed $k\geq 2$:

\begin{equation}\label{eq:slutskyUB} 
 c(G_n) \geq \sum_{k=2}^K c(k;G_n) \cdot (N_{G_n}(k)/n) \xrightarrow[n\to\infty]{\Pee} \sum_{k=2}^K \gamma(k)\cdot \pmf(k), 
\end{equation}

\noindent
where Slutsky's theorem justifies the convergence in probability.
On the other hand we have 

\begin{equation}\label{eq:slutskyLB} 
c(G_n) \leq \sum_{k=2}^K c(k;G_n) \cdot (N_{G_n}(k)/n) + \sum_{k>K} (N_{G_n}(k)/n)
\xrightarrow[n\to\infty]{\Pee} \sum_{k=2}^K \gamma(k)\cdot \pmf(k) + \sum_{k>K} \pmf(k), 
\end{equation}

\noindent
where the convergence in probability can be justified using Slutsky's theorem together 
with the fact that $\sum_{k=0}^\infty \pmf(k) = 1$ (one convenient way to convince oneself that this is true, is to note that 
$D$, the degree of the typical point, is a.s.~finite). In more detail, 

$$\sum_{k>K} (N_{G_n}(k)/n) 
= 1 - \sum_{k=0}^K (N_{G_n}(k)/n) \xrightarrow[n\to\infty]{\Pee} 1 - \sum_{k=0}^K \pmf(k) = \sum_{k>K} \pmf(k). $$

The result follows from~\eqref{eq:slutskyLB} and~\eqref{eq:slutskyUB}, by sending $K\to\infty$. 
\end{proofof}

\section{Degrees when \texorpdfstring{$k\to\infty$}{k tends to infinity}: proof of Theorem~\ref{thm:degrees_hyperbolic}\label{sec:degrees}}

Since the new contribution of Theorem~\ref{thm:degrees_hyperbolic} concerns the cases where the degree $k_n \to \infty$, we 
will assume that this holds throughout this section.

\subsection{Proof overview}

We start by using the Campbell-Mecke formula to compare the degree distribution in $\GPo$ with that of the typical point in $\Ginf$. 
As we've already seen this equals 
\[
	\pmf(k) = \int_0^\infty \rho(y,k) \alpha e^{-\alpha y} \, dy.
\]
We will relate this to the Poissonized KPKVB model $\GPo$. 
More precisely, let $N_{\Po}(k)$ denote the set of degree $k$ vertices in $\GPo$. 
We then show in Section~\ref{ssec:expected_degrees_GPo} that for any $1 \ll k_n \le n -1$ with 
$k_n = \bigO{n^{\frac{1}{2\alpha + 1}}}$, $\Exp{N_{\Po}(k_n)} = (1+\smallO{1}) n \pmf(k_n)$ and more generally,

\begin{equation}\label{eq:def_factorial_moments_GPo}
	\Exp{\binom{N_{\Po}(k_n)}{r}} = (1+\smallO{1}) \frac{\Exp{N_{\Po}(k_n)}^r}{r!},
\end{equation}

\noindent
for any integer $r \ge 1$ in Section~\ref{ssec:factorial_moments_GPo}. 
The latter result requires us to analyze the joint degree distribution in $\GPo$, which we do in Section~\ref{ssec:joint_degrees_GPo}. 
The above result in particular implies concentration of $N_{\Po}(k_n)$ from which the result on the degree distribution in $\GPo$ follows for $k_n = \smallO{n^{\frac{1}{2\alpha + 1}}}$. When $k_n = (1+\smallO{1}) c n^{\frac{1}{2\alpha + 1}}$ we use the above result to show that the fraction of degree $k_n$ nodes in $\GPo$ converges to a Poisson distribution. 

To extend these results to $G_n$ we couple the construction of the KPKVB model to that of the Poissonized version $\GPo$ in Section~\ref{ssec:coupling_Gn_GPo} to show that a.a.s, $N_n(k_n) = (1+\smallO{1}) N_{\Po}(k_n)$. With these results we then prove Theorem~\ref{thm:degrees_hyperbolic} in Section~\ref{ssec:proof_thm_degrees}.

We will also establish all the above mentioned results for the finite box model $\Gbox$, since the proofs only require small 
alterations and we will need these results later on when analyzing the clustering coefficient and function.

\subsection{Expected degrees in \texorpdfstring{$\Gbox$}{G box} and \texorpdfstring{$\GPo$}{G Po}}\label{ssec:expected_degrees_GPo}

We proceed with the expected degrees in the finite box and Poissonized KPKVB model. Recall the definition of the neighbourhood balls $\BallPon{y}$ and $\BallHyp{y}$ of a point $(0,y)$ in, respectively $\Gbox$ and $\GPo$. We introduce the short hand notation
\[
	\mu_{\Po}(y) := \Mu{\BallHyp{y}} \quad \text{and} \quad \mu_{\mathrm{box}}(y) := \Mu{\BallPon{y}}.
\]

Our first results relate these measures to the measure $\mu(y)$, of the ball $\BallPo{y}$ in the infinite model $\Ginf$.

\begin{lemma}[Expected degree given height in $\GPo$]\label{lem:average_degree_P_n}
Let $\alpha > \frac{1}{2}$, $\varepsilon >0$ and $0 \leq y \leq (1-\epsilon)R$. Then as $n \to \infty$, uniformly in $y$,
\[
	\mu_{\Po}(y) = (1+\smallO{1}) \mu(y).
\]
\end{lemma}

\begin{proof}
Recall that when $y^\prime \ge R - y$ then $p^\prime \in \BallHyp{y}$ while for $y^\prime < R - y$ this is true when~\eqref{eq:def_Omega_hyperbolic} holds. We split the integral for $\mu_{Po,n}(y)$ accordingly, into two integrals $I_1$ and $I_2$,
\begin{align*}
\mu_{\Po}(y) 
&= \int_0^{R-y} 2\Phi(y,y_1)\frac{\alpha\nu}{\pi}e^{-\alpha y_1} dy_1+\int_{R-y}^R \frac{\pi n}{\nu} \frac{\alpha \nu}{\pi}e^{-\alpha y_1}dy_1 =: I_1+I_2.
\end{align*}
Firstly, we will show that the second integral $I_2=o(\mu(y))$ and then we will show that $I_1 = (1+o(1))\mu(y)$ (both with  convergence uniform in $0\leq y\leq (1-\epsilon)R$).

For the second integral $I_2$, we compute
\begin{align*}
I_2 &= \int_{R-y}^R \frac{\pi n}{\nu} \frac{\alpha \nu}{\pi}e^{-\alpha y_1}dy_1 =n(e^{-\alpha(R-y)}-e^{-\alpha R}) =ne^{-\alpha R}(e^{\alpha y} -1) =n^{1-2\alpha}\nu^{2\alpha}(e^{\alpha y} -1). 	
\end{align*}
To see that $n^{1-2\alpha}\nu^{2\alpha}(e^{\alpha y} -1) = o(\mu(y))$, recall that $\mu(y)=\xi e^{\frac{y}{2}}$. So, we need to show that
\begin{align*}
\frac{n^{1-2\alpha}\nu^{2\alpha}(e^{\alpha y} -1)}{\xi e^{\frac{y}{2}}} = o(1)
\end{align*}
or equivalently that
\begin{align*}
\frac{e^{\alpha y}-1}{\xi e^{\frac{y}{2}}}=\smallO{n^{2\alpha-1}}.
\end{align*}
For this, note that
\begin{align*}
\frac{e^{\alpha y}-1}{\xi e^{\frac{y}{2}}}=\bigO{e^{\alpha y-\frac{y}{2}}} = \bigO{e^{\left(\alpha-\frac{1}{2}\right)y}}.
\end{align*}
As $y \leq (1-\epsilon)R =(1-\epsilon)2\ln \frac{n}{\nu}$ and $\alpha > \frac{1}{2}$, we have 
\[
	e^{\left(\alpha-\frac{1}{2}\right)y} \leq e^{\left(\alpha-\frac{1}{2}\right)(1-\epsilon)R} = \left(\frac{n}{\nu}\right)^{2\left(\alpha-\frac{1}{2}\right)(1-\epsilon)} = \smallO{n^{2\alpha-1}},
\]
where the convergence is uniform in $y$, $0\leq y\leq (1-\epsilon)R$, as the last upper bound does not depend on $y$.

For the first integral $I_1$, we first recall from Lemma~\ref{lem:asymptotics_Omega_hyperbolic} that there is a positive constant $K$ such that for any $\varepsilon >0$, for all $y_1, y_2 \in [0,(1-\varepsilon)R]$, $y_1+y_2 <R$, we have
\[
	e^{\frac{1}{2}(y_1+y_2)} - Ke^{\frac{3}{2}(y_1+y_2)-R} \leq \Phi(y_1,y_2) \leq e^{\frac{1}{2}(y_1+y_2)} + Ke^{\frac{3}{2}(y_1+y_2)-R}.
\]
We thus define the main and error term of $I_1$ as
\begin{align*}
I_{1,main} &=\int_0^{R-y} 2e^{\frac{y+y_1}{2}}\frac{\alpha\nu}{\pi}e^{-\alpha y_1}dy_1, \\
I_{1,error} &= \int_0^{R-y} 2Ke^{\frac{3}{2}(y+y_1)-R} \frac{\alpha\nu}{\pi} e^{-\alpha y_1} dy_1.
\end{align*}
From the error bounds for $\Phi$ as given in Lemma~\ref{lem:asymptotics_Omega_hyperbolic}, it follows that
$$I_{1,main}-I_{1,error} \leq I_1 \leq I_{1,main}+I_{1,error}. $$
We will firstly show that $I_{1,main} =(1+o(1))\mu(y) $ and then that $I_{1,error} = o(\mu(y))$.

For the main term, we obtain, as $R-y \geq \epsilon R \rightarrow \infty$, uniformly in $0\leq y\leq (1-\epsilon)R$,
\begin{align*}
I_{1,main} &= \int_0^{R-y} 2e^{\frac{y+y_1}{2}}\frac{\alpha\nu}{\pi}e^{-\alpha y_1}dy_1 
		= \frac{2\alpha\nu}{\pi}e^{\frac{y}{2}}\int_0^{R-y} e^{\left(\frac{1}{2}-\alpha\right)y_1}dy_1\\
	&=\frac{2\alpha\nu}{\pi\left(\alpha-\frac{1}{2}\right)}e^{\frac{y}{2}}
		\left(1-e^{\left(\frac{1}{2}-\alpha\right)(R-y)}\right)
		=(1+o(1))\xi e^{\frac{y}{2}} = (1+o(1))\mu(y).
\end{align*}

For the error term, we obtain, for $\alpha \not = \frac{3}{2}$, uniformly in $0\leq y\leq (1-\epsilon)R$,
\begin{align*}
I_{1,error}=\int_0^{R-y} 2Ke^{\frac{3}{2}(y+y_1)-R} \frac{\alpha\nu}{\pi} e^{-\alpha y_1} dy_1 &= 2K\frac{\alpha\nu}{\pi} e^{\frac{3}{2}y-R} \int_0^{R-y} e^{\left(\frac{3}{2}-\alpha\right)y_1}dy_1 \\
&= \frac{2K\alpha\nu}{\pi\left(\frac{3}{2}-\alpha\right)}e^{\frac{3}{2}y-R} \left(e^{\left(\frac{3}{2}-\alpha\right)(R-y)}-1\right)\\
&= \frac{2K\alpha\nu}{\pi\left(\frac{3}{2}-\alpha\right)}e^{\frac{1}{2}y} \left(e^{\left(\frac{1}{2}-\alpha\right)(R-y)}-e^{-(R-y)}\right) =o\left(\xi e^{\frac{y}{2}}\right).
\end{align*}
For the error term with $\alpha=\frac{3}{2}$, uniformly in $0\leq y\leq (1-\epsilon)R$,
\begin{align*}
\int_0^{R-y} 3Ke^{\frac{3}{2}(y+y_1)-R} \frac{\nu}{\pi} e^{-\frac{3}{2} y_1} dy_1 &= 3K\frac{\nu}{\pi} e^{\frac{3}{2}y-R} \int_0^{R-y} dy_1 = \frac{3K\nu}{\pi}e^{\frac{3}{2}y-R} (R-y) =o\left(\xi e^{\frac{y}{2}}\right).
\end{align*}
We conclude that $I_{1,error} = o(\mu(y))$ and hence $I_{1,main}\pm I_{1,error} = (1+o(1))\mu(y)$, which finishes the proof.
\end{proof}

\begin{lemma}[Expected degree given height in $\Gbox$]\label{lem:average_degree_G_box}
Let $\alpha > \frac{1}{2}$, $\varepsilon >0$ and $0 \leq y \leq (1-\epsilon)R$. Then as $n \to \infty$, uniformly in $y$,
\[
	\mu_{\mathrm{box}}(y) = (1+\smallO{1}) \mu(y).
\]
\end{lemma}

\begin{proof}
First note that since we have identified the boundaries of $[-\frac{\pi}{2}e^{\frac{R}{2}}, \frac{\pi}{2}e^{\frac{R}{2}}]$ we can assume, without loss of generality, that $p = (0,y)$. We then have that the boundaries of $\BallPon{p}$ are given by the equations $x^\prime = \pm e^{\frac{y+y^\prime}{2}}$, which intersect the left and right boundaries of $[-\frac{\pi}{2}e^{\frac{R}{2}}, \frac{\pi}{2}e^{\frac{R}{2}}]$ at height
\[
	h(y) = R + 2 \log\left(\frac{\pi}{2}\right) - y.
\]
Therefore, if $y \le 2 \log(\pi/2)$ this intersection occurs above the height $R$ of the box $\Rcal$ while in the other case the full region of the box above $h(y)$ is connected to $p$. 

We will first consider the case where $y \le 2 \log(\pi/2)$. Here we have
\begin{align*}
	\mu(\BallPon{p})
	&= \int_0^{R} \int_{-e^{\frac{y+y^\prime}{2}}}^{e^{\frac{y+y^\prime}{2}}} 
		f(x^\prime,y^\prime) \, dx^\prime \, dy^\prime\\
	&= \frac{2 \alpha \nu}{\pi} e^{\frac{y}{2}} \int_0^{R} e^{-(\alpha - \frac{1}{2})y^\prime} \, dy^\prime\\
	&= \mu(y)\left(1 - e^{-(\alpha - \frac{1}{2})R}\right),
\end{align*}
where the error term is $\smallO{1}$, uniformly in $y$.

Now let $y > 2 \log(\pi/2)$ and recall that $\mu(y) = \xi e^{\frac{y}{2}}$ where $\xi = \frac{4\alpha \nu}{(2\alpha - 1)\pi}$. Then, after some simple algebra, we have that
\begin{align*}
	\mu_{\mathrm{box}}(y)
	&= \int_0^{h(y)} \int_{-\frac{\pi}{2}e^{\frac{R}{2}}}^{\frac{\pi}{2}e^{\frac{R}{2}}} 
		\ind{|x^\prime| \le e^{\frac{y+y^\prime}{2}}} f(x^\prime,y^\prime) \, dx^\prime \, dy^\prime
		+ \int_{h(y)}^{R} \int_{-\frac{\pi}{2}e^{\frac{R}{2}}}^{\frac{\pi}{2}e^{\frac{R}{2}}} 
		f(x^\prime,y^\prime) \, dx^\prime \, dy^\prime\\
	&= \frac{2 \alpha \nu}{\pi} e^{\frac{y}{2}} \int_0^{h(y)} e^{-(\alpha - \frac{1}{2})y^\prime} \, dy^\prime
		+ \alpha \nu e^{\frac{R}{2}} \int_{h(y)}^{R} e^{-\alpha y^\prime} \, dy^\prime \\
	&= \xi e^{\frac{y}{2}}\left(1 - \left(\frac{\pi}{2}\right)^{-(2\alpha - 1)} 
		e^{-(\alpha - \frac{1}{2})(R - y)}\right)
	+ \nu e^{\frac{R}{2}}\left(\left(\frac{\pi}{2}\right)^{-2\alpha} e^{-\alpha(R - y)} 
		- e^{-\alpha R}\right)\\
	&= \mu(y)\left(1 - \phi_n(y) \right).
\end{align*}
where 
\begin{equation}\label{eq:def_average_degree_Gbox_phi}
	\phi_n(y) :=  \left(\frac{\pi}{2}\right)^{-(2\alpha - 1)} \hspace{-3pt} e^{-(\alpha - \frac{1}{2})(R - y)}
				+ \frac{\nu}{\xi}e^{-(\alpha - \frac{1}{2})R - \frac{y}{2}} - \frac{\nu}{\xi}\left(\frac{\pi}{2}\right)^{-2\alpha} e^{-(\alpha-\frac{1}{2})(R - y)}.
\end{equation}
Since $R - y \ge \varepsilon R$ we have that $|\phi_n(y)|$ is uniformly bounded by
$\bigO{e^{-(\alpha - \frac{1}{2})\varepsilon R}}$, which is $\smallO{1}$ for $\alpha > \frac{1}{2}$. 
\end{proof}

We can now use a concentration of heights argument to show that the integration of the Poisson probabilities $\Prob{\Po(\mu_{\Po}(y)) = k_n}$ over $0 \le y \le (1-\varepsilon) R$ is asymptotically equivalent to $\pmf(k_n)$. And the same holds if we instead consider $\mu_{\mathrm{box}}(y)$. The proof contains some technical elements that are contained in the Appendix to not hinder the flow of the argument. 

\begin{lemma}\label{lem:degree_integral}
Let $0 < \varepsilon < 1$. Then for all $0 \le k_n \le n - 1$, as $n \to \infty$,

\begin{equation}\label{eq:degree_integral_BallHyp}
	\int_0^{(1-\varepsilon)R} \Prob{\Po(\Mu{\BallHyp{y}} = k_n} \alpha e^{-\alpha y} \dd y
	= (1+\smallO{1}) \pmf(k_n).
\end{equation}

Moreover, the same holds if we replace $\Mu{\BallHyp{y}}$ with $\Mu{\BallPon{y}}$ in~\eqref{eq:degree_integral_BallHyp}.
\end{lemma}

\begin{proof}
We will show that
\[
	\int_0^{(1-\varepsilon)R} \Prob{\Po(\mu_{\Po}(y)) = k_n} \alpha e^{-\alpha y} \dd y
	= (1+o(1))\int_0^\infty \Pee(\Po(\xi e^{\frac{z}{2}})=k_n) \alpha e^{-\alpha z} \dd z.
\]
This implies the result because the last integral equals $(1+o(1))2\alpha \xi^{2\alpha} k_n^{-(2\alpha+1)} = (1+o(1))\pmf(k_n)$. Where the last asymptotic equality follows from~\eqref{eq:degree_distribution_P_asymptotics}.

Define the function $z(y)=2\ln\frac{\mu_{\Po}(y)}{\xi}$ (note that $z(y)$ is well-defined as $\mu_{Po,n}(y)\geq 0$ and that $z(y)$ is bijective because $\mu_{\Po}(y)$ is strictly monotone increasing and continuous, see Lemma~3.3. in \cite{gugelmann2012random}). By rearranging, we have that
\begin{align*}
\mu_{\Po}(y) = \xi e^{\frac{z(y)}{2}}.
\end{align*}
From Lemma~\ref{lem:average_degree_P_n}, it follows that uniformly for all $0\leq y\leq (1-\varepsilon)R$, $\xi e^{\frac{y}{2}} = (1+o(1))\mu_{\Po}(y) = (1+o(1))\xi e^{\frac{z(y)}{2}}$, and hence that
\begin{align*}
e^{-\alpha y} = (1+o(1))e^{-\alpha z(y)}.
\end{align*}
Next we need a similar result regarding the derivative of $\mu_{\Po}(y)$, i.e.
\[
	\mu_{\Po}^\prime(y) = (1+o(1))\frac{1}{2}\mu(y) = (1+o(1))\frac{1}{2}\mu_{\Po}(y)
	= (1+\smallO{1}) \mu^\prime(y),
\]
uniformly for $0 \leq  y \leq (1-\varepsilon)R$. This result is given by Lemma~\ref{lem:derivative_mu_Po} in the Appendix. The lemma is placed there since the proof is a straightforward though cumbersome use of function analysis and we do not want to break the flow of the argument.

We now have that
\[
	z^\prime(y) = \frac{2\mu_{\Po}^\prime(y)}{\mu_{\Po}(y)} = 1+o(1).
\]
which implies that
\begin{align*}
	\int_0^{(1-\varepsilon)R} \Prob{\Po(\mu_{\Po}(y))=k_n} \alpha e^{-\alpha y}dy 
	= (1+o(1)) \int_0^{(1-\varepsilon)R} \Prob{\Po(\xi e^{\frac{z(y)}{2}})=k_n}\alpha e^{-\alpha z(y)} z^\prime(y) dy.
\end{align*}
We now apply integration by substitution to the integral, i.e. use the new variable $z=z(y)$, to obtain 
\begin{align*}
	\int_{z(0)}^{z((1-\varepsilon)R)} \Prob{\Po(\xi e^{\frac{z}{2}})=k_n}\alpha e^{-\alpha z} dz
	= \int_{z(0)}^{z((1-\varepsilon)R)} \Prob{\Po(\mu(z))=k_n}\alpha e^{-\alpha z} dz.
\end{align*}

Note that since the function $y \mapsto 2\ln \frac{y}{\xi}$ is monotone increasing it follows that for large enough $n$, 
$\Kcal_C(k_n)\subset [z(0),z((1-\epsilon)R)]$. Therefore, by a concentration of heights argument (Proposition~\ref{prop:concentration_height_general}) it follows that
\begin{align*}
	\int_{z(0)}^{z((1-\varepsilon)R)} \Prob{\Po(\mu(z))=k_n}\alpha e^{-\alpha z} dz 
	&= (1+o(1))\int_0^\infty \Pee(Po(\mu(z))=k_n)\alpha e^{-\alpha z}dz.
\end{align*}
and hence
\[
	\int_0^{(1-\varepsilon)R} \Prob{\Po(\mu_{\Po}(y))=k_n} \alpha e^{-\alpha y}dy
	= (1+o(1))\int_0^\infty \Pee(Po(\mu(z))=k_n)\alpha e^{-\alpha z}dz,
\]
which finishes the proof for $\mu_{\Po}(y)$.

The proof for $\mu_{\mathrm{box}}(y)$ follows similar arguments. First, we define $z(y) = 2 \log (\mu_{\mathrm{box}}(y)/\xi)$ and use Lemma~\ref{lem:average_degree_G_box} instead of Lemma~\ref{lem:average_degree_P_n} to establish that $e^{-\alpha y} = (1+\smallO{1})e^{-\alpha z(y)}$. For the derivative $z^\prime(y)$ we recall from the proof of Lemma~\ref{lem:average_degree_G_box} that $\mu_{\mathrm{box}}(y) = (1 + \phi_n(y)) \mu(y)$ with $\phi_n(y)$ given by~\eqref{eq:def_average_degree_Gbox_phi}. The derivative of $\phi_n(y)$ can be uniformly bounded by $\smallO{1}$ for $0 \le y \le (1-\varepsilon)R$. Hence we get $\mu_{\mathrm{box}}^\prime(y) = (1+\smallO{1}) \mu^\prime(y)$ and thus $z^\prime(y) = 1 + \smallO{1}$ uniformly for $0 \le y \le (1-\varepsilon)R$. We can now apply the same change of variables and a concentration of heights argument to arrive at the required statement.

\end{proof}

The main result of this section now follows almost immediately.

\begin{lemma}[First moment of number of degree $k$ vertices]\label{lem:expnnkn}
Let $N_{Po}(k)$ and $\Nbox(k)$ denote the number of vertices with degree $k$ vertices in the Poissonized KPKVB model $\GPo$ and the finite box model $\Gbox$, respectively. Consider a sequence of integers $k_n\rightarrow\infty$ with $0 \leq k_n \leq n-1$.	If $k_n = O\left(n^{\frac{1}{2\alpha+1}}\right)$, then as $n \to \infty$,
\[
	\Exp{N_{Po}(k_n)} =(1+o(1)) n \pmf(k_n) \quad \text{and} \quad \Exp{\Nbox(k_n)} =(1+o(1)) n \pmf(k_n).
\]
Moreover, if $k_n \gg n^{\frac{1}{2\alpha + 1}}$ then
\[
	\Exp{N_{\Po}(k_n)} = \smallO{1} \quad \text{and} \quad \Exp{\Nbox(k_n)} = \smallO{1}.
\]
\end{lemma}

\begin{proof}
We shall consider $\GPo$. The proof of the statements for $\Gbox$ follows using the same arguments and we omit it here. 

Let $D_{\Po}(p)$ denote the degree of a node $p \in \GPo$. Then since $N_{Po}(k_n)= \sum_{v\in V(\GPo)}\1_{\{D_{\GPo}(v)=k_n\}}$ and $\Prob{D_{\Po}(p) = k}$ is invariant under translations in the $x$-axis, we can apply the Campbell-Mecke formula to obtain
\begin{align*}
	\Exp{N_{Po}(k_n)} 
	&= \int_0^R \Prob{\Po(\mu_{\Po}(y)) = k_n} \frac{\pi n}{\nu}\frac{\alpha \nu}{\pi} e^{-\alpha y}dy \\
	&=n \int_0^R \Prob{\Po(\mu_{\Po}(y)) = k_n} \alpha e^{-\alpha y}dy \\
	&=n\left( \int_0^{(1-\epsilon)R} \Prob{\Po(\mu_{\Po}(y)) = k_n} \alpha e^{-\alpha y}dy
		+ \int_{(1-\varepsilon)R}^R \Prob{\Po(\mu_{\Po}(y)) = k_n} \alpha e^{-\alpha y}dy \right),
\end{align*}
where $0 < \varepsilon < 1$ is a constant to be chosen later.
Note that
\begin{align*}
	\int_{(1-\varepsilon)R}^R \Prob{\Po(\mu_{\Po}(y)) = k_n}\alpha e^{-\alpha y}dy 
	\le \int_{(1-\varepsilon)R}^R \alpha e^{-\alpha y}dy 
	= \Theta(e^{-\alpha (1-\varepsilon) R}) = \Theta(n^{-2\alpha(1-\varepsilon)}) .
\end{align*}

Now first consider the case where $k_n = \bigO{n^{\frac{1}{2\alpha + 1}}}$. Then, for $\alpha > \frac{1}{2}$ and $0 < \varepsilon < \frac{1}{2\alpha} - 1$, we have $2\alpha(1-\varepsilon)>1$. Therefore, $\frac{2\alpha (1-\varepsilon)}{2\alpha+1}>\frac{1}{2\alpha+1}$ and thus
$k_n = \bigO{n^{\frac{1}{2\alpha+1}}} = \smallO{n^{\frac{2\alpha(1-\varepsilon)}{2\alpha+1}}}$. This implies that $k_n^{-(2\alpha+1)} \gg n^{-2\alpha(1-\varepsilon)}$ or, stated differently, $\Theta(n^{-2\alpha(1-\varepsilon)}) = \smallO{k_n^{-(2\alpha+1)}}$. The first statement of the lemma now follows from Lemma~\ref{lem:degree_integral}.

When $k_n \gg n^{\frac{1}{2\alpha + 1}}$, Lemma~\ref{lem:degree_integral} implies that
\[
	n \int_0^{(1-\varepsilon)R} \Pee(Po(\mu_{Po}(y))=k_n) \alpha e^{-\alpha y}dy
	= (1+\smallO{1}) n \pmf(k_n)= \bigO{n k^{-(2\alpha + 1}} = \smallO{1}.
\]
On the other hand,
\[
	n \int_{(1-\varepsilon)R}^R \Pee(Po(\mu_{Po}(y))=k_n) \alpha e^{-\alpha y}dy = \bigO{n^{1 - 2\alpha(1-\varepsilon)}}.
\]
which is $\smallO{1}$ since by our choice $2\alpha(1-\varepsilon) > 1$. Thus the second claim of the lemma follows.
\end{proof}

\subsection{Joint degrees in \texorpdfstring{$\Gbox$}{G box} and \texorpdfstring{$\GPo$}{G Po}}\label{ssec:joint_degrees_GPo}

To prove the factorization of higher moments of $N_{\Po}(k_n)$ and $\Nbox(k_n)$ as in~\eqref{eq:def_factorial_moments_GPo}, we first have to understand the joint degree distribution in $\GPo$ and $\Gbox$, respectively. This subsequently requires us to analyze the joint neighbourhoods of two points $p, p^\prime$ in these models. To explaining the proof strategy we will use the finite box model, since the formulas there are slightly easier. The results for the Poissonized KPKVB model $\GPo$ follow the same idea.

%
%

For $r\in \N$ and $p_1, \dots, p_r \in \Rcal$, we write $\Gbox\cup \{p_1,\dots,p_r\}$ for the finite box model obtained by adding $p_1,\dots,p_r$ to the vertex set of the graph and adding all corresponding edges according to the connection rule. Then we define, for any positive integer $s$ and $V \subset \{p_1,\dots,p_s\}$,
\begin{equation}\label{eq:def_joint_degree_distribution}
\varphi_{\mathrm{box}}(V, k ; p_1,\dots,p_s) = \Prob{\text{every }p\in V\text{ has degree }k\text{ in }\Gbox \cup \{p_1,\dots,p_s\}}.
\end{equation}
In particular, for two points $p, p^\prime \in \Rcal$ and $V = \{p,p^\prime\}$, $\varphi_{\mathrm{box}}(V, k ;p,p^\prime)$ is the joint degree distribution of $p, p^\prime $ in $\Gbox$. We will use similar notation for the Poissonized KPKVB model. That is, $\GPo \cup \{p_1,\dots,p_r\}$ denotes the Poissonized KPKVB model obtained by adding $p_1,\dots,p_r$ to the vertex set of the graph and adding all corresponding edges and $\varphi_{\Po}(V, k ; p_1,\dots,p_s)$ the corresponding joint degree function.

If we define,
\begin{align*}
	X_1(p,p^\prime) &:= \Po\left(\Mu{\BallPon{p}\setminus \BallPon{p^\prime}}\right),\\
	X_2(p,p^\prime) &:= \Po\left(\Mu{\BallPon{p^\prime}\setminus \BallPon{p}}\right),\\
	Y(p,p^\prime) &:= \Po\left(\Mu{\BallPon{p} \cap \BallPon{p^\prime}}\right)
\end{align*}
then each of these are independent Poisson random variables, while
\[
	\varphi_{\mathrm{box}}(\{p,p^\prime\}, k ;p,p^\prime) 
	= \Prob{X_1(p,p^\prime) + Y(p,p^\prime) = k, X_2(p,p^\prime) + Y(p,p^\prime) = k_n}.
\]
Recall the definition of $y_{k_n,C}^\pm$ from equation~\eqref{eq:def_y_k_C}. We will show (see Lemma~\ref{cor:expected_common_neighbours_Ecal_set}) that for any two points $p, p^\prime$ whose $y$-coordinate is in $\Kcal_C(k_n)$ and whose $x$-coordinates are sufficiently separated, it holds that $\Mu{\BallPon{p} \cap \BallPon{p^\prime}} = \bigO{k_n^{1-\varepsilon}}$. Since the mean of $X_1$ and $X_2$ for such two points is $k_n$, the contribution of the Poisson random variable $Y(p,p^\prime)$ to their degrees becomes negligible as $k_n \to \infty$ and hence the joint degree distribution will factorizes on this set. The main idea is that if $p$ and $p^\prime$ are sufficiently separated in the $x$-direction, then the overlap of their neighbourhoods $\BallPon{p} \cap \BallPon{p^\prime}$ is of smaller order than $\Mu{\BallPon{p}} + \Mu{\BallPon{p^\prime}}$. We now proceed with analyzing theses joint neighbourhoods.

Let $p, p^\prime \in \Rcal$ and denote by $\Ncal_{\text{box}}(p,p^\prime)$ the number of common neighbours of $p$ and $p^\prime$ in $\Gbox \cup \{p,p^\prime\}$. We shall establish an upper bound on the expected number of joint neighbours when $p$ and $p^\prime$ are sufficiently separated. Observe that $\Exp{\Ncal_{\text{box}}(p,p^\prime)} = \Mu{\BallPon{p} \cap \BallPon{p^\prime}}$. 

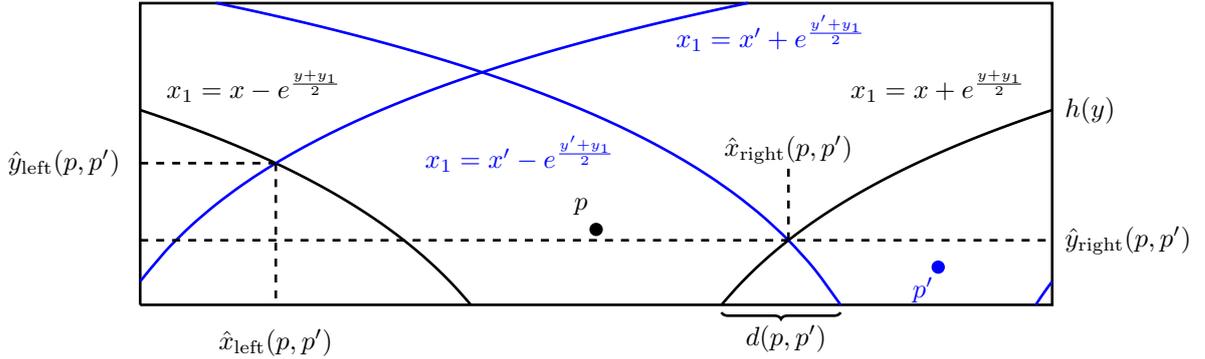
\begin{figure}[!t]
\centering
\begin{tikzpicture}

	\pgfmathsetmacro{\u}{0} 
	\pgfmathsetmacro{\v}{1} 
	\pgfmathsetmacro{\uu}{4.5} 
	\pgfmathsetmacro{\vv}{0.5} 
	\pgfmathsetmacro{\r}{6}
	\pgfmathsetmacro{\t}{4}
	
	\draw[line width=1pt] (-\r,0) -- (\r,0) -- (\r,\t) -- (-\r,\t) -- (-\r,0);
	
    \draw node[fill, circle, inner sep=0pt, minimum size=5pt] at (\u,\v) {};
    \draw node at (-0.2,1.3) {$p$};
    
    \draw node[fill,blue, circle, inner sep=0pt, minimum size=5pt] at (\uu,\vv) {};
    \draw node at (4.3,0.2) {\color{blue}$p^\prime$};
	
	\draw[domain=1.6487:6,smooth,variable=\x,black,line width=1pt] plot (\x, {2*ln(\x)-1});
    \draw[domain=-1.6487:-6,smooth,variable=\x,black,line width=1pt] plot (\x, {2*ln(-\x)-1});
    \draw[domain=5.7840:6,smooth,variable=\x,blue,line width=1pt] plot (\x, {2*ln(\x-4.5)-0.5});
    \draw[domain=3.2160:-5,smooth,variable=\x,blue,line width=1pt] plot (\x, {2*ln(4.5-\x)-0.5});
    \draw[domain=2:-6,smooth,variable=\x,blue,line width=1pt] plot (\x, {2*ln(\x+7.5)-0.5});

    \draw [decorate,decoration={brace},line width=1pt] (3.2160,-0.1) -- (1.6487,-0.1);
    \draw node at (2.5,-0.45) {$d(p,p^\prime)$};
    
    
    \pgfmathsetmacro{\vast}{2*ln((2*\r - \uu)/(exp(\v/2) + exp(\vv/2)))}
    \pgfmathsetmacro{\uast}{(\uu - 2*\r)/(1+exp((\vv-\v)/2))}
    \pgfmathsetmacro{\vvast}{2*ln((\uu)/(exp(\v/2) + exp(\vv/2)))}
    \pgfmathsetmacro{\uuast}{(\u*exp(\vv/2) + \uu*exp(\v/2))/(exp(\v/2) + exp(\vv/2))}
    \pgfmathsetmacro{\h}{2*ln(\r)-\v}
    
    
    \draw node at (\r+0.5,\h) {$h(y)$};
    
    \draw node at (2.3,3.6) {\color{blue}$x_1 = x^\prime + e^{\frac{y^\prime + y_1}{2}}$};
    \draw node at (-1,2) {\color{blue}$x_1 = x^\prime - e^{\frac{y^\prime + y_1}{2}}$};
    \draw node at (-4.5,2.9) {$x_1 = x - e^{\frac{y + y_1}{2}}$};
    \draw node at (4.5,2.9) {$x_1 = x + e^{\frac{y + y_1}{2}}$};

	\draw[dashed,black,line width=1pt] (-\r,\vast) -- (\uast,\vast);
	\draw node at (-\r-1,\vast) {\color{black}$\hat{y}_{\text{left}}(p,p^\prime)$};
	\draw[dashed,line width=1pt] (\uast,\vast) -- (\uast,0);
	\draw node at (\uast,-0.5) {$\hat{x}_{\text{left}}(p,p^\prime)$};

	\draw[dashed,black,line width=1pt] (-\r,\vvast) -- (\r,\vvast);
	\draw node at (\r+1,\vvast) {\color{black}$\hat{y}_{\text{right}}(p,p^\prime)$};
	\draw[dashed,line width=1pt] (\uuast,\vvast) -- (\uuast,\vvast+1);
	\draw node at (\uuast,\vvast+1.2) {$\hat{x}_{\text{right}}(p,p^\prime)$};

\end{tikzpicture}
\caption{Schematic representation of the neighbourhoods of $p$ and $p^\prime$ in $\Gbox$ when $|x - x^\prime| > e^{\frac{y}{2}} + e^{\frac{y^\prime}{2}}$ used for the proof of Lemma \ref{lem:common_neighbours_Pcal_n}. Note that although here $p^\prime \notin \BallPon{p}$, this is not true in general. This situation was merely chosen to improve readability of the figure.}
\label{fig:representation_disjoint_neighbourhoods_P_n_3}
\end{figure}

We start by analyzing the shape of the joint neighbourhood. Due to symmetry and the fact that we have identified the left and right boundaries of the box $\Rcal$, we can, without loss of generality, assume that $p = (0,y)$ and $p^\prime = (x^\prime,y^\prime)$ with $x^\prime > 0$. To understand the computation it is helpful to have a picture. Figure~\ref{fig:representation_disjoint_neighbourhoods_P_n_3} shows such an example. There are several different quantities that are important. The first are the heights where the left and right boundaries of the ball $\BallPon{p}$ hit the boundaries of the box $\Rcal$. Since $x = 0$ these heights are the same and we denote their common value by $h(y)$. We also need to know the coordinates $\hat{y}_{\text{right}}(p,p^\prime)$ and $\hat{x}_{\text{right}}(p,p^\prime)$ of the intersection of the right boundary of the neighbourhood of $p$ with the left boundary of the neighbourhood of $p^\prime$ and those for the intersection of the left boundary of the neighbourhood of $p$ with the right boundary of the neighbourhood of $p^\prime$, which we denote by $\hat{y}_{\text{left}}(p,p^\prime)$ and $\hat{x}_{\text{left}}(p,p^\prime)$. Finally we will denote by $d(p,p^\prime)$ the distance between the lower right boundary of $\BallPon{p}$ and the lower left of $\BallPon{p^\prime}$, which is positive only when the bottom parts of both neighbourhoods do not intersect, as is the case in Figure~\ref{fig:representation_disjoint_neighbourhoods_P_n_3}. The condition $d(p,p^\prime) > 0$ is exactly the right notion for $p$ and $p^\prime$ being sufficiently separated.

We start by deriving expressions for these important coordinates. For this we introduce some notation. For any $p = (x,y) \in \Rcal$ we will define the left and right boundary functions as, respectively,
\begin{align}
	b_p^-(z) &= \begin{cases}
		2 \log\left(x-z\right) - y &\mbox{if }  -\frac{\pi}{2} e^{R/2} \le z \le x - e^{y/2}  \\
		2\log\left(\pi e^{R/2} + x - z\right) - y 
			&\mbox{if } x - e^{(y + R)/2} + \pi e^{R/2} \le z \le \frac{\pi}{2} e^{R/2}\\
		0 &\mbox{otherwise}
	\end{cases} \label{eq:def_left_boundary_Bp} \\ 
	b_p^+(z) &= \begin{cases}
		2 \log\left(z-x\right) - y &\mbox{if } x + e^{y/2} \le z \le \frac{\pi}{2} e^{R/2} \\
		2\log\left(\pi e^{R/2} + z - x\right) - y 
			&\mbox{if } -\frac{\pi}{2} e^{R/2} \le z \le x + e^{(y + R)/2} - \pi e^{R/2}\\
		0 &\mbox{otherwise}
	\end{cases}
\end{align}
Note that these functions describe the boundaries of the ball $\BallPon{p}$. In particular, $p^\prime = (x^\prime, y^\prime) \in \BallPon{p}$ if and only if $y^\prime \ge \min\left\{b_p^-(x^\prime), b_p^+(x^\prime)\right\}$.

We shall derive the expressions for the point $(\hat{x}_{\text{left}}(p,p^\prime), \hat{y}_{\text{left}}(p,p^\prime))$. The $x$-coordinate $\hat{x}_{\text{left}}(p,p^\prime)$ is the solution to the equation $b_{p}^+(z) = b_{p^\prime}^-(z)$ for $-\frac{\pi}{2} e^{R/2} \le z \le + e^{y/2}$. This equation becomes
\[
	2\log\left(\pi e^{R/2} + z - x^\prime\right) - y^\prime = 2 \log\left(x^\prime-z\right) - y^\prime,
\]
whose solution is $\frac{x^\prime - \pi e^{R/2}}{1 + e^{(y^\prime - y)/2}}$. Plugging this into either the left or right hand side of the above equation yields the $y$-coordinate $\hat{y}_{\text{left}}(p,p^\prime) = 2 \log\left(\frac{\pi e^{R/2} - x^\prime}{e^{y/2} + e^{y^\prime/2}}\right)$. The expressions for $\hat{x}_{\text{right}}(p,p^\prime$ and $\hat{y}_{\text{right}}(p,p^\prime)$ are derived in a similar way. The expression for $d(p,p^\prime)$ follows as the difference $b_{p^\prime}^-(x^\prime - e^{y^\prime/2}) - b_p^+(e^{y/2})$.

The full expressions of all coordinates are given below for further reference.

\begin{align}
	h(y) &= R - y + 2\log\left(\frac{\pi}{2}\right) \label{eq:def_height_y_P_n}\\
	\hat{y}_{\text{right}}(p,p^\prime) &= 2\log\left(\frac{x^\prime}{e^{\frac{y}{2}} + e^{\frac{y^\prime}{2}}}\right)\\
	\hat{x}_{\text{right}}(p,p^\prime) &= \frac{x^\prime}{1 + 	
		e^{\frac{y^\prime - y}{2}}},\\
	\hat{y}_{\text{left}}(p,p^\prime) &= 2 \log\left(\frac{\pi e^{R/2} - x^\prime}{e^{\frac{y}{2}} + e^{\frac{y^\prime}{2}}}\right),\\
	\hat{x}_{\text{left}}(p,p^\prime) &= \frac{x^\prime - \pi e^{R/2}}{1 + e^{\frac{y^\prime - y}{2}}}, \\
	d(p,p^\prime) &= |x - x^\prime|_n - \left(e^{\frac{y}{2}} + e^{\frac{y^\prime}{2}}\right).
	\label{eq:def_d_p_p_prime}
\end{align}

The following result shows that if $d(p,p^\prime) > 0$, then the expected number of common neighbours is $\smallO{\Mu{\BallPon{p}} + \Mu{\BallPon{p^\prime}}}$.

\begin{lemma}\label{lem:common_neighbours_Pcal_n}
Let $p, p^\prime \in \Rcal$. Then, whenever $|x - x^\prime|_n > \left(e^{\frac{y}{2}} + e^{\frac{y^\prime}{2}}\right)$,
\[
	\Exp{\Ncal_{\text{box}}(p,p^\prime)} \le \Mu{\BallPon{p}}
	\left(\left(\frac{|x - x^\prime|}{e^{\frac{y}{2}} + e^{\frac{y^\prime}{2}}}\right)^{-(2\alpha - 1)}  
	+ \frac{\nu}{\xi} e^{-(\alpha - \frac{1}{2})(R-y)}\right).
\]
\end{lemma}

\begin{proof}
Again, without loss of generality we assume that $p = p_0 = (0,y)$ and $p^\prime = (x^\prime, y^\prime)$ with $0 \le x^\prime \le \frac{\pi}{2} e^{R/2}$. Note that since $0 < x^\prime \le \frac{\pi}{2} e^{R/2}$, $ \hat{y}_{\mathrm{right}}(p,p^\prime) \le \hat{y}_{\mathrm{left}}(p,p^\prime)$. We write $\hat{y}$ for $\hat{y}_{\mathrm{right}}(p,p^\prime)$ and observe that below $\hat{y}$ the balls $\BallPon{p}$ and $\BallPon{p^\prime}$ are disjoint. Therefore, if we define $A := \left\{p_1 = (x_1,y_1) \in \Rcal \cap \BallPon{p} \, : \, y_1 \ge \hat{y} \right\}$. Then
\[
	\Exp{\Ncal_{\text{box}}(p,p^\prime)} \le \Mu{A}.
\]
We proceed with computing the right hand side
\begin{align*}
	\Mu{A} &= \int_{\hat{y}}^{h(y)} 
		\int_{- e^{\frac{y + y_1}{2}}}^{e^{\frac{y^\prime+y_1}{2}}} 
		f(x_1,y_1) \dd x_1 \dd y_1
		+ \int_{h(y)}^{R} \int_{-\frac{\pi}{2}e^{R/2}}^{\frac{\pi}{2} e^{R/2}} 
		\hspace{-10pt}  f(x_1,y_1) \dd x_1 \dd y_1 \\
	&= \frac{2\alpha \nu}{\pi}e^{\frac{y}{2}}
		\int_{\hat{y}}^{h(y)} e^{-(\alpha - \frac{1}{2})y_1} \dd y_1 
		+ \alpha \nu e^{R/2} \int_{h(y)}^R e^{-\alpha y_1} \dd y_1\\
	&\le \xi\left(e^{\frac{y}{2}} + e^{\frac{y^\prime}{2}}\right) e^{-(\alpha-\frac{1}{2})\hat{y}}
		+ \nu e^{R/2} e^{-\alpha h(y)} \\
	&= \Mu{\BallPon{p}}\left(
		e^{-(\alpha-\frac{1}{2})\hat{y}} + \frac{\nu}{\xi} e^{-(\alpha - \frac{1}{2})(R-y)}\right).
\end{align*}
The result follows by plugging in 
\[
	\hat{y} := \hat{y}_{\text{right}}(p,p^\prime) = 2 \log\left(\frac{x^\prime}{e^{\frac{y}{2}} + e^{\frac{y^\prime}{2}}}\right),
\]
and noting that $x^\prime$ is the same as $|x - x^\prime|$, by our generalization step.

\end{proof} 

We can also prove a similar result for the Poissonized KPKVB model $\GPo$, denoting by $\Ncal_{\Po}(p,p^\prime)$ the number of joint neighbours in $\GPo \cup \{p,p^\prime\}$.

\begin{lemma}\label{lem:common_neighbours_KPKVB}
Let $0 < \varepsilon < 1$, $p, p^\prime \in \Rcal$ with $y,y^\prime \le (1-\varepsilon)R$, denote by $\Ncal_{\Po}(p,p^\prime)$ the number of joint neighbours of $p, p^\prime$ in $\GPo$ and let $K$ be the constant from Lemma~\ref{lem:asymptotics_Omega_hyperbolic}. Then, whenever $|x - x^\prime|_n > \left(e^{\frac{y}{2}} + e^{\frac{y^\prime}{2}}\right)\left(1 + \frac{\pi^2 K}{4}\right)$,
\[
	\Exp{\Ncal_{\Po}(p,p^\prime)} \le \Mu{\BallHyp{p}}
	\left(e^{(2\alpha - 1)\lambda} \left(\frac{|x - x^\prime|}{e^{\frac{y}{2}} + e^{\frac{y^\prime}{2}}}\right)^{-(2\alpha - 1)}  
	+ \frac{\nu}{\xi} e^{-(\alpha - \frac{1}{2})(R-y)}\right),
\]
where
\[
	\lambda = \log\left(1 + \frac{\pi^2 K}{4}\right).
\]
\end{lemma}

\begin{proof}
We will proceed in a similar fashion as for Lemma~\ref{lem:common_neighbours_Pcal_n}. That is, we will bound the expected number of common neighbours by the number of neighbors of $p$ whose $y$-coordinate is above the intersection of the right boundary of $\BallHyp{p}$ and the left boundary of $\BallHyp{p^\prime}$. Denote by $\hat{y}$ the height of this intersection point. Then
\begin{align*}
	\Exp{\Ncal_{\Po}(p,p^\prime)} &\le \frac{2\alpha \nu}{\pi} \int_{\hat{y}}^{R-y} \Phi(y,y_1) e^{-\alpha y_1} \dd y_1
		+ \alpha \nu e^{R/2} \int_{R-y}^{R} e^{-\alpha y_1} \dd y_1.
\end{align*} 
The second integral is bounded by $\frac{\nu}{\xi} \Mu{\BallPon{y}} e^{-(\alpha - \frac{1}{2})(R-y)}$. We bound the first integral using Lemma~\ref{lem:asymptotics_Omega_hyperbolic} as
\begin{align*}
	\frac{2\alpha \nu}{\pi} \int_{\hat{y}}^{R-y} \Phi(y,y_1) e^{-\alpha y_1} \dd y_1
	&\le \frac{2\alpha \nu}{\pi} \int_{\hat{y}}^{R-y} \left(e^{\frac{y+y_1}{2}} + K e^{\frac{3}{2}(y + y_1)-R}\right) 
		e^{-\alpha y_1} \dd y_1 \\
	&\le \frac{2\alpha \nu}{\pi} (1+K)e^{\frac{y}{2}} \int_{\hat{y}}^{R-y} e^{-(\alpha - \frac{1}{2}) y_1} \dd y_1\\
	&\le (1 + K)\Mu{\BallPon{p}}  e^{-(\alpha - \frac{1}{2})\hat{y}},
\end{align*}
where we used that $\frac{3y_1}{2} \le R - y + \frac{y_1}{2}$ for all $y_1 \le R-y$ for the second line. 

It remains to compute $\hat{y}$, for which we will establish the following bound
\begin{equation}\label{eq:joint_neighbors_KPKVB_intersection}
	\hat{y} \ge 2 \log\left(\frac{x^\prime}{e^{\frac{y}{2}} + e^{\frac{y^\prime}{2}}}\right) - 2\lambda.
\end{equation}

To show~\eqref{eq:joint_neighbors_KPKVB_intersection} we note that for any point $y_1 \ge \hat{y}$, the corresponding $x$-coordinate of the left boundary of $\BallHyp{p^\prime}$ must be to the left of that of the ball $\BallHyp{p}$, i.e.
$x^\prime - \Phi(y^\prime, y_1) \le \Phi(y,y_1)$. Therefore it is enough to show that for all 
\[
	y_1 \le 2 \log\left(\frac{x^\prime}{e^{\frac{y}{2}} + e^{\frac{y^\prime}{2}}}\right) - 2\lambda
\]
it holds that $\Phi(y,y_1) \le x^\prime - \Phi(y^\prime, y_1)$, with $\lambda$ as defined in the statement of the lemma. Note that by assumption on $x^\prime := |x - x^\prime|_n$ (since we can take $x = 0$) the above upper bound is non-negative. Using Lemma~\ref{lem:asymptotics_Omega_hyperbolic} it suffices to prove that for all such $y_1$,
\[
	e^{\frac{y + y_1}{2}} + K e^{\frac{3}{2}(y + y_1) - R} \le x^\prime - e^{\frac{y^\prime + y_1}{2}}
	- K e^{\frac{3}{2}(y^\prime + y_1) - R},
\]
which is equivalent to
\[
	\left(e^{\frac{y}{2}} + e^{\frac{y^\prime}{2}}\right) e^{\frac{y_1}{2}} 
	+ K e^{-R} \left(e^{\frac{y}{2}} + e^{\frac{y^\prime}{2}}\right)^3 e^{\frac{3 y_1}{2}} \le x^\prime.
\]
Plugging the upper bound for $y_1$ into the left hand side and using that $(e^{y/2} + e^{y^\prime/2})^3 \ge e^{3y/2} + e^{3y^\prime/2}$, we obtain
\begin{align*}
	\left(e^{\frac{y}{2}} + e^{\frac{y^\prime}{2}}\right) e^{\frac{y_1}{2}} 
		+ K e^{-R} \left(e^{\frac{y}{2}} + e^{\frac{y^\prime}{2}}\right)^3 e^{\frac{3 y_1}{2}}
	&\le x^\prime e^{-\lambda} + K e^{-R} (x^\prime)^3 e^{-3\lambda}
		\le x^\prime \left(e^{-\lambda} + \frac{\pi^2 K}{4} e^{-3\lambda}\right)\\
	&\le x^\prime e^{-\lambda} \left(1 + \frac{\pi^2 K}{4}\right) = x^\prime,
\end{align*} 
where we also used that $x^\prime \le \frac{\pi}{2} e^{-R/2}$.
\end{proof}

Let us now define the \emph{stripe}
\begin{equation}\label{eq:def_stripe}
	\stripknc = \Rcal \cap (\R_+ \times \Kcal_C(k_n)).
\end{equation} 
and in addition define, for any $0 < \varepsilon < 1$, the following set
\begin{equation}\label{eq:def_joint_degree_set_E_growing_k}
	\mathcal{E}_{\varepsilon}(k_n) = \left\{(p,p^\prime) \in \stripknc
		\, : \,  |x - x^\prime|_n > k_n^{1 + \varepsilon} \right\}, 
\end{equation}
where $|x|_n = \min\{|x|, \pi e^{R/2} - |x|\}$ denotes the norm on the finite box $\Rcal$ where the left and right boundaries are identified. Then for any two points $p,p^\prime \in \mathcal{E}_\varepsilon(k_n)$ the expected number of joint neighbours is $\smallO{k_n}$.

\begin{lemma}\label{cor:expected_common_neighbours_Ecal_set}
Fix $0 < \varepsilon < 1$ and let $\varepsilon^\prime = \min\{\varepsilon(2\alpha - 1),\varepsilon\}$. Then for all $(p,p^\prime) \in \mathcal{E}_\varepsilon(k_n)$, as $n \to \infty$,
\[
	\Mu{\BallPon{p} \cap \BallPon{p^\prime}} = \bigO{k_n^{1-\varepsilon^\prime}}.
\] 
\end{lemma}

\begin{proof}
Since for all $(p,p^\prime) \in \mathcal{E}_\varepsilon(k_n)$ we have $\Mu{\BallPon{p}}, \Mu{\BallPon{p^\prime}} = \bigT{k_n}$, Lemma~\ref{lem:common_neighbours_Pcal_n} implies that
\[
	\Mu{\BallPon{p} \cap \BallPon{p^\prime}} \le \bigO{k_n} \phi_n(p,p^\prime),
\]
where
\[
	\phi_n(p,p^\prime) = 2\left(\frac{|x - x^\prime|_n}{e^{\frac{y}{2}} + e^{\frac{y^\prime}{2}}}\right)^{-(2\alpha - 1)} 
		+ \frac{3 \nu^{2\alpha + 1}e^{-(\alpha - \frac{1}{2})R} e^{\alpha y}}{2 \pi^{2\alpha}\left(
		e^{\frac{y}{2}} + e^{\frac{y^\prime}{2}}\right)}
		+ \frac{\nu e^{-(\alpha - \frac{1}{2})R}}{e^{\frac{y}{2}} + e^{\frac{y^\prime}{2}}}. 
\]
We thus need to show that $\phi_n(p,p^\prime) = \bigO{k_n^{-\varepsilon^\prime}}$. For $(p,p^\prime) \in \mathcal{E}_\varepsilon(k_n)$, it holds that $e^{y/2}, e^{y^\prime/2} = \bigT{k_n}$ and $|x - x^\prime|_n > k_n^{1+\varepsilon}$ and hence
\begin{align*}
	 2\left(\frac{|x - x^\prime|_n}{e^{\frac{y}{2}} + e^{\frac{y^\prime}{2}}}\right)^{-(2\alpha - 1)}
	 =	\bigO{k_n^{-\varepsilon(2\alpha - 1)}}.
\end{align*}
For the second term in $\phi_n(p,p^\prime)$ we use that $e^{\alpha y^\ast} = \bigT{k_n^{2\alpha}}$ and $e^{R} = \bigT{n^2}$ to obtain
\[
	\frac{3 \nu^{2\alpha + 1}e^{-(\alpha - \frac{1}{2})R} e^{\alpha y}}{2 \pi^{2\alpha}\left(
			e^{\frac{y}{2}} + e^{\frac{y^\prime}{2}}\right)}
	= \bigO{1} n^{-(2\alpha - 1)} k_n^{2\alpha - 1} = \bigO{n^{-(\alpha - \frac{1}{2})}}.
\]
Finally, the third term satisfies $\bigO{n^{-(2\alpha - 1)}k_n^{-1}}$, and we conclude that
\[
	\phi_n(p,p^\prime) = \bigO{k_n^{-\varepsilon(2\alpha - 1)} + n^{-(\alpha - \frac{1}{2})}
	+ n^{-(2\alpha - 1)}k_n^{-1}} = \bigO{k_n^{-\varepsilon^\prime}},
\]
where we used that $\varepsilon^\prime = \min\{\varepsilon(2\alpha - 1),\varepsilon\}$. 
\end{proof}

It is clear that using Lemma~\ref{lem:common_neighbours_KPKVB} instead of Lemma~\ref{lem:common_neighbours_Pcal_n}, the above proof applies to the Poissonized KPKVB model, yielding the following result.

\begin{lemma}\label{cor:expected_common_neighbours_KPKVB}
Fix $0 < \varepsilon < 1$ and let $\varepsilon^\prime = \min\{\varepsilon(2\alpha - 1),\varepsilon\}$. Then for all $(p,p^\prime) \in \mathcal{E}_\varepsilon(k_n)$, as $n \to \infty$,
\[
	\Mu{\BallHyp{p} \cap \BallHyp{p^\prime}} = \bigO{k_n^{1-\varepsilon^\prime}}.
\] 
\end{lemma}

Recall the definition of the stripe $\stripknc = \Rcal \cap (\R_+ \times [y_{k_n,C}^-, y_{k_n,C}^+])$ from~\eqref{eq:def_stripe}. Consider a fixed number of points $p_1, \dots, p_r$. Then, if their $x$-coordinates are sufficient far apart and their $y$ coordinates lie within the stripe $\stripknc$, their degrees are asymptotically independent.

\begin{lemma}[Asymptotic factorization of degree probabilities]\label{lem:probdegFact}
Let $(k_n)$ be a sequence of integers with $0\leq k_n \leq n-1$, $k_n = \bigO{n^{\frac{1}{2\alpha+1}}}$ and $k_n \rightarrow \infty$.
Let $C, C^\prime>0$ and $r,s$ be positive integers with $r+1\leq s$. Fix $0<\varepsilon<1$. Then, it holds uniformly for all $(v_1,\dots,v_s) \in \left(\stripknc\right)^{s} = \stripknc \times \dots \times \stripknc$, satisfying $|x_{v_i}-x_{v_{r+1}}|\geq k_n^{1+\varepsilon}$ for all  $1\leq i \leq r$,  that
\[
	\varphi(\{v_1,\dots,v_{r+1}\};v_1,\dots,v_s)
	=(1+o(1))\varphi(\{v_1,\dots,v_r\};v_1,\dots,v_s)\varphi(\{v_{r+1}\};v_1,\dots,v_s)+\bigO{k_n^{-C^\prime}},
\]
where $\varphi$ is either $\varphi_{\Po}$ or $\varphi_{\mathrm{box}}$.
(Here uniformity means that the $\smallO{1}$ and $\bigO{k_n}$ terms do not depend on $v_1,\dots,v_{s}$.)
\end{lemma}

\begin{proof}
Let $H=\GPo \cup\{p_1,\dots,p_s\}$ or $H=\Gbox \cup \{p_1,\dots,p_s\}$ and $1\leq r\leq s$. For $1\leq j \leq r$, let $Y_j$ be the number of vertices of $H$ which are adjacent to both $p_j$ and $p_{r+1}$. Let $X_j$ be the number of vertices of $H$ which are adjacent to $p_j$, but not to $p_{r+1}$. Then, $X_j+Y_j= D_{H}(p_j)$ is the degree of $p_j$ in $H$. 

Now let $X_{r+1}$ be the number of vertices of $H$ which are adjacent to $p_{r+1}$, but to none of $p_1,\dots,p_r$. Let $Y_{r+1}$ be the number of vertices of $H$ which are adjacent to $p_{r+1}$, and at least one of $p_1,\dots,p_r$. Then, $X_{r+1}+Y_{r+1}=D_H(p_{r+1})$ is the degree of $p_{r+1}$ in $H$. By definition, we therefore have
\begin{align*}
\varphi(\{p_1,\dots,p_{r+1}\};p_1,\dots,p_s)&=\Prob{X_1+Y_1=\dots=X_{r+1}+Y_{r+1}=k_n}, \\
\varphi(\{p_1,\dots,p_r\};p_1,\dots,p_s)&=\Prob{X_1+Y_1=\dots=X_r+Y_r=k_n}, \\
\varphi(\{p_{r+1}\};p_1,\dots,p_s)&=\Prob{X_{r+1}+Y_{r+1}=k_n},
\end{align*}
and the claim of the lemma is that
\begin{equation}\label{eq:probdegfact_main}
	\begin{aligned}
		&\Prob{X_1+Y_1=\dots=X_{r+1}+Y_{r+1}=k_n} \\
		&= (1+o(1))\Prob{X_1+Y_1=\dots=X_r+Y_r=k_n}\Prob{X_{r+1}+Y_{r+1}=k_n}+\bigO{k_n^{-C}}.
	\end{aligned}
\end{equation}

To prove~\eqref{eq:probdegfact_main} let $\varepsilon'=\min(\varepsilon,\varepsilon(2\alpha-1))\in (0,1)$. Since for $1\leq i \leq r$, it is given that $|x_{p_i}-x_{p_{r+1}}|\geq k_n^{1+\varepsilon}$, then in the case where $H = \GPo \cup\{p_1,\dots,p_s\}$ it follows from Lemma~\ref{lem:common_neighbours_KPKVB} that $\E[Y_j] = \bigO{k_n^{1-\varepsilon'}}$. When $H=\Gbox \cup \{p_1,\dots,p_s\}$ we get the same result using Lemma~\ref{lem:common_neighbours_Pcal_n}. The rest of the proof is independent of which of the two models we consider and only uses that $\E[Y_j] = \bigO{k_n^{1-\varepsilon'}}$.

As $Y_{r+1} \leq Y_1+\dots+Y_r$, we also have $\mu:=\Exp{Y_{r+1}} = \bigO{k_n^{1-\varepsilon'}}$. In particular, there is $c_0>0$ such that $c_0\sqrt{k_n^{1-\varepsilon'}\ln k_n^{1-\varepsilon'}}\geq c_1\sqrt{\mu\ln\mu}$ (where $c_1=\sqrt{\frac{2C}{1-\varepsilon'}}$, which is well-defined because $1-\varepsilon'>0$). Now define
\begin{align*}
A_n=[\mu-c_0\sqrt{k_n^{1-\varepsilon'}\ln k_n^{1-\varepsilon'}},\mu+c_0\sqrt{k_n^{1-\varepsilon'}\ln k_n^{1-\varepsilon'}}]\cap \N_0.
\end{align*}
By equation~\eqref{eq:def_chernoff_bound_poisson_C}, we have
\begin{align*}
	\Prob{Y_{r+1} \not \in A_n} 
	&= \Prob{|Y_{r+1}-\mu|\geq c_0\sqrt{k_n^{1-\varepsilon'}\ln k_n^{1-\varepsilon'}}}\\
	&\le \Prob{|Y_{r+1}-\mu|\geq c_1\sqrt{\mu\ln\mu}}
	=\bigO{k_n^{-\frac{(1-\varepsilon')c_1^2}{2}}}.
\end{align*}
As by definition $c_1$ satisfies $\frac{(1-\varepsilon')c_1^2}{2}=C$, this implies that for the event $S_n = \{Y_{r+1}\in A_n\}$,
\[
	\Prob{S_n^c}=\bigO{k_n^{-C}}.
\]
Beginning with the left-hand side of the claim of the lemma, the law of total probability applied to the events $\{Y_{r+1}=y_{r+1}\}$, for all $y_{r+1}\in A_n$, and $S_n^c$ implies that
\begin{align*}
	&\Prob{X_1+Y_1=\dots=X_{r+1}+Y_{r+1}=k_n} \\
	&=\sum_{y_{r+1}\in A_n} \Prob{X_1+Y_1=\dots=X_{r+1}+y_{r+1}=k_n|Y_{r+1}=y_{r+1}}\Prob{Y_{r+1}=y_{r+1}}\\ 
	&\hspace{10pt}+ \Prob{X_1+Y_1=\dots=X_{r+1}+Y_{r+1}=k_n|S_n^c} \Prob{S_n^c}\\
	&=\sum_{y_{r+1}\in A_n} \Prob{X_1+Y_1=\dots=X_{r+1}+y_{r+1}=k_n|Y_{r+1}=y_{r+1}} 
		\Prob{Y_{r+1}=y_{r+1}}+\bigO{k_n^{-C}}.
\end{align*}
As $X_{r+1}$ is independent of $X_1, \dots, X_r, Y_1, \dots, Y_r$ by the properties of a Poisson process (as $X_{r+1}$ counts the number of points in a set which is disjoint of $X_1,\dots,X_r, Y_1, \dots, Y_r$), it follows that
\begin{align*}
	&\Prob{X_1+Y_1=\dots=X_{r+1}+Y_{r+1}=k_n} \\
	&=\sum_{y_{r+1} \in A_n} \Prob{X_1+Y_1=\dots=X_{r}+Y_r=k_n|Y_{r+1}=y_{r+1}}
		\Prob{X_{r+1}+y_{r+1}=k_n} \Prob{Y_{r+1}=y_{r+1}}\\
	&\hspace{10pt} +\bigO{k_n^{-C}}.
\end{align*}
We will now show that uniformly for all $y_{r+1}, s \in A_n$, it holds that,
\begin{align}\label{eq:yrp1tos}
	\Prob{X_{r+1}=k_n-y_{r+1}} = (1+o(1))\Prob{X_{r+1}=k_n-s}.
\end{align}
To see this, observe that for all $y_{r+1},s\in A_n$, we have that $|y_{r+1}-s|\leq 2c\sqrt{k_n^{1-\varepsilon'}\ln k_n^{1-\varepsilon'}}$. Denote the expectation of $X_{r+1}$ by $\lambda$, write $\delta_n = k_n-y_{r+1}-\lambda$ and note that
\begin{align*}
	\frac{\Prob{X_{r+1}=k_n-y_{r+1}}}{\Prob{X_{r+1}=k_n-s}} 
	&= \frac{\Prob{X_{r+1}=k_n-y_{r+1}}}{\Prob{X_{r+1}=k_n-y_{r+1}+(y_{r+1}-s)}}\\
	&=\frac{(k_n-y_{r+1}+(y_{r+1}-s))!}{(k_n-y_{r+1})!}\lambda^{s-y_{r+1}}
\end{align*}
We will now use that $\frac{(a+b)!}{a!} = (1+o(1))(a+b)^b$ for $b^2=o(a)$, applied to $a=k_n-y_{r+1}$ and $b=y_{r+1}-s$. To see this auxiliary fact, note that by Stirling's approximation to the factorial (see e.g.~\cite{Dutkay2013},~\cite{Nanjundiah1959}), it follows that
\begin{align*}
	\frac{(a+b)!}{a!}=(1+o(1))\frac{(a+b)^{a+b+\frac{1}{2}}e^{-a-b}}{a^{a+\frac{1}{2}}e^{-a}}
	=(1+o(1))\left(1+\frac{b}{a}\right)^{a+\frac{1}{2}}(a+b)^b e^{-b}.
\end{align*}
Since $1+\frac{b}{a}\leq e^{\frac{b}{a}}$, it holds that $(1+\frac{b}{a})^a \leq e^b$. Furthermore, as $\ln(1+x)\geq x-\frac{x^2}{2}$ (for $x\in (-1,1)$), we have that $(1+\frac{b}{a})^a =e^{a\ln(1+\frac{b}{a})}\geq e^{a(\frac{b}{a}-\frac{b^2}{2a})} =e^{b-\frac{b^2}{2a}}=(1+o(1))e^{b}$ because $b^2=o(a)$. Finally, $b^2=o(a)$ also implies that $(1+\frac{b}{a})^\frac{1}{2}=1+o(1)$. This finishes the proof of the auxiliary fact and we can continue with
\begin{align*}
	\frac{\Prob{X_{r+1}=k_n-y_{r+1}}}{\Prob{X_{r+1}=k_n-s}}
	&=(1+o(1)) (k_n-y_{r+1}+(y_{r+1}-s))^{y_{r+1}-s} \lambda^{s-y_{r+1}} \\
	&=(1+o(1)) (\lambda+\delta_n+(y_{r+1}-s))^{y_{r+1}-s} \lambda^{s-y_{r+1}} \\
	&=(1+o(1))\left(1+\frac{\delta_n+(y_{r+1}-s)}{\lambda}\right)^{y_{r+1}-s} \\
	&=(1+o(1)) e^{\frac{\delta_n(y_{r+1}-s)}{\lambda}} e^{\frac{(y_{r+1}-s)^2}{\lambda}} = 1+o(1),
\end{align*}
where the last line follows since $\delta_n, |y_{r+1}-s| \leq 2c_0\sqrt{k_n^{1-\varepsilon'}\ln k_n^{1-\varepsilon'}}$ and $\lambda=\Theta(k_n)$ and therefore, with convergence uniform in $y_{r+1},s$,
\[
	\frac{\delta_n(y_{r+1}-s)}{\lambda}, \frac{(y_{r+1}-s)^2}{\lambda} \leq \frac{4c_0^2 k_n^{1-\varepsilon'}\ln k_n^{1-\varepsilon'}}{\lambda}  \rightarrow 0.
\]
As $\Prob{S_n^c}=\bigO{k_n^{-C}}$, we have
\begin{align*}
	1=\sum_{s\in A_n}\Prob{Y_{r+1}=s}+\bigO{k_n^{-C}}.
\end{align*}
From~\eqref{eq:yrp1tos}, it then follows that
\begin{align*}
	\Prob{X_{r+1}=k_n-y_{r+1}} 
	&=\Prob{X_{r+1}=k_n-y_{r+1}} \sum_{s\in A_n} \Prob{Y_{r+1}=s}+\bigO{k_n^{-C}}\\
	&= (1+o(1)) \sum_{s\in A_n} \Prob{Y_{r+1}=s} \Prob{X_{r+1}=k_n-s}+\bigO{k_n^{-C}} \\
	&=(1+o(1))\sum_{s\in A_n} \Prob{X_{r+1}+Y_{r+1}=k_n, Y_{r+1}=s}+\bigO{k_n^{-C}} \\
	&=(1+o(1))\Prob{X_{r+1}+Y_{r+1}=k_n, Y_{r+1}\in A_n}+\bigO{k_n^{-C}} \\
	&=(1+o(1))\left(\Prob{X_{r+1}+Y_{r+1}=k_n}-\Prob{X_{r+1}+Y_{r+1}=k_n, Y_{r+1}\not \in A_n}\right)+\bigO{k_n^{-C}} \\
	&=(1+o(1))\Prob{X_{r+1}+Y_{r+1}=k_n}+\bigO{k_n^{-C}}.
\end{align*}
Note that $\Prob{X_{r+1}+Y_{r+1}=k_n}$ no longer depends on $y_{r+1}$ and neither does the $\bigO{k_n^{-C}}$ error term. Therefore we have
\begin{align*}
	&\Prob{X_1+Y_1=\dots=X_{r+1}+Y_{r+1}=k_n} \\
	&\sum_{y_{r+1} \in A_n} \Prob{X_1+Y_1=\dots=X_{r}+Y_r=k_n|Y_{r+1}=y_{r+1}} 	
		\Prob{X_{r+1}+y_{r+1}=k_n}\Prob{Y_{r+1}=y_{r+1}}+\bigO{k_n^{-C}}\\
	&=(1+o(1))\Prob{X_{r+1}+Y_{r+1}=k_n}\sum_{y_{r+1} \in A_n} 	
		\Prob{X_1+Y_1=\dots=X_{r}+Y_r=k_n|Y_{r+1}=y_{r+1}}\Prob{Y_{r+1}=y_{r+1}}\\
	&\hspace{10pt}+\bigO{k_n^{-C}}.
\end{align*}
For the last summation we have
\begin{align*}
	&\sum_{y_{r+1} \in A_n} \Prob{X_1+Y_1=\dots=X_{r}+Y_r=k_n|Y_{r+1}=y_{r+1}}\Prob{Y_{r+1}=y_{r+1}} \\
	&= \Prob{X_1+Y_1=\dots=X_r+Y_r=k_n,Y_{r+1}\in A_n} \\
	&= (1+o(1))\Prob{X_1+Y_1=\dots=X_r+Y_r=k_n}+\bigO{k_n^{-C}}.
\end{align*}
Finally, plugging this into the previous step gives 
\begin{align*}
	&\Prob{X_1+Y_1=\dots=X_{r+1}+Y_{r+1}=k_n} \\
	&=(1+o(1))\Prob{X_{r+1}+Y_{r+1}=k_n}\sum_{y_{r+1} \in A_n} 	
		\Prob{X_1+Y_1=\dots=X_{r}+Y_r=k_n|Y_{r+1}=y_{r+1}}\Prob{Y_{r+1}=y_{r+1}}\\
	&\hspace{10pt}+\bigO{k_n^{-C}} \\
	&=(1+o(1))\Prob{X_{r+1}+Y_{r+1}=k_n}\Prob{X_1+Y_1=\dots=X_r+Y_r=k_n}+\bigO{k_n^{-C}}.
\end{align*}
which establishes~\eqref{eq:probdegfact_main} and thus the claim of the lemma.
\end{proof}

\subsection{Factorial moments of degrees}\label{ssec:factorial_moments_GPo}

Now that we have analyzed the joint neighbourhoods and degree distributions in both the Poissonized KPKVB and finite box model, we can show convergence of the factorial moments of the number of nodes of degree $k$ in both models.

\begin{lemma}[Factorial moments]\label{lem:factmoment}
Recall that $N_{Po}(k)$ denotes the number of degree $k$ vertices in $\GPo$. Let $(k_n)$ be a sequence of integers with $0\leq k_n \leq n-1$, $k_n = \bigO{n^{\frac{1}{2\alpha+1}}}$ and $k_n \rightarrow \infty$. Then, for any positive integer $r$, it holds that
\[
	\Exp{{N_{Po}(k_n) \choose r}} =(1+o(1)) \frac{(\Exp{ N_{Po}(k_n)})^r}{r!}.
\]
\end{lemma}

The proof of this result requires the following technical lemma which states that the integration of the joint degree distribution can be factorized.

\begin{lemma}[Factorization of degrees]
	\label{lem:asympind}
	Let $(k_n)$ be a sequence of integers with $0\leq k_n \leq n-1$, $k_n = \bigO{n^{\frac{1}{2\alpha+1}}}$ and $k_n \rightarrow \infty$.
	
	Let $\varphi$ be either $\varphi_{\mathrm{box}}$ or $\varphi_{\Po}$. Then we have that
	\begin{align*}
	&\int_\Rcal \cdots \int_\Rcal \varphi(\{p_1,\dots,p_r \}, k_n ;p_1,\dots,p_r) 
		 \dd \mu(p_1) \cdots \dd \mu(p_r) \\
	&= (1+o(1))\left(\int_\Rcal  \varphi(\{p_1\}, k_n ;p_1) \dd \mu(p)\right)^r .
	\end{align*}
\end{lemma}

\begin{proof}
Let $C>r(2\alpha+1)$ and define the set $A = (\Rcal \times \cdots \times\Rcal) \backslash (\stripknc \times \cdots\times \stripknc)$. We will first show that the contribution of the integration of $\varphi$ over this range is negligible. 

For $(p_1,\dots,p_r) \in (\Rcal \times \cdots \times\Rcal) \backslash (\stripknc \times \cdots\times \stripknc)$, there is a $j$, $1\leq j\leq r$, such that $y_j \not \in [y_{k_n,c}^-,y_{k_n,c}^+]$, so that the Chernoff bound (see in~\ref{eq:chernoff_bound_degrees}) yields that $\Prob{D_{\GPo}(v_j)=k_n } = \bigO{k_n^{-C}}$. As, for $1\leq j\leq r$, the event $\{D_{\GPo}(p_1)=\dots=D_{\GPo}(p_r)=k_n\}$ implies the event $\{D_{\GPo}(v_j)=k_n\}$, it follows that $\Prob{D_{\GPo}(p_1)=\dots=D_{\GPo}(p_r)=k_n} = \bigO{k_n^{-C}}$ and hence,
\begin{align*}
	\int_A \Pee(D_{\GPo}(p_1)=\dots=D_{\GPo}(p_r)=k_n) \dd \mu(p_1) \cdots \dd \mu(p_r)	= O\left(n^r k_n^{-C}\right).
\end{align*}

Next we note that by the concentration of heights (Proposition~\ref{prop:concentration_height_general}) it follows that
\begin{align*}
	\int_{\stripknc} \varphi(\{p_1\}, k_n ;p_1) \dd \mu(p_1) 
	&= (1+\smallO{1}) \int_\Rcal  \varphi(\{p_1\}, k_n ;p_1) \dd \mu(p_1) \\
	&= (1+\smallO{1}) 2\alpha \xi^{2\alpha} n k_n^{-(2\alpha + 1)} \numberthis \label{eq:integration_average_degree_stripe}
\end{align*}
Since for $C>r(2\alpha+1)$ it holds that $n^r k_n^{-C} = \smallO{\left(n k_n^{-(2\alpha + 1)}\right)^r}$ it now suffices to show that
\begin{align*}
	&\int_{\stripknc} \cdots \int_{\stripknc} \varphi(\{p_1,\dots,p_r \}, k_n ;p_1,\dots,p_r) 
		\dd \mu(p_1)\cdots \dd \mu(p_r) \\
	&= (1+o(1))\left(\int_{\stripknc}  \varphi(\{p_1\}, k_n ;p_1) \dd \mu(p_1) \right)^r.
\end{align*}

We will prove this using induction. More precisely, for every fixed positive integer $s$, we will show by induction on $r$, $1\leq r\leq s$, that for all $p_{r+1},\dots,p_s \in \stripknc$, it holds that
\begin{align*}
	&\int_{\stripknc} \cdots \int_{\stripknc} \varphi(\{p_1,\dots,p_r\}, k_n ;p_1,\dots,p_s)
		\dd \mu(p_1) \cdots \dd \mu(p_r)\\
	&= (1+o(1))\left(\int_{\stripknc} \varphi(\{p_1\}, k_n ;p_1) \dd \mu(p_1)\right)^r .
\end{align*}
Note that for $r=s$ this is the claim of the lemma. Throughout the proof $H$ is either $\GPo \cup\{p_1,\dots,p_s\}$ or $\Gbox \cup \{p_1,\dots,p_s\}$, depending on whether $\varphi$ is, respectively, $\varphi_{\Po}$ or $\varphi_{\mathrm{box}}$.

For $r=1$, we only need to show that uniformly for $p_1 \in \stripknc$,
\begin{align*}
\varphi(\{p_1\}, k_n ;p_1,\dots,p_s) = (1+o(1))\varphi(\{p_1\}, k_n ;p_1).
\end{align*}
To see this, note that as $p_1\in \stripknc$, the expected degree of $p_1$ in $\GPo$ is $\Theta(k_n)$. Assume that $p_1$ is adjacent to $s'<s$ many vertices among $p_2,\dots,p_s$. Then, as $s'<s$ is bounded and $k_n \rightarrow \infty$ as $n\rightarrow \infty$,  we have that
\begin{align*}
\Pee(D_{\GPo}(p_1)=k_n-s') = (1+o(1))\Pee(D_{\GPo}(p_1)=k_n).
\end{align*}
So, we have the base case of the induction,
\begin{align*}
\varphi(\{p_1\}, k_n ;p_1,\dots,p_s)&=\Pee(D_H(p_1)=k_n)
= \Pee(D_{\GPo}(p_1)=k_n-s')\\
&= (1+o(1))\Pee(D_{\GPo}(p_1)=k_n)=(1+o(1))\varphi(\{p_1\}, k_n ;p_1).
\end{align*}
Assuming the claim holds for integer $r <s$, we will show that it holds for $r+1$.

Let $p_{r+2},\dots,p_s\in \stripknc$ (if $r+2>s$, this definition is void and the corresponding points will never be used in the proof).
Fix $0<\epsilon<1$ small enough s.t. $\frac{1}{2}+\epsilon<\alpha$. Define the region that the $(r+1)$-th vertex $p_{r+1}$ is far apart from all other vertices horizontally,
$$\Fcal_\epsilon(k_n)=\{(p_1,\dots,p_{r+1}) \in \left(\stripknc\right)^{r+1}: \forall 1\leq i \leq r: |x_{p_i}-x_{p_{r+1}}|\geq k_n^{1+\epsilon}\}.$$

We will split the integration into this region and its complement $\Fcal_\epsilon(k_n)^c = \left(\stripknc\right)^{r+1} \backslash \Fcal_\epsilon(k_n)$.

Firstly, we derive an upper bound for the complement $\Fcal_\epsilon(k_n)^c$. Note that 
\[
	\varphi(\{p_1,\dots,p_{r+1}\}, k_n ;p_1,\dots,p_s) \leq \varphi(\{p_1,\dots,p_r\}, k_n ;p_1,\dots,p_s),
\]
and so,
\begin{align*}
	&\int \int_{\Fcal_\epsilon(k_n)^c} \varphi(\{p_1,\dots,p_{r+1}\}, k_n ;p_1,\dots,p_s)
		\dd \mu(p_1) \cdots \dd \mu(p_{r+1})\\
	&\leq \int \int_{\Fcal_\epsilon(k_n)^c} \varphi(\{p_1,\dots,p_r\}, k_n ;p_1,\dots,p_s)
		\dd \mu(p_1) \cdots \dd \mu(p_{r+1}).
\end{align*}

For $(p_1,\dots,p_{r+1})\in \Fcal_\epsilon(k_n)^c$, we have that $(p_1,\dots,p_r) \in \left(\stripknc\right)^r$ and $p_{r+1}=(x_{r+1},y_{r+1})$ satisfies $y_{k_n,c}^-\leq y_{r+1}\leq y_{k_n,c}^+$ and $x_{r+1}$ falls into an interval $I_n$ of length $2k_n^{1+\epsilon}$.
As the integrand $\varphi(\{p_1,\dots,p_r\};p_1,\dots,p_s)$ is constant in $x_{r+1}$, we can upper bound the corresponding integration as follows,
\begin{align*}
	&\int_{\{(x_{r+1},y_{r+1})\in\stripknc: x_{r+1}\in I_n\}} \varphi(\{p_1,\dots,p_r\}, k_n ;p_1,\dots,p_s) 
		\dd \mu(p_{r+1}) \\
	&= \int_{y_{k_n,c}^-}^{y_{k_n,c}^+} \int_{I_n} \varphi(\{p_1,\dots,p_r\}, k_n ;p_1,\dots,p_s)\frac{\alpha \nu}{\pi} e^{-\alpha y_{r+1}} dx_{r+1}dy_{r+1} \\
	&\leq 2k_n^{1+\epsilon} \varphi(\{p_1,\dots,p_r\}, k_n ;p_1,\dots,p_s)\int_{y_{k_n,c}^-}^{y_{k_n,c}^+}  \frac{\alpha \nu}{\pi} e^{-\alpha y_{r+1}} dy_{r+1}.
\end{align*}
Furthermore, we have that
\begin{align*}
	\int_{y_{k_n,c}^-}^{y_{k_n,c}^+} e^{-\alpha y_{r+1}}dy_{r+1} 
	\leq \int_{y_{k_n,c}^-}^\infty e^{-\alpha y_{r+1}}dy_{r+1} 
	= \bigO{e^{-\alpha y_{k_n,c}^-}} = \bigO{\left( \frac{k_n-c\sqrt{k_n\ln k_n}}{\xi}\right)^{-2\alpha}} 
	= \bigO{k_n^{-2\alpha}}.
\end{align*}
We have thus established that
\begin{align*}
	&\int \int_{\Fcal_\epsilon(k_n)^c} \varphi(\{p_1,\dots,p_r\}, k_n ;p_1,\dots,p_s)
		\dd \mu(p_1) \cdots \dd \mu(p_{r+1})\\
	&\leq \bigO{k_n^{1+\epsilon} k_n^{-2\alpha}}\int_{\stripknc} \cdots \int_{\stripknc} 
		\varphi(\{p_1,\dots,p_r\}, k_n ;p_1,\dots,p_s) \dd \mu(p_1) \cdots \dd \mu(p_{r}).
\end{align*}
The $r$-fold integral can the bounded from above using the induction hypothesis as
\begin{align}\label{eq:indhyp}
	\bigO{k_n^{1+\epsilon-2\alpha}}\left(\int_{\stripknc} \varphi(\{p_1\};p_1) \dd \mu(p_1) \right)^r .
\end{align}
Finally, we use that $k_n^{1+\epsilon-2\alpha} = \smallO{\int_{\stripknc} \varphi(\{p_1\};p_1)\mu_n(dp_1)}$. To see this, we will firstly show that $k_n^{1+\epsilon-2\alpha} = \smallO{nk_n^{-(2\alpha+1)}}$ and then apply~\eqref{eq:integration_average_degree_stripe}, which says that
\begin{align*}
\int_{\stripknc} \varphi(\{p_1\};p_1) \dd \mu(p_1) = \bigT{n k_n^{-(2\alpha+1)}}.
\end{align*}
By our choice of $\epsilon$, we have that $\frac{1}{2}+\epsilon<\alpha$, which implies that $\frac{2+\epsilon}{1+2\alpha}<1$, and hence, using $k_n = \bigO{n^{\frac{1}{2\alpha+1}}}$,
\begin{align*}
\frac{k_n^{1+\epsilon-2\alpha}}{nk_n^{-(2\alpha+1)}} = n^{-1} k_n^{2+\epsilon} = \bigO{n^{-1} n^{\frac{2+\epsilon}{1+2\alpha}}} = \smallO{n^{-1}n}= \smallO{1}.
\end{align*}
Having shown that $k_n^{1+\epsilon-2\alpha} = \smallO{\int_{\stripknc} \varphi(\{p_1\};p_1)\mu_n(dp_1)}$, we therefore conclude from~\eqref{eq:indhyp} that
\begin{align*}
	&\int \int_{\Fcal_\epsilon(k_n)^c} \varphi(\{p_1,\dots,p_r\}, k_n ;p_1,\dots,p_s)
		\dd \mu(p_1) \cdots \dd \mu(p_{r+1})\\
	&=\smallO{\int_{\stripknc} \varphi(\{p_1\}, k_n ;p_1)\mu_n(dp_1)}
		\left(\int_{\stripknc} \varphi(\{p_1\}, k_n ;p_1) \dd \mu(p_1) \right)^r  \\
	&=\smallO{\left(\int_{\stripknc} \varphi(\{p_1\}, k_n ;p_1) \dd \mu(p_1)\right)^{r+1}}.
\end{align*}

For the integration over $\Fcal_\epsilon(k_n)$, recall that by Lemma~\ref{lem:probdegFact}, \[
	\varphi(\{p_1,\dots,p_{r+1}\}, k_n ;p_1,\dots,p_s)
	=(1+o(1))\varphi(\{p_1,\dots,p_r\}, k_n ;p_1,\dots,p_s)
	\varphi(\{p_{r+1}\}, k_n ;p_1,\dots,p_s)+\bigO{k_n^{-C}}.
\]
This implies that
\begin{align*}
	&\int \int_{\Fcal_\epsilon(k_n)} \varphi(\{p_1,\dots,p_{r+1}\}, k_n ;p_1,\dots,p_s) 
		 \dd \mu(p_1) \cdots \dd \mu(p_{r+1})\\
	&=(1+o(1))\int \int_{\Fcal_\epsilon(k_n)} \varphi(\{p_1,\dots,p_r\}, k_n ;p_1,\dots,p_s) 
		\varphi(\{p_{r+1}\};p_1,\dots,p_s) \dd \mu(p_1) \cdots \dd \mu(p_{r+1}) \\
	&\hspace{10pt}+\bigO{k_n^{-C}}\int \int_{\Fcal_\epsilon(k_n)} 
		\dd \mu(p_1) \cdots \dd \mu(p_{r+1}) =:M+E.
\end{align*}
To finish the induction step we have to show that
\begin{align*}
M=(1+o(1)) \left(\int_{\stripknc} \varphi(\{p_1\}, k_n ;p_1)\mu_n(dp_1)\right)^{r+1}
\end{align*}
and
\begin{align*}
E=\smallO{\left(\int_{\stripknc} \varphi(\{p_1\}, k_n ;p_1)\mu_n(dp_1)\right)^{r+1}}.
\end{align*}

For $M$, note that we can factorize to integration over $p_1,\dots,p_r$ and over $p_{r+1}$.

Furthermore, we note that the condition $(p_1,\dots,p_{r+1})\in \Fcal_\epsilon(k_n)$ implies that (writing $p_{r+1}=(x_{r+1},y_{r+1})$) the horizontal coordinate $x_{r+1}$ falls into an interval $I_n$ of length $L_n$, satisfying
\[
	\frac{\pi n}{\nu}-2rk_n^{1+\epsilon}\leq L_n\leq \frac{\pi n}{\nu}-2k_n^{1+\epsilon}.
\]
As $k_n^{1+\epsilon} = \bigO{n^{\frac{1+\epsilon}{2\alpha+1}}} = o(n)$ for $\epsilon<1$ and $\alpha>\frac{1}{2}$, this shows that the length of the integration range in $x_{r+1}$ satisfies $L_n=(1+o(1))\frac{\pi n}{\nu}$. Thus, we have that
\begin{align*}
	M &= (1+o(1))\int_{I_n}\int_{y_{k_n,c}^-}^{y_{k_n,c}^+} \varphi(\{p_{r+1}\}, k_n ;p_1,\dots,p_s)
		\frac{\alpha\nu}{\pi} e^{-\alpha y_{r+1}}dy_{r+1}dx_{r+1}\\
	&\hspace{15pt} \times \int_{\stripknc} \cdots \int_{\stripknc} 
		\varphi(\{p_1,\dots,p_r\}, k_n ;p_1,\dots,p_s) \dd \mu(p_1) \cdots \dd \mu(p_{r})\\
	&=(1+o(1))n\int_{y_{k_n,c}^-}^{y_{k_n,c}^+} \varphi(\{p_{r+1}\}, k_n ;p_1,\dots,p_s)\alpha e^{-\alpha y_{r+1}}dy_{r+1}\\
	&\hspace{15pt} \times \int_{\stripknc} \cdots \int_{\stripknc} \varphi(\{p_1,\dots,p_r\};p_1,\dots,p_s)
		\dd \mu(p_1) \cdots \dd \mu(p_{r})\\
	&=(1+o(1))\int_{\stripknc} \varphi(\{p_{r+1}\}, k_n ;p_1,\dots,p_s)\mu_n(dp_1)\\
	&\hspace{15pt} \times\int_{\stripknc} \cdots \int_{\stripknc} \varphi(\{p_1,\dots,p_r\}, k_n ;p_1,\dots,p_s)
		\dd \mu(p_1) \cdots \dd \mu(p_{r}).
\end{align*}
By applying the base case of the induction to the first factor and the induction hypothesis to the second one, we have derived that
\begin{align*}
M&=(1+o(1))\left(\int_{\stripknc} \varphi(\{p_1\}, k_n ;p_1) \dd \mu(p_1) \right)^{r+1}.
\end{align*}
Finally, for $E$ we observe that
\begin{align*}
	E=\bigO{k_n^{-C}}\int \int_{\Fcal_\epsilon(k_n)} \dd \mu(p_1) \cdots \dd \mu(p_{r+1}) = \bigO{k_n^{-C}}(1+o(1))n^{r+1}.
\end{align*}
Recall that again by~\eqref{eq:integration_average_degree_stripe},
\begin{align*}
\int_{\stripknc} \varphi(\{p_1\}, k_n ;p_1) \dd \mu(p_1) = \bigT{n k_n^{-(2\alpha+1)}},
\end{align*}
which implies that
\begin{align*}
\left(\int_{\stripknc} \varphi(\{p_1\}, k_n ;p_1) \dd \mu(p_1) \right)^{r+1} 
= \bigT{n^{r+1}k_n^{-r(2\alpha+1)}}.
\end{align*}
For $C>r(2\alpha+1)$, we can conclude that indeed
\begin{align*}
E = \bigO{n^{r+1}k_n^{-C}}=\smallO{n^{r+1}k_n^{-r(2\alpha+1)}}=\smallO{\left(\int_{\stripknc} \varphi(\{p_1\}, k_n ;p_1) \dd \mu(p_1)\right)^{r+1}}.
\end{align*}
\end{proof}

We now prove the result for the factorial moments.

\begin{proofof}{Lemma~\ref{lem:factmoment}}
We give the proof for the Poissonized KPKVB model. The proof for the finite box model $\Gbox$ follows using similar arguments.

First of all, we observe that
\begin{align*}
{N_{\Po}(k_n) \choose r} = \frac{1}{r!} \sum_{p_1,\dots,p_r\in V(\GPo), \atop \text{distinct}} \1_{\{D_{\GPo}(p_1)=\dots=D_{\GPo}(p_r)=k_n\}}.
\end{align*}
This can be seen by induction on $r$. For $r=1$, the claim is clear. Assuming it holds for $r\geq 1$, by the induction hypothesis,
\begin{align*}
&{N_{\Po}(k_n) \choose r+1} = {N_{\Po}(k_n) \choose r} \frac{N_{\Po}(k_n)-r}{r+1} \\
&= \frac{1}{(r+1)!}\sum_{p_1,\dots,p_r\in V(\GPo), \atop \text{distinct}} \1_{\{D_{\GPo}(p_1)=\dots=D_{\GPo}(p_r)=k_n\}} (N_{\Po}(k_n)-r).
\end{align*}
Now, we can write
\[
	N_{\Po}(k_n) = \sum_{\stackrel{p_{r+1} \in V(\GPo),}{p_{r+1} \not \in\{ p_1,\dots,p_r\}} } \1_{\{D_{\GPo}(p_{r+1})=k_n\}}+\sum_{p_{r+1} \in \{p_1,\dots,p_r\}}\1_{\{D_{\GPo}(p_{r+1})=k_n\}}.
\]
The first sum leads to the right-hand side of the claim for $r+1$, whereas the second sum will cancel with the $-r$.

By the Campbell-Mecke formula 
\begin{align*}
	&\Exp{ {N_{\Po}(k_n) \choose r}}=\frac{1}{r!} \E\left[ \sum_{p_1,\dots,p_r\in V(\GPo), \atop \text{distinct}}  
		\1_{\{D_{\GPo}(p_1)=\dots=D_{\GPo}(p_r)=k_n\}}\right] \\
	&=\frac{1}{r!} \int_\Rcal \cdots \int_\Rcal \varphi_{\Po}(\{p_1,\dots,p_r\},p_1,\dots,p_r)
		\dd \mu(p_1) \cdots \dd \mu(p_{r}),
\end{align*}
where we integrate over $r$ additional points which we can think of as being added independently and with the same distribution as the vertices of the Poissonized KPKVB model $\GPo$ in the upper half-plane coordinates.

With $r=1$, it follows that
\[
	\Exp{N_{\Po}(k_n)} = \int_\Rcal \varphi_{\Po}(\{p_1\}, k_n ; p_1) f(x_1,y_1) \dd x_1 \dd y_1,
\]
which yields that the right-hand side of the claim of the lemma can be rewritten as
\[
	\frac{1}{r!}(\Exp{N_{\Po}(k_n)})^r = \frac{1}{r!}\left(\int_\Rcal \phi_{\Po}(\{p_1\}, k_n; p_1) \dd \mu(p_1)\right)^r.
\]

Using Lemma~\ref{lem:asympind}, we conclude that
\begin{align*}
	\Exp{ { N_{\Po}(k_n) \choose r}} 
	&= \frac{1}{r!} \int_\Rcal \cdots \int_\Rcal 
		\varphi_{\Po}(\{p_1,\dots,p_r\}, k_n; p_1,\dots,p_r) 
		\dd \mu(p_1) \cdots \dd \mu(p_{r}) \\
	&= (1+o(1))\frac{1}{r!}\left(\int_\Rcal \varphi_{\Po}(\{p_1\}, k_n;p_1) f(x_1,y_1) \dd x_1 \dd y_1 \right)^r \\
	&= (1+o(1))\frac{1}{r!}(\Exp{N_{\Po}(k_n)})^r.
\end{align*}
\end{proofof}

\subsection{Coupling $G_n$ to $\GPo$}\label{ssec:coupling_Gn_GPo}

In the previous sections we have established results for the degrees and the factorial moments of the degree $k_n$ nodes in the Poissonized KPKVB and finite box model. Our intended result, however, was for the degree distribution in the original KPKVB model. In order to extend the result for the Poissonized KPKVB model to the original model we will use a coupling argument to show that the expected difference between the number of degree $k_n$ nodes is negligible.

\begin{lemma}\label{lem:diff_Nk_hyperbolic_binomial_poisson}
As $n \rightarrow \infty$, it holds that for $0 \le k_n \le n - 1$,
\[
\Exp{\left|N_{n}(k_n)-N_{Po}(k_n)\right|} = \smallO{\Exp{N_{Po}(k_n)}}.
\]
\end{lemma}

\begin{proof}
We couple both models by taking an infinite supply of i.i.d.~points $u_1, u_2, \dots$ chosen according
to the $(\alpha,R)$-quasi uniform distribution and letting the vertices of $G(n;\alpha,\nu)$ be $u_1,\dots, u_n$ and
the vertices of $\GPo(n;\alpha,\nu)$ be $u_1,\dots, u_N$ with $N\isd \Po(n)$ independently of
$u_1, u_2, \dots$. Thus, under this coupling, the only difference between $G_n = G(n;\alpha,\nu)$ and $\GPo = \GPo(n;\alpha,\nu)$ is the number of points. Note that since $N$ is Poisson with mean $n$, it follows from the Chernoff bound (see also equation (135) in the paper) that we may assume that $n - C\sqrt{n \log n} \le N \le n + C \sqrt{n \log n}$. To keep notation simple we will suppress this conditioning in the derivations.

Clearly, if $N = n$ the graphs are the same. So we will consider the two cases $n - C\sqrt{n \log n} \le N < n$ and $n < N \le n + C \sqrt{n \log n}$. We will prove the latter case. The other case uses similar arguments and hence we omit the details here.

If $n < N \le n + C \sqrt{n \log n}$ then the $G_n$ has less vertices that $\GPo$. Write $V_n(k_n)$ and $V_{\Po}(k_n)$ to denote the set of vertices that have degree $k_n$ in $G_n$ and $G_{\Po}$, respectively. Then since the vertices $u_{n+1}, \dots, u_N$ are not present in $G_n$,
\begin{align*}
	\left|N_n(k_n) - N_{\Po}(k_n)\right| &= \sum_{i = 1}^N \1_{\left\{u_i \in V_n(k_n) \Delta V_{\Po}(k_n)\right\}}\\
	&= \sum_{i = 1}^n \1_{\left\{u_i \in V_n(k_n) \Delta V_{\Po}(k_n)\right\}}
		+ \sum_{i = n+1}^N \1_{\left\{u_i \in V_{\Po}(k_n)\right\}}.
\end{align*}
Let $D_{\Po}$ denote the degree in the Poissonized KPKVB model of a point $u$ placed according to the $(\alpha,R)$-quasi uniform distribution. Then, using the Campbell-Mecke formula, the expectation of the second summation equals
\begin{align*}
	(N-n) \Prob{D_{\Po} = k_n} &= (N-n) \int_0^R \Prob{\Po(\mu_{\Po,n}(y)) = k_n} \alpha e^{-\alpha y} \dd y\\
	&\le (1+o(1)) C \sqrt{n \log n} p_{k_n} = \smallO{\Exp{N_{\Po}(k_n)}}.
\end{align*}
Therefore it remains to consider the first summation.

Let $D_n(u)$ and $D_{\Po}(u)$ denote the degree of a point $u$ in $G_n$ and $G_{\Po}$, respectively. Then there are two scenarios to consider: 1) either $D_n(u_i) = k_n$ and $D_{\Po}(u_i) \ne k_n$ or 2) $D_n(u_i) \ne k_n$ and $D_{\Po}(u_i) = k_n$. In the first case, since $u_i$ is present in both graphs it follows that $D_{\Po}(u_i) > k_n$. Similarly, for the second case it must hold that $D_n(u_i) < k_n$. Hence we have
\begin{align*}
	&\hspace{-30pt}\sum_{i = 1}^n \1_{\left\{u_i \in V_n(k_n) \Delta V_{\Po}(k_n)\right\}}\\
	&= \sum_{i = 1}^n \1_{\left\{D_n(u_i) = k_n, D_{\Po}(u_i) > k_n\right\}}
		+ \sum_{i = 1}^n \1_{\left\{D_n(u_i) < k_n, D_{\Po}(u_i) = k_n\right\}}.
\end{align*}

Let us first consider the second summation, i.e. the case where the node has degree smaller than $k_n$ in $G_n$. Taking the expectation gives $n \Prob{D_n < k_n, D_{\Po} = k_n}$, where $D_n$ denotes the degree in the KPKVB model of a point $u$ placed according to the $(\alpha,R)$-quasi uniform distribution. We now observe that because the points $u_1, \dots, u_N$ used to couple the graphs are independent, we can view the graph $G_n$ as being obtained from $\GPo$ by removing $N-n$ points, uniformly at random. Therefore if a point has degree $k_n$ in $\GPo$ but smaller degree in $G_n$, this means that at least one of its neighbors was removed. Denote by $Z(n)$ a random variable with a Hypergeometric distribution, for taking $N-n$ draws from a population of size $N$, where there are $k_n$ good objects. That is, $Z(n)$ denote the number of removed neighbors of a node $u$ with degree $k_n$ in $\GPo$. We then have
\begin{align*}
	\Prob{D_n < k_n, D_{\Po} = k_n} &= \Prob{Z(n) > 1}\Prob{D_{\Po} = k_n}\\
	&\le \Exp{Z(n)} \Prob{D_{\Po} = k_n} = \frac{(N-n) k_n}{N} \Prob{D_{\Po} = k_n}.
\end{align*}
Because $\alpha > 1/2$ and $k_n = \bigO{n^{\frac{1}{2\alpha+ 1}}}$ it holds that $k_n = \smallO{\sqrt{\frac{n}{\log n}}}$. Since $N = \bigT{n}$ and $N - n \le \bigO{\sqrt{n \log n}}$ it then follows that $\frac{(N-n) k_n}{N} = \smallO{1}$, from which we conclude that
\begin{align*}
	\Exp{\sum_{i = 1}^n \1_{\left\{d_n(u_i) < k_n, D_{\Po}(u_i) = k_n\right\}}}
	&= n \Prob{D_n < k_n, D_{\Po} = k_n}\\
	&\le \smallO{1} n \Prob{D_{\Po} = k_n} = \smallO{\Exp{N_{\Po}(k_n)}}.
\end{align*}

We now proceed with the other summation, for the case where a vertex has degree $k_n$ in $G_n$ but larger degree in $G_{\Po}$. Since the degree of $u$ in $\GPo$ can be a most $N-n$ larger we have
\[
	\Prob{D_n(u) = k_n, D_{\Po}(u) > k_n} = \sum_{t = 1}^{N - n} \Prob{D_n(u) = k_n, D_{\Po}(u) = k_n + t}.
\]
Using that the graph $G_n$ can be seen as being obtained from $\GPo$ by removing $N-n$ points uniformly at random, a point with degree $k_n + t$ in $G_{\Po}$ can only have degree $k_n$ in $G_n$ if exactly $t$ of its neighbors where removed. Let us therefore denote by $Z(n,t)$ a random variable with a Hypergeometric distribution, for taking $N-n$ draws from a population of size $N$, where there are $k_n + t$ good objects. Then
\[
	\Prob{D_n(u) = k_n, D_{\Po}(u) = k_n + t} = \Prob{Z(n,t) = t}\Prob{D_{\Po} = k_n + t}.
\]

Recall that, for any $0 < \varepsilon < 1$
\begin{align*}
	\Prob{D_{\Po} = k_n + t} &= \int_0^R \Prob{\Po(\mu_{Po,n}(y)) = k_n+ t} \alpha e^{-\alpha y} dy\\
	&= \int_0^{(1-\varepsilon)R} \Prob{\Po(\mu_{Po,n}(y)) = k_n+ t} \alpha e^{-\alpha y} dy \\
	&\hspace{10pt}+ \int_{(1-\varepsilon)R}^R \Prob{\Po(\mu_{Po,n}(y)) = k_n+ t} \alpha e^{-\alpha y} dy
\end{align*}
By Lemma~\ref{lem:degree_integral} the first part is $(1+o(1)) \pmf(k_n + t)$ while the second part is $O(n^{-2\alpha(1-\varepsilon)})$ and hence
\[
	\Prob{D_{\Po} = k_n + t} \le \bigO{1}\left(\pmf(k_n + t) + n^{-2\alpha(1-\varepsilon)}\right).
\]
In addition we have that $\Prob{Z(n,t) = t} \le \bigO{1} \frac{\Exp{Z(n,t)}}{t}$. We thus obtain
\begin{align*}
	\sum_{t = 1}^{N-n} \Prob{Z(n,t) = t}\Prob{D_{\Po} = k_n + t}
	&\le \bigO{1} \sum_{t = 1}^{N-n} \frac{\Exp{Z(n,t)}}{t}\left(\pmf(k_n + t) + n^{-2\alpha(1-\varepsilon)}\right)\\
	&= \bigO{1} \sum_{t = 1}^{N-n} \frac{(N-n)(k_n + t)}{N}\left(\pmf(k_n + t) + n^{-2\alpha(1-\varepsilon)}\right)\\
	&= \bigO{\sqrt{\frac{\log n}{n}}} \sum_{t = 1}^{N-n} \frac{k_n}{t}
		\left(\pmf(k_n + t) + n^{-2\alpha(1-\varepsilon)}\right)\\
	&\hspace{10pt}+ \bigO{\sqrt{\frac{\log n}{n}}} \sum_{t = 1}^{N-n}
		\left(\pmf(k_n + t) + n^{-2\alpha(1-\varepsilon)}\right),
\end{align*}
where we used that $\frac{N-n}{N} = \bigO{\sqrt{\frac{\log n}{n}}}$.

We will show that both summations are $\smallO{\pmf(k_n)}$. For the first summation we recall that $\frac{k_n (\log n)^{3/2}}{\sqrt{n}} = \smallO{1}$ for $k_n = \bigO{n^{\frac{1}{2\alpha + 1}}}$, while for $\varepsilon > 0$ small enough $n^{-2\alpha(1-\varepsilon)} = \smallO{k_n^{-(2\alpha + 1)}} = \smallO{\pmf(k_n)}$. Hence, since $\pmf(k_n + t)\le \pmf(k_n)$,
\begin{align*}
	\bigO{\sqrt{\frac{\log n}{n}}} \sum_{t = 1}^{N-n} \frac{k_n}{t}
		\left(\pmf(k_n + t) + n^{-2\alpha(1-\varepsilon)}\right)
	&\le \bigO{ k_n \sqrt{\frac{\log n}{n}} \pmf(k_n)} \sum_{t = 1}^{N - m} \frac{1}{t}\\
	&= \bigO{\frac{k_n (\log n)^{3/2}}{\sqrt{n}}} \pmf(k_n) = \smallO{\pmf(k_n)}.
\end{align*}

For the other summation we use that
\[
	\sum_{t = 1}^{N-n} \pmf(k_n + t) \le \sum_{t = 1}^\infty \pmf(k_n + t)
	\le \bigO{k_n^{-2\alpha}} = \bigO{k_n \pmf(k_n)},
\]
together with the fact that for $\varepsilon$ small enough, $\log(n) n^{-2\alpha(1-\varepsilon)} = \smallO{k_n^{-(2\alpha + 1)}} = \smallO{\pmf(k_n)}$. This implies that
\begin{align*}
	\bigO{\sqrt{\frac{\log n}{n}}} \sum_{t = 1}^{N-n}
		\left(\pmf(k_n + t) + n^{-2\alpha(1-\varepsilon)}\right)
	&\le \bigO{\sqrt{\frac{\log n}{n}} k_n \pmf(k_n)} + \bigO{(N-n)\sqrt{\frac{\log n}{n}} n^{-2\alpha(1-\varepsilon)}}\\
	&\le \smallO{\pmf(k_n)} + \bigO{(\log n) n^{-2\alpha(1-\varepsilon)}} = \smallO{\pmf(k_n)}.
\end{align*}

It now follows that
\begin{align*}
	n \Prob{D_n(u) = k_n, D_{\Po}(u) > k_n}
	&= \sum_{t = 1}^{N-n} \Prob{D_n = k_n, D_{\Po} = k_n + t}\\
	&= n \sum_{t = 1}^{N-n} \Prob{Z(n,t) = t}\Prob{D_{\Po} = k_n + t}\\
	&= n \, \smallO{\Prob{D_{\Po} = k_n}} = \smallO{\Exp{N_{\Po}(k_n)}},
\end{align*}
which finishes the proof for the case where $N > n$.
\end{proof}

\subsection{Proof of Theorem~\ref{thm:degrees_hyperbolic}}\label{ssec:proof_thm_degrees}

We now have all necessary ingredients to prove the main result on the degrees, Theorem~\ref{thm:degrees_hyperbolic}.

\begin{proofof}{Theorem~\ref{thm:degrees_hyperbolic}}

Recall that we shall only give the proof for the case where $k_n \to \infty$, since result (i) for fixed $k = \bigO{1}$ follows from~\cite{gugelmann2012random}.

\begin{enumerate}[\upshape (i)]
\item  First we recall that the statements regarding $\pmf(k)$ and its asymptotic behavior follow from Equation~\ref{eq:def_pk} and Equation~\ref{eq:degree_distribution_P_asymptotics}.

By Lemma~\ref{lem:expnnkn},
\[
\Exp{{ N_{Po}(k_n) }}=(1+o(1))n \pmf(k_n).
\]
Using Lemma~\ref{lem:factmoment} with $r=2$ we have that $\Exp{{N_{Po}(k_n) \choose 2}} = (1+o(1))\frac{(\Exp{ N_{Po}(k_n)})^2}{2}$, which implies $\Exp{N_{Po}(k_n)^2} = 2\Exp{{N_{Po}(k_n) \choose 2}} +\Exp{N_{Po}(k_n)} = (1+o(1))(\Exp{ N_{Po}(k_n)})^2 +o((\Exp{N_{Po}(k_n)})^2)$ (because $\Exp{N_{Po}(k_n)} = (1+o(1)) n\pmf(k_n) \rightarrow \infty$). Hence, by Chebychev for any $\epsilon >0$,
\[
	\Prob{|N_{Po}(k_n)-\Exp{N_{Po}(k_n)}|\geq \epsilon \Exp{N_{Po}(k_n)}}
	\leq \frac{Var(N_{Po}(k_n))}{\epsilon^2 (\Exp{ N_{Po}(k_n)})^2} = o(1).
\]
As $N_n(k_n)=N_{Po}(k_n)+N_n(k_n)-N_{Po}(k_n) = N_{Po}(k_n)\pm |N_n(k_n)-N_{Po}(k_n)|$ (where the sign depends on whether $N_n(k_n)>N_{Po}(k_n)$ or not), due to Lemma~\ref{lem:diff_Nk_hyperbolic_binomial_poisson}, we have that
\[
	\Exp{N_n(k_n)} = \Exp{N_{Po}(k_n)} \pm \Exp{ |N_n(k_n)-N_{Po}(k_n)|} = (1+o(1))\Exp{N_{Po}(k_n)} = (1+o(1))n \pmf(k_n).
\]

\item Let $\zeta = 2\alpha \xi^{2\alpha}c^{-(2\alpha+1)} \in \R$.
The proof consists of showing that $\Exp{{N_{Po}(k_n) \choose r} } \rightarrow \frac{\zeta^r}{r!}$ for every positive integer $r$.

If $k_n=(1+o(1))c n^{\frac{1}{2\alpha+1}}$, for some positive constant $c > 0$. then by Lemma~\ref{lem:expnnkn},
\begin{align*}
	\Exp{N_{Po}(k_n)} &= (1+o(1))2\alpha \xi^{2\alpha} n (1+o(1))^{-(2\alpha+1)} c^{-(2\alpha+1)}n^{-1} \\
	&= (1+o(1))2\alpha \xi^{2\alpha} c^{-(2\alpha+1)}=(1+o(1))\zeta,
\end{align*}
which implies $\Exp{N_{Po}(k_n)} \rightarrow \zeta$ (as $\zeta$ is a positive constant). From Lemma~\ref{lem:factmoment}, it then follows that $\Exp{{N_{Po}(k_n) \choose r}} =(1+o(1))\frac{(\Exp{ N_{Po}(k_n)})^r}{r!}\rightarrow \frac{\zeta^r}{r!}$. Thus, it follows from \cite[Theorem 8.3.1]{alon2016probabilistic} that $N_{Po}(k_n) \xrightarrow{d} \mathrm{Po}(\zeta)$ for the Poissonized version.

Finally, since $k_n = \bigO{n^{\frac{1}{2\alpha+1}}}$, by Lemma~\ref{lem:diff_Nk_hyperbolic_binomial_poisson}, $\Exp{| N_n(k_n) - N_{Po}(k_n)|} = o(\Exp{N_{Po}(k_n)}) = o(\zeta)$, from which it follows that $\Prob{|N_n(k_n)-N_{Po}(k_n)|\geq 1} \leq \Exp{ |N_n(k_n)-N_{Po}(k_n)|} = o(\zeta)$. Hence, it also holds in the original KPKVB model that $N_n(k_n) \xrightarrow{d} \mathrm{Po}(\zeta)$.

\item We will show that in this case $\Exp{N_n(k_n)} = o(1)$. This then implies, by Markov's inequality,
\[
	\Prob{N_n(k_n)>0}\leq \Exp{N_n(k_n)} = o(1).
\]

First we observe that as the Poissonized KPKVB model $\GPo$ has the same intensity measure as the original KPKVB model with a fixed number $n$ of points, the expected degree of a vertex of the KPKVB model with radial coordinate $r=R-y$ is given by $\mu_{\Po}(y)$ and hence,
\[
	\Exp{N_n(k_n)} = n \int_0^R \Prob{\mathrm{Bin}(n-1,\mu_{\Po}(y)/n) = k_n} 
	\frac{\alpha \sinh(\alpha (R-y))}{\cosh(\alpha R)-1} dy.
\]
Fix $0 < \varepsilon < \frac{4\alpha - 1}{4\alpha + 2} \wedge \frac{2\alpha - 1}{2\alpha}$. We first show that we only need to consider integration up to $y \le (1-\varepsilon)R$. By our choice of $\varepsilon$, $2\alpha(1-\varepsilon) > 1$, so that
\[
	\frac{\cosh(\alpha \varepsilon R) - 1}{\cosh(\alpha R) - 1}
	= \bigO{n^{-2\alpha (1 - \varepsilon)}} = \smallO{n^{-1}}.
\]
This implies
\[
	n \int_{(1-\varepsilon)R}^R \Prob{\mathrm{Bin}(n-1,\mu_{\Po}(y)/n) = k_n}
	\frac{\alpha \sinh(\alpha (R-y))}{\cosh(\alpha R)-1} dy
	\le n \frac{\cosh(\alpha \epsilon R)}{\cosh(\alpha R)-1} = \smallO{1},
\]
and thus it is enough to show that
\[
	n \int_0^{(1-\varepsilon)R} \Prob{\mathrm{Bin}(n-1,\mu_{\Po}(y)) = k_n}
	\frac{\alpha \sinh(\alpha (R-y))}{\cosh(\alpha R)-1} dy = \smallO{1}.
\]
Note that for all $0 \le y \le (1-\varepsilon)R$ we have
\[
	\frac{\alpha \sinh(\alpha (R-y))}{\cosh(\alpha R)-1} = (1+\smallO{1}) \alpha e^{-\alpha y}.
\]
Hence, by bounding the Binomial probability (see Lemma~\ref{lem:binomial_poisson_bound})
\begin{align*}
	&\hspace{-20pt} n \int_0^{(1-\varepsilon)R} \Prob{\mathrm{Bin}(n-1,\mu_{\Po}(y)) = k_n}
		\frac{\alpha \sinh(\alpha (R-y))}{\cosh(\alpha R)-1} dy\\
	&\le (1+\smallO{1}) \frac{e}{\sqrt{2\pi}} \sqrt{\frac{n}{n-k_n}}
		n \int_0^{(1-\varepsilon)R} \Prob{\Po(\mu_{Po}(y)) = k_n} \alpha e^{-\alpha y} dy \\
	&\le (1+\smallO{1}) \frac{e}{\sqrt{2\pi}} \sqrt{\frac{n}{n-k_n}} n \pmf(k_n)
	= \bigO{\sqrt{\frac{n}{n-k_n}} n k_n^{-(2\alpha + 1)}}.
\end{align*}
We shall now consider two cases: $n^{\frac{1}{2\alpha + 1}} \ll k_n < n^{1-\varepsilon}$ and $n^{1-\varepsilon} \le k_n \le n - 1$.

If $n^{\frac{1}{2\alpha + 1}} \ll k_n < n^{1-\varepsilon}$ then $\sqrt{\frac{n}{n-k_n}} = 1 + \smallO{1}$ and hence
\[
	\sqrt{\frac{n}{n-k_n}} n k_n^{-(2\alpha + 1)} = \bigO{n k_n^{-(2\alpha + 1)}} = \smallO{1}.
\]

For $k_n \ge n^{1 - \varepsilon}$ we have, by our choice of $\varepsilon$, that $\frac{3}{2} - (2\alpha + 1)(1-\varepsilon) < 0$, and thus
\[
	\sqrt{\frac{n}{n-k_n}} n k_n^{-(2\alpha + 1)} = \bigO{n^{\frac{3}{2}} k_n^{-(2\alpha + 1)}}
	= \bigO{n^{\frac{3}{2} - (2\alpha + 1)(1-\varepsilon)}} = \smallO{1}.
\]
\end{enumerate}
\end{proofof}

\section{Clustering when \texorpdfstring{$k\to\infty$}{k tends to infinity} : overview of the proof strategy}

The proof of Theorem~\ref{thm:mainktoinfty} follows the same strategy as outlined in Section~\ref{sec:proof_outline} and executed in Section~\ref{sec:proofs_fixed_k}. However, the fact that $k = k_n \to \infty$ as $n \to \infty$, introduces significant technical challenges, especially for $k_n$ close the the maximum scale $n^{\frac{1}{2\alpha + 1}}$. For example, the coupling between $\GPo$ and $\Gbox$ we use becomes less exact so that we can no longer use Lemma~\ref{lem:coupling_edges} to conclude that triangle counts in $\GPo$ and $\Gbox$ are asymptotically equivalent. Moreover, since we are ultimately interested in recovering the scaling of $c(k_n;G_n)$, which Theorem~\ref{thm:mainktoinfty} claims is $\gamma(k_n)$, we need to show that each step in the strategy outlined in Section~\ref{sec:proof_outline} only introduces error terms that are of smaller order, i.e. that are $\smallO{\gamma(k_n)}$. This will turn out to require a great deal of care in bounding all error terms we encounter.

In this section we explain the challenges associated with each step and give a detailed overview of the structure for the proof of Theorem~\ref{thm:mainktoinfty} using intermediate results for each of the steps. We first define the scaling function
\begin{equation}\label{eq:def_scaling_function}
	s(k) = \begin{cases}
		k^{-(4\alpha - 2)} &\mbox{if } \frac{1}{2} < \alpha < \frac{3}{4},\\
		\log(k) k^{-1} &\mbox{if } \alpha = \frac{3}{4},\\
		k^{-1} &\mbox{if } \alpha > \frac{3}{4},
	\end{cases}
\end{equation}
so that $\gamma(k) = \bigT{s(k)}$ as $k \to \infty$. We will end this section with the proof of Theorem~\ref{thm:mainktoinfty}, based on the intermediate results.

\begin{remark}[Diverging $k_n$]
Throughout the remainder of this paper, unless stated otherwise, $\{k_n\}_{n \ge 1}$ will always denote a sequence of non-negative integers satisfying $k_n \to \infty$ and $k_n = \smallO{n^{\frac{1}{2\alpha + 1}}}$, as $n \to \infty$.
\end{remark}

We start with introducing a slightly modified version of the local clustering function, which will be convenient for computations later,
\begin{equation}\label{eq:def_local_clustering_ast_general}
	c^\ast(k;G) = \frac{1}{\Exp{N(k)}} \sum_{v \in V(G) \atop \text{deg}(v)=k} c(v).
\end{equation}

Notice that the only difference between $c(k;G)$ and $c^\ast(k;G)$ is that we replace $N(k)$ by its expectation $\Exp{N(k)}$. The advantage is that now, the only randomness is in the formation of triangles. In addition, note that since $\Exp{N(k)} > 0$ a case distinction for $N(k)$ is no longer needed for $c^\ast(k;G)$. It is however still relevant since we are eventually interested in $c(k;G)$. Following the notational convention, throughout the remainder of this paper we write $c^\ast(k; \GPo)$ and $c^\ast(k; \Gbox)$ to denote the modified local clustering function in $\GPo$ and $G_{\mathcal{P},n}(\alpha,\nu)$, respectively.

Figure~\ref{fig:overview_proof} shows a schematic overview of the proof of Theorem~\ref{thm:mainktoinfty} based on the different propositions described below, plus the sections in which theses propositions are proved. Observe that the order in which the intermediate results are proved is reversed with respect to the natural order of reasoning. This does not create any circular logic, since each intermediate result is independent of the others. We choose this order because results proved in the later stages are helpful to deal with error terms coming up in proofs at earlier stages and hence help streamline those proofs. Below we briefly describe each of the intermediate steps leading up to the proof of Theorem~\ref{thm:mainktoinfty}.

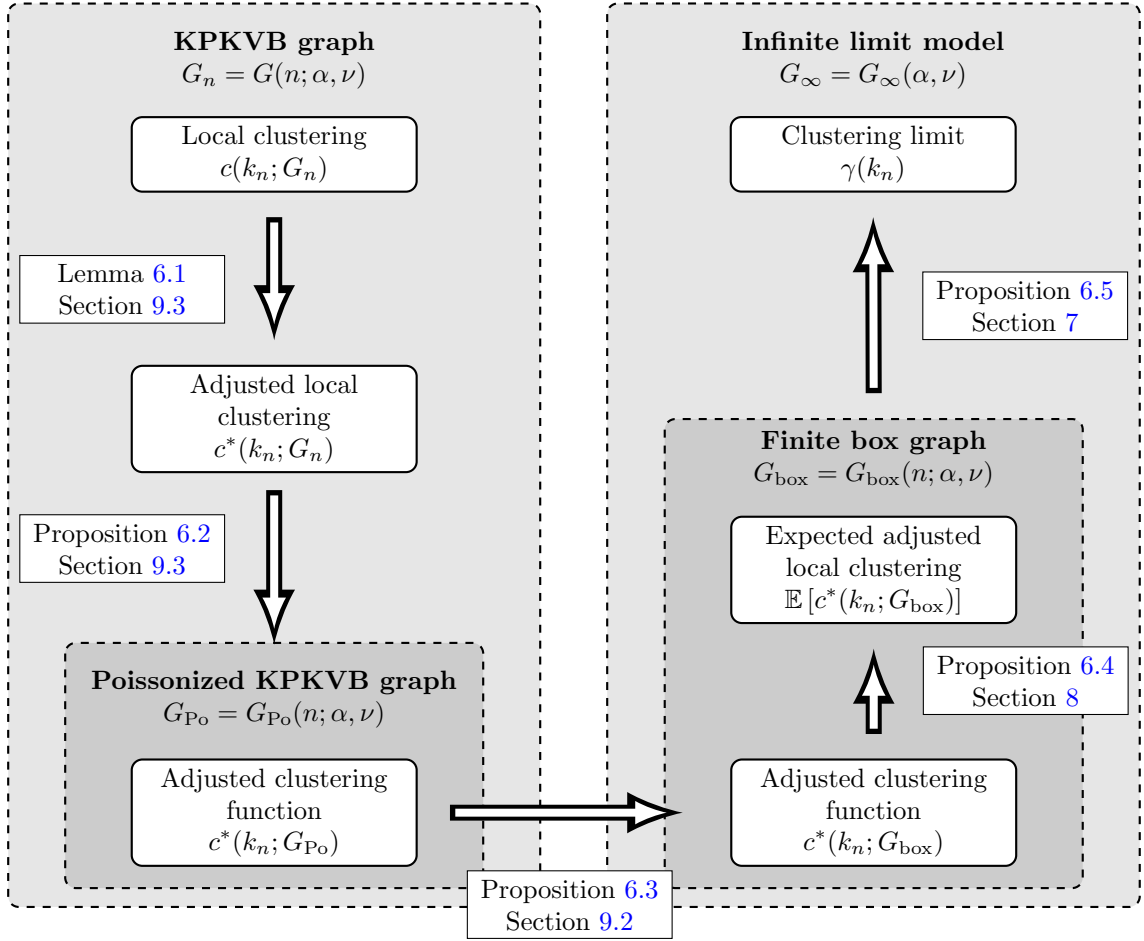
\begin{figure}[!t]
\centering

\begin{tikzpicture}

\pgfdeclarelayer{background}
\pgfdeclarelayer{foreground}
\pgfsetlayers{background,main,foreground}

\tikzstyle{block} = [draw, text centered, rounded corners, draw=black, thick, fill=white]
\tikzstyle{textblock} = [draw, text centered, draw=black, fill=white]

\tikzset{
  double arrow/.style args={#1 colored by #2 and #3}{
    -latex,line width=#1,#2, 
    postaction={draw,-latex,#3,line width=(#1)/2,
                shorten <=(#1)/3,shorten >=(#1)}, 
  }
}

\draw node (anchor) at (0,0) {};


\path (anchor)+(0,0) node (hyp_header) {\begin{minipage}[c]{10em}
	\begin{center}
		\textbf{KPKVB graph}\\
		$G_n = G(n; \alpha, \nu)$
	\end{center}
\end{minipage}};

\path (hyp_header.south)+(0,-0.75) node (c_hyp) [block] {\begin{minipage}[c]{10em}
	\begin{center}
	Local clustering\\
	$\displaystyle c(k_n;G_n)$
	\end{center}
\end{minipage}};


\path (c_hyp.south)+(0,-3) node (c_ast_hyp) [block] {\begin{minipage}[c]{10em}
	\begin{center}
	Adjusted local clustering\\
	$\displaystyle c^\ast(k_n;G_n)$
	\end{center}
\end{minipage}};

\path (c_ast_hyp.north)+(-2,1) node (hyp_text) [textblock] {\begin{minipage}[c]{7em}
	\begin{center}
		Lemma \ref{lem:clustering_ast_H}\\
		Section~\ref{ssec:coupling_H_HP}
	\end{center}
\end{minipage}};



\path (c_ast_hyp.south)+(0,-3) node (pois_hyp_header) {\begin{minipage}[c]{14em}
	\begin{center}
		\textbf{Poissonized KPKVB graph}\\
		$\GPo = \GPo(n;\alpha,\nu)$
	\end{center}
\end{minipage}};

\path (pois_hyp_header.north)+(-2,1.5) node (pois_hyp_text) [textblock] {\begin{minipage}[c]{7em}
	\begin{center}
		Proposition \ref{prop:clustering_ast_H_Pois}\\
		Section~\ref{ssec:coupling_H_HP}
	\end{center}
\end{minipage}};


\path (pois_hyp_header.south)+(0,-1) node (c_ast_pois_hyp) [block] {\begin{minipage}[c]{10em}
	\begin{center}
	Adjusted clustering function\\
	$\displaystyle c^\ast(k_n; \GPo)$
	\end{center}
\end{minipage}};


\path (hyp_header.east)+(6,0) node (pois_header) {\begin{minipage}[c]{10em}
	\begin{center}
		\textbf{Infinite limit model}\\
		$\Ginf= \Ginf(\alpha, \nu)$
	\end{center}
\end{minipage}};

\path (pois_header)+(0,-1.25) node (c_infty) [block] {\begin{minipage}[c]{10em}
	\begin{center}
	Clustering limit\\
	$\displaystyle \gamma(k_n)$
	\end{center}
\end{minipage}};

\path (c_infty.south)+(2,-1.5) node (exp_c_pois_text) [textblock] {\begin{minipage}[c]{7em}
	\begin{center}
		Proposition \ref{prop:convergence_average_clustering_P_n}\\
		Section \ref{sec:clustering_Pn_to_P}
	\end{center}
\end{minipage}};


\path (c_ast_pois_hyp.east)+(6,0) node (c_ast_pois_n) [block] {\begin{minipage}[c]{10em}
	\begin{center}
	Adjusted clustering function\\
	$\displaystyle c^\ast(k_n; \Gbox)$
	\end{center}
\end{minipage}};

\path (c_ast_pois_n.south)+(-4,-0.5) node (c_ast_pois_n_text) [textblock] {\begin{minipage}[c]{7em}
	\begin{center}
		Proposition \ref{prop:couling_c_H_P}\\
		Section \ref{ssec:coupling_HP_ast_P}
	\end{center}
\end{minipage}};

\path (c_ast_pois_n.north)+(0,2.5) node (exp_c_pois_n) [block] {\begin{minipage}[c]{10em}
	\begin{center}
	Expected adjusted local clustering\\
	$\displaystyle \Exp{c^\ast(k_n; \Gbox)}$
	\end{center}
\end{minipage}};

\path (exp_c_pois_n.south)+(2,-0.75) node (exp_c_pois_n_text) [textblock] {\begin{minipage}[c]{7em}
	\begin{center}
		Proposition \ref{prop:concentration_local_clustering_P_n}\\
		Section \ref{sec:concentration_c_P_n}
	\end{center}
\end{minipage}};

\path (exp_c_pois_n.north)+(0,0.75) node (pois_n_header) {\begin{minipage}[c]{13em}
	\begin{center}
		\textbf{Finite box graph}\\
		$\Gbox = \Gbox(n; \alpha, \nu)$
	\end{center}
\end{minipage}};


\path (c_hyp.south)+(0,-0.2) node (arrow_1l) {};
\path (c_ast_hyp.north)+(0,0.2) node (arrow_1r) {};

\draw [double arrow=6pt colored by black and white] (arrow_1l) -- (arrow_1r);

\path (c_ast_hyp.south)+(0,-0.1) node (arrow_2l) {};
\path (pois_hyp_header.north)+(0,0.1) node (arrow_2r) {};

\draw [double arrow=6pt colored by black and white] (arrow_2l) -- (arrow_2r);

\path (c_ast_pois_hyp.east)+(0.3,0) node (arrow_3l) {};
\path (c_ast_pois_n.west)+(-0.5,0) node (arrow_3r) {};

\draw [double arrow=6pt colored by black and white] (arrow_3l) -- (arrow_3r);

\path (c_ast_pois_n.north)+(0,0.2) node (arrow_4l) {};
\path (exp_c_pois_n.south)+(0,-0.2) node (arrow_4r) {};

\draw [double arrow=6pt colored by black and white] (arrow_4l) -- (arrow_4r);

\path (exp_c_pois_n.north)+(0,1.5) node (arrow_5l) {};
\path (c_infty.south)+(0,-0.2) node (arrow_5r) {};

\draw [double arrow=6pt colored by black and white] (arrow_5l) -- (arrow_5r);

\begin{pgfonlayer}{background}

\path (c_hyp)+(-3.5,2) node (hyp_a) {};
\path (c_ast_pois_hyp)+(3.5,-1.25) node (hyp_b) {};
\path[rounded corners, draw=black, dashed, thick, fill=black!10] (hyp_a) rectangle (hyp_b);

\path (pois_hyp_header)+(-2.75,0.75) node (pois_hyp_a) {};
\path (c_ast_pois_hyp)+(2.75,-1) node (pois_hyp_b) {};
\path[rounded corners, draw=black, dashed, thick, fill=black!20] (pois_hyp_a) rectangle (pois_hyp_b);

\path (c_infty)+(-3.5,2) node (pois_a) {};
\path (c_ast_pois_n)+(3.5,-1.25) node (pois_b) {};
\path[rounded corners, draw=black, dashed, thick, fill=black!10] (pois_a) rectangle (pois_b);

\path (exp_c_pois_n)+(-2.75,2) node (pois_n_a) {};
\path (c_ast_pois_n)+(2.75,-1) node (pois_n_b) {};
\path[rounded corners, draw=black, dashed, thick, fill=black!20] (pois_n_a) rectangle (pois_n_b);

\end{pgfonlayer}

\end{tikzpicture}

\caption{Overview of the proof strategy for Theorem \ref{thm:mainktoinfty}. The left column denote the models in which the true hyperbolic balls are used while the right column contains the models that use an approximation of these. The most important part is the transition between these to setting which is accomplished by Proposition~\ref{prop:couling_c_H_P}.}
\label{fig:overview_proof}
\end{figure}

\subsection{Adjusted clustering and the Poissonized KPKVB model}\label{ssec:KPKVB_to_GPo_infinite_k}

Recall that the first step for the fixed $k$ case was to show that the transition from the KPKVB graph $G_n = G(n;\alpha,\nu)$ to the Poissonized version $\GPo$ did not influence clustering. Here we first make a transition from the local clustering function $c(k_n; G_n)$ to the adjusted version $c^\ast(k_n; G_n)$. The following lemma justifies working with this modified version. The proof uses a concentration result for $N_{n}(k_n)$ and full details can be found in Section~\ref{ssec:coupling_H_HP}.
For a sequence of random variables $(X_n)_{n\in \mathbb{N}}$ and a sequence $(a_n)_{n \in \mathbb{N}}$, we write $X_n = \smallOp{a_n}$ to 
denote that $X_n/a_n \stackrel{\mathbb{P}}{\to} 0$ as $n\to \infty$. 
\begin{lemma}\label{lem:clustering_ast_H}
As $n\to \infty$ 
\[
\left|c^\ast(k_n; G_n) - c(k_n; G_n)\right| = \smallOp{s(k_n)}.
\]
\end{lemma}

We then establish that the modified local clustering function for KPKVB graphs $G_n$ behaves similarly to that in $\GPo$. The proof, found in Section~\ref{ssec:coupling_H_HP}, is based on a standard coupling between a Binomial Point Process and Poisson Point Process. 

\begin{proposition}\label{prop:clustering_ast_H_Pois}
As $n \to \infty$,
\[
	\Exp{\left|c^\ast(k_n;G_n) - c^\ast(k_n; \GPo)\right|} = \smallO{s(k_n)}.
\]
\end{proposition}

\subsection{Coupling of local clustering between \texorpdfstring{$\GPo$}{G Po} and \texorpdfstring{$\Gbox$}{G box}}

The next step is to show that the modified clustering is preserved under the coupling described in Section~\ref{ssec:coupling_H_P}. The proof can be found in Section~\ref{ssec:coupling_HP_ast_P}. This step is one of the key technical challenges we face in proving Theorem~\ref{thm:mainktoinfty}. 

To understand why, recall that the degree $k$ of a node is related to its height $y$, roughly speaking, by $k \approx \xi e^{y/2}$. Therefore, when $k$ is fixed we have that the heights of nodes with that degree are also fixed, in particular $y < R/4$ for large enough $n$. In addition, the main contribution of triangles would also come from nodes with heights $y^\prime < R/4$. This allowed us to use Lemma~\ref{lem:coupling_edges} and conclude that the triangles present in the graph $\GPo$ where exactly those present in $\Gbox$ and therefore the local clustering function was the same in both models. When $k_n \to \infty$ this is no longer true in general. For instance, suppose $k_n = n^{\frac{1-\varepsilon}{2\alpha + 1}}$, for some small $0 < \varepsilon < 1$. Then the relation $k_n \approx \xi e^{y_n/2}$ implies that $y_n \approx \frac{2(1-\varepsilon)}{2\alpha + 1}\log(n) - 2\log(\xi)$. Since
$R/4 = \frac{1}{2}\log(n) - \frac{1}{2}\log(\nu)$ we get that $R/4 = \smallO{y_n}$ for all $\alpha > (3 - 4\varepsilon)/2$ and hence $y_n > R/4$ for large enough $n$, violating the conditions of Lemma~\ref{lem:coupling_edges}. However, by carefully analyzing the difference between the adjusted local clustering function in both models we can still make the same conclusion. This is summarized in the following proposition whose proof is found in Section~\ref{ssec:coupling_HP_ast_P}.

\begin{proposition}[Coupling result for adjusted clustering function]\label{prop:couling_c_H_P}
As $n \to \infty$,
\[
	\Exp{\left|c^\ast(k_n; \GPo) - c^\ast(k_n; \Gbox)\right|} = \smallO{s(k_n)}.
\]
\end{proposition}


Together, the three results described so far imply that the difference between the clustering function for a KPKVB graph and the adjusted clustering function for the finite box graph $\Gbox$ converges to zero faster than the proposed scaling $\gamma(k_n)$ in Theorem~\ref{thm:mainktoinfty}. Hence, it is enough to prove the result for $c^\ast(k; \Gbox)$. 

\subsection{From the finite box to the infinite model}

To compute the limit of the adjusted clustering function $c^\ast(k; \Gbox)$ we first prove in Section~\ref{sec:concentration_c_P_n} that it is concentrated around its mean $\Exp{c^\ast(k_n; \Gbox)}$.


\begin{proposition}[Concentration for adjusted clustering function in $\Gbox$]\label{prop:concentration_local_clustering_P_n}
As $n \to \infty$,
\[
	\Exp{\left|c^\ast(k_n; \Gbox) - \Exp{c^\ast(k_n; \Gbox)}\right|} = \smallO{s(k_n)}.
\]
\end{proposition}

This result represents another technical challenge we face when considering $k_n \to \infty$. For the proof, we first identify the specific range of heights that give the main contribution to the triangle count, showing that the triangles coming from nodes with heights outside this range is of smaller order. Then we prove a concentration result for the main term, using that the neighbourhoods of two nodes whose $x$-coordinates are sufficiently separated can be considered to be disjoint (see Section~\ref{ssec:joint_degrees_GPo}). The full details are found in Section~\ref{sec:concentration_c_P_n}.

Assuming this concentration result, we are left to compute the expectation $\Exp{c^\ast(k_n; \Gbox)}$ and show that it is asymptotically equivalent to $\gamma(k_n)$ as $n \to \infty$. To accomplish this we move to the infinite limit model $\Ginf$ and show that the difference between the expected value of $c^\ast(k;\Gbox)$ and $\gamma(k_n)$ goes to zero faster than the proposed scaling in Theorem \ref{thm:local_clustering_hyperbolic}.

\begin{proposition}[Transition to the infinite limit model]\label{prop:convergence_average_clustering_P_n}
As $n \to \infty$,
\[
	\left|\Exp{c^\ast(k_n; \Gbox)} - \gamma(k_n)\right| = \smallO{s(k_n)}.
\]
\end{proposition}

Recall that for the finite box model the left and right boundaries of $\Rcal_n$ where identified, so that graph $\Gbox$ contains some additional edge with respect to the induced subgraph of $\Ginf$ on $\Rcal_n$. The proof of Proposition~\ref{prop:convergence_average_clustering_P_n} therefore relies on analyzing the number of triangles coming from these additional edges and showing that their contribution to the clustering function are of negligible order, see Section~\ref{sec:clustering_Pn_to_P}. 

\begin{remark}[Notations for different graphs]
We will use the subscripts $n$, $\Po$, $\text{box}$ and $\infty$ to identify properties of, respectively, the KPKVB mode $G_n$, the Poisson version $\GPo$, the finite box model $\Gbox$ and the infinite model $\Ginf$. For example $N_{\Po}(k)$ denotes number of nodes with degree $k$ in $\GPo$ and $\rho_{\text{box}}(y,k) = \Prob{\Po(\mu(\BallPon{y})) = k}$, i.e. the degree distribution in $\Gbox$ for a point $p = (x,y)$.
\end{remark}

\subsection{Proof of the main results}\label{ssec:proof_main_result_diverging_k}

We are now ready to prove Theorem~\ref{thm:mainktoinfty}, using the results stated in the previous sections.

\begin{proofof}{Theorem~\ref{thm:mainktoinfty}}
We first rewrite $c(k_n; G_n)$ as
\begin{align*}
    c(k_n;G_n)-\gamma(k_n) &= \left(c(k_n;G_n)-c^\ast(k_n;G_n)\right)
    	+ \left(c^\ast(k_n;G_n)-c^\ast(k_n; \GPo)\right)\\
    &\hspace{10pt}+ \left(c^\ast(k_n; \GPo)-c^\ast(k_n; \Gbox)\right)
    	+ \left(c^\ast(k_n; \Gbox)-\Exp{c^\ast(k_n; \Gbox)}\right)\\
    &\hspace{10pt}+ \Exp{c^\ast(k_n; \Gbox)}-\gamma(k_n)
\end{align*}
Then, we take absolute values and apply the triangle inequality. 
\begin{align*}
    \left|c(k_n;G_n)-\gamma(k_n)\right|
    &\le \left|c(k_n;G_n)-c^\ast(k_n;G_n)\right|
    	+\left|c^\ast(k_n;G_n)-c^\ast(k_n; \GPo)\right|\\
    &\hspace{10pt}+ \left|c^\ast(k_n; \GPo)-c^\ast(k_n; \Gbox)\right|
    	+ \left|c^\ast(k_n; \Gbox)-\Exp{c^\ast(k_n; \Gbox)}\right|\\
    &+ \left|\Exp{c^\ast(k_n; \Gbox)}-\gamma(k_n)\right|
\end{align*}
The first term is $\smallOp{s(k_n)}$ by Lemma~\ref{lem:clustering_ast_H}. For the other terms, the propositions presented above in this section can be applied in order to show that the expectation of each difference is $\smallO{s(k_n)}$: Proposition~\ref{prop:clustering_ast_H_Pois} for the Poissonization in the second term, Proposition~\ref{prop:couling_c_H_P} for the coupling between the Poissonized KPKVB and the finite box model in the third term, Proposition~\ref{prop:concentration_local_clustering_P_n} for the concentration in the fourth term and finally Proposition~\ref{prop:convergence_average_clustering_P_n} for the transition to the infinite limit model. 

In particular, this implies that all terms are $\smallOp{s(k_n)}$. We thus conclude that
\begin{align*}
    \left|c(k_n;G_n)-\gamma(k_n)\right| = \smallOp{s(k_n)} = \smallOp{\gamma(k_n)},
\end{align*}
which establishes the first statement of the theorem and finishes the proof.  
\end{proofof}

\section{From \texorpdfstring{$\Gbox$}{G box} to \texorpdfstring{$\Ginf$}{G infinity} (Proving Proposition \ref{prop:convergence_average_clustering_P_n})}\label{sec:clustering_Pn_to_P}

In this section we shall relate the clustering in the finite box model $\Gbox$ to that of the infinite model. The main goal is to prove Proposition~\ref{prop:convergence_average_clustering_P_n} which states that
\[
	\left|\Exp{c^\ast(k_n; \Gbox)} - \gamma(k_n)\right| = \smallO{s(k_n)}.
\] 

Recall that $\Gbox$ is obtained by restricting the Poisson point process $\PPP$ to the box $\Rcal = (-I_n, I_n] \times (0, R]$, with $I_n = \frac{\pi}{2} e^{R/2}$ and connecting two points $p_1, p_2 \in \Rcal$ if and only if $|x_1 - x_2|_{\pi e^{R/2}} \le e^{(y_1 + y_2)/2}$. We also recall that by definition of the norm $|.|_{\pi e^{R/2}}$ the left and right boundaries of $\Rcal$ are identified. See Section~\ref{ssec:finite_model} for more details. Due to this identification of the boundaries some triples of nodes that form a triangle in the finite box model do not form a triangle in the infinite model. Therefore, to establish the required result we need to compute the asymptotic difference between triangle counts in both models. To keep notation concise we write $|\cdot|_n$ for the norm $|\cdot|_{\pi e^{R/2}}$.

For any $p \in \R \times \R_+$ we define for the finite box model,
\[
	T_{\text{box}}(p) = \sum_{p_1, p_2 \in \Pcal \setminus \{p\}, \atop \text{distinct}} T_{\text{box}}(p,p_1,p_2)
\]
where the sum is over all distinct pairs in $\Pcal \setminus p$ and
\[
	T_{\text{box}}(p,p_1,p_2) = \ind{p_1 \in \BallPon{p}}\ind{p_2 \in \BallPon{p}}\ind{p_2 \in \BallPon{p_1}}.
\]
Similarly, for the infinite model we define
\[
	T_\infty(y) = \sum_{p_1, p_2 \in \Pcal \setminus (0,y), \atop \text{distinct}} T_\infty(y,p_1,p_2),
\]
where
\[
	T_{\infty}(y,p_1,p_2) = \ind{p_1 \in \BallPo{y}}\ind{p_2 \in \BallPo{y}}\ind{p_2 \in \BallPo{p_1}}.
\]
Recall that, slightly abusing notation, we write $\BallPo{y}$ for $\BallPo{(0,y)}$ and that $\Nbox(k)$ denotes the number of vertices with degree $k$ in $\Gbox$.

We will first relate $\gamma(k_n)$ to an integral expression involving $T_\infty(y)$ and $\Exp{c^\ast(k_n;\Gbox)}$ to one involving $\Exp{T_{\text{box}}(y)}$. Recall the definition of $y_{k,C}^\pm$ from~\eqref{eq:def_y_k_C} and the interval $\Kcal_{C}(k_n) = [y_{k_n,C}^-, y_{k_n,C}^+]$. Note that for any $y \in \Kcal_C(k_n)$ it holds that
\begin{equation}\label{eq:def_K_C_set}
	\frac{k_n - C \sqrt{k_n \log(k_n)}}{\xi} \le e^{\frac{y}{2}}
	\le \frac{k_n + C \sqrt{k_n \log(k_n)}}{\xi}	
\end{equation}
and thus $k_n - C \sqrt{k_n \log(k_n)} \le \mu(y) \le k_n + C \sqrt{k_n \log(k_n)}$.

\begin{lemma}
Let $\gamma(k_n)$ be defined as in~\eqref{eq:gammakint}. Then as $n \to \infty$
\begin{equation}\label{eq:alt_gamma_kn}
	\gamma(k_n) = (1+\smallO{1})\frac{1}{k_n^2 \pmf(k_n)} \int_{\Kcal_{C}(k_n)} \Exp{T_\infty(y)} \rho(y,k) 
				\alpha e^{-\alpha y} \dd y. 
\end{equation}
Moreover,
\begin{equation}\label{eq:alt_c_ast_kn}
	\Exp{c^\ast(k_n;\Gbox)} = (1 + \smallO{1}) \frac{1}{k_n^2 \pmf(k_n)} \int_{\Kcal_{C}(k_n)} \Exp{T_{\text{box}}(y)}
		\rho_{\mathrm{box}}(y,k_n) \alpha e^{-\alpha y} \dd y
\end{equation}
as $n \to \infty$,
\end{lemma}

\begin{proof}
Recall that 
\[
	P(y) = \Exp{\ind{u_1 \in \BallPo{u_2}}},
\]
where $u_1$ and $u_2$ are independent and distributed according to the probability density \\$\Mu{\BallPo{y}}^{-1} \ind{u_i \in \BallPo{y}} f(x_i, y_i)$. It then follows from the Campbell-Mecke formula that
\begin{align*}
	\Exp{T_\infty(y)} &= \int \ind{p_1 \in \BallPo{y}}\ind{p_2 \in \BallPo{y}}
		\ind{p_2 \in \BallPo{p_1}} f(x_1,y_1) f(x_2,y_2) \dd x_1 \dd x_2 \dd y_1 \dd y_2\\
	&= \Mu{\BallPo{y}}^2 P(y).
\end{align*}
It then follows that,
\begin{align*}
	\gamma(k_n) &= \frac{1}{\pmf(k_n)} \cdot \int_0^\infty P(y) \rho(y,k) \alpha e^{-\alpha y} \dd y
	= \frac{1}{\pmf(k_n)} \int_0^\infty \Exp{T_\infty(y)} \Mu{\BallPo{y}}^{-2} \rho(y,k) 
		\alpha e^{-\alpha y} \dd y.
\end{align*}

Because $\Exp{T_\infty(y)} \Mu{\BallPo{y}}^{-2} = \bigO{1}$, uniformly in $y \in \R_+$ and $\mu(y) = (1+\smallO{1}) k_n$ uniformly for $y \in \Kcal_C(k_n)$, by the concentration of heights (Proposition~\ref{prop:concentration_height_general})
\begin{align*}
	&\hspace{-30pt} \frac{1}{\pmf(k_n)} 
		\int_0^\infty \Exp{T_\infty(y)} \Mu{\BallPo{y}}^{-2} \rho(y,k) \alpha e^{-\alpha y} \dd y \\
	&= (1+\smallO{1}) \frac{1}{\pmf(k_n)} 
		\int_{\Kcal_C(k_n)} \Exp{T_\infty(y)} \Mu{\BallPo{y}}^{-2} \rho(y,k) \alpha e^{-\alpha y} \dd y \\
	&= (1+ \smallO{1}) \frac{1}{k_n^2 \pmf(k_n)} \int_{\Kcal_C(k_n)} \Exp{T_\infty(y)} \rho(y,k) \alpha e^{-\alpha y} \dd y.
\end{align*}

For~\eqref{eq:alt_c_ast_kn} we recall that
\[
	c^\ast(k_n; \Gbox) = \frac{1}{\Exp{\Nbox(k_n)}} \sum_{p \in \Pcal} c_{\text{box}}(p)\ind{\Dbox(p) = k_n},
\]
where $c_{\text{box}}(p)$ can be expressed as
\[
	c_{\text{box}}(p) = \frac{1}{\binom{\Dbox(p)}{2}} \sum_{p_1, p_2 \in \Pcal \setminus p, \atop \text{distinct}} T_{\text{box}}(p,p_1,p_2)
	= \frac{T_{\text{box}}(p)}{\binom{\Dbox(p)}{2}}.
\]
By the Campbell-Mecke formula
\begin{align*}
	\Exp{c^\ast(k_n; \Gbox)} 
	&= \frac{1}{\Exp{\Nbox(k_n)}} \int_{\Rcal} \Exp{c_{\text{box}}(p)\ind{\Dbox(p) = k_n}} f(x,y) \dd x \dd y\\
	&= \frac{1}{\Exp{\Nbox(k_n)}} \int_{\Rcal} \CExp{c_{\text{box}}(p)}{\Dbox(p) = k_n}\rho_{\text{box}}(p,k_n) 
		f(x,y) \dd x \dd y\\
	&= (1+\smallO{1})\frac{n}{\Exp{\Nbox(k_n)}} \int_{\Kcal_C(k_n)} \CExp{c_{\text{box}}(y)}{\Dbox(y) = k_n}
		\rho_{\mathrm{box}}(y,k_n) \alpha e^{-\alpha y} \dd y,
\end{align*}
where the last line follows from the concentration of heights, for which we used the upper bound $\CExp{c_{\text{box}}(y)}{\Dbox(y) = k_n} \le 1$.

To analyze the conditional expectation we observe that, similar to the analysis of $\gamma(k_n)$, conditioned on there being $k_n$ points in $\BallPon{y}$, each point $u_i = (x_i,y_i)$ is independently distributed according to the probability density $\Mu{\BallPon{y}}^{-1} \ind{u_i \in \BallPon{y}} f(x_i,y_i)$. Therefore,
\begin{align*}
	\CExp{c_{\text{box}}(y)}{\Dbox(y) = k_n}
	&= \binom{k_n}{2}^{-1} \Exp{\sum_{1 \le i < j \le k_n} \ind{u_i \in \BallPon{u_j}}} \\
	&= \Exp{u_1 \in \BallPon{u_2}} \\
	&= \Mu{\BallPon{y}}^{-2} \iint T_{\text{box}}(y,p_1,p_2) f(x_1,y_1) f(x_2,y_2) 
		\dd x_1 \dd y_1 \dd x_2 \dd y_2\\
	&= \Mu{\BallPon{y}}^{-2} \Exp{T_{\text{box}}(y)}.
\end{align*}
and thus, by applying a concentration of heights argument on $\Mu{\BallPon{y}}^{-2}$,
\[
	\Exp{c^\ast(k_n; \Gbox)} = (1 + \smallO{1}) \frac{n \Mu{\BallPon{2\log(k_n/\xi)}}^{-2}}{\Exp{\Nbox(k_n)}} \int_{\Kcal_C(k_n)} \Exp{T_{\text{box}}(y)} \rho_{\mathrm{box}}(y,k_n) \alpha e^{-\alpha y} \dd y.
\]
To finish the argument, we first note that $\Mu{\BallPon{2\log(k_n/\xi)}}^{-2} = (1+\smallO{1})k_n^2$, while
\[
	\Exp{N_{\text{box}}(k_n)} = (1+\smallO{1}) n \pmf(k_n)
\]
by Lemma~\ref{lem:average_degree_G_box}. We therefore conclude that
\[
	\Exp{c^\ast(k_n; \Gbox)} = (1 + \smallO{1}) \frac{1}{k_n^2 \pmf(k_n)} \int_{\Kcal_C(k_n)} 
		\Exp{T_{\text{box}}(y)} \rho_{\mathrm{box}}(y,k_n) \alpha e^{-\alpha y} \dd y.
\]
\end{proof}

As a result of this lemma we only need to compare the difference in triangles between both models for height in the interval $\Kcal_C(k_n)$. This will significantly help the analysis.


%

\subsection{Comparing triangles between \texorpdfstring{$\Ginf$}{G infinity} and \texorpdfstring{$\Gbox$}{G box}}

To analyze the difference $\left|T_{\text{box}}(y) - T_\infty(y)\right|$ we first reiterate that the difference between the indicator $\ind{p_1 \in \BallPon{p}}$ in the finite box model and $\ind{p_1 \in \BallPo{p}}$ is that in $\Gbox$ we identified the boundaries of the interval $[-\frac{\pi}{2}e^{R/2}, \frac{\pi}{2} e^{R/2}]$ and we stop at height $y = R$. This induces a difference in triangle counts between both models. 
To see this, note that for any $p = (x,y)$ with $0 \le y \le R$ we have that $\BallPon{p} = \BallPo{p} \cap \Rcal$. This means that if $p^\prime, p_2 \in \BallPon{p}$ and $p_2 \in \BallPo{p^\prime} \cap \Rcal$ then $p_2 \in \BallPon{p} \cap \BallPon{p^\prime}$ and hence $(p,p^\prime,p_2)$ form a triangle both in $\Gbox$ and $\Ginf$. However, it could happen that there are points in the intersection $\BallPon{p} \cap \BallPon{p^\prime}$ that are not in $\BallPo{p} \cap \BallPo{p^\prime}$. Let us denote this region by $\mathcal{T}(p,p^\prime)$, see Figure~\ref{fig:comparing_triangles_diff_intersections} for an example of this region. Then, any $p_2 \in \mathcal{T}(p,p^\prime)$ creates a triangle with $p$ and $p^\prime$ in $\Gbox$ that is not present in $\Ginf$. Finally, any point $p_2 \in \BallPo{p} \cap \BallPo{p^\prime}$ with height $y_2 > R$ creates a triangle with $p, p^\prime$ in $\Ginf$ but not in $\Gbox$.

Let us now define the following triangle count function
\[
	\widetilde{T}_{\text{box}}(p_0) = \sum_{(p_1, p_2) \in \Pcal \setminus \{p_0\}, \atop \text{distinct}}
		\widetilde{T}_{\text{box}}(p_0,p_1,p_2).
\]
where
\[
	\widetilde{T}_{\text{box}}(p_0,p_1,p_2) = \ind{p_1 \in \BallPon{p}}\ind{p_2 \in \BallPon{p}}\ind{p_2 \in \BallPo{p_1} \cap \Rcal}.
\]
Then $\widetilde{T}_{\text{box}}(p_0)$ only counts those triangles attached to $p_0$ that exist in both $\Gbox$ and $\Ginf$ and thus, by definition of the region $\mathcal{T}(p_0,p_1)$,
\[
	T_{\text{box}}(p_0) - \widetilde{T}_{\text{box}}(p_0)
	= \sum_{p_1, p_2 \in \Pcal \setminus \{p_0\}, \atop \text{distinct}} \ind{p_1 \in \BallPon{p_0}} 
			\ind{p_2 \in \mathcal{T}(p_0, p_1)}.
\]

The next result, which is crucial for the proof of Proposition~\ref{prop:convergence_average_clustering_P_n}, computes the expected measure of $\mathcal{T}(p,p^\prime)$ with respect to $p^\prime$. 

\begin{lemma}\label{lem:clustering_error_T_term}
Let $p_0 = (0,y)$ with $y \in \Kcal_{C}(k_n)$. Then as $n \to \infty$,
\[
	\Exp{\left|T_{\text{box}}(p_0) - \widetilde{T}_{\text{box}}(p_0)\right|}
	= y \, \bigO{n^{-(2\alpha - 1)}} + e^{y}\,\bigO{n^{-(4\alpha-2)}}.
\]
\end{lemma}

The proof of the lemma is not difficult but cumbersome, since it involves computing many different integrals. We postpone this proof till the end of this section and proceed with the main goal, proving Proposition~\ref{prop:convergence_average_clustering_P_n}. First we state a small lemma about the scaling of $s(k_n)$ that will be very useful.  

\begin{lemma}\label{lem:scaling_s_alpha}
Let $s(k_n)$ be as defined in \eqref{eq:def_scaling_function}. Then for any $k_n = \smallO{n^{\frac{1}{2\alpha + 1}}}$, as $n \to \infty$,
\[
	n^{-(2\alpha - 1)} = \smallO{s(k_n)}.
\]
\end{lemma}

\begin{proof}
First let $\frac{1}{2} < \alpha < \frac{3}{4}$. Then
\[
	n^{-(2\alpha - 1)}s(k_n)^{-1} = n^{-(2\alpha -1)}k_n^{4\alpha - 2}
	= \smallO{n^{-(2\alpha - 1) + \frac{4\alpha - 2}{2\alpha + 1}}} 
	= \smallO{n^{-\frac{4\alpha^2 - 4\alpha + 1}{2\alpha + 1}}}
	= \smallO{1},
\]
since $4\alpha^2 - 4\alpha + 1 > 0$ for all $\alpha > \frac{1}{2}$. Similarly, for $\alpha \ge \frac{3}{4}$ we have
that $4\alpha^2 > 2$ and hence,
\[
	n^{-(2\alpha - 1)} s_{\alpha}(k_n) = \smallO{n^{-(2\alpha - 1)} k_n} = \smallO{n^{-\frac{4\alpha^2 - 2}{2\alpha + 1}}}
	= \smallO{1}.
\]
\end{proof}

The first key implication of Lemma~\ref{lem:clustering_error_T_term} is that the triangle count in the finite box model is equivalent to $k_n^2 P(y)$, where $P(y)$ is defined by~\eqref{eq:def_Py}.

\begin{lemma}\label{lem:appox_triangle_count_box}
Let $p_0 = (0,y)$. Then uniformly for $y \in \Kcal_C(k_n)$,
\[
	\Exp{T_{\mathrm{box}}(p_0)} = (1+\smallO{1})k_n^2 P(y) + \smallO{s(k_n) k_n^2},
\]
as $n \to \infty$.
\end{lemma}

\begin{proof}
Recall that $\Exp{T_\infty(y)} = \mu(y)^2 P(y) = (1+\smallO{1})k_n^2 P(y)$ on $\Kcal_C(k_n)$. We will show that
\[
	\Exp{\left|\Exp{T_{\mathrm{box}}(p_0)} - T_\infty(y)\right|} = \smallO{s(k_n) k_n^2},
\]
which implies the result.

Define $\Rcal^\prime := (\R \times \R_+) \setminus \Rcal$. Then we can write
\begin{align*}
	\left|T_{\mathrm{box}}(p_0) - T_\infty(y)\right|
	&\le \left|T_{\mathrm{box}}(p_0) - \widetilde{T}_{\text{box}}(p_0)\right| 
		+ \left|\widetilde{T}_{\text{box}}(p_0) - T_\infty(y)\right|\\
	&= \left|T_{\mathrm{box}}(p_0) - \widetilde{T}_{\text{box}}(p_0)\right| 
		+ \sum_{p_1, p_2 \in \Pcal \cap \Rcal^\prime, \atop \text{distinct}} T_{\infty}(p_0,p_1,p_2).
\end{align*}
Then by the Campbell-Mecke formula
\begin{align*}
	\left|\Exp{T_{\text{box}}(p_0) - T_{\infty}(p_0)}\right|
	&\le \Exp{\left|T_{\text{box}}(p_0) - \widetilde{T}_{\text{box}}(p_0)\right|}
		+ \int_{\Rcal^\prime}\int_{\Rcal^\prime} T_{\infty}(p_0,p_1,p_2) \dd \mu(p_1) \dd \mu(p_2). 
\end{align*}

The first part is taken care of by Lemma~\ref{lem:clustering_error_T_term}. For the other integral we have
\begin{align*}
	\iint_{\Rcal^\prime} T_{\infty}(p_0,p_1,p_2) \dd \mu(p_1) \dd \mu(p_2)
	&\le \left(\int_{\Rcal^\prime} \ind{p_1 \in \BallPo{p_0}} f(x_1,y_1) \dd x_1 \dd y_1\right)^2\\
	&= \bigO{\left(e^{y/2} \int_{R}^\infty e^{-(\alpha - \frac{1}{2})y_1} \dd y_1\right)^2}\\
	&= \bigO{e^{y} e^{-(2\alpha-1)R}} = \bigO{e^y n^{-(4\alpha - 2)}}.
\end{align*}
Thus we conclude, using Lemma~\ref{lem:clustering_error_T_term}, that,
\begin{equation}\label{eq:triangle_count_diff_scaling}
	\left|\Exp{T_{\text{box}}(p_0) - T_{\infty}(p_0)}\right| = \bigO{y n^{-(2\alpha - 1)} + n^{-(4\alpha-2)} e^{y}}.
\end{equation}

Therefore, on $\Kcal_C(k_n)$, 
\begin{align*}
	\left|\Exp{T_{\text{box}}(p_0) - T_{\infty}(p_0)}\right| 
	&= \bigO{\log(k_n) n^{-(2\alpha - 1)} + k_n^2 n^{-(4\alpha - 2)}} \\
	&= \bigO{\log(k_n) n^{-(2\alpha - 1)} + k_n^2 n^{-(4\alpha - 2)}} = \smallO{s(k_n) k_n^2},
\end{align*}
where the last part follows from Lemma~\ref{lem:scaling_s_alpha} and the fact that $s(k_n)^2 = \smallO{s(k_n)}$.
\end{proof}

We can now prove the main result of this section.

\begin{proofof}{Proposition~\ref{prop:convergence_average_clustering_P_n}}
First, by Lemma~\ref{lem:appox_triangle_count_box} and~\eqref{eq:alt_c_ast_kn} we have
\begin{align*}
	\Exp{c^\ast(k_n;\Gbox)} &= (1 + \smallO{1}) \frac{1}{\pmf(k_n)} \int_{\Kcal_{C}(k_n)} P(y)
		\rho_{\mathrm{box}}(y,k_n) \alpha e^{-\alpha y} \dd y \\
	&\hspace{10pt}+ \smallO{s(k_n)}	\frac{1}{\pmf(k_n)} \int_{\Kcal_{C}(k_n)} P(y)
			\rho_{\mathrm{box}}(y,k_n) \alpha e^{-\alpha y} \dd y.
\end{align*}
By Lemma~\ref{lem:degree_integral} the integral in the second term is $(1+\smallO{1}) \pmf(k_n)$ and thus the second term is $\smallO{s(k_n)} = \smallO{\gamma(k_n)}$. Hence it remain to prove that the first term is $(1+\smallO{1}) \gamma(y)$. Using~\eqref{eq:alt_gamma_kn} it is enough to show that
\[
	\int_{\Kcal_{C}(k_n)} P(y) \rho_{\mathrm{box}}(y,k_n) \alpha e^{-\alpha y} \dd y
	= (1+\smallO{1}) \int_{\Kcal_{C}(k_n)} P(y) \rho(y,k_n) \alpha e^{-\alpha y} \dd y.
\]
Now recall the substitution of variables from the proof of Lemma~\ref{lem:degree_integral}: $z(y) = 2 \log(\mu_{\mathrm{box}}(y)/\xi)$. We will apply this change of variables to the right hand side in the above equation. This yields
\begin{align*}
	&\hspace{-30pt}\int_{\Kcal_{C}(k_n)} P(y) \rho_{\mathrm{box}}(y,k_n) \alpha e^{-\alpha y} \dd y\\
	&= (1+\smallO{1}) \int_{\Kcal_C(k_n)} P(z(y) - 2\log(1+\smallO{1})) \rho(z(y),k_n) 
		z^\prime(y) \alpha e^{-\alpha z(y)} \dd y\\
	&= (1+\smallO{1}) \int_{\Kcal_C(k_n)} P(z - 2\log(1+\smallO{1})) \rho(z,k_n) \alpha e^{-\alpha z} \dd z.
\end{align*}
Next we recall that by Proposition~\ref{prop:asymptotics_P}
\[
	P(y) \sim \begin{cases}
		e^{-\frac{y}{2}(4\alpha - 2)} c_\alpha \xi^{4\alpha - 2} &\mbox{if } \frac{1}{2} < \alpha < \frac{3}{4}\\
		\frac{y}{2} e^{-\frac{y}{2}} &\mbox{if } \alpha = \frac{3}{4}\\
		e^{-\frac{y}{2}} \frac{\alpha - \frac{1}{2}}{\alpha - \frac{3}{4}} &\mbox{if } \alpha > \frac{3}{4}.
	\end{cases}
\]
In particular, this implies that $P(z - 2\log(1+\smallO{1})) = (1+\smallO{1})P(z)$, uniformly on $\Kcal_C(k_n)$. We therefore conclude that
\[
	\int_{\Kcal_{C}(k_n)} P(y) \rho_{\mathrm{box}}(y,k_n) \alpha e^{-\alpha y} \dd y
		= (1+\smallO{1}) \int_{\Kcal_{C}(k_n)} P(z) \rho(z,k_n) \alpha e^{-\alpha z} \dd z,
\]
which finishes the proof.
\end{proofof}

From the proof of Proposition~\ref{prop:convergence_average_clustering_P_n} we immediately obtain the following useful corollary, which will be used in Section~\ref{sec:concentration_c_P_n}. Recall from~\eqref{eq:def_stripe} that $\stripknc = \Rcal \cap (\R_+ \times [y_{k_n,C}^-, y_{k_n,C}^+])$.

\begin{corollary}\label{cor:adjusted_triangle_counting_P_n}
Let $p_0 = (0,y)$. Then, as $n \to \infty$,
\[
	\int_{-I_n}^{I_n}\int_{\Kcal_{C}(k_n)} \rho_{\text{box}}(y,k_n) \Exp{\widetilde{T}_{\text{box}}(p_0)} f(x,y) \dd x \dd y
	= (1+\smallO{1}) n k_n^2 \int_0^\infty P(y) \rho(y,k_n) \alpha e^{-\alpha y} \dd y. 
\]
In particular,
\[
	\int_{\stripknc} \rho_{\text{box}}(y,k_n) \Exp{\widetilde{T}_{\text{box}}(p_0)} f(x,y) \dd x \dd y
	= \bigT{n k_n^{-(2\alpha - 1)} s(k_n)}.
\]
\end{corollary}

\subsection{Counting missing triangles}\label{ssec:missing_triangles}

We now come back to computing the expected number of triangles attached to a node at height $y$ in $\Gbox$ that are not present in $\Ginf$. 

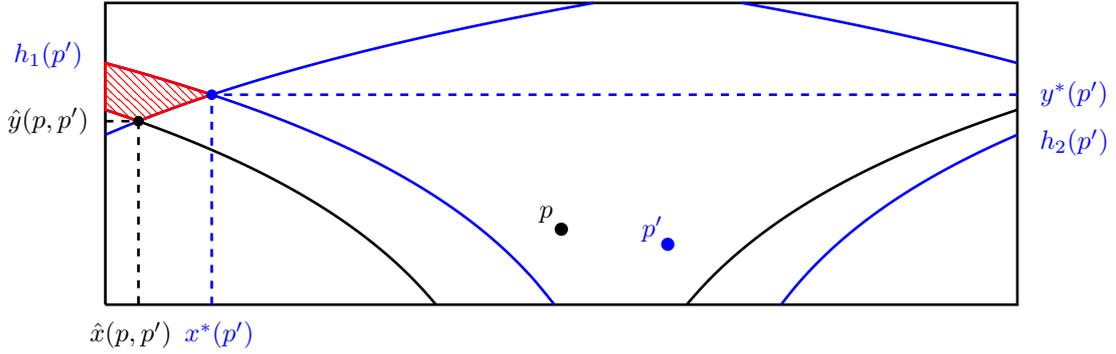
\begin{figure}[!t]
\centering
\begin{tikzpicture}
	\pgfmathsetmacro{\u}{0} 
	\pgfmathsetmacro{\v}{1} 
	\pgfmathsetmacro{\uu}{1.4} 
	\pgfmathsetmacro{\vv}{0.8} 
	\pgfmathsetmacro{\r}{6}
	\pgfmathsetmacro{\t}{4}
	
	\draw[line width=1pt] (-\r,0) -- (\r,0) -- (\r,\t) -- (-\r,\t) -- (-\r,0);

    \draw node[fill, circle, inner sep=0pt, minimum size=5pt] (p1) at (\u,\v) {};
    \path (p1)+(-0.2,0.2) node {$p$};
    \draw node[fill,blue, circle, inner sep=0pt, minimum size=5pt] (p2) at (\uu,\vv) {};
    \path (p2)+(-0.2,0.2) node {\color{blue}$p^\prime$};

	
	\pgfmathsetmacro{\rightbounduv}{\u+exp((\v)/2)}
	\draw[domain=\rightbounduv:\r,smooth,variable=\x,black,line width=1pt] plot (\x, {2*ln(\x)-\v});
    \pgfmathsetmacro{\leftbounduv}{\u-exp((\v)/2)}
    \draw[domain=\leftbounduv:-\r,smooth,variable=\x,black,line width=1pt] plot (\x, {2*ln(-\x)-\v});
    
    
    \pgfmathsetmacro{\rightbounduuvv}{\uu+exp((\vv)/2)}
    \draw[domain=\rightbounduuvv:\r,smooth,variable=\x,blue,line width=1pt] plot (\x, {2*ln(\x-\uu)-\vv});
    \pgfmathsetmacro{\shiftrightbounduuvv}{\uu+exp((\vv + \t)/2)-2*\r}
    \draw[domain=\shiftrightbounduuvv:-\r,smooth,variable=\x,blue,line width=1pt] plot (\x, {2*ln(\x+(2*\r-\uu))-\vv});
    \pgfmathsetmacro{\leftbounduuvv}{\uu-exp((\vv)/2)}
    \draw[domain=\leftbounduuvv:-\r,smooth,variable=\x,blue,line width=1pt] plot (\x, {2*ln(\uu-\x)-\vv});
    \pgfmathsetmacro{\shiftleftbounduuvv}{\uu-exp((\vv + \t)/2)+2*\r}
    \draw[domain=\shiftleftbounduuvv:\r,smooth,variable=\x,blue,line width=1pt] plot (\x, {2*ln(2*\r + \uu-\x)-\vv});
    
    \pgfmathsetmacro{\uuast}{\uu-\r}
    \pgfmathsetmacro{\vvast}{2*ln(\r)-\vv}
    
    \pgfmathsetmacro{\vast}{2*ln((2*\r - \uu)/(exp(\v/2) + exp(\vv/2)))}
    \pgfmathsetmacro{\uast}{(\uu - 2*\r)/(1 + exp((\vv - \v)/2))}    
   
    \pgfmathsetmacro{\hp}{2*ln(\r-\u)-\v}
    
    \pgfmathsetmacro{\hh}{2*ln(\uu+\r)-\vv}
    \pgfmathsetmacro{\hhh}{2*ln(\r-\uu)-\vv}
    
	\draw[red,line width=1pt,pattern=north west lines, pattern color=red] 
		plot[domain=\uuast:\uast,smooth,variable=\x,red] (\x, {2*ln(\x+(2*\r-\uu))-\vv}) 
		-- 
		plot[domain=\uast:-\r,smooth,variable=\x,red] (\x, {2*ln(-\x)-\v})
		-- 
		(-\r,\hh)
		-- 
		plot[domain=-\r:\uuast,smooth,variable=\x,red] (\x, {2*ln(\uu-\x)-\vv});
	
%
	
	 \draw[dashed,line width=1pt,blue] (\uuast,0) -- (\uuast,\vvast);
	 \draw node at (\uuast+0.1,-0.4) {\color{blue}$x^\ast(p^\prime)$};
	 \draw[dashed,line width=1pt,blue] (\r,\vvast) -- (\uuast,\vvast);
	 \draw node at (\r+0.75,\vvast) {\color{blue}$y^\ast(p^\prime)$};
	 \draw[dashed,black,line width=1pt] (-\r,\vast) -- (\uast,\vast);
	 \draw node at (-\r-0.75,\vast) {\color{black}$\hat{y}(p,p^\prime)$};
	 \draw[dashed,black,line width=1pt] (\uast,\vast) -- (\uast,0);
	 \draw node at (\uast-0.1,-0.4) {\color{black}$\hat{x}(p,p^\prime)$};
	 
	 \draw node[fill, circle, inner sep=0pt, minimum size=4pt, blue] at (\uuast,\vvast) {};
	 \draw node[fill, circle, inner sep=0pt, minimum size=4pt, black] at (\uast,\vast) {};
	
    \draw node at (-\r-0.75,\hh+0.1) {\color{blue}$h_1(p^\prime)$};
    \draw node at (\r+0.75,\hhh-0.1) {\color{blue}$h_2(p^\prime)$};

\end{tikzpicture}
\caption{Example configuration of two points $p$ and $p^\prime$ for which $\BallPon{p} \cap \BallPon{p^\prime}$ is not a subset of $\BallPo{p} \cap \BallPo{p^\prime}$. The red region indicates the area belonging to $\BallPon{p} \cap \BallPon{p^\prime}$ but not to $\BallPo{p} \cap \BallPo{p^\prime}$.}
\label{fig:comparing_triangles_diff_intersections}
\end{figure}

Recall that $\mathcal{T}(p,p^\prime)$ denotes the region of points which form triangles with $p$ and $p^\prime$ in $\Gbox$ but not in $\Ginf$. Figure \ref{fig:comparing_triangles_diff_intersections} shows an example of a configuration where $\mathcal{T}(p,p^\prime) \ne \emptyset$. We observe that $\mathcal{T}(p,p^\prime) \ne \emptyset$ because the right boundary of the ball $\BallPon{p^\prime}$ exits the right boundary of the box $\Rcal$ and then, since we identified the boundaries, continues from the left so that $\BallPon{p^\prime}$ covers part of the ball $\BallPon{p}$ which would not be covered in the infinite limit model. 

The point $(\hat{x}(p,p^\prime), \hat{y}(p,p^\prime))$ is the same as $(\hat{x}_{\mathrm{left}}, \hat{y}_{\mathrm{left}})$ from Section~\ref{ssec:joint_degrees_GPo}). Using the same approach as there we can compute the other two coordinates, $x^\ast(p^\prime)$ and $y^\ast(p^\prime)$. In total we have the following four expressions:
\begin{align*}
	x^\ast(p^\prime) &= x^\prime - \frac{\pi}{2} e^{R/2}\\
	y^\ast(p^\prime) &= 2\log\left(\frac{\pi}{2}e^{R/2}\right) - y^\prime\\
	\hat{x}(p,p^\prime) &= \frac{x^\prime - \pi e^{R/2}}{1 + e^{(y^\prime - y)/2}} \\
	\hat{y}(p,p^\prime) &= 2 \log\left(\frac{\pi e^{R/2} - x^\prime}{e^{y/2} + e^{y^\prime/2}}\right)
\end{align*}

The crucial observation is that $\mathcal{T}(p,p^\prime) = \emptyset$ as long as the point $(x^\ast(p^\prime), y^\ast(p^\prime))$ is above the left boundary of $p$. This happens exactly when $y^\ast(p^\prime) > b_p^-(x^\ast(p^\prime))$, where $b_p^-(z)$ is defined in~\eqref{eq:def_left_boundary_Bp}. Therefore the boundary of this event is given by the equation $y^\ast(p^\prime) = b_p^-(x^\ast(p^\prime))$ which reads
\[
	2\log\left(\frac{\pi}{2}e^{R/2}\right) - y^\prime = 2\log\left(\frac{\pi}{2} e^{R/2} -x^\prime\right) - y.
\]
Solving this equation gives us the function
\begin{equation}
	b^\ast_p(z) = y - 2\log\left(1 - \frac{z}{\frac{\pi}{2} e^{R/2}}\right),
\end{equation}
which is displayed by the red curve in Figure \ref{fig:comparing_triangles_diff_analysis}. It holds that $y^\ast(p^\prime) > b_p^-(x^\ast(p^\prime))$ if and only if $y^\prime < b^\ast_p(x^\prime)$ and hence we have that $\mathcal{T}(p,p^\prime) = \emptyset$ for all $p^\prime \in \Rcal$ for which $y^\prime \ge b^\ast_p(x^\prime)$. We also note that when $y^\prime = b^\ast_p(x^\prime)$ the two points $(x^\ast(p^\prime), y^\ast(p^\prime))$ and $(\hat{x}(p,p^\prime),\hat{y}(p,p^\prime))$ coincide.

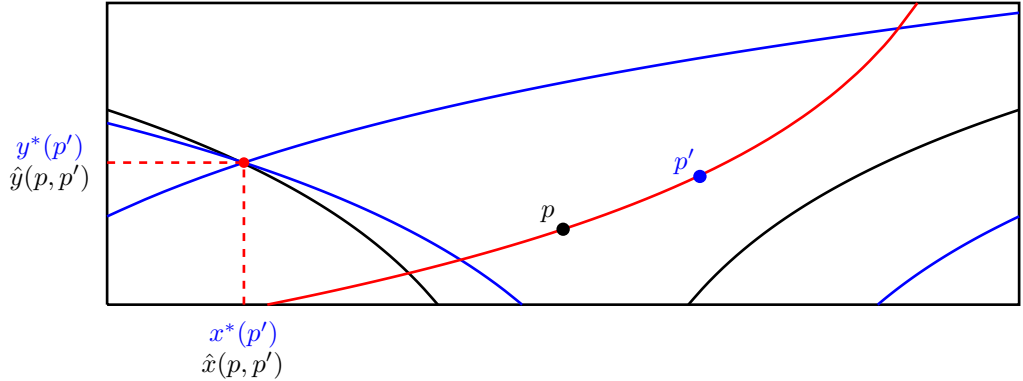
\begin{figure}[!t]

\centering
\begin{tikzpicture}
	\pgfmathsetmacro{\u}{0}
	\pgfmathsetmacro{\v}{1}
	\pgfmathsetmacro{\uu}{1.8}
	\pgfmathsetmacro{\vv}{1.7}
	\pgfmathsetmacro{\r}{6}
	\pgfmathsetmacro{\t}{4}
	
	\draw[line width=1pt] (-\r,0) -- (\r,0) -- (\r,\t) -- (-\r,\t) -- (-\r,0);	
	
	
	\pgfmathsetmacro{\rightbounduv}{\u+exp((\v)/2)}
	\draw[domain=\rightbounduv:\r,smooth,variable=\x,black,line width=1pt] plot (\x, {2*ln(\x)-\v});
    \pgfmathsetmacro{\leftbounduv}{\u-exp((\v)/2)}
    \draw[domain=\leftbounduv:-\r,smooth,variable=\x,black,line width=1pt] plot (\x, {2*ln(-\x)-\v});
    
    
    \pgfmathsetmacro{\rightbounduuvv}{\uu+exp((\vv)/2)}
    \draw[domain=\rightbounduuvv:\r,smooth,variable=\x,blue,line width=1pt] plot (\x, {2*ln(\x-\uu)-\vv});
    \pgfmathsetmacro{\shiftrightbounduuvv}{\uu+exp((\vv + \t)/2)-2*\r}
    \draw[domain=-\r:\r,smooth,variable=\x,blue,line width=1pt] plot (\x, {2*ln(\x+(2*\r-\uu))-\vv});
    \pgfmathsetmacro{\leftbounduuvv}{\uu-exp((\vv)/2)}
    \draw[domain=\leftbounduuvv:-\r,smooth,variable=\x,blue,line width=1pt] plot (\x, {2*ln(\uu-\x)-\vv});

    \pgfmathsetmacro{\starleftbound}{\r*(1-exp(\v/2))}
    \pgfmathsetmacro{\starrightbound}{\r*(1-exp((\v-\t)/2))}
    \draw[domain=\starleftbound:\starrightbound,smooth,variable=\x,red, line width=1pt] plot (\x, {2*ln((\r/(\r-\x)))+\v});
    
    \pgfmathsetmacro{\hp}{2*ln(\r-\u)-\v}
    \draw node at (\r+1,\hp) {};
    
    \pgfmathsetmacro{\hh}{2*ln(\uu+\r)-\vv}
    \pgfmathsetmacro{\hhh}{2*ln(\r-\uu)-\vv}

    
    \pgfmathsetmacro{\uuast}{\uu-\r}
    \pgfmathsetmacro{\vvast}{2*ln(\r)-\vv}

    \pgfmathsetmacro{\vast}{2*ln((2*\r - \uu)/(exp(\v/2) + exp(\vv/2)))}
    \pgfmathsetmacro{\uast}{(\uu - 2*\r)/(1 + exp((\vv - \v)/2))} 

    \draw[dashed,line width=1pt,red] (\uuast,0) -- (\uuast,\vvast);
    \draw node at (\uuast,-0.4) {\color{blue}$x^\ast(p^\prime)$};
    \draw[dashed,line width=1pt,red] (-\r,\vvast) -- (\uuast,\vvast);
    \draw node at (-\r-0.75,\vvast+0.2) {\color{blue}$y^\ast(p^\prime)$};
    \draw node at (-\r-0.75,\vast-0.2) {\color{black}$\hat{y}(p,p^\prime)$};
    \draw node at (\uast,-0.8) {\color{black}$\hat{x}(p,p^\prime)$};
    
    \draw node[fill, circle, inner sep=0pt, minimum size=4pt, red] at (\uuast,\vvast) {};
    
    \draw node[fill, circle, inner sep=0pt, minimum size=5pt] (p1) at (\u,\v) {};
    \path (p1)+(-0.2,0.2) node {$p$};
    \draw node[fill,blue, circle, inner sep=0pt, minimum size=5pt] (p2) at (\uu,\vv) {};
    \path (p2)+(-0.2,0.2) node {\color{blue}$p^\prime$};

    

    
    

\end{tikzpicture}
\caption{Example for a given $p$ of the boundary function $x^\prime \mapsto b^\ast_p(x^\prime)$, given by the red curve, which determines whether $\mathcal{T}(p,p^\prime) = \emptyset$. We see that when $y^\prime = b^\ast_p(x^\prime)$ then $(\hat{x}(p,p^\prime), \hat{y}(p,p^\prime)) = (x^\ast(p^\prime), y^\ast(p^\prime))$.}
\label{fig:comparing_triangles_diff_analysis}
\end{figure}

This analysis allows us to compute the expected difference in the number of triangles for the finite box model and the infinite model, for a typical node with height $y$, i.e. prove Lemma~\ref{lem:clustering_error_T_term}. 

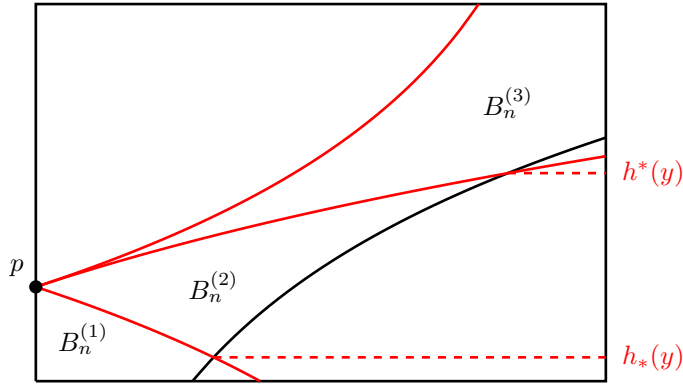
\begin{figure}[!t]

\centering
\begin{tikzpicture}[scale=1.25]


	\pgfmathsetmacro{\u}{0}
	\pgfmathsetmacro{\v}{1}
	\pgfmathsetmacro{\uu}{1.8}
	\pgfmathsetmacro{\vv}{1.7}
	\pgfmathsetmacro{\r}{6}
	\pgfmathsetmacro{\t}{4}
	
	\draw[line width=1pt] (0,0) -- (\r,0) -- (\r,\t) -- (0,\t) -- (0,0);	
	
	
	\pgfmathsetmacro{\rightbounduv}{\u+exp((\v)/2)}
	\draw[domain=\rightbounduv:\r,smooth,variable=\x,black,line width=1pt] plot (\x, {2*ln(\x)-\v});
    
    \pgfmathsetmacro{\bone}{\r*(1-exp(-\v/2))}
    \pgfmathsetmacro{\btwo}{\r*(exp(-\v/2)-1)}
    \pgfmathsetmacro{\bthreeleft}{\r*(1-exp(\v/2))}
    \pgfmathsetmacro{\bthreeright}{\r*(1-exp((\v-\t)/2))}
    
    \draw[domain=0:\bone,smooth,variable=\x,red,line width=1pt] plot (\x, {\v-2*ln(\r/(\r-\x))});
    \draw[domain=0:\r,smooth,variable=\x,red,line width=1pt] plot (\x, {\v+2*ln(1+(\x/\r))});
    \draw[domain=0:\bthreeright,smooth,variable=\x,red, line width=1pt] plot (\x, {2*ln((\r/(\r-\x)))+\v});    
    
    \pgfmathsetmacro{\ybottom}{\v+2*ln(\r/(\r+exp(\v)))}
    \pgfmathsetmacro{\xbottom}{\r*exp(\v)/(\r + exp(\v))}
    \pgfmathsetmacro{\ytop}{\v+2*ln(\r/(\r-exp(\v)))}
    \pgfmathsetmacro{\xtop}{\r*exp(\v)/(\r - exp(\v))}
    
    \draw[dashed,line width=1pt,red] (\xbottom,\ybottom) -- (\r,\ybottom);
    \draw[dashed,line width=1pt,red] (\xtop,\ytop) -- (\r,\ytop);
    \draw node at (\r+0.5,\ybottom) {\color{red}$h_\ast(y)$};
    \draw node at (\r+0.5,\ytop) {\color{red}$h^\ast(y)$};
    
    \draw node at (0.5,\ybottom+0.2) {$B_n^{(1)}$};
    \draw node at (\xbottom,\ybottom+0.75) {$B_n^{(2)}$};
    \draw node at (\xtop,\ytop+0.75) {$B_n^{(3)}$};

    \pgfmathsetmacro{\hp}{2*ln(\r-\u)-\v}
  
    \pgfmathsetmacro{\hh}{2*ln(\uu+\r)-\vv}
    \pgfmathsetmacro{\hhh}{2*ln(\r-\uu)-\vv}

    
    \draw node[fill, circle, inner sep=0pt, minimum size=5pt] (p1) at (\u,\v) {};
    \path (p1)+(-0.2,0.2) node {$p$};

\end{tikzpicture}
\caption{Three different areas $B_n^{(i)}$ used in the proof of Lemma \ref{lem:clustering_error_T_term}.}
\label{fig:comparing_triangles_B_areas}
\end{figure}

\begin{proofof}{Lemma \ref{lem:clustering_error_T_term}}
Due to symmetry it is enough to show that
\begin{equation}\label{eq:clustering_error_T_main}
	\int_0^{R}\int_0^{I_n} \mu\left(\mathcal{T}(p,p_1)\right) \dd \mu(p_1) = \bigO{y n^{-(2\alpha - 1)} + n^{-(2\alpha-1)} e^{y}}
\end{equation}
The proof goes in two stages. First we compute $\mu\left(\mathcal{T}(p,p_1)\right)$ by splitting it over three disjoint regimes with respect to $p_1$, with $x_1 \ge 0$. Then we do the integration with respect to $p_1$.

\subsubsection*{Computing $\bm{\mu\left(\mathcal{T}(p,p_1)\right)}$}

Recall that $I_n = \frac{\pi}{2} e^{R/2}$ and define the sets
\begin{align*}
	A_n^{(1)} &= \left\{p_1 = (x_1,y_1) \in \Rcal \, : \, 0 \le y_1 \le y - 2\log(I_n/(I_n-x_1)) \right\},\\
	A_n^{(2)} &= \left\{p_1 = (x_1,y_1) \in \Rcal \, : \, y - 2\log(I_n/(I_n-x_1)) < y_1 
		\le y + 2 \log\left(1 + \frac{x_1}{I_n}\right)\right\},\\
	A_n^{(3)} &= \left\{p_1 = (x_1,y_1) \in \Rcal \, : \, y + 2 \log\left(1 + \frac{x_1}{I_n}\right) < y_1 
			\le y + 2 \log\left(\frac{I_n}{I_n-x_1}\right)\right\},
\end{align*}
and let $B_n^{(i)} = \BallPon{p} \cap A_n^{(i)}$, for $i = 1, 2, 3$, see Figure~\ref{fig:comparing_triangles_B_areas}. Here the heights of the two intersections are given by
\begin{align}
	h_\ast(y) &= y + 2 \log\left(\frac{I_n}{I_n + e^y}\right)\\
	h^\ast(y) &= y + 2 \log\left(\frac{I_n}{I_n - e^y}\right).
\end{align}

With these definitions we have that the union $B_n := \bigcup_{i = 1}^n B_n^{(i)}$ denotes the area under the red curve in Figure~\ref{fig:comparing_triangles_diff_analysis} and hence, for all $p_1 \in \Rcal\setminus B_n$ with $x_1 \ge 0$ we have that $\mathcal{T}(p,p_1) = \emptyset$. So we only need to consider $p_1 \in B_n$. We shall establish the following result:
\begin{equation}\label{eq:mu_triangle_diff}
	\mu\left(\mathcal{T}(p,p_1)\right) = 
	\begin{cases}
		\bigO{I_n^{-2\alpha} e^{\alpha y_1}} &\mbox{if } p_1 \in B_n^{(1)}\\
		\bigO{I_n^{-2\alpha} e^{\alpha y}} &\mbox{if } p_1 \in B_n^{(2)} \cup B_n^{(3)}
	\end{cases}
\end{equation}

Depending on which set $p_1$ belongs to, the set $\mathcal{T}(p,p_1)$ has a different shape. We displayed these shapes in Figure~\ref{fig:shapes_triangle_mismatches} as a visual aid to follow the computations below. 

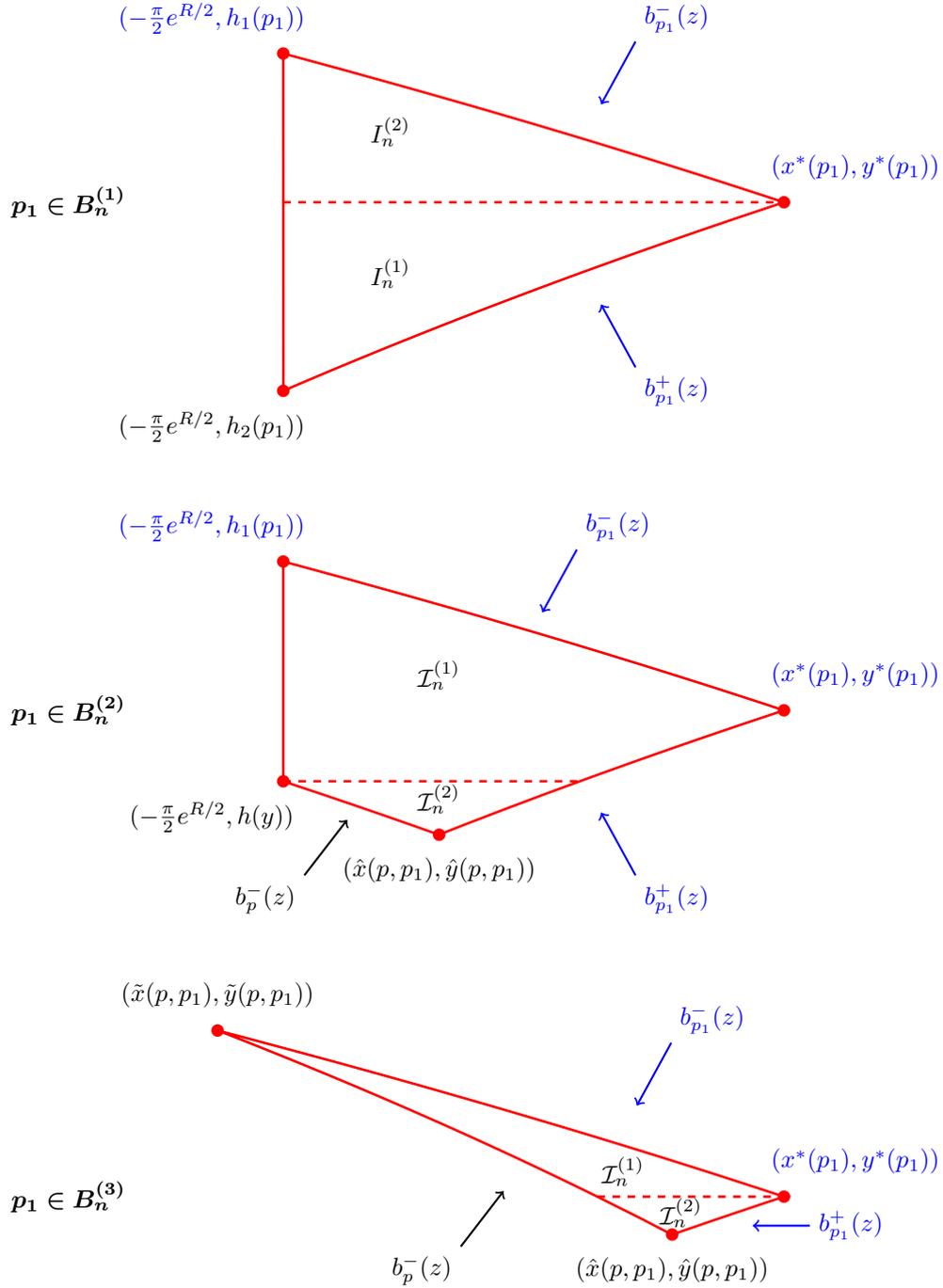
\begin{figure}[!tp]
\centering
\begin{tikzpicture}[scale=5]


	\pgfmathsetmacro{\u}{0}
	\pgfmathsetmacro{\v}{1}
	\pgfmathsetmacro{\uu}{1.4}
	\pgfmathsetmacro{\vv}{0.2}
	\pgfmathsetmacro{\r}{6}
	\pgfmathsetmacro{\t}{4}
    
    \pgfmathsetmacro{\uuast}{\uu-\r}
    \pgfmathsetmacro{\vvast}{2*ln(\r)-\vv}
    
    \pgfmathsetmacro{\vast}{2*ln((2*\r - \uu)/(exp(\v/2) + exp(\vv/2)))}
    \pgfmathsetmacro{\uast}{(\uu - 2*\r)/(1 + exp((\vv - \v)/2))}    
   
    \pgfmathsetmacro{\hp}{2*ln(\r-\u)-\v}
    
    \pgfmathsetmacro{\hh}{2*ln(\uu+\r)-\vv}
    \pgfmathsetmacro{\hhh}{2*ln(\r-\uu)-\vv}
    
	\draw[red,line width=1pt] 
		plot[domain=\uuast:-\r,smooth,variable=\x,red] (\x, {2*ln(\x+(2*\r-\uu))-\vv}) 
		-- 
		(-\r,\hh)
		-- 
		plot[domain=-\r:\uuast,smooth,variable=\x,red] (\x, {2*ln(\uu-\x)-\vv});
	
	\pgfmathsetmacro{\top}{2*ln(\uu-\uast)-\vv}
	
	\draw[red, dashed,line width=1pt] (-\r,\vvast) -- (\uuast,\vvast);

    \draw node at (-\r-0.2,\hh+0.1) {\color{blue}$(-\frac{\pi}{2} e^{R/2}, h_1(p_1))$};
    \draw node at (-\r-0.2,\hhh-0.1) {$(-\frac{\pi}{2} e^{R/2}, h_2(p_1))$};
    \draw node at (\uuast+0.2,\vvast+0.1) {\color{blue}$(x^\ast(p_1), y^\ast(p_1))$};
    
    \draw node at (\uuast-2,\vvast) {$\bm{p_1 \in B^{(1)}_n}$};
    
    \draw node at (-\r+0.3,\vvast+0.2) {$I_n^{(2)}$};
    \draw node at (-\r+0.3,\vvast-0.2) {$I_n^{(1)}$};
    
    \draw node (f2) at (\uuast-0.3,\hh+0.1) {\color{blue}$b_{p_1}^-(z)$};
    \draw node (f3) at (\uuast-0.3,\hhh) {\color{blue}$b_{p_1}^+(z)$};
    \path (f2)+(230:0.35) node (f2_arrow) {};
    \path (f3)+(130:0.35) node (f3_arrow) {};
    \draw[->,thick,blue] (f2.south west) -- (f2_arrow);
    \draw[->,thick,blue] (f3.north west) -- (f3_arrow);

	\draw node[red, fill, circle, inner sep=0pt, minimum size=5pt] at (-\r,\hhh) {};
	\draw node[red, fill, circle, inner sep=0pt, minimum size=5pt] at (-\r,\hh) {};
	\draw node[red, fill, circle, inner sep=0pt, minimum size=5pt] at (\uuast,\vvast) {};

\end{tikzpicture}\\
\vspace{20pt}
\begin{tikzpicture}[scale=5]
	\pgfmathsetmacro{\u}{0}
	\pgfmathsetmacro{\v}{1}
	\pgfmathsetmacro{\uu}{1.4}
	\pgfmathsetmacro{\vv}{0.8}
	\pgfmathsetmacro{\r}{6}
	\pgfmathsetmacro{\t}{4}
    
    \pgfmathsetmacro{\uuast}{\uu-\r}
    \pgfmathsetmacro{\vvast}{2*ln(\r)-\vv}
    
    \pgfmathsetmacro{\vast}{2*ln((2*\r - \uu)/(exp(\v/2) + exp(\vv/2)))}
    \pgfmathsetmacro{\uast}{(\uu - 2*\r)/(1 + exp((\vv - \v)/2))}    
   
    \pgfmathsetmacro{\hp}{2*ln(\r-\u)-\v}
    
    \pgfmathsetmacro{\hh}{2*ln(\uu+\r)-\vv}
    \pgfmathsetmacro{\hhh}{2*ln(\r-\uu)-\vv}
    
	\draw[red,line width=1pt] 
		plot[domain=\uuast:\uast,smooth,variable=\x,red] (\x, {2*ln(\x+(2*\r-\uu))-\vv}) 
		-- 
		plot[domain=\uast:-\r,smooth,variable=\x,red] (\x, {2*ln(-\x)-\v})
		-- 
		(-\r,\hh)
		-- 
		plot[domain=-\r:\uuast,smooth,variable=\x,red] (\x, {2*ln(\uu-\x)-\vv});
	
	\pgfmathsetmacro{\rb}{\uu + exp((\vv+\hp)/2) -2*\r}

	\draw[red,dashed,line width=1pt] (-\r,\hp) -- (\rb,\hp);

    \draw node at (-\r-0.2,\hh+0.1) {\color{blue}$(-\frac{\pi}{2} e^{R/2}, h_1(p_1))$};
    \draw node at (-\r-0.2,\hp-0.1) {$(-\frac{\pi}{2} e^{R/2}, h(y))$};
    \draw node at (\uuast+0.2,\vvast+0.1) {\color{blue}$(x^\ast(p_1), y^\ast(p_1))$};
    \draw node at (\uast,\vast-0.1) {$(\hat{x}(p,p_1), \hat{y}(p, p_1))$};
    
    \draw node at (\uuast-2,\vvast) {$\bm{p_1 \in B^{(2)}_n}$};
    
    \draw node at (\uast,\vvast+0.1) {$\mathcal{I}_n^{(1)}$};
    \draw node at (\uast,\hp-0.05) {$\mathcal{I}_n^{(2)}$};
    
    \draw node (f1) at (-\r-0.05,\hhh) {$b_p^-(z)$};
    \draw node (f2) at (\uast+0.5,\hh+0.1) {\color{blue}$b_{p_1}^-(z)$};
    \draw node (f3) at (\uuast-0.3,\hhh) {\color{blue}$b_{p_1}^+(z)$};
    \path (f1)+(45:0.35) node (f1_arrow) {};
    \path (f2)+(230:0.35) node (f2_arrow) {};
    \path (f3)+(130:0.35) node (f3_arrow) {};
    \draw[->,thick] (f1.north east) -- (f1_arrow);
    \draw[->,thick,blue] (f2.south west) -- (f2_arrow);
    \draw[->,thick,blue] (f3.north west) -- (f3_arrow);

	\draw node[red, fill, circle, inner sep=0pt, minimum size=5pt] at (-\r,\hp) {};
	\draw node[red, fill, circle, inner sep=0pt, minimum size=5pt] at (-\r,\hh) {};
	\draw node[red, fill, circle, inner sep=0pt, minimum size=5pt] at (\uuast,\vvast) {};
	\draw node[red, fill, circle, inner sep=0pt, minimum size=5pt] at (\uast,\vast) {};

\end{tikzpicture}\\
\vspace{20pt}
\begin{tikzpicture}[scale=5]


	\pgfmathsetmacro{\u}{0}
	\pgfmathsetmacro{\v}{1}
	\pgfmathsetmacro{\uu}{2.5}
	\pgfmathsetmacro{\vv}{1.8}
	\pgfmathsetmacro{\r}{6}
	\pgfmathsetmacro{\t}{4}
    
    \pgfmathsetmacro{\uuast}{\uu-\r}
    \pgfmathsetmacro{\vvast}{2*ln(\r)-\vv}
    
    \pgfmathsetmacro{\vast}{2*ln((2*\r - \uu)/(exp(\v/2) + exp(\vv/2)))}
    \pgfmathsetmacro{\uast}{(\uu - 2*\r)/(1 + exp((\vv - \v)/2))}    
    
    \pgfmathsetmacro{\utilde}{\uu/(1-exp((\vv-\v)/2))} 
    \pgfmathsetmacro{\vtilde}{2*ln(\uu-\utilde)-\vv}
   
    \pgfmathsetmacro{\hp}{2*ln(\r-\u)-\v}
    
    \pgfmathsetmacro{\hh}{2*ln(\uu+\r)-\vv}
    \pgfmathsetmacro{\hhh}{2*ln(\r-\uu)-\vv}
    
	\draw[red,line width=1pt] 
		plot[domain=\utilde:\uuast,smooth,variable=\x,red] (\x, {2*ln(\uu-\x)-\vv})
		--
		plot[domain=\uuast:\uast,smooth,variable=\x,red] (\x, {2*ln(\x+(2*\r-\uu))-\vv})
		--
		plot[domain=\uast:\utilde,smooth,variable=\x,red] (\x, {2*ln(-\x)-\v}); 
	
	\pgfmathsetmacro{\i}{-exp((\v+\vvast)/2)}
	\draw[red, dashed,line width=1pt] (\i,\vvast) -- (\uuast,\vvast);

    \draw node at (\utilde,\vtilde+0.1) {\color{black}$(\tilde{x}(p,p_1), \tilde{y}(p,p_1))$};
    \draw node at (\uast,\vast-0.1) {$(\hat{x}(p,p_1), \hat{y}(p,p_1))$};
    \draw node at (\uuast+0.2,\vvast+0.1) {\color{blue}$(x^\ast(p_1), y^\ast(p_1))$};
   
   \draw node at (\uuast-2,\vvast) {$\bm{p_1 \in B^{(3)}_n}$};
   
    \draw node at (\uuast-0.45,\vvast+0.07) {$\mathcal{I}_n^{(1)}$};
    \draw node at (\uast+0.025,\vvast-0.045) {$\mathcal{I}_n^{(2)}$};
   
    \draw node (f1) at (\uast-0.7,\vast-0.1) {\color{black}$b_{p}^-(z)$};
    \draw node (f2) at (\uuast-0.2,\vvast+0.5) {\color{blue}$b_{p_1}^-(z)$};
    \draw node (f3) at (\uast+0.5,\vast+0.025) {\color{blue}$b_{p_1}^+(z)$};
    \path (f1)+(45:0.35) node (f1_arrow) {};
    \path (f2)+(230:0.35) node (f2_arrow) {};
    \path (f3)+(180:0.3) node (f3_arrow) {};
	\draw[->,thick] (f1.north east) -- (f1_arrow);
    \draw[->,thick,blue] (f2.south west) -- (f2_arrow);
    \draw[->,thick,blue] (f3.west) -- (f3_arrow);
	\draw node[red, fill, circle, inner sep=0pt, minimum size=5pt] at (\utilde,\vtilde) {};
	\draw node[red, fill, circle, inner sep=0pt, minimum size=5pt] at (\uast,\vast) {};
	\draw node[red, fill, circle, inner sep=0pt, minimum size=5pt] at (\uuast,\vvast) {};

\end{tikzpicture}
\caption{The different shapes of $\mathcal{T}(p,p_1)$ depending on the regime to which $p_1$ belongs. The top figure is for $p_1 \in B_n^{(1)}$, the middle one for $p_1 \in B_n^{(2)}$ and the bottom one for $p_1 \in B_n^{(3)}$.}
\label{fig:shapes_triangle_mismatches}
\end{figure}

\paragraph{Case $p_1 \in B_n^{(1)}$: $0 \le y_1 \le y - 2\log(I_n/(I_n-x_1))$}

In this case the integral over $p_2$ splits into two parts
\begin{align*}
	\mathcal{I}_n^{(1)}(p_1) &:= \int_{h_2(p_1)}^{y^\ast(p_1)} \int_{-I_n}^{x_1 + e^{(y_1+y_2)/2}-2I_n} e^{-\alpha y_2}
		\dd x_2 \dd y_2\\
	\mathcal{I}_n^{(2)}(p_1) &:= \int_{y^\ast(p_1)}^{h_1(p_1)} \int_{x^\ast(p_1)}^{x_1 - e^{(y_1+y_2)/2}} e^{-\alpha y_2}
		\dd x_2 \dd y_2.
\end{align*}

We first compute $\mathcal{I}_n^{(1)}$.
\begin{align*}
	\mathcal{I}_n^{(1)}(p_1) &= \int_{h_2(p_1)}^{y^\ast(p_1)} \left(x_1 + e^{(y_1+y_2)/2} - I_n\right) e^{-\alpha y_2} 
		\dd y_2\\
	&\le e^{y_1/2} \int_{h_2(p_1)}^{y^\ast(p_1)} e^{-(\alpha-\frac{1}{2}) y_2} \dd y_2\\
	&= \frac{2 e^{y_1/2}}{2\alpha - 1} \left(e^{-(\alpha-\frac{1}{2})h_2(p_1)} - e^{-(\alpha-\frac{1}{2})y^\ast(p_1)}
		\right) \\
	&= \frac{2 e^{\alpha y_1}}{2\alpha - 1} I_n^{-(2\alpha -1)}\left(\left(1 - \frac{x_1}{I_n}\right)^{-(2\alpha - 1)}-1\right)\\
	&= \bigO{I_n^{-2\alpha} x_1 e^{\alpha y_1}},
\end{align*}
where we used that $x_1 \le e^{(y+y_1)/2} = \smallO{I_n}$ for all $y_1\le y$ and $y \in \Kcal_{C}(k_n)$ so that
\[
	\left(\left(1 - \frac{x_1}{I_n}\right)^{-(2\alpha - 1)}-1\right) = \bigO{\frac{x_1}{I_n}} \quad 
	\text{as } n \to \infty.
\]

For $\mathcal{I}_n^{(2)}(p_1)$ we have
\begin{align*}
	\mathcal{I}_n^{(2)}(p_1) &= \int_{y^\ast(p_1)}^{h_1(p_1)} \left(I_n + x_1 - e^{(y_1+y_2)}\right) e^{-\alpha y_2}
		\dd y_2\\
	&\le 2 I_n \int_{y^\ast(p_1)}^{h_1(p_1)} e^{-\alpha y_2} \dd x_2 \dd y_2\\
	&= \frac{2}{\alpha} I_n \left(I_n^{-2\alpha}e^{\alpha y_1} - \left(I_n + x_1\right)^{-2\alpha} e^{-\alpha y_1}\right)\\
	&= \bigO{I_n^{-2\alpha} x_1 e^{\alpha y_1}} = \bigO{I_n^{-(2\alpha - 1)} e^{\alpha y_1}}.
\end{align*}

We conclude that for $p_1 \in B_n^{(1)}$:
\[
	\mu\left(\mathcal{T}(p,p_1)\right) = \bigO{I_n^{-2\alpha} x_1 e^{\alpha y_1}},
\]
which establishes the first part of \eqref{eq:mu_triangle_diff}.

\paragraph{Case $p_1 \in B_n^{(2)}$: $y - 2\log(I_n/(I_n-x_1)) < y_1 \le y + 2 \log\left(1 + \frac{x_1}{I_n}\right)$}

Here we split the integration into two parts (see Figure~\ref{fig:shapes_triangle_mismatches}). Recall that $x^\ast(p,p_1) = x_1 - I_n$. Then, for the first part we have
\begin{align*}
	\mathcal{I}_n^{(1)}(p,p_1) &\le \int_{h(y)}^{h_1(p_1)} \int_{-I_n}^{x^\ast(p,p_1)} f(x_2, y_2) 
		\dd x_2 \dd y_2\\
	&= \bigO{x_1 \left(e^{-\alpha h(y)} - e^{-\alpha h_1(p_1)}\right)}\\
	&= \bigO{x_1 I_n^{-2\alpha}\left(e^{\alpha y} - e^{\alpha y_1}\left(1 + \frac{x_1}{I_n}\right)^{-2\alpha}\right)}\\
	&= \bigO{I_n^{-2\alpha} x_1 e^{\alpha y_1}\left(\left(1 - \frac{x_1}{I_n}\right)^{-2\alpha} 
		- \left(1 + \frac{x_1}{I_n}\right)^{-2\alpha}\right)}\\
	&= \bigO{I_n^{-2\alpha} x_1 e^{\alpha y_1}} = \bigO{I_n^{-(2\alpha - 1)} e^{\alpha y}}, 
\end{align*}
were we used that $y \le y_1 + 2\log(I_n/(I_n-x_1))$ for $p_1 \in B_n^{(2)}$ for the third line and 
\[
	\left(1 - \frac{x_1}{I_n}\right)^{-2\alpha} - \left(1 + \frac{x_1}{I_n}\right)^{-2\alpha}
	= \bigO{\frac{x_1}{I_n}} = \bigO{1},
\]
for the last line.

For the second part we first us the upper bound on $y_1$ to compute that 
\begin{align*}
	x_1 + e^{(y_1+y_2)/2} - 2 I_n + e^{(y + y_2)/2} &\le \left(e^{y/2} + e^{y_1/2}\right)e^{y_2/2}\\
	&\le e^{y/2}\left(2 + \frac{x_1}{I_n}\right)e^{y_2/2} = \bigO{e^{(y+y_2)/2}},
\end{align*}
since $|x_1| \le I_n$. Then we have
\begin{align*}
	\mathcal{I}_n^{(2)} &= \int_{\hat{y}(p,p_1)}^{h(y)} \int_{-e^{(y + y_2)/2}}^{x_1 + e^{(y+y_1)/2} - 2 I_n} 
		f(x_2, y_2) \dd x_2 \dd y_2\\
	&= \bigO{e^{y/2} \int_{\hat{y}(p,p_1)}^{h(y)} e^{-(\alpha -\frac{1}{2}) y_2} \dd y_2}\\
	&= \bigO{e^{y/2} \left(e^{-(\alpha -\frac{1}{2}) \hat{y}(p,p_1)} - e^{-(\alpha -\frac{1}{2}) h(y)}\right)}\\
	&= \bigO{e^{y/2} \left(\left(\frac{2I_n - x_1}{e^{y/2} + e^{y_1/2}}\right)^{-(2\alpha-1)} 
		- I_n^{-(2\alpha-1)} e^{(\alpha -\frac{1}{2}) y}\right)}\\
	&= \bigO{I_n^{-(2\alpha-1)} e^{\alpha y}},
\end{align*}
where for the last line we first used that $(2I_n - x_1)^{-(2\alpha-1)} \le I_n^{-(2\alpha-1)}$ and then
\[
	\left(\left(e^{y/2} + e^{y_1/2}\right)^{2\alpha-1}- e^{(\alpha -\frac{1}{2}) y}\right)
	\le e^{(\alpha -\frac{1}{2}) y}\left(\left(1 + \sqrt{1+\frac{x_1}{I_n}} \, \right)^{2\alpha - 1} - 1\right)
	= \bigO{e^{(\alpha -\frac{1}{2}) y}}.
\]

It then follows that for $p_1 \in B_n^{(2)}$
\[
	\mu\left(\mathcal{T}(p,p_1)\right) = \bigO{I_n^{-(2\alpha-1)} e^{\alpha y}}.
\]

\paragraph{Case $p_1 \in B_n^{(3)}$: $y +2\log(1+x_1/I_n) < y_1 \le y + 2\log(I_n/(I_n-x_1))$}

\begin{align*}
	\mathcal{I}_n^{(1)} &= \int_{y^\ast}^{\tilde{y}} \int_{-e^{(y+y_2)/2}}^{x_1-e^{(y_1+y_2)/2}} f(x_2,y_2)
		\dd x_2 \dd y_2\\
	&= \bigO{\int_{y^\ast}^{\tilde{y}} x_1 e^{-\alpha y_2} - \left(e^{y_1/2} - e^{y/2}\right)e^{-(\alpha - \frac{1}{2})y_2}
		\dd y_2}\\
	&= \bigO{x_1 \int_{y^\ast}^{\tilde{y}}  e^{-\alpha y_2} \dd y_2}.
\end{align*}

Now 
\begin{align*}
	\int_{y^\ast}^{\tilde{y}}  e^{-\alpha y_2} \dd y_2 
	&= \frac{1}{\alpha}\left(e^{-\alpha y^\ast} - e^{-\alpha \tilde{y}}\right) 
		= \frac{1}{\alpha}\left(I_n^{-2\alpha} e^{\alpha y_1} 
		- \left(\frac{x_1}{e^{y_1/2} - e^{y/2}}\right)^{-2\alpha}\right) \\
	&= \frac{I_n^{-2\alpha} e^{\alpha y_1}}{\alpha}\left(1 - \left(1 - e^{(y - y_1)/2}\right)^{2\alpha}
		\left(\frac{x_1}{I_n}\right)^{-2\alpha}\right) = \bigO{I_n^{-2\alpha} e^{\alpha y_1}},
\end{align*}
and hence we have
\[
	\mathcal{I}_n^{(1)} = \bigO{I_n^{-2\alpha} x_1 e^{\alpha y_1}}.
\]

For the second integral we have, using that $y \le y_1$ for $p_1 \in B_n^{(3)}$,
\begin{align*}
	\mathcal{I}_n^{(2)} &= \int_{\hat{y}}^{y^\ast} \int_{-e^{(y+y_2)/2}}^{e^{(y_1 + y_2)/2} + x_1 - 2I_n} 
		f(x_2,y_2)\dd x_2 \dd y_2\\
	&= \bigO{\int_{\hat{y}}^{y^\ast} \left(e^{y/2} + e^{y_1/2}\right)e^{-(\alpha - \frac{1}{2})y_2} \dd y_2}\\
	&= \bigO{ e^{y_1/2} \int_{\hat{y}}^{y^\ast}e^{-(\alpha - \frac{1}{2})y_2} \dd y_2}.
\end{align*}

For the integral we have
\begin{align*}
	&\hspace{-30pt}\int_{\hat{y}}^{y^\ast}e^{-(\alpha - \frac{1}{2})y_2} \dd y_2
		= \frac{2}{2\alpha - 1} \left(e^{-(\alpha - \frac{1}{2})\hat{y}} - e^{-(\alpha - \frac{1}{2})y^\ast}\right)\\
	&= \frac{2}{2\alpha - 1}\left(\left(\frac{2I_n - x_1}{e^{y/2} + e^{y_1/2}}\right)^{-(2\alpha - 1)} 
		- I_n^{-(2\alpha - 1)} e^{(\alpha - \frac{1}{2})y_1}\right) = \bigO{I_n^{-(2\alpha-1)} e^{(\alpha - \frac{1}{2})y_1}},
\end{align*}
where we used the upper bound on $y_1$ and the fact that $2I_n - x_1 = \bigT{I_n}$ for all $x_1 \in [-I_n,I_n]$. We conclude that
\[
	\mathcal{I}_n^{(2)} = \bigO{I_n^{-(2\alpha-1)} x_1 e^{\alpha y}}
\]
and hence for $p_1 \in B_n^{(3)}$
\[
	\mu\left(\mathcal{T}(p,p_1)\right) = \bigO{I_n^{-2\alpha} x_1 e^{\alpha y}}
	= \bigO{I_n^{-(2\alpha - 1)} e^{\alpha y}}.
\]

\subsubsection*{Integration $\mu(\mathcal{T}(p,p_1))$ with respect to $p_1$}

We now proceed with the second part of the computation leading to \eqref{eq:clustering_error_T_main}. Here we will integrate $\mu(\mathcal{T}(p,p^\prime))(p,p_1)$ over the region $B_n := B_n^{(1)} \cup B_n^{(2)} \cup B_n^{(3)}$, see Figure~\ref{fig:comparing_triangles_B_areas}. Let us first identify the boundaries of these areas. 

The area $B_n^{(1)}$ is bounded from above by the line given by the equation
\[
	y_1 = y - 2\log\left(\frac{I_n}{I_n - x_1}\right).
\]
Solving this for $x_1$ yields $x_1 = I_n\left(1 - e^{(y_1-y)/2}\right)$ and hence the area $B_n^{(1)}$ is given by
\[
	B_n^{(1)} = \left\{(x_1, y_1) \, : \, 0 \le y_1 \le y, \quad 0 \le x_1 \le I_n\left(1 - e^{(y_1-y)/2}\right) \wedge e^{(y + y_1)/2} \right\}.
\]

In a similar way we have that $B_n^{(2)}$ is bounded from above by line
\[
	y_1 = y + 2\log\left(\frac{I_n}{I_n + x_1}\right),
\]
which yields $x_1 = I_n\left(e^{(y_1 - y)/2} - 1\right)$. The lower red boundary is the upper boundary of $B_n^{(2)}$ and hence we have
\[
	B_n^{(2)} = \left\{(x_1, y_1) \, : \, h_\ast(y) \le y_1 \le h^\ast(y), \,\, I_n\left(1 - e^{(y_1-y)/2}\right) \vee 
	I_n\left(e^{(y_1 - y)/2} - 1\right) \le x_1 \le e^{(y + y_1)/2} \right\}.
\]

We continue in the same way for $B_n^{(3)}$
\[
	B_n^{(3)} = \left\{(x_1, y_1) \, : \, y \le y_1 \le R, \,\,
	I_n\left(1 - e^{(y - y_1)/2}\right) \le x_1 \le I_n\left(e^{(y_1 - y)/2} - 1\right) \wedge e^{(y + y_1)/2} \wedge I_n \right\}.
\]

We these characterizations of the areas we now integrate $\mu(\mathcal{T}(p,p_1))$ over $B_n$, splitting the computations over the three different areas.

\paragraph{Integration over $\bm{B_n^{(1)}}:$}

We use that $I_n\left(1 - e^{(y_1-y)/2}\right) \wedge e^{(y + y_1)/2} \le I_n\left(1 - e^{(y_1-y)/2}\right)$ so that
\begin{align*}
	&\hspace{-30pt}\int_{B_n^{(1)}} \mu\left(\mathcal{T}(p,p_1)\right) 
		f(x_1,y_1)	\dd x_1 \dd y_1 \\
	&\le  \int_0^y \int_0^{I_n(1-e^{(y_1-y)/2})} \mu\left(\mathcal{T}(p,p_1)\right) 
		f(x_1,y_1) \dd x_1 \dd y_1\\
	&= \bigO{ I_n^{-2\alpha} \int_0^y \int_0^{e^{(y+y_1)/2}}  x_1 \dd x_1 \dd y_1 }\\
	&= \bigO{I_n^{-(2\alpha-1)} \int_0^y \left(1 - e^{(y_1-y)/2}\right)^2 \dd y_1} \\
	&= \bigO{I_n^{-(2\alpha - 1)} y} = \bigO{y n^{-(2\alpha - 1)}}.
\end{align*} 

\paragraph{Integration over $\bm{B_n^{(2)}}:$}

We will show that
\begin{equation}\label{eq:mu_triangle_diff_2}
	\mu(B_n^{(2)}) = \bigO{I_n^{-1} e^{(2-\alpha)y}},
\end{equation}
which together with \eqref{eq:mu_triangle_diff} yields
\begin{align*}
	\int_{B_n^{(2)}} \mu\left(\mathcal{T}(p,p_1)\right) 
		f(x_1,y_1)	\dd x_1 \dd y_1
	&= \bigO{\mu(B_n^{(2)}) I_n^{-(2\alpha - 1)} e^{\alpha y}}\\
	&= \bigO{I_n^{-2\alpha} e^{2y}}.
\end{align*}

The integration is split into two parts determined by $I_n\left(1 - e^{(y_1-y)/2}\right) \vee 
	I_n\left(e^{(y_1 - y)/2} - 1\right)$:
\begin{align*}
	\mu(B_n^{(3)}) &= \int_{h_\ast(y)}^{y} \int_{I_n(1-e^{(y_1-y)/2})}^{e^{(y + y_1)/2}} 
		f(x_1,y_1) \dd x_1 \dd y_1\\
	&\hspace{10pt} + \int_y^{h^\ast(y)} \int_{I_n(e^{(y_1-y)/2}-1)}^{e^{(y + y_1)/2}} 
		f(x_1,y_1) \dd x_1 \dd y_1.
\end{align*}

For the first integral we use that $e^{(y + y_1)/2} - I_n(1-e^{(y_1-y)/2}) \le e^{y_1/2}\left(e^{y/2} + e^{-y/2}\right)$ to obtain
\begin{align*}
	&\hspace{-30pt}\int_{h_\ast(y)}^{y} \int_{I_n(1-e^{(y_1-y)/2})}^{e^{(y + y_1)/2}} f(x_1,y_1) 
		\dd x_1 \dd y_1\\
	&= \bigO{e^{y/2} \int_{h_\ast(y)}^{y} e^{-(\alpha - \frac{1}{2})y_1} \dd y_1}\\
	&= \bigO{e^{y/2}\left(e^{-(\alpha - \frac{1}{2})y} - e^{-(\alpha - \frac{1}{2})y} 
		\left(\frac{I_n}{I_n + e^y}\right)^{-(2\alpha - 1)}\right)}\\
	&= \bigO{I_n^{-1} e^{(2-\alpha)y}}.
\end{align*}
For the second integral note that $e^{(y + y_1)/2} - I_n(e^{(y_1-y)/2}-1) \le e^{(y + y_1)/2}$ and hence
\begin{align*}
	&\hspace{-30pt}\int_y^{h^\ast(y)} \int_{I_n(e^{(y_1-y)/2}-1)}^{e^{(y + y_1)/2}} f(x_1,y_1) 
		\dd x_1 \dd y_1\\
	&= \bigO{e^{y/2} \int_y^{h^\ast(y)} e^{-(\alpha - \frac{1}{2})y_1} \dd y_1}\\
	&= \bigO{e^{y/2} \left(e^{-(\alpha - \frac{1}{2})y} - e^{-(\alpha - \frac{1}{2})y}
		\left(\frac{I_n}{I_n - e^y}\right)^{-(2\alpha - 1)}\right)}\\
	&= \bigO{I_n^{-1} e^{(2-\alpha)y}},
\end{align*}
so that \eqref{eq:mu_triangle_diff_2} follows.

\paragraph{Integration over $\bm{B_n^{(3)}}:$}

For this case we show that 
\begin{equation}\label{eq:mu_triangle_diff_3}
	\mu(B_n^{(3)}) = \bigO{e^{(1-\alpha)y}}
\end{equation} 
so that
\begin{align*}
	\int_{B_n^{(3)}} \mu\left(\mathcal{T}(p,p_1)\right) 
		f(x_1,y_1)	\dd x_1 \dd y_1
	&= \bigO{\mu(B_n^{(2)}) I_n^{-(2\alpha - 1)} e^{\alpha y}}\\
	&= \bigO{I_n^{-(2\alpha-1)} e^{y}}.
\end{align*}

Here the integral is split into three parts:
\begin{align*}
	\mu(B_n^{(3)}) &= \int_y^{h^\ast(y)} \int_{I_n(1-e^{(y-y_1)/2})}^{I_n(e^{(y_1-y)/2}-1)}
		f(x_1,y_1) \dd x_1 \dd y_1\\
	&\hspace{10pt}+ \int_{h^\ast(y)}^{h(y)} \int_{I_n(1-e^{(y-y_1)/2})}^{e^{(y+y_1)/2}}
		f(x_1,y_1) \dd x_1 \dd y_1\\
	&\hspace{10pt}+ \int_{h(y)}^{R} \int_{I_n(1-e^{(y-y_1)/2})}^{I_n}
		f(x_1,y_1) \dd x_1 \dd y_1.
\end{align*}

Let us first focus on the first integral. Since	$I_n(e^{(y_1-y)/2}-1) - I_n(1-e^{(y-y_1)/2}) \le I_n e^{(y_1-y)/2}$ we get,
using similar arguments as above
\begin{align*}
	\int_y^{h^\ast(y)} \int_{I_n(1-e^{(y-y_1)/2})}^{I_n(e^{(y_1-y)/2}-1)} f(x_1,y_1) \dd x_1 \dd y_1
	&= \bigO{I_n e^{-y/2} \int_y^{h^\ast(y)} e^{-(\alpha - \frac{1}{2})y_1} \dd y_1}\\
	&= \bigO{I_n e^{-\alpha y} \left(1 - \left(\frac{I_n}{I_n - e^y}\right)^{-(2\alpha - 1)}\right)}\\
	&= \bigO{e^{(1-\alpha)y}}.
\end{align*}

Proceeding to the second integral, we first note that $e^{(y+y_1)/2} - I_n(1-e^{(y-y_1)/2}) = \bigO{I_n e^{(y_1-y)/2}}$ so that similar calculations as before yield
\begin{align*}
	\int_{h^\ast(y)}^{h(y)} \int_{I_n(1-e^{(y-y_1)/2})}^{e^{(y+y_1)/2}}	f(x_1,y_1) \dd x_1 \dd y_1
	&= \bigO{I_n e^{-y/2} \int_{h^\ast(y)}^{h(y)} e^{-(\alpha - \frac{1}{2})y_1} \dd y_1}
		= \bigO{e^{(1-\alpha)y}}.
\end{align*}

\end{proofof}

\section{Concentration for \texorpdfstring{$c(k; \Gbox)$}{c(k;G box)} (Proving Proposition \ref{prop:concentration_local_clustering_P_n})}
\label{sec:concentration_c_P_n}

In this section we establish a concentration result for the local clustering function $c^\ast(k; \Gbox)$ in the finite box model $\Gbox$. Similar to the previous section we will focus on typical points $p = (0,y)$ with $y \in \Kcal_{C}(k_n)$. 

\subsection{The main contribution of triangles}

Recall that $\Nbox(k_n)$ denotes the number of vertices in $\Gbox$ with degree $k_n$. We first write
\[
	c^\ast(k_n ; \Gbox) = \frac{T_{\text{box}}(k_n)}{\binom{k_n}{2}\Exp{\Nbox(k_n)}},
\]
where
\[
	 T_{\text{box}}(k_n) = \sum_{p \in \Pcal} \ind{\Dbox(p) = k_n} \sum_{(p_1, p_2) \in \Pcal \setminus \{p\}, \atop \text{distinct}} 
	 \ind{p_1 \in \BallPon{p}}\ind{p_2 \in \BallPon{p}}\ind{p_2 \in \BallPon{p_1}}
\]
In particular, the variance of $c^\ast(k_n ; \Gbox)$ is determined by the variance of $T_{\text{box}}(k_n)$.

Next, recall the adjusted triangle count function
\[
	\widetilde{T}_{\text{box}}(p_0) = \sum_{(p_1, p_2) \in \Pcal \setminus \{p_0\}, \atop \text{distinct}}
		\widetilde{T}_{\text{box}}(p_0,p_1,p_2).
\]
where
\[
	\widetilde{T}_{\text{box}}(p_0,p_1,p_2) = \ind{p_1 \in \BallPon{p_0}}\ind{p_2 \in \BallPon{p_0}}\ind{p_2 \in \BallPo{p_1} \cap \Rcal},
\]
as well as the definition of $\Kcal_{C}(k_n)$
\[
	\Kcal_{C}(k_n) = \left\{y \in \R_+ : \frac{k_n - C \sqrt{k_n \log(k_n)}}{\xi} \vee 1 \le e^{\frac{y}{2}}
		\le \frac{k_n + C \sqrt{k_n \log(k_n)}}{\xi} \right\},
\]
and write $\Rcal(k_n,C) = [-I_n,I_n] \times \Kcal_{C}(k_n)$ for the part of the box $\Rcal$ with heights in $\Kcal_{C}(k_n)$.
Slightly abusing notation, we will define the corresponding triangle degree function
\begin{equation}\label{eq:def_degree_triangle_count_in_K}
	\widetilde{T}_{\text{box}}(k_n, C) = \sum_{p \in \Pcal \cap \Rcal(k_n,C)} \ind{\text{deg}_{\text{box}}(p) = k_n} \widetilde{T}_{\text{box}}(p).
\end{equation}
and with that a different clustering function.
\begin{equation}\label{eq:def_tilde_c_box}
	\widetilde{c}_{\text{box}}(k_n) = \frac{\widetilde{T}_{\text{box}}(k_n,C)}{\binom{k_n}{2}\Exp{N_{\text{box}}(k_n)}}.
\end{equation}
The idea is that the main contribution of triangles of degree $k_n$ to the triangle count $T_{\text{box}}(k_n)$ is given by $\widetilde{T}_{\text{box}}(k_n, C)$. Therefore, in order to prove Proposition \ref{prop:concentration_local_clustering_P_n} it suffices to show that $\widetilde{T}_{\text{box}}(k_n,C)$ is sufficiently concentrated around its mean. This last part is done in the following proposition.

%

\begin{proposition}[Concentration $\widetilde{T}_{\text{box}}(k_n,C)$]\label{prop:concentration_tilde_T_P_n}
Let $\alpha > \frac{1}{2}$, $\nu > 0$ and let $(k_n)_{n \ge 1}$ be any positive sequence satisfying $k_n = \smallO{n^{\frac{1}{2\alpha+1}}}$. Then for any $C > 0$, as $n \to \infty$,
\[
	\Exp{\widetilde{T}_{\text{box}}(k_n,C)^2} = \left(1 + \smallO{1}\right)\Exp{\widetilde{T}_{\text{box}}(k_n, C)}^2.
\]
\end{proposition}

We first use this result to prove Proposition~\ref{prop:concentration_local_clustering_P_n}. The remainder of this section is devoted to the proof of Proposition~\ref{prop:concentration_tilde_T_P_n}. The final proof can be found in Section~\ref{ssec:concentration_tilde_T}.

\begin{proofof}{Proposition \ref{prop:concentration_local_clustering_P_n}}
We bound the expectation as follows,
\begin{align*}
	\Exp{\left|c^\ast(k_n; \Gbox) - \Exp{c^\ast(k_n; \Gbox)}\right|} 
	&\le \frac{\Exp{\left|\widetilde{T}_{\text{box}}(k_n,C) - \Exp{\widetilde{T}_{\text{box}}(k_n,C)}\right|}}
		{\binom{k_n}{2}\Exp{N_{\text{box}}(k_n)}}\\
	&\hspace{10pt} + 2 \Exp{\left|c^\ast(k_n; \Gbox) - \widetilde{c}_{\text{box}}(k_n)\right|}.
\end{align*}
We will show that both terms are $\smallO{s(k_n)}$.

First we note that $\ind{p_2 \in \BallPo{p_1} \cap \Rcal} \le \ind{p_2 \in \BallPon{p_1}}$ and hence $\widetilde{T}_{\text{box}}(p) \le T_{\text{box}}(p)$. This implies that
\[
	 \widetilde{c}_{\text{box}}(k_n) = \frac{\widetilde{T}_{\text{box}}(k_n,C)}{\binom{k_n}{2}\Exp{N_{\text{box}}(k_n)}} \le c^\ast(k_n; \Gbox). 
\]
and therefore
\[
	\Exp{\left|c^\ast(k_n; \Gbox) - \widetilde{c}_{\text{box}}(k_n)\right|}
	= \Exp{c^\ast(k_n; \Gbox)} - \Exp{\widetilde{c}_{\text{box}}(k_n)}.
\]
For the expectation of $\widetilde{T}_{\text{box}}(k_n,C)$ we use that 
\[
	\CExp{\widetilde{T}_{\text{box}}(p)}{\Dbox(p) = k_n}
= \binom{k_n}{2} \Mu{\BallPon{y}}^{-2} \Exp{\widetilde{T}_{\text{box}}(p)}.
\] 
Recall that for $y \in \Rcal(k_n,C)$
\[
	\Mu{\BallPon{y}}^{-2} = (1+\smallO{1})\Mu{y}^{-2} = (1+\smallO{1})k_n^{-2},
\]
where the error term is uniform in $y$.

We thus obtain
\begin{align*}
	\Exp{\widetilde{T}_{\text{box}}(k_n,C)} 
	&= \int_{\Rcal(k_n,C)} \CExp{\widetilde{T}_{\text{box}}(p)}{\Dbox(p) = k_n}
		\rho_{\text{box}}(y,k_n) f(x,y) \dd x \dd y\\
	&= (1+\smallO{1}) \binom{k_n}{2} \int_{\Rcal(k_n,C)} \Mu{\BallPon{y}}^{-2} \Exp{\widetilde{T}_{\text{box}}(y)}
		\rho_{\text{box}}(y,k_n) \alpha e^{-\alpha y} \dd y\\
	&= (1+\smallO{1}) \frac{1}{2} \int_{\Rcal(k_n,C)} \Exp{\widetilde{T}_{\text{box}}(y)}
			\rho_{\text{box}}(y,k_n) \alpha e^{-\alpha y} \dd y\\
	&= (1+\smallO{1}) n \binom{k_n}{2} \int_0^\infty P(y) \rho(y,k_n) \alpha e^{-\alpha y} \dd y,
\end{align*}
where the last line is due to Corollary~\ref{cor:adjusted_triangle_counting_P_n}. In particular, since the last integral is $\bigT{k_n^{-(2\alpha + 1)} s(k_n)}$ we conclude that
\begin{equation}\label{eq:scaling_exected_tilde_T}
	\Exp{\widetilde{T}_{\text{box}}(k_n,C)} = \bigT{n k_n^{-(2\alpha - 1)}s(k_n)}. 
\end{equation}

Since $\Exp{N_{\text{box}}(k_n)} = (1+\smallO{1}n \pmf(k_n)$ it follows that
\[
	\widetilde{c}_{\text{box}}(k_n) = \frac{\Exp{\widetilde{T}_{\text{box}}(k_n,C)}}{\binom{k_n}{2}\Exp{N_{\text{box}}(k_n)}}
	= (1+\smallO{1}) \frac{\int_0^\infty P(y) \alpha e^{-\alpha y} \dd y}{\pmf(k_n)}
	= (1 + \smallO{1}) \gamma(k_n).
\]
On the other hand, Proposition~\ref{prop:convergence_average_clustering_P_n} implies that $\Exp{c^\ast(k_n; \Gbox)} = (1+\smallO{1})\gamma(k_n)$ and thus we conclude that
\[
	2\Exp{\left|c^\ast(k_n; \Gbox) - \widetilde{c}_{\text{box}}(k_n)\right|}
	= \smallO{\gamma(k_n)} = \smallO{s(k_n)}.	
\]

For the remaining term we use H\"{o}lder's inequality and Proposition \ref{prop:concentration_tilde_T_P_n} to obtain
\begin{align*}
	\Exp{\left|\widetilde{T}_{\text{box}}(k_n,C) - \Exp{\widetilde{T}_{\text{box}}(k_n,C)}\right|}
	&\le \left(\Exp{\widetilde{T}_{\text{box}}(k_n,C)^2} 
		- \Exp{\widetilde{T}_{\text{box}}(k_n,C)}^2\right)^{\frac{1}{2}}\\
	&= \smallO{\Exp{\widetilde{T}_{\text{box}}(k_n,C)}}.
\end{align*}
This implies
\begin{align*}
	\frac{\Exp{\left|\widetilde{T}_{\text{box}}(k_n,C) - \Exp{\widetilde{T}_{\text{box}}(k_n,C)}\right|}}
		{\binom{k_n}{2}\Exp{N_{\text{box}}(k_n)}}
	&= \smallO{\frac{\Exp{\widetilde{T}_{\text{box}}(k_n,C)}}{\binom{k_n}{2}\Exp{N_{\text{box}}(k_n)}}}
	= \smallO{s(k_n)},
\end{align*}
which finishes the proof.
\end{proofof}

We note that the above proof establishes the following important result

\begin{corollary}\label{cor:c_ast_box_2_tilde_c_box}
Let $k_n \to \infty$. Then, as $n \to \infty$,
\[
	\Exp{\left|c^\ast(k_n; \Gbox) - \widetilde{c}_{\mathrm{box}}(k_n)\right|} = \smallO{s(k_n)}.
\]
\end{corollary}

\subsection{Joint degrees in \texorpdfstring{$\Gbox$}{G box}}

To prove Proposition~\ref{prop:concentration_tilde_T_P_n} will use results from Section~\ref{ssec:joint_degrees_GPo} regarding the joint degree distribution in $\Gbox$. For any two points $p,p^\prime \in \Rcal$ we will denote by 
\begin{equation}\label{eq:def_joint_degree_distribution_2}
	\rho_{\text{box}}(p,p^\prime,k,k^\prime) 
	:= \Prob{\Po\left(\Mu{\BallPon{p}}\right) = k, \Po\left(\Mu{\BallPon{p^\prime}}\right) = k^\prime}.
\end{equation}
the joint degree distribution.

Recall the definition of $\mathcal{E}_\varepsilon(k_n)$
from Section~\ref{ssec:joint_degrees_GPo},
\[
	\mathcal{E}_{\varepsilon}(k_n) = \left\{(p,p^\prime) \in \Rcal \times \Rcal
			\, : \, y,y^\prime \in [y_{k_n,C}^-, y_{k_N,C}^+] \text{ and } |x - x^\prime|_n > k_n^{1 + \varepsilon} \right\},
\]
where
\[
	y_{k,C}^\pm = 2 \log\left(\frac{k \pm C \sqrt{k \log(k)}}{\xi}\right),
\]
as defined in~\eqref{eq:def_y_k_C}. Furthermore, we recall that by Lemma~\ref{lem:probdegFact} the joint degree distribution of two point $p, p^\prime \in \mathcal{E}_{\varepsilon}(k_N)$ factorizes, i.e. on the set $\mathcal{E}_\varepsilon(k_n)$ the joint degree distribution in $\Gbox$ is asymptotically equivalent to the product of the degree distributions. We shall now prove a slightly stronger result (Lemma~\ref{lem:joint_degree_factorization}) which also takes care of bounded shifts in the joint degree distribution $\rho_{\text{box}}(p,p^\prime,k_n - t,k_n - t^\prime)$, for some uniformly bounded $t, t^\prime \in \mathbb{Z}$. For this we first need the following simple result for Poisson distributions.

\begin{lemma}\label{lem:finite_shifts_poisson}
Let $k_n \to \infty$ be a sequence of non-negative integers and $X = \Po(\lambda_n)$ be a Poisson random variable with mean $\lambda_n$ satisfying
\[
		k_n - C\sqrt{k_n \log(k_n)} \le \lambda_n \le k_n + C\sqrt{k_n \log(k_n)}
\]
for some $C > 0$. Then, for any $t_n, s_n = O(1)$, as $n \to \infty$,
\[
	\Prob{X = k_n - t_n} \sim \Prob{X = k_n - s_n}.
\]
\end{lemma}

\begin{proof}
Note that $k_n > t_n, s_n$ for large enough $n$. Hence, using Stirling's formula, as $n \to \infty$,
\begin{align*}
	\frac{\Prob{X = k_n - t_n}}{\Prob{X = k_n - s_n}}
	&= \frac{(k_n - t_n - (s_n-t_n))!}{(k_n - t_n)!} \lambda_n^{s_n - t_n} \\
	&\sim \sqrt{\frac{k_n - s_n}{k_n - t_n}} \, \frac{(k_n-s_n)^{k_n - s_n}}{(k_n - t_n)^{k_n - t_n}} \, e^{t_n - s_n} 
		\lambda_n^{s_n-t_n} \\
	&= \sqrt{\ell_n} (\ell_n)^{k_n - t_n} e^{t_n - s_n} (k_n - s_n)^{t_n - s_n} \lambda_n^{s_n-t_n}\\
	&= \sqrt{\ell_n} e^{(k_n - t_n)\log(\ell_n) + t_n - s_n} \left(\frac{k_n - s_n}{\lambda_n}\right)^{t_n - s_n}
\end{align*}
where we wrote $\ell_n = (k_n - s_n)/(k_n - t_n)$. Note that $\ell_n \to 1$ and hence $\sqrt{\ell_n} \to 1$. Moreover, since $(k_n - s_n)/\lambda_n \to 1$ and $|s_n - t_n| = \bigO{1}$ we have that $\left(\frac{k_n - s_n}{\lambda_n}\right)^{t_n - s_n} \sim 1$ Therefore it remains to show that
\[
	\lim_{n \to \infty} e^{(k_n - t_n)\log(\ell_n) + t_n - s_n} = 1.
\] 
For this we note that for any $x$, such that $|x| \le 1/2$, we have 
\[
	x - x^2 \le \log(1+x) \le x.
\]
Write $x_n = \ell_n - 1 = \frac{t_n - s_n}{k_n - t_n}$. Then by the assumptions of the lemma, $x_n \to 0$, and thus, for $n$ large enough,
\[
	t_n - s_n - \frac{(t_n - s_n)^2}{k_n - t_n}
	\le (k_n-t)\log\left(\ell_n\right)
	\le t_n - s_n.
\]
In particular
\[
	e^{-\frac{(t_n-s_n)^2}{k_n-t_n}}
	\le e^{(k_n - t_n)\log(\ell_n) + t_n - s_n} \le 1,
\]
and the result follows since $\frac{(t_n-s_n)^2}{k_n-t_n} \to 0$.
\end{proof}

We can now prove the main result of this section.

\begin{lemma}\label{lem:joint_degree_factorization}
Let $0 < \varepsilon < 1$, $k_n \to \infty$ and let $t_n,t^\prime_n, s_n, s_n^\prime \in \mathbb{Z}$ be uniformly bounded.
Then for any $(p,p^\prime) \in \mathcal{E}_\varepsilon(k_n)$, as $n \to \infty$,
\begin{align*}
	\rho_{\mathrm{box}}(p,p^\prime,k_n - t_n,k_n - t_n^\prime)
	&= (1+\smallO{1})\rho_{\mathrm{box}}(p,k_n - s_n)\rho_{\mathrm{box}}(p^\prime,k_n-s_n^\prime).
\end{align*}
\end{lemma}

\begin{proof}
Define the random variables
\begin{align*}
	X_1(p,p^\prime) &:= \Po\left(\Mu{\BallPon{p}\setminus \BallPon{p^\prime}}\right),\\
	X_2(p,p^\prime) &:= \Po\left(\Mu{\BallPon{p^\prime}\setminus \BallPon{p}}\right),\\
	Y(p,p^\prime) &:= \Po\left(\Mu{\BallPon{p} \cup \BallPon{p^\prime}}\right),
\end{align*}
so that
\begin{align*}
	\rho_{\text{box}}(p,p^\prime,k_n-t_n,k_n-t_n^\prime) &= \Prob{X_1(p,p^\prime) + Y(p,p^\prime) = k_n-t_n, X_2(p,p^\prime) + Y(p,p^\prime) = k_n - t_n^\prime}.
\end{align*}
Since by Lemma~\ref{cor:expected_common_neighbours_Ecal_set} $\Mu{\BallPon{p} \cap \BallPon{p^\prime}} = \bigO{k_n^{1-\varepsilon^\prime}}$, it follows from Lemma~\ref{lem:probdegFact} that
\[
	\rho_{\text{box}}(p,p^\prime,k_n-t_n,k_n-t_n^\prime) = (1+\smallO{1})\rho_{\text{box}}(p,k_n-t_n)\rho_{\text{box}}(p^\prime,k_n-t_n^\prime).
\]
The result then follows by applying Lemma~\ref{lem:finite_shifts_poisson} twice.
\end{proof}

\subsection{Concentration result for main triangle contribution}\label{ssec:concentration_tilde_T}

We now turn to Proposition \ref{prop:concentration_tilde_T_P_n}. Before we dive into the proof let us first give a high level overview of the strategy and the flow of the arguments. 

Recall (see \eqref{eq:def_degree_triangle_count_in_K}) that for any $C > 0$
\[
	\widetilde{T}_{\text{box}}(k_n,C) = \sum_{p \in \Pcal_n \cap \Kcal_{C,n}(k_n)} \ind{\text{deg}_{\text{box}}(p) = k} \widetilde{T}_{\text{box}}(p)
\]
Then we have
\[
	\widetilde{T}_{\text{box}}(k_n,C)^2 = \hspace{-5pt} \sum_{p, p^\prime \in \Pcal_n \cap \Kcal_C(k_n)}
		\hspace{-3pt} \ind{\Dbox(p), \, \Dbox(p^\prime) = k_n} 
		\sum_{(p_1, p_2), (p_1^\prime, p_2^\prime) \in \Pcal_n, \atop \text{distinct}} \hspace{-3pt}
		\widetilde{T}_{\Pcal}(p,p_1,p_2) \widetilde{T}_{\Pcal}(p^\prime, p_1^\prime, p_2^\prime),
\]
This expression can be written as the sum of several terms, depending on how $\{p, p_1, p_2\}$ and $\{p^\prime, p_1^\prime, p_2^\prime\}$ intersect. To this end we define, for $a \in \{0,1\}$ and $b \in \{0,1,2\}$,
\[
	I_{a,b} = \hspace{-3pt} \sum_{p, p^\prime \in \Pcal_n \cap \Kcal_C(k) \atop |\{p\} \cap \{p^\prime\}| = a}
	\hspace{-5pt} \ind{\Dbox(p), \, \Dbox(p^\prime) = k_n} J_b(p,p^\prime),
\]
where
\[
	J_b(p,p^\prime) = \hspace{-10pt} \sum_{{p_1, p_2, p_1^\prime, p_2^\prime \in \Pcal_n 
		\atop |\{p_1, p_2\} \cap \{p_1^\prime, p_2^\prime\}| = b,} \atop \text{distinct}}
		\hspace{-5pt} T_{\Pcal,n}(p,p_1,p_2) T_{\Pcal,n}(p^\prime, p_1^\prime, p_2^\prime),
\]
with the sum taken over all two distinct pairs $(p_1, p_2)$ and $(p_1^\prime, p_2^\prime)$. Then we have
\[
	\widetilde{T}_{\text{box}}(k, C)^2 = \sum_{a = 0}^1 \sum_{b = 0}^2 I_{a,b}.
\]

To prove Proposition \ref{prop:concentration_tilde_T_P_n} we will deal with each of the $I_{a,b}$ separately, showing that 
\begin{equation}\label{eq:variance_T_I_00}
	\Exp{I_{0,0}} = (1+\smallO{1})\Exp{\widetilde{T}_{\text{box}}(k_n,C)}^2
\end{equation}
and for all other combinations
\begin{equation}\label{eq:variance_T_I_ab}
	\Exp{I_{a,b}} = \smallO{\Exp{\widetilde{T}_{\text{box}}(k_n,C)}^2}.
\end{equation}
Note $I_{1,2} = \widetilde{T}_{\text{box}}(k_n,C)$ and since~\eqref{eq:scaling_exected_tilde_T} implies that $\Exp{\widetilde{T}_{\text{box}}(k_n,C)} \to \infty$, it follows that \eqref{eq:variance_T_I_ab} holds for $I_{1,2}$. 

Recall that $\Rcal(k_n,C) = [-I_n,I_n] \times \Kcal_{C}(k_n)$ and~\eqref{eq:def_joint_degree_set_E_growing_k}
\[
	\mathcal{E}_{\varepsilon}(k_n) = \left\{(p,p^\prime) \in \Rcal \times \Rcal
			\, : \, y,y^\prime \in \Kcal_{C}(k_n) \text{ and } |x - x^\prime|_n > k_n^{1 + \varepsilon} \right\}.
\]
Let $\mathcal{E}_\varepsilon(k_n)^c$ be the same set but with $|x - x^\prime|_n \le k_n^{1 + \varepsilon}$ and denote by $I_{a,b}^\ast$ the the part of $I_{a,b}$ where $(p,p^\prime) \in \mathcal{E}_\varepsilon(k_n)$. Will split the analysis between $I_{a,b}^\ast$ and $I_{a,b} - I_{a,b}^\ast$. The idea for these two cases is that by Lemma~\ref{lem:joint_degree_factorization} it follows that on the set $\mathcal{E}_\varepsilon(k_n)$ and for any uniformly bounded $t, t^\prime \in \mathbb{Z}$, the joint degree distribution factorizes,
\[
	\rho_{\text{box}}(p,p^\prime,k_n + t,k_n+t^\prime) = (1 + \smallO{1})\rho_{\text{box}}(p,k_n)\rho_{\text{box}}(p,k_n).
\]
In particular this allows us to prove that $\Exp{I_{0,0}^\ast} = (1+\smallO{1})\Exp{\widetilde{T}_{\text{box}}(k_n,C)}^2$. 
On the other hand, the expected number of points in $\mathcal{E}_\varepsilon(k_n)^c$ is $\bigO{k_n^{1+\varepsilon} k_n^{-2\alpha} \Exp{\Nbox(k_n)}} = \smallO{\Exp{\Nbox(k_n)}^2}$, where the latter is the expected number of points in $\Rcal(k_n,C) \times \Rcal(k_n,C)$. Hence we expect the contributions coming from $\mathcal{E}_\varepsilon(k_n)^c$ to be negligible.

\begin{proofof}{Proposition \ref{prop:concentration_tilde_T_P_n}}

Throughout this proof we set $i = |\{p^\prime, p_1, p_2, p_1^\prime, p_2^\prime\} \cap \BallPon{p}|$, $j = |\{p^\prime\} \cap \BallPon{p}|$ and define $i^\prime, j^\prime$ in a similar way by interchanging the primed and non-primed variables. In addition, we write $\widetilde{D}_{\text{box}}(p,p^\prime,k,\ell)$ to denote the indicator that $|\BallPon{p} \cap (\Pcal \setminus \{p, p^\prime, p_1, p_2, p_1^\prime, p_2^\prime\})| = k$ and $|\BallPon{p^\prime} \cap (\Pcal \setminus \{p, p^\prime, p_1, p_2, p_1^\prime, p_2^\prime\})| = \ell$. Note that this also depend on $\{p_1, p_2, p_1^\prime, p_2^\prime\}$ but we suppressed this to keep notation concise. Similarly we write $D_{\text{box}}(p,p^\prime,k,\ell)$ to denote the indicator that $|\BallPon{p} \cap (\Pcal \setminus \{p, p^\prime\})| = k$ and $|\BallPon{p^\prime} \cap (\Pcal \setminus \{p, p^\prime\})| = \ell$, which now only depends on $p$ and $p^\prime$. Then, by the Campbell-Mecke formula
\begin{align*}
	&\Exp{\ind{\Dbox(p) = k_n, \Dbox(p^\prime) = k_n} J_b(p,p^\prime)}\\
	&= \Exp{\sum_{{p_1, p_2, p_1^\prime, p_2^\prime \in \Pcal_n 
		\atop |\{p_1, p_2\} \cap \{p_1^\prime, p_2^\prime\}| = b,} \atop \text{distinct}}
			\hspace{-7pt} \widetilde{D}_{\text{box}}(p,p^\prime, k_n-i,k_n- i^\prime) \,
			\widetilde{T}_{\text{box}}(p,p_1,p_2) \widetilde{T}_{\text{box}}(p^\prime, p_1^\prime, p_2^\prime)}
\end{align*}
where the sum is over all distinct pairs $(p_1, p_2)$ and $(p_1^\prime, p_2^\prime)$. We also know that 
\[
	\Exp{T_\Pcal(k_n)} = \bigT{n k_n^{-(2\alpha - 1)} s_\alpha(k_n)}.
\]

We will now proceed to establish \eqref{eq:variance_T_I_00} and \eqref{eq:variance_T_I_ab}. 

\paragraph{Computing $\bm{I_{0,0}}$}
We first show that
\begin{equation}\label{eq:I_ab_ast_main}
	\Exp{I_{0,0} - I_{0,0}^\ast} = \smallO{\Exp{T_{\text{box}}(k_n,C)}^2},
\end{equation}
so that for the remainder of the proof we only need to consider $p, p^\prime \in \mathcal{E}_\varepsilon(k_n)$ and hence, we can apply Lemma \ref{lem:joint_degree_factorization}. 

For $J_0$ we have, using Lemma~\ref{lem:joint_degree_factorization}
\begin{align*}
	&\Exp{\ind{\Dbox(p) = k_n, \Dbox(p^\prime) = k_n} J_0(p,p^\prime)}\\
	&= \Exp{\sum_{{p_1, p_2, p_1^\prime, p_2^\prime \in \Pcal \setminus \{p,p^\prime\} 
		\atop |\{p_1, p_2\} \cap \{p_1^\prime, p_2^\prime\}| = 0,} \atop \text{distinct}}
		\hspace{-7pt} \widetilde{D}_{\text{box}}(p,p^\prime, k_n-i,k_n- i^\prime) \,
		\widetilde{T}_{\text{box}}(p,p_1,p_2) \widetilde{T}_{\text{box}}(p^\prime, p_1^\prime, p_2^\prime)}\\
	&= \Exp{D_{\text{box}}(p,p^\prime, k_n-j-2,k_n- j^\prime-2) \, \sum_{p_1,p_2 \in \Pcal \setminus p, \atop \text{distinct}}
		\widetilde{T}_{\text{box}}(p,p_1,p_2)
		\sum_{p_1^\prime,p_2^\prime \in \Pcal \setminus p^\prime, \atop \text{distinct}} 
			\widetilde{T}_{\text{box}}(p^\prime,p_1^\prime,p_2^\prime)}\\
	&= (1+\smallO{1})\rho_{\text{box}}(p,p^\prime,k_n,k_n) \CExp{\widetilde{T}_{\text{box}}(p)}{\Dbox(p) = k_n}
		\CExp{\widetilde{T}_{\text{box}}(p^\prime)}{\Dbox(p^\prime) = k_n},
\end{align*}
Next we recall that for all $y^\prime \in \Kcal_{C}(k_n)$ (see~\eqref{eq:def_K_C_set}), 
\[
	\CExp{\widetilde{T}_{\text{box}}(p^\prime)}{\Dbox(p^\prime) = k_n} = \binom{k_n}{2} \Mu{\BallPon{p^\prime}}^{-2} \Exp{\widetilde{T}_{\text{box}}(p^\prime)} = \bigO{1}k_n^2 P(y^\prime),
\] 
where $p^\prime = (x^\prime, y^\prime)$ and we used that $\Exp{\widetilde{T}_{\text{box}}(p^\prime)} = (1+\smallO{1}) k_n^2 P(y^\prime)$, for all $y^\prime \in \Kcal_{C}(k_n)$. Therefore, using that $\rho_{\text{box}}(p,p^\prime,k_n,k_n) \le \rho_{\text{box}}(p,k_n)$,
\begin{align*}
	&\Exp{\ind{\Dbox(p) = k_n, \Dbox(p^\prime) = k_n} J_0(p,p^\prime)}\\
	&\le \bigO{k_n^2} \rho_{\text{box}}(p,k_n) \CExp{\widetilde{T}_{\text{box}}(p)}{\Dbox(p) = k_n} P(y^\prime)
\end{align*}
and thus
\begin{align*}
	&\Exp{I_{0,0} - I_{0,0}^\ast}\\
	&= \int_{\mathcal{E}_\varepsilon(k_n)^c}
		\Exp{\ind{\Dbox(p), \Dbox(p^\prime) = k_n} J_0(p,p^\prime)} f(x,y) f(x^\prime,y^\prime) \, dx^\prime \, dx \, dy^\prime \, dy\\
	&\le \bigO{k_n^2} k_n^{1+\varepsilon} \left(\int_{a_n^-}^{a_n^+} P(y^\prime) 
		e^{-\alpha y^\prime} \, dy^\prime\right) \Exp{\widetilde{T}_{\text{box}}(k_n,C)}\\
	&= \bigO{k_n^{3 + \varepsilon -2\alpha} s_\alpha(k_n) \Exp{\widetilde{T}_{\text{box}}(k_n,C)}}\\
	&= \smallO{n k_n^{-(2\alpha - 1)}s_\alpha(k_n) \Exp{\widetilde{T}_{\text{box}}(k_n,C)}} 
		= \smallO{\Exp{\widetilde{T}_{\text{box}}(k_n,C)}^2},
\end{align*}
which proves \eqref{eq:I_ab_ast_main}. Here we used that $k_n^{2 + \varepsilon} = \smallO{n}$ and
\[
	\Exp{\widetilde{T}_{\text{box}}(k_n,C)} = \bigT{\Exp{\widetilde{T}_{\text{box}}(k_n)}} 
	= \bigT{n k_n^{-(2\alpha - 1)}s_\alpha(k_n)}
\] 
for the last line.

We will now show that
\[
	\Exp{I^\ast_{0,0}} = (1+\smallO{1})\Exp{\Exp{T_{\text{box}}(k_n,C)}^2}.
\]
Recall the result from Lemma~\ref{lem:joint_degree_factorization}, that for $(p,p^\prime) \in \mathcal{E}_{\varepsilon}(k_n)$ and any two uniformly bounded $t, t^\prime \in \mathbb{Z}$,
\[
	\rho_{\text{box}}(p,p^\prime,k_n + t,k_n+t^\prime) = (1 + \smallO{1})\rho_{\text{box}}(p,k_n)\rho_{\text{box}}(p,k_n).
\]
Therefore, by defining $h(y) = \CExp{\widetilde{T}_{\text{box}}(y)}{\Dbox(y) = k_n}$
\[
	\Exp{I_{0,0}^\ast} = (1 + \smallO{1})\int_{\mathcal{E}_{\varepsilon}(k_n)}\rho_{\text{box}}(p,k_n)
		\rho_{\text{box}}(p^\prime,k_n)
		h(y) h(y^\prime) f(x,y)	f(x^\prime,y^\prime) \, dx^\prime \, dx \, dy^\prime \, dy.
\]
The difference with $\Exp{\Exp{T_{\text{box}}(k_n,C)}^2}$ is in that the above integral is over $\mathcal{E}_\varepsilon(k_n)$ instead of $\Rcal(k_n,C) \times \Rcal(k_n,C)$. Since the difference between the two sets is 
$\mathcal{E}_\varepsilon(k_n)^c$ and $n k_n^{1+\varepsilon} = \smallO{n^2}$ it follows that
\begin{align*}
	&\Exp{\Exp{T_{\text{box}}(k_n,C)}^2}
		- \int_{\mathcal{E}_{\varepsilon}(k_n)}\rho_{\text{box}}(p,k_n)\rho_{\text{box}}(p^\prime,k_n)
		h(y) h(y^\prime) f(x,y)	f(x^\prime,y^\prime) \dd x^\prime \dd x \dd y^\prime \dd y\\
	&= \int_{\mathcal{E}_{\varepsilon}(k_n)^c}\rho_{\text{box}}(p,k_n)\rho_{\text{box}}(p^\prime,k_n)
		h(y) h(y^\prime) f(x,y)	f(x^\prime,y^\prime) \dd x^\prime \dd x \dd y^\prime \dd y\\
	&= \bigO{k_n^{1+\varepsilon}n} 
		\left(\int_{\Kcal_{C}(k_n)}h(y)\rho_{\text{box}}(y,k_n)\alpha e^{-\alpha y} \dd y\right)^2
		= \smallO{\Exp{\Exp{T_{\text{box}}(k_n,C)}^2}}.
\end{align*}
Thus we conclude that $\Exp{I^\ast_{0,0}} = (1+\smallO{1})\Exp{\Exp{T_{\text{box}}(k_n,C)}^2}$, which finishes the proof of \eqref{eq:variance_T_I_00}.

\paragraph{Computing $\bm{\Exp{I_{0,1}}}$}

We first write
\begin{equation}\label{eq:bound_E_J1}
	\Exp{\ind{\Dbox(p) = k_n, \Dbox(p^\prime) = k_n} J_1}
	\le \bigO{1} k_n \rho_{\text{box}}(p,p^\prime,k_n,k_n)\CExp{\widetilde{T}_{\text{box}}(p)}{\Dbox(p) = k_n}
\end{equation}
Then, using that $\rho_{\text{box}}(p,p^\prime,k_n,k_n) \le \rho_{\text{box}}(p,k_n)$,
\begin{align*}
	&\Exp{I_{0,1} - I_{0,1}^\ast}\\
	&\int_{\mathcal{E}_\varepsilon(k_n)^c}
		\Exp{\ind{\Dbox(p), \Dbox(p^\prime) = k_n} J_1(p,p^\prime)} f(x,y) f(x^\prime,y^\prime) \, dx^\prime \, dx \, dy^\prime \, dy\\
	&= k_n \int_{\Kcal_C(k_n)^2} \ind{|x-x^\prime| \le k_n^{1 + \varepsilon}}
		\rho_{\text{box}}(p,k_n)\CExp{\widetilde{T}_{\text{box}}(p)}{\Dbox(p) = k_n} f(x,y) f(x^\prime,y^\prime) \, dx^\prime \, dx \, dy^\prime \, dy\\
	&\le \bigO{k_n^{2+\varepsilon}} \left(\int_{a_n^-}^{a_n^+} e^{-\alpha y^\prime} \, dy^\prime\right) 
		\Exp{\widetilde{T}_{\text{box}}(k_n,C)}\\
	&= \bigO{k_n^{2 + \varepsilon - 2\alpha} \Exp{\widetilde{T}_{\text{box}}(k_n,C)}}.
\end{align*}
Recall that $\Exp{\widetilde{T}_{\text{box}}(k_n,C)} = \bigT{n k_n^{-(2\alpha - 1)}s(k_n)}$. Therefore to show that $\Exp{I_{0,1} - I_{0,1}^\ast} = \smallO{\Exp{\widetilde{T}_{\text{box}}(k_n,C)}^2}$ it suffices to show that $k_n^{2 + \varepsilon - 2\alpha} = \smallO{n k_n^{-(2\alpha - 1)}s(k_n)}$. When $\frac{1}{2} < \alpha \le \frac{3}{4}$ we have
\[
	\frac{4\alpha - 1 + \varepsilon}{2\alpha + 1} < 1,
\]
for $\varepsilon$ small enough. Hence
\[
	n^{-1} k_n^{2\alpha - 1} s(k_n)^{-1} k_n^{2 + \varepsilon -2\alpha} = n^{-1} k_n^{4\alpha - 1 + \varepsilon} 
	= \smallO{n^{-1}n^{\frac{4\alpha - 1 + \varepsilon}{2\alpha + 1}}} = \smallO{1}
\]
When $\alpha \ge \frac{3}{4}$,
\[
	n^{-1} k_n^{2\alpha - 1} s(k_n)^{-1} k_n^{2 + \varepsilon -2\alpha} = \bigO{\log(k_n)} n^{-1} k_n^{2 + \varepsilon} = \smallO{1},
\]
for $\varepsilon$ small enough.

For $(p,p^\prime) \in \mathcal{E}_\varepsilon(k_n)$ we assume without loss of generality that $p_1^\prime = p_1 = (x_1,y_1)$, i.e.
\[
	J_{0,1} = \sum_{(p_1, p_2) \in \Pcal \setminus \{p\},\atop \text{distinct}} \widetilde{T}_{\text{box}}(p,p_1,p_2) 
	\sum_{p_2^\prime \in \Pcal \setminus \{p^\prime, p_1\}} \widetilde{T}_{\text{box}}(p^\prime,p_1,p_2^\prime). 
\] 
Now let $Z_{0,1}$ denote the part of $J_{0,1}$ where $y_1 \le 4\log(k_n)$ and $y_2, y_2^\prime \le \varepsilon \log(k_n)$. 

We first analyze $\CExp{Z_{0,1}}{\Dbox(p), \Dbox(p^\prime) = k_n}$. When $y_1 \le 4\log(k_n)$ and both $y_2, y_2^\prime \le \varepsilon \log(k_n)$ we have that
\[
	|x_2 - x_2^\prime| \le |x_1 - x_2| + |x_1 - x_2^\prime| \le e^{\frac{y_1}{2}}\left(e^{\frac{y_2}{2}} + e^{\frac{y_2^\prime}{2}}\right) \le 2k_n^{2+\varepsilon}
\]
whenever $\widetilde{T}_{\text{box}}(p,p_1,p_2) \widetilde{T}_{\text{box}}(p^\prime,p_1,p_2^\prime) > 0$ while both $|x - x_2|, |x^\prime - x_2^\prime| = \bigO{k_n^{1 + \varepsilon}}$. Hence it follows that $\widetilde{T}_{\text{box}}(p,p_1,p_2) \widetilde{T}_{\text{box}}(p^\prime,p_1,p_2^\prime) > 0$ implies that
\[
	|x - x^\prime| \le |x - x_2| + |x_2 - x_2^\prime| + |x_2^\prime - x^\prime| = \bigO{k_n^{2 + \varepsilon}}.
\]
Next, by integrating only over $x_2^\prime$ and $y_2^\prime $ we get
\begin{align*}
	\CExp{Z_{0,1}}{\Dbox(p), \Dbox(p^\prime) = k_n} &=
	\bigO{e^{\frac{y^\prime}{2}}\ind{|x-x^\prime|\le \bigO{1}k_n^{2 + \varepsilon}} 
		\CExp{\widetilde{T}_{\text{box}}(p)}{\Dbox(p) =k_n}}\\
	&= \bigO{k_n \CExp{\widetilde{T}_{\text{box}}(p)}{\Dbox(p) =k_n}}.
\end{align*}
Thus
\begin{align*}
	&\int_{\mathcal{E}_\varepsilon(k_n)} \rho_{\text{box}}(p,p^\prime,k_n,k_n)
		\CExp{Z_{0,1}}{\Dbox(p), \Dbox(p^\prime) = k_n} f(x,y) f(x^\prime,y^\prime) \dd x \dd y \dd x^\prime \dd y^\prime\\
	&= \bigO{k_n^{3+\varepsilon}} \Exp{\widetilde{T}_{\text{box}}(k_n,C)} 
		\int_{\Kcal_{C}(k_n)}\rho_{\text{box}}(y^\prime,k_n) e^{-\alpha y^\prime} \dd y^\prime\\
	&= \bigO{k_n^{2+\varepsilon} k_n^{-2\alpha} \Exp{\widetilde{T}_{\text{box}}(k_n,C)} } 
		= \smallO{\Exp{\widetilde{T}_{\text{box}}(k_n,C)}^2}, 
\end{align*}
where the last line follows from the analysis done for $\Exp{I_{0,0}-I_{0,0}^\ast}$.

It now remains to consider $J_{0,1} - Z_{0,1} := Z_{0,1}^\ast$. We will show that 
\begin{equation}\label{eq:expectation_J01}
	\CExp{Z_{0,1}^\ast}{\Dbox(p), \Dbox(p^\prime) = k_n} = \smallO{k_n^4 s(k_n)^2}.
\end{equation} 
Using that the joint degree distribution factorizes on $\mathcal{E}_\varepsilon(k_n)$ this then implies that
\begin{align*}
	\Exp{I_{0,1}^\ast}
	&= \smallO{k_n^4 s(k_n)^2} \left(\int_{\Rcal(k_n,C)} \rho_{\text{box}}(y,k_n) f(x,y) \dd x \dd y \right)^2\\ 
	&= \smallO{\left(n s(k_n) k_n^{-2\alpha + 1}\right)^2} = \smallO{\Exp{\widetilde{T}_{\text{box}}(k_n,C)}^2},
\end{align*}
which finished the proof of~\eqref{eq:variance_T_I_ab} for $a = 0, b = 1$.

We first consider the part with $y_1 > 4\log(k_n)$. Since the integration over $x_1, x_2$ and $x_2^\prime$ of $\CExp{Z_{0,1}^\ast}{\Dbox(p), \Dbox(p^\prime) = k_n}$ is bounded by $\bigO{e^{y}e^{\frac{y^\prime}{2}}}$ we get that the contribution to $\CExp{Z_{0,1}^\ast}{\Dbox(p), \Dbox(p^\prime) = k_n}$ with $y > 4\log(k_n)$ and $(p,p^\prime) \in \mathcal{E}_\varepsilon(k_n)$ is
\begin{align*}
	\bigO{e^{y}e^{\frac{y^\prime}{2}} \int_{4\log(k_n)}^{R} e^{-(\alpha - \frac{1}{2})y_1} \, dy_1}
	&= \bigO{k_n^3 \int_{4\log(k_n)}^{R} e^{-(\alpha - \frac{1}{2})y_1} \, dy_1}\\
	&= \bigO{k_n^{3 - (4\alpha - 2)}} = \smallO{k_n^4 s_\alpha(k_n)^2}.
\end{align*}
Here the last step follows since for $\frac{1}{2} < \alpha < \frac{3}{4}$
\[
	k_n^{3 - (4\alpha - 2) - 4} s(k_n)^{-2} 
	= k_n^{3 - (4\alpha - 2) - 4 + 2(4\alpha - 2)} = k_n^{-5 + 4\alpha} = \smallO{1},
\]
while for $\alpha = \frac{3}{4}$
\[
	k_n^{3 - (4\alpha - 2) - 4} s(k_n)^{-2} 
	= \bigO{\log(k_n)^{-2}} k_n^{3 - (4\alpha - 2) - 2} = \bigO{\log(k_n)^{-2}} = \smallO{1},
\]
and for $\alpha > \frac{3}{4}$
\[
	k_n^{3 - (4\alpha - 2) - 4} s(k_n)^{-2} = k_n^{3 - (4\alpha - 2) - 2} = \smallO{1}.
\]

Next we consider the case where $y_1 \le 4\log(k_n)$ and at least one of $y_2, y_2^\prime$ is larger than $\varepsilon \log(k_n)$. Due to symmetry it is enough to consider the case with $y_2 > \varepsilon \log(k_n)$. Here the contribution to $\CExp{Z_{0,1}^\ast}{\Dbox(p), \Dbox(p^\prime) = k_n}$ is
\begin{align*}
	\Exp{\widetilde{T}_{\text{box}}(p)}
		\bigO{e^{\frac{y^\prime}{2}}\int_{\varepsilon \log(k_n)}^{R} e^{-(\alpha -\frac{1}{2})y_2} \, dy_2}
	&= \bigO{k_n^{1 - \varepsilon(\alpha -\frac{1}{2})}}\Exp{\widetilde{T}_{\text{box}}(p)}\\
	&= \bigO{k_n^{3 - \varepsilon(\alpha -\frac{1}{2})} s(k_n)} = \smallO{k_n^4 s(k_n)^2}.
\end{align*}
The last line follows since $k_n^{-1} = \smallO{s(k_n)}$ for $\frac{1}{2} < \alpha < \frac{3}{4}$ and $k_n^{-1} = \bigO{s(k_n)}$ for $\alpha \ge \frac{3}{4}$. 

\paragraph{Computing $\bm{\Exp{I_{0,2}}}$} In this case we have
\[
	\Exp{\ind{\Dbox(p) = k_n, \Dbox(p^\prime) = k_n}J_2} = (1+\smallO{1})\rho_{\text{box}}(p,p^\prime,k_n,k_n)
	\CExp{\widetilde{T}_{\text{box}}(p)}{\Dbox(p) = k_n}.
\]
We then use that $\rho_{\text{box}}(p,p^\prime,k_n,k_n) \le \rho_{\text{box}}(p,k_n)$ to obtain
\begin{align*}
	\Exp{I_{0,2}-I_{0,2}^\ast} 
	&= \bigO{k_n^{1+\varepsilon}} \left(\int_{\Kcal_{C}(k_n)} e^{-\alpha y^\prime} \dd y^\prime\right) 
		\Exp{\widetilde{T}_{\text{box}}(k_n,C)}\\
	&= \bigO{k_n^{\varepsilon -(2\alpha - 1)}}\Exp{\widetilde{T}_{\text{box}}(k_n,C)}
		= \smallO{\Exp{\widetilde{T}_{\text{box}}(k_n,C)}}
\end{align*}
were the last line follows since $\Exp{\widetilde{T}_{\text{box}}(k_n,C)} = \bigT{n k_n^{-(2\alpha - 1)}s(k_n)}$ and $k_n^{\varepsilon}n^{-1} = \smallO{s(k_n)}$.

For the other term we use the fact that the degree distribution factorizes;
\begin{align*}
	\Exp{I_{0,2}^\ast} 
	&= \bigO{1}\left(\int_{\Rcal(k_n,C)} \rho_{\text{box}}(y^\prime,k_n) f(x^\prime,y^\prime) 
		\dd x^\prime \dd y^\prime \right)\Exp{\widetilde{T}_{\text{box}}(k_n,C)}\\
	&= \bigO{n k_n^{-(2\alpha + 1)}}\Exp{\widetilde{T}_{\text{box}}(k_n,C)}
		= \smallO{\Exp{\widetilde{T}_{\text{box}}(k_n,C)}^2},
\end{align*}
where we also used that $k_n^{-2} = \smallO{s(k_n)}$.

\paragraph{Computing $\bm{\Exp{I_{1,1}}}$}

Using~\eqref{eq:expectation_J01} we get
\begin{align*}
	\Exp{I_{1,1}}
	&= \bigO{k_n} \int_{\Rcal(k_n,C)} \rho_{\text{box}}(y,k_n) \CExp{\widetilde{T}_{\text{box}}}{\Dbox(p) = k_n}
		f(x,y) \dd x \dd y\\
	&= \bigO{k_n}\Exp{\widetilde{T}_{\text{box}}(k_n,C)}.
\end{align*}
Now observe that for $\frac{1}{2} < \alpha < \frac{3}{4}$
\[
	k_n n^{-1} k_n^{(2\alpha -1)} s(k_n)^{-1} = k_n^{6\alpha - 2} n^{-1}
	= \bigO{n^{\frac{4\alpha - 3}{2\alpha + 1}}} = \smallO{1},
\]
while for $\alpha \ge \frac{3}{4}$
\[
	k_n  n^{-1} k_n^{(2\alpha -1)} s(k_n)^{-1} = \bigO{n^{-1} k_n^{-(2\alpha - 1)}} = \smallO{1}.
\]
We conclude that $k_n = \smallO{n k^{-(2\alpha - 1)}s(k_n)}$ and hence $\Exp{I_{1,1}} = \smallO{\Exp{\widetilde{T}_{\text{box}}(k_n,C)}^2}$.

\end{proofof}

\section{Equivalence for local clustering in \texorpdfstring{$\GPo$}{G Po} and \texorpdfstring{$\Gbox$}{G box}}\label{sec:coupling_H_P_n}

In this section we establish the equivalence between $c^\ast(k; G_n)$ and $c^\ast(k; \Gbox)$ as expressed in Proposition~\ref{prop:couling_c_H_P}, using the coupling procedure explained in Section~\ref{ssec:coupling_H_P}. As in the previous section we write $|\cdot|_n$ for the norm $|\cdot|_{\pi e^{R/2}}$.

%
%

Recall the map $\Psi$ from~\eqref{eq:def_Psi}
\[
	\Psi(r,\theta) = \left(\theta \frac{e^{R/2}}{2}, R - r\right),
\] 
and that $\BallHyp{p}$ denotes the image under $\Psi$ of the ball of hyperbolic radius $R$ around the point $\Psi^{-1}(p)$. Under the coupling between the hyperbolic random graph and the finite box model, described in Section~\ref{ssec:coupling_H_P}, two points $p = (x,y)$ and $p^\prime = (x^\prime, y^\prime)$ are connected if and only if
\[
	|x-x^\prime|_n \le \Phi(y, y^\prime)
	= \frac{1}{2}e^{R/2} \arccos\left( \frac{\cosh(R-y) \cosh(R-y^\prime) - \cosh R}{\sinh(R-y) \sinh(R-y^\prime)} \right),
\]
see~\eqref{eq:def_Omega_hyperbolic}. We will often use the result from Lemma~\ref{lem:asymptotics_Omega_hyperbolic} to approximate the function $\Phi$, for $y + y^\prime < R$, by  
\[
	e^{\frac{1}{2}(y+y^\prime)} - K e^{\frac{3}{2}(y+y^\prime) - R} \leq \Phi(R - y, R - y^\prime) 
		\leq  e^{\frac{1}{2}(y+y^\prime)} + K e^{\frac{3}{2}(y+y^\prime) - R},
\]
where $K$ is a constant determined by the lemma.


\subsection{Some results on the hyperbolic geometric graph}

We start with some basic results for the hyperbolic random geometric graph. Recall that $\BallPo{p} = \{p^\prime \in \R \times \R_+ \, : \, |x - x^\prime| \le e^{(y +y^\prime)/2}\}$ and observe that~\eqref{eq:asymp2} from Lemma~\ref{lem:asymptotics_Omega_hyperbolic} implies the following. 
\begin{corollary}\label{cor:balls_inclusion}
For sufficiently large $n$ and $p \in \Rcal$,
\begin{equation*}
 \BallPo{p} \cap \Rcal ([K,R]) \subseteq \BallHyp{p} \cap \Rcal ([K,R]),
\end{equation*}
where $K$ is the constant from Lemma~\ref{lem:asymptotics_Omega_hyperbolic}.
\end{corollary}

Furthermore, Lemma~\ref{lem:asymptotics_Omega_hyperbolic} enables us to determine the measure of a ball around a given point $p=(0,y)$ - this is will be fairly useful in our subsequent analysis. 

Let $p \in \Rcal$. Then we can see that the curve $x^\prime = e^{\frac{1}{2} (y + y^\prime)}$ with $x^\prime \geq 0$ meets the right boundary of $\Rcal$, that is, the line $x^\prime = \frac{\pi}{2} e^{R/2}$ at $y^\prime = R - y + 2\ln \frac{\pi}{2}$. Hence, any point $p^{\prime} \in \Rcal ([R - y + 2\ln \frac{\pi}{2}, R])$ is included in $\BallPo{p}$. In other words,
\begin{equation*} \label{eq:P_ball_inclusion_lower}
\BallPo{p} \cap \Rcal ([R - y +2\ln \frac{\pi}{2},R]) = \Rcal ([R - y + 2\ln \frac{\pi}{2},R]).
\end{equation*}
This together with the fact that for any $u^\prime = (r^\prime, \theta^\prime)$,
\[
	r^\prime < y = R -r \Rightarrow d_\H(\Psi^{-1}(p),u^\prime) \le R
\]
implies that 
\begin{equation}\label{eq:symm_diff_upper_P} 
(\BallHyp{p} \bigtriangleup \BallPo{p})  \cap \Rcal ([R - y + 2 \ln \frac{\pi}{2},R]) = \emptyset,
\end{equation}
where $A \bigtriangleup B$ denotes the symmetric difference of the sets $A$ and $B$. We can now compute the expected number of points in $\BallHyp{p} \bigtriangleup \BallPo{p}$, i.e. those vertices that are neighbors of $p$ in only one of the two models.

\begin{lemma}\label{lem:sym_diff_measure_H_P}
Let $0 \le y_n <R$ be such that $R - y_n \to \infty$ and write $p_n = (x_n, y_n)$. Then we have, as $n \to \infty$,
\[
	\mu (\BallHyp{p_n} \bigtriangleup \BallPo{p_n} ) 
	= \Theta (1) \cdot \begin{cases} 
		e^{(1/2-\alpha)R + \alpha y_n}, & \mbox{if } \alpha < 3/2 \\
		(R-y_n) e^{3y_n/2 - R}, & \mbox{if }\alpha = 3/2\\
		e^{3y_n/2 - R}, &  \mbox{if } \alpha > 3/2 
	\end{cases}.
\]
\end{lemma}

\begin{proof}
Let $r_n := R - y_n$. Lemma~\ref{lem:asymptotics_Omega_hyperbolic} implies that for such a $p_n$, if a point $p$ belongs to $\BallHyp{p_n} \bigtriangleup \BallPo{p_n} \cap \Rcal ([0,r_n])$ then 
\[
	|x_n - x| = \Theta(1) \cdot e^{\frac{3}{2} (y_n + y) - R}.
\]
Now, if $p \in [r_n, r_n + 2 \ln \frac{\pi}{2})]$ and also $p \in \BallHyp{p_n} \bigtriangleup \BallPo{p_n}$, then 
\[
	|x_n-x|_n =\frac{\pi}{2} e^{R/2} - e^{\frac{1}{2} (y_n + y)}.
\]
Finally,~\eqref{eq:symm_diff_upper_P} implies that no point in $\Rcal ([r_n+2\ln \frac{\pi}{2},R])$ belongs to $\BallHyp{p_n} \bigtriangleup \BallPo{p_n}$. We first compute the expected number of points in $p \in \BallHyp{p_n} \bigtriangleup \BallPo{p_n}$ that have $R - y \le r_n$. The result depends on the value of $\alpha$, yielding the following three cases
\begin{align*}
	\mu (\BallHyp{p_n} \bigtriangleup \BallPo{p_n} \cap \Rcal ([0,r_n])) 
	&= \Theta(1) \cdot e^{3y_n/2 - R}\int_0^{r_n} e^{(3/2-\alpha) y} \, dy \\
	&= \Theta(1)\cdot \begin{cases} e^{(1/2-\alpha)R + \alpha y_n}, & \mbox{if } \alpha < 3/2 \\
		(R - y_n) e^{3y_n/2 - R}, & \mbox{if }\alpha = 3/2\\
		e^{3y_n/2 - R}, &  \mbox{if } \alpha > 3/2
	\end{cases}.
\end{align*}
Next we compute the number of remaining points in $\BallHyp{p_n} \bigtriangleup \BallPo{p_n}$, 
\begin{align*}
	\mu (\BallHyp{p_n} \bigtriangleup \BallPo{p_n} \cap \Rcal ([r_n, R])) 
	&= \frac{\nu\alpha}{\pi} \int_{r_n}^{r_n + 2 \ln \frac{\pi}{2}} 
		\left(\frac{\pi}{2} e^{R/2} - e^{\frac{1}{2} (y_n + y)}\right)e^{-\alpha y} \, dy \\
	&= O(1) \cdot e^{R/2} \int_{r_n}^{r_n + 2 \ln \frac{\pi}{2}} e^{-\alpha y} \, dy 
		= O(1) \cdot e^{R/2} e^{-\alpha r_n} \\
	&=O(1) \cdot e^{(1/2 - \alpha) R + \alpha y_n}.
\end{align*}
Now note that for any $\alpha > 3/2$, we have 
\[
	\left( (1/2 - \alpha) R + \alpha y_n\right) - \left(3y_n/2 - R \right)
	= (3/2 -\alpha) (R- y_n) \to -\infty,
\]
by our assumption on $y_n$. For $\alpha = 3/2$, these two quantities are equal. From these observations, we deduce that 
\[
	\mu (\BallHyp{p_n} \bigtriangleup \BallPo{p_n} ) 
	= \Theta (1) \cdot \begin{cases} 
		e^{(1/2-\alpha)R + \alpha y_n}, & \mbox{if } \alpha < 3/2 \\
		r_n e^{3y_n/2 - R}, & \mbox{if }\alpha = 3/2\\
		e^{3y_n/2 - R}, &  \mbox{if } \alpha > 3/2 
	\end{cases}.
\]
\end{proof}

\subsection{Equivalence clustering \texorpdfstring{$\GPo$}{G Po} and \texorpdfstring{$\Gbox$}{G box}}\label{ssec:coupling_HP_ast_P}

Here we prove Proposition~\ref{prop:couling_c_H_P}. We first note that Lemma~\ref{lem:average_degree_P_n} and Lemma~\ref{lem:average_degree_G_box} imply the following
\begin{equation} \label{eq:n_k_Hyp}
	\Exp{N_{\Po}(k_n)} = \bigT{1} n k_n^{-(2\alpha + 1)},
\end{equation}
and
\begin{equation} \label{eq:n_k_Po}
	\Exp{N_{\mathrm{box}}(k_n)} = \bigT{1} n k_n^{-(2\alpha + 1)}.
\end{equation}
Moreover,
\begin{equation}\label{eq:equivalence_N_HP_P}
	\lim_{n \to \infty} \frac{\Exp{N_{\Po}(k_n)}}{\Exp{N_{\mathrm{box}}(k_n)}} = 1.
\end{equation}

Recall that Proposition~\ref{prop:couling_c_H_P} states
\[
	\lim_{n \to \infty} s(k_n)^{-1} \, \Exp{\left|c^\ast(k_n; \GPo) - c^\ast(k_n;\Gbox)\right|} = 0.
\]

Next recall the definition of $\Kcal_{C}(k_n)$
\[
	\Kcal_C(k_n) = \left\{y \in \R_+ : \frac{k_n - C \sqrt{k_n \log(k_n)}}{\xi} \vee 0 \le e^{\frac{y}{2}}
	\le \frac{k_n + C \sqrt{k_n \log(k_n)}}{\xi} \wedge e^{R/2} \right\},
\]
and~\eqref{eq:def_tilde_c_box}
\[
	\widetilde{c}_{\text{box}}(k_n) = \frac{\widetilde{T}_{\text{box}}(k_n,C)}{\binom{k_n}{2}\Exp{N_{\text{box}}(k_n)}},
\]
where $\widetilde{T}_{\text{box}}(k_n,C)$ counts for all nodes $p = (x,y)$ with $y \in \Kcal_{C}(k_n)$ the pairs $(p_1,p_2)$ that form a triangle with $p$, with the exception that it considers $p_2 \in \BallPo{p_1} \cap \Rcal$ instead of $\BallPon{p_1}$. Then using Corollary~\ref{cor:c_ast_box_2_tilde_c_box} we get
\[
	\Exp{\left|c^\ast(k_n; \GPo) - c^\ast(k_n;\Gbox)\right|}
	\le \Exp{\left|c^\ast(k_n; \GPo) - \widetilde{c}_{\text{box}}(k_n)\right|} 
	+ \smallO{s(k_n)},
\]
and hence it is enough to prove that
\[
	\lim_{n \to \infty} s(k_n)^{-1}\Exp{\left|c^\ast(k_n; \GPo) - \widetilde{c}_{\text{box}}(k_n)\right|} = 0.
\]


The following lemma will be frequently used in the proof of Proposition~\ref{prop:couling_c_H_P}.

\begin{lemma} \label{lem:gamma_approx}
Let $t, r \in \mathbb{R}$ be fixed and let $\hat{\rho}(y,k)$ be any of the three probability functions $\rho_{\Po}(y,k), \rho_{\text{box}}(y,k)$ or $\rho(y,k)$. Then for any sequence $k_n$ of non-negative integers with $k_n = \bigO{n^{\frac{1}{2\alpha + 1}}}$ and $C > 0$ large enough,
\[
	\int_{\Kcal_{C}} e^{t y} \hat{\rho}_n(y,k_n-r) e^{-\alpha y} \dd y = \bigO{1} k_n^{-2\alpha - 1 + 2t}
\]
as $n \to \infty$.
\end{lemma}
\begin{proof}
Note that on $\Kcal_{C}(k_n)$ we have that $e^{ty} = \bigT{k_n^{2t}}$. Hence, by Lemma~\ref{lem:degree_integral}
\begin{align*}
	\int_{\Kcal_{C}} e^{t y} \hat{\rho}_n(y,k_n-r) e^{-\alpha y} \dd y
	&= \bigT{k_n^{2t}} \int_{\Kcal_{C}} \hat{\rho}_n(y,k_n-r) e^{-\alpha y} \dd y\\
	&= \bigO{k_n^{2t}} (k_n-r)^{-(2\alpha + 1)} = \bigO{1} k_n^{-2\alpha - 1 + 2t}.
\end{align*}
\end{proof}

\begin{proofof}{Proposition~\ref{prop:couling_c_H_P}}

To keep notation concise we abbreviate $\Exp{N_{\Po}(k_n)}$ and $\Exp{\Nbox(k_n)}$ by $\expH$ and $\expP$, respectively. We will also suppress the subscript $n$ in most expressions regarding the graphs $\GPo$ and $\Gbox$. Finally we will write 
\[
	T_{\Po}(p) = \sum_{(p_1, p_2) \in \Pcal \setminus \{p\}, \atop \text{distinct}} T_{\Po}(p,p_1,p_2),
\] 
with
\[
	T_{\Po}(p,p_1,p_2) = \ind{p_1 \in \BallHyp{p}} \ind{p_2 \in \BallHyp{p}} \ind{p_2 \in \BallHyp{p_1}}
\]
to denote the triangle count function for $p$ in $\GPo$. Then we have
\begin{align*} 
	&\Exp{\left|  c^\ast(k_n; \GPo) - \widetilde{c}_{\text{box}}(k_n)\right|}= 
	\binom{k_n}{2}^{-1}\Exp{\left|\sum_{p \in \Pcal} 
    	\frac{\ind{\mathrm{deg}_{\Po}(p) = k_n}}{\expH} T_{\Po}(p)
        - \frac{\ind{\mathrm{deg}_{\mathrm{box}}(p) = k_n}}{\expP}  \widetilde{T}_{\mathrm{box}}(p)\right|} \\
    &\le\binom{k_n}{2}^{-1} \expH^{-1} \Exp{\left|\sum_{p \in \Pcal} \ind{\mathrm{deg}_{\Po}(y) = k_n} T_{\Po}(p) 
    	- \ind{\mathrm{deg}_{\mathrm{box}}(p) = k_n} \widetilde{T}_{\text{box}}(p)\right|} \\
    &\hspace{10pt}+ \binom{k_n}{2}^{-1} \left|\frac{1}{\expH} - \frac{1}{\expP}\right|\Exp{
        	\sum_{p \in \Pcal} \ind{\mathrm{deg}_{\mathrm{box}}(p) = k_n} \widetilde{T}_{\text{box}}(p) }
\end{align*}
The last term can be rewritten as
\begin{align*}
	\left|1 - \frac{\expH}{\expP}\right| \Exp{\widetilde{c}_{\text{box}}(k_n)} = \left|1 - \frac{\expH}{\expP}\right|\gamma(k_n)(1+o(1)),
\end{align*}
where we used Proposition~\ref{prop:convergence_average_clustering_P_n} (See Section~\ref{sec:clustering_Pn_to_P}). The first term in this product converges to zero by~\eqref{eq:equivalence_N_HP_P} while the second term scales as $s(k_n)$. Hence
\[
	\left|1 - \frac{\expH}{\expP}\right| \Exp{\widetilde{c}_{\text{box}}(k_n)} = \smallO{s(k_n)},
\]
and therefore we are left to analyze the other term. By the Campbell-Mecke formula we have that
\begin{align*}
	    &\Exp{\left|\sum_{p \in \Pcal} \ind{\mathrm{deg}_{\Po}(p) = k_n} T_{\Po}(p) 
	        	- \ind{\mathrm{deg}_{\mathrm{box}}(p) = k_n} \widetilde{T}_{\text{box}}(p)\right|} \\
	    &= \int_{\Rcal} 
	        \Exp{\left|\ind{\mathrm{deg}_{\Po}(y) = k_n} T_{\Po}(y)
	        - \ind{\mathrm{deg}_{\mathrm{box}}(y) = k_n} \widetilde{T}_{\text{box}}(y)\right|} 
	        	f(x,y) \dd y \dd x.
\end{align*}
Since 
\begin{align*}
	\Exp{\frac{\ind{\mathrm{deg}_{\Po}(y) = k_n}}{\expH} T_{\Po}(y)}
	&\le \binom{k_n}{2} \rho_{\Po}(y,k_n)\expH^{-1} \\
	&= \binom{k_n}{2} \rho_{\Po}(y,k_n)\bigT{\expP^{-1}}\\
	&= \bigT{n^{-1} k_n^{2\alpha + 3}}\rho_{\Po}(y,k_n)
\end{align*}
and similar for the other term, it follows that
\begin{align*}
	&\hspace{-60pt}\Exp{ \left| \frac{\ind{\mathrm{deg}_{\Po}(y) = k_n}}{\expH} T_{\Po}(y)
		- \frac{\ind{\mathrm{deg}_{\mathrm{box}}(y) = k_n}}{\expH}  \widetilde{T}_{\text{box}}(y)\right|} \\
	&\le \bigT{n^{-1} k_n^{2\alpha + 3}}\left(\rho_{\Po}(y,k_n) + \rho_{\text{box}}(y,k_n)\right).
\end{align*}
Therefore, by a concentration of heights argument (c.f. Proposition~\ref{prop:concentration_height_general}), it is enough to consider the integral
\begin{equation} \label{eq:expectation_total}
	n \int_{\Kcal_{C}(k_n)} \Exp{\left|\ind{\mathrm{deg}_{\Po}(y) = k_n} T_{\Po}(y)
		- \ind{\mathrm{deg}_{\mathrm{box}}(y) = k_n} \widetilde{T}_{\mathrm{box}}(y)\right|} e^{-\alpha y} \dd y,
\end{equation}
where we also used that $f(x,y)$ is simply a constant multiple of the function $e^{-\alpha y}$. Since $\binom{k_n}{2}\expH
= \bigT{n k_n^{-(2\alpha - 1)}}$ we have to show that
\[
	\lim_{n \to \infty} k_n^{2\alpha - 1} s(k_n)^{-1} \int_{\Kcal_{C}(k_n)} 
		\Exp{\left|\ind{\mathrm{deg}_{\Po}(y) = k_n} T_{\Po}(y)
			- \ind{\mathrm{deg}_{\infty}(y) = k_n} \widetilde{T}_{\mathrm{box}}(y)\right|} e^{-\alpha y} \dd y = 0.
\] 

For $\alpha > 3/4$, $s_{3/4}(k_n) = \log(k_n)^{-1} s_\alpha(k_n) = \smallO{s_\alpha(k_n)}$ and thus it suffices to prove the following two cases:
\begin{enumerate}
\item if $1/2 < \alpha \leq 3/4$, then
\[
	\lim_{n\to \infty} k_n^{6\alpha - 3} \, \int_{\Kcal_{C}(k_n)} 
		\Exp{\left|\ind{\mathrm{deg}_{\Po}(y) = k_n} T_{\Po}(y)
		- \ind{\mathrm{deg}_{\mathrm{box}}(y) = k_n} \widetilde{T}_{\mathrm{box}}(y)\right|} e^{-\alpha y} \dd y = 0,
\]
\item if $3/4 < \alpha$, then
\[ 
	\lim_{n\to \infty} k_n^{2\alpha} \, \int_{\Kcal_{C}(k_n)} 
		\Exp{\left|\ind{\mathrm{deg}_{\Po}(y) = k_n} T_{\Po}(y)
		- \ind{\mathrm{deg}_{\mathrm{box}}(y) = k_n} \widetilde{T}_{\mathrm{box}}(y)\right|} e^{-\alpha y} \dd y = 0.
\]
\end{enumerate}
We shall proceed by expanding the integrand and analyzing the individual terms. With a slight abuse of notation we shall write $y$ instead of $(0,y)$ in an expression such as $\BallHyp{y}$. In addition we write $D_{\Po}(y,k_n;\Pcal)$ for the indicator which is equal to 1 if and only if $\BallHyp{y}$ contains $k_n$ points from $\Pcal \setminus \{(0,y)\}$. We define $D_{\text{box}}(y,k_n;\Pcal)$ analogously for the ball $\BallPon{y}$. It is important to note that for any $p^\prime \in \Rcal$ it holds that $p^\prime \in \BallPon{y} \iff p^\prime \in \BallPo{y}$.

We need to split the integrand over several terms and then analyze each of these separately. Applying the Campbell-Mecke formula yields
\begin{align*} 
 &\hspace{-30pt}\Exp{ \left| \ind{\mathrm{deg}_{\Po}(y) = k_n} P_{\Po}(y)
        - \ind{\mathrm{deg}_{\infty}(y) = k_n}  \widetilde{T}_{\mathrm{box}}(y)
        \right|}\leq \\
 & {\mathbb E} \left[ \sum_{(p_1,p_2) \in \Pcal \setminus \{(0,y)\}, \atop \text{distinct}}
  \left| D_{\Po}(y,k_n-2; \Pcal \setminus \{p_1,p_2 \}) T_{\Po}(y,p_1,p_2) \right. \right.\\
  & \hspace{5cm} 
\left. \left. - D_{\text{box}} (y,k_n-2;\Pcal \setminus \{p_1,p_2\}) \widetilde{T}_{\text{box}}(y,p_1,p_2)
   \right| \vphantom{\sum_{p_1,p_2 \in \Pcal \setminus \{(0,y)\}, \atop \text{distinct}}}\right],
\end{align*}
where the sum ranges over all distinct pairs of points in $\Pcal \setminus \{ (0,y)\}$. In what follows, we will set $\BallSym{p'} = \BallHyp{p'} \bigtriangleup (\BallPo{p'} \cap \Rcal)$ and $\BallInter{p'} =\BallHyp{p'} \cap \BallPon{p'}$ and observe that $\BallInter{y} = \BallHyp{y} \cap \BallPo{y}$. We will now bound the sum that is inside the expectation.
We will split the sum into different parts, depending on combinations of $p_1, p_2\in \Pcal \setminus \{(0,y)\}$ for which only one of the two terms of the difference is non-zero. Clearly, for this we need that either $p_1 \in \BallInter{y}$ and $p_2 \in \BallSym{p_1}$ or $p_1 \in \BallSym{y}$ and $p_2 \in \BallInter{p_1}$. We will consider the following four cases:
\begin{enumerate}
\item$p_1 \in \BallInter{y}$ and $p_2 \in \BallSym{p_1}$
\begin{enumerate}
\item $y_1,y_2 < (1-\eps ) R \wedge (R-y)$
\item $y_1 \ge (1-\eps ) R \wedge (R-y)$
\end{enumerate}
\item$p_1 \in \BallHyp{y} \setminus \BallPo{y}$ with $y_1 < K$ and $p_2 \in \BallInter{y}$.
\item $p_1 \in \BallSym{y}$ with $y_1 \ge K$ and $p_2 \in \BallInter{y}$,
\end{enumerate}
where $K$ in the last two cases is the constant from Lemma~\ref{lem:asymptotics_Omega_hyperbolic}.

Observe that when $y_1 < (1-\eps ) R \wedge (R-y)$ and $y_2 \ge (1-\eps ) R \wedge (R-y)$ it follows from Corollary~\ref{cor:balls_inclusion} that $p_2 \in \BallInter{p_1}$ and thus we do not have to consider this case when $p_1 \in \BallInter{y}$ and $p_2 \in \BallSym{p_1}$. Similarly, when $y_1 \ge K$ and $p_1 \in \BallSym{y}$ Corollary~\ref{cor:balls_inclusion} implies that $p_1 \in \BallHyp{y} \setminus \BallPo{y}$ which explains the setting of case 2.

We can now bound the sum by the following expression:
\begin{align} 
	&\sum_{(p_1,p_2) \in \Pcal \setminus \{(0,y)\}, \atop \text{distinct}}
	  \left| D_{\Po}(y,k_n-2; \Pcal \setminus \{p_1,p_2 \}) T_{\Po} (y,p_1,p_2) \right. \notag \\
	& \hspace{5cm} \left. - D_{\text{box}} (y,k_n-2;\Pcal \setminus \{p_1,p_2\}) \widetilde{T}_{\text{box}} (y,p_1,p_2)
	   \right|  \notag \\
	&\le \hspace{-15pt} \sum_{{p_1,p_2\in \Pcal \setminus \{(0,y)\} \atop  y_1, y_2 < (1-\eps)R \wedge (R-y),} \atop \text{distinct}}
		\hspace{-10pt} \ind{p_1 \in \BallInter{y}} \cdot \ind{p_2 \in \BallSym{p_1}} 
		\, D_{\Po} (y,k_n-2;\Pcal \setminus \{ p_1,p_2\}) \label{eq:sum_1}\\
	&\hspace{10pt}+ \hspace{-15pt} \sum_{{p_1,p_2\in \Pcal \setminus \{(0,y)\} \atop  y_1, y_2 < (1-\eps)R \wedge (R-y),} \atop \text{distinct}}
		\hspace{-10pt} \ind{p_1 \in \BallInter{y}} \cdot \ind{p_2 \in \BallSym{p_1}} 
		\, D_{\text{box}} (y,k_n-2;\Pcal \setminus \{ p_1,p_2\}) \label{eq:sum_2}\\
	&\hspace{10pt}+  \hspace{-15pt} \sum_{{p_1,p_2 \ \in \Pcal \setminus \{(0,y) \} 
		\atop y_1 \geq (1-\eps) R \wedge (R-y),} \atop \text{distinct}} \hspace{-10pt}
		\ind{p_1 \in \BallInter{y}} \cdot \ind{p_2 \in \BallSym{p_1}\cap \BallHyp{y}} 
		\, D_{\Po} (y,k_n-2;\Pcal \setminus \{ p_1,p_2\}) \label{eq:sum_3} \\
	&\hspace{10pt}+ \hspace{-15pt} \sum_{{p_1,p_2 \ \in \Pcal \setminus \{(0,y) \} 
		\atop y_1 \geq (1-\eps) R \wedge (R-y),} \atop \text{distinct}} \hspace{-10pt}
		\ind{p_1 \in \BallInter{y}} \cdot \ind{p_2 \in \BallSym{p_1}\cap \BallHyp{y}} 
		\, D_{\text{box}} (y,k_n-2;\Pcal \setminus \{ p_1,p_2\}) \label{eq:sum_4} \\
	&\hspace{10pt}+ \hspace{-10pt} \sum_{{p_1,p_2 \in\Pcal \setminus \{(0,y)\} 
		\atop y(p_1) \geq K,} \atop \text{distinct}} \hspace{-10pt} \ind{p_1\in \BallHyp{y}\setminus \BallPo{y}} \ind{p_2\in \BallHyp{y}\cap \BallPo{y}} 
		\, D_{\Po} (y,k_n-2;\Pcal \setminus \{ p_1,p_2\}) \label{eq:sum_5}\\
	&\hspace{10pt}+ \hspace{-10pt} \sum_{{p_1,p_2 \in\Pcal \setminus \{(0,y)\} 
			\atop y(p_1) \geq K,} \atop \text{distinct}} \hspace{-10pt} \ind{p_1\in \BallHyp{y}\setminus \BallPo{y}} \ind{p_2\in \BallHyp{y}\cap \BallPo{y}} 
			\, D_{\text{box}} (y,k_n-2;\Pcal \setminus \{ p_1,p_2\}) \label{eq:sum_5b}\\
	&\hspace{10pt}+ \hspace{-10pt} \sum_{{p_1,p_2 \in\Pcal \setminus \{(0,y)\} \atop \ y(p_1) < K,} \atop \text{distinct}}
		\ind{p_1\in \BallSym{y}} \ind{p_2\in \BallHyp{y}\cap \BallPo{y}}. \label{eq:sum_6}
\end{align}
In the following paragraphs we will give upper bounds on the expected values of each one of these partial sums. 

\paragraph{The sums~\eqref{eq:sum_1} and~\eqref{eq:sum_2}}

We will analyze~\eqref{eq:sum_1}. The analysis of the other sum~\eqref{eq:sum_2} is similar.
Note first that for any two points $p_1,p_2$ the following holds: $p_1 \in \BallHyp{y}$ and $p_2 \in \BallSym{p_1}\cap \BallHyp{y}$, then $p_2 \in \BallHyp{y}$ and $p_1 \in \BallSym{p_2}\cap \BallHyp{y}$.
Using this symmetry, it suffices to consider distinct pairs $(p_1,p_2) \in \Pcal \setminus \{(0,y)\}$ with $0\leq y_2 \leq y_1 \leq R- y$. Let $\dom{}$ denote the set of these pairs. 

We are going to consider several sub-cases and, thereby, split the domain $\dom{}$ into the corresponding sub-domains. 
Let $\omega =\omega (n) \to \infty$ as $n\to \infty$ be a slowly growing function and set $y_\omega := y +\omega$. 
We let 
\begin{align*}
	\dom{1} &= \{(p_1,p_2) \in \dom{} \cap \Pcal \ : \ y \leq y_1 \leq R/2, \, y_\omega \leq y_2 \leq y_1 \},\\
	\dom{2} &= \{(p_1,p_2) \in \dom{} \cap \Pcal \ : \ y_1 \leq R/2, \, y_2 \leq y_\omega \} \text{ and }\\
	\dom{3} &=  \{(p_1,p_2) \in \dom{} \cap \Pcal \ : R/2 < y_1 \leq R -y, \, y_2 \leq y_1 \}.
\end{align*} 
Note that $\dom{} \subseteq \dom{1} \cup \dom{2}\cup \dom{3}$.
Hence, we can write 
\begin{equation} \label{eq:1sum-rewriting}
\begin{split} 
	&\Exp{\sum_{p_1, p_2\in \Pcal \setminus \{(0,y)\} 
	\atop y_1, y_2 \leq (1-\eps) R\wedge (R-y)} \hspace{-10pt} \ind{p_1 \in \BallHyp{y}} 
	\ind{p_2 \in \BallSym{p_1}\cap \BallHyp{y}} 
	\, D_{\Po} (y,k_n-2;\Pcal \setminus \{p_1,p_2\})} \\ 
	&\le \sum_{i=1}^3 \Exp{\sum_{(p_1, p_2)\in \dom{i}} \ind{p_1 \in \BallHyp{y}} 
	\ind{p_2 \in \BallSym{p_1}\cap \BallHyp{y}} \cdot D_{\Po} (y,k_n-2;\Pcal \setminus \{p_1,p_2\})}.
\end{split}
\end{equation}
We bound each one of the above three summands as follows:  
\begin{equation} \label{eq:term1}
\begin{split}
	&\Exp{\sum_{(p_1, p_2)\in \dom{1}} \ind{p_1 \in \BallHyp{y}} \cdot \ind{p_2 \in \BallSym{p_1}\cap \BallHyp{y}} 
		\, D_{\Po} (y,k_n-2;\Pcal \setminus \{p_1,p_2\})} \\
	&\le \Exp{\sum_{(p_1, p_2)\in \dom{1}} \ind{p_1 \in \BallHyp{y}} \cdot \ind{p_2 \in  \BallHyp{y}} 
		\, D_{\Po} (y,k_n-2;\Pcal \setminus \{p_1,p_2\})} := \mathcal{I}_n^{(1)}(y),
\end{split}
\end{equation}

\begin{equation} \label{eq:term2}
\begin{split}
	&\Exp{\sum_{(p_1, p_2)\in \dom{2}} \ind{p_1 \in \BallHyp{y}} \cdot \ind{p_2 \in \BallSym{p_1}\cap \BallHyp{y}} 
		\, D_{\Po} (y,k_n-2;\Pcal \setminus \{p_1,p_2\})} \\
	&\le \Exp{\sum_{(p_1, p_2)\in \dom{2}} \ind{p_1 \in \BallHyp{y}} \cdot \ind{p_2 \in \BallSym{p_1}} 
		\, D_{\Po} (y,k_n-2;\Pcal \setminus \{p_1,p_2\})} := \mathcal{I}_n^{(2)}(y)
\end{split}
\end{equation}
and 
\begin{equation}\label{eq:term3}
\begin{split}
	&\Exp{\sum_{(p_1, p_2)\in \dom{3}} \ind{p_1 \in \BallHyp{y}} \cdot \ind{p_2 \in \BallSym{p_1}\cap \BallHyp{y}} 
		\, D_{\Po} (y,k_n-2;\Pcal \setminus \{p_1,p_2\})} \\
	&\le \Exp{\sum_{(p_1, p_2)\in \dom{3}} \ind{p_1 \in \BallHyp{y}} \cdot \ind{p_2 \in \BallHyp{y}} 
		\, D_{\Po} (y,k_n-2;\Pcal \setminus \{p_1,p_2\})} := \mathcal{I}_n^{(3)}(y).
\end{split}
\end{equation}
We will bound each term using the Campbell-Mecke formula and show for $i = 1,2,3$ that for $1/2 < \alpha < 3/4$
\begin{equation}
	\lim_{n \to \infty} k_n^{6\alpha -3} 
		\int_{\Kcal_{C}(k_n)} \mathcal{I}_n^{(i)}(y) e^{-\alpha} \dd y = 0,
\end{equation}
and for $\alpha \ge 3/4$
\begin{equation}
	\lim_{n \to \infty} k_n^{2\alpha}
		\int_{\Kcal_{C}(k_n)} \mathcal{I}_n^{(i)}(y) e^{-\alpha} \dd y = 0.
\end{equation}

For the first term~\eqref{eq:term1}, we note that
\[ 
\Exp{D_{\Po}(y, k_n-2;\Pcal \setminus \{p_1, p_2\}} =\rho_{\Po}(y,k_n-2).
\] 
and hence $\mathcal{I}_n^{(1)}(y)$ becomes
\begin{equation} \label{eq:1sum-expansion} 
\begin{split}
	\rho_{\Po}(y,k_n-2) \int_{-I_n}^{I_n} \int_y^{R/2}\int_{-I_n}^{I_n} \int_{y_\omega}^{y_1}
  	\ind{p_1 \in \BallInter{y}} \, \ind{p_2 \in \BallHyp{y}} 
  	e^{-\alpha (y_1 + y_2)} \dd y_2 \dd x_2 \dd y_1 \dd x_1.
\end{split}
\end{equation}

Next, Lemma~\ref{lem:asymptotics_Omega_hyperbolic} implies that for $y'\leq R -y$, we have that if $(x',y') \in \BallHyp{y}$, then $|x'| < (1+ K) e^{y/2 + y'/2}$, where $K >0$ is as in Lemma~\ref{lem:asymptotics_Omega_hyperbolic}. Using these observations, we obtain: 
\begin{align*}
	&\hspace{-30pt} \Exp{\sum_{p_1, p_2\in \dom{1}} \ind{p_1 \in \BallInter{(0,y)}} \cdot \ind{p_2 \in \BallHyp{y}} \cdot 
	D_{\Po} (y,k_n-2;\Pcal \setminus \{p_1, p_2\})} \\
	&= \rho_{\Po}(y,k_n-2)
	e^{y}\int_y^{R/2} e^{y_1/2}\int_{y_\omega}^{y_1} e^{y_2/2} e^{-\alpha y_2} \cdot e^{-\alpha y_1} dy_2 dy_1.
\end{align*}

Now, the double integral becomes
\begin{equation}
\begin{split}
& \int_y^{R/2} e^{y_1/2}\int_{y_\omega}^{y_1} e^{y_2/2} e^{-\alpha y_2} \cdot e^{-\alpha y_1} dy_2 dy_1 = \\
&  O(1) \cdot  \int_y^{R/2} e^{y_1/2 - \alpha y_1} \cdot 
e^{(1/2 - \alpha) y_\omega} dy_1 \\
& =O(1) \cdot e^{(1/2 - \alpha) y_\omega} \cdot \int_y^{R/2} e^{y_1/2 - \alpha y_1} d y_1 \\
& =O(1) \cdot e^{(1/2 - \alpha) y_\omega + (1/2 - \alpha) y} \\ 
& \ll e^{(1 - 2\alpha) y},
\end{split}
\end{equation}
since $y_\omega = y + \omega$ and $\omega \to \infty$. 
We then deduce that 
\begin{equation}
\begin{split}
&\hspace{-50pt}\Exp{\sum_{p_1, p_2\in \dom{1}} \ind{p_1 \in \BallInter{(0,y)}} \cdot \ind{p_2 \in \BallHyp{y}} \cdot 
D_{\Po} (y,k_n-2;\Pcal \setminus \{p_1, p_2\})}\\
& \ll \rho_{\Po}(y,k_n-2) e^{(1-2\alpha)y}.
\end{split}
\end{equation}
We now integrate this with respect to $y$ and determine its contribution to~\eqref{eq:expectation_total};
\begin{align*} 
	& \int_{\Kcal_{C}(k_n)} \rho_{\Po}(y,k_n-2) e^{(1-2\alpha)y} e^{-\alpha y} \dd y \dd x \\
	&= \bigO{k_n^{-6\alpha + 1}}
\end{align*}
where we used Lemma~\ref{lem:gamma_approx} with $t = 1 - 2\alpha$.

Since $1 - 6\alpha + \min\{6\alpha - 3, 2\alpha\} < 0$ for all $\alpha > 1/2$ we deduce that for $1/2 < \alpha < 3/4$
\[
	\lim_{n \to \infty} k_n^{6\alpha - 3} 
	\int_{\Kcal_{C}(k_n)} \mathcal{I}_n^{(1)}(y) e^{-\alpha y} \dd y = 0,
\]
while for $\alpha \ge 3/4$
\[
	\lim_{n \to \infty} k_n^{2\alpha}  
		\int_{\Kcal_{C}(k_n)} \mathcal{I}_n^{(1)}(y) e^{-\alpha y} \dd y = 0.
\]

We will now bound the term in~\eqref{eq:term2}. Using similar observations as for the previous term we get that $\mathcal{I}_n^{(2)}(y)$ equals
\[
	 \rho_{\Po}(y,k_n-2) \int_{-I_n}^{I_n} \int_0^{R/2} \hspace{-2pt} \int_{-I_n}^{I_n} \int_0^{y_\omega}
	 \ind{p_1 \in \BallHyp{y}} \, \ind{p_2 \in \BallSym{(0,y)}} e^{-\alpha(y_1 + y_2)} \dd y_2 
	 \dd x_2 \dd y_1 \dd x_1.
\]

Now, Lemma~\ref{lem:asymptotics_Omega_hyperbolic} implies that for 
$y_2 \leq R -y_1$, we have that if 
$(x_2,y_2) \in \BallSym{(x_1,y_1)}$, then $x_2$ lies in an interval of length 
$Ke^{3 y/2 + 3y'/2 - R}$, where $K >0$ is again the constant in Lemma~\ref{lem:asymptotics_Omega_hyperbolic}. 
Using these observations we obtain: 
\begin{equation} \label{eq:term2_intermediate}
	\mathcal{I}_n^{(2)}(y) = \rho_{\Po}(y,k_n-2) e^{y/2}\int_0^{R/2} e^{y_1/2 + 3y_1/2}
	\int_0^{y_\omega} e^{3y_2/2 - R} e^{-\alpha y_2} \cdot e^{-\alpha y_1} \dd y_2 \dd y_1. 
\end{equation} 
The integrals satisfy
\begin{align*}
	&\hspace{-30pt} e^{-R}  \left(\int_0^{R/2} e^{(2-\alpha )y_1} \dd y_1\right) 
		\left( \int_0^{y_\omega} e^{(3/2 - \alpha) y_2} \dd y_2 \right)\\
	&= \bigO{1} e^{-R} \left(\begin{cases}
		e^{(1-\alpha/2)R} &\mbox{if } \frac{1}{2} < \alpha < 2\\
		R &\mbox{if } \alpha \ge 2
	\end{cases}\right)
	\left(\begin{cases}
		e^{(3/2 - \alpha)y_\omega} &\mbox{if } \frac{1}{2} < \alpha < \frac{3}{2}\\
		y &\mbox{if } \alpha \ge \frac{3}{2}
	\end{cases}\right)\\
	&= \bigO{1} \begin{cases}
		e^{-\frac{\alpha}{2}R} e^{(3/2 - \alpha)y} &\mbox{if } \frac{1}{2} < \alpha < \frac{3}{2}\\
		(y +\omega(n)) e^{-\frac{\alpha}{2} R} &\mbox{if } \frac{3}{2} \le \alpha < 2\\
		(y + \omega(n)) R e^{-R} &\mbox{if } \alpha \ge 2
	\end{cases}.
\end{align*}
Since $y_\omega := y + \omega(n) \le R = \bigO{\log(n)}$ we conclude that on $\Kcal_{C}(k_n)$
\[
	\mathcal{I}_n^{(2)}(y) = \bigO{1} \rho_{\Po}(y,k_n-2) \begin{cases}
				n^{-\alpha} k_n^{3 - 2\alpha} &\mbox{if } \frac{1}{2} < \alpha < \frac{3}{2}\\
				n^{-\alpha} \log(n) &\mbox{if } \frac{3}{2} \le \alpha < 2\\
				n^{-2} \log(n)^{2} &\mbox{if } \alpha \ge 2
		\end{cases},
\]
and hence
\begin{align*}
	\int_{\Kcal_{C}(k_n)} \mathcal{I}_n^{(2)}(y) e^{-\alpha y} \dd y
	&= \bigO{1} k_n^{-(2\alpha + 1)}\begin{cases}
				n^{-\alpha} k_n^{3 - 2\alpha} &\mbox{if } \frac{1}{2} < \alpha < \frac{3}{2}\\
				n^{-\alpha} \log(n) &\mbox{if } \frac{3}{2} \le \alpha < 2\\
				n^{-2} \log(n)^{2} &\mbox{if } \alpha \ge 2
		\end{cases},	\\
	&= \bigO{1} \begin{cases}
				n^{-\alpha} k_n^{2 - 4\alpha} &\mbox{if } \frac{1}{2} < \alpha < \frac{3}{2}\\
				n^{-\alpha} \log(n)k_n^{-(2\alpha + 1)} &\mbox{if } \frac{3}{2} \le \alpha < 2\\
				n^{-2} \log(n)^{2} k_n^{-(2\alpha + 1)} &\mbox{if } \alpha \ge 2
		\end{cases}.
\end{align*}

Now for $1/2 < \alpha < 3/4$ it holds that $4\alpha^2 - \alpha + 1 > 0$. Hence since $k_n = \bigO{n^{\frac{1}{2\alpha + 1}}}$, we have
\[
	k_n^{6\alpha - 3} n^{-\alpha} k_n^{2 - 4\alpha} = n^{-\alpha} k_n^{2\alpha - 1} = \bigO{n^{-\alpha + \frac{2\alpha - 1}{2\alpha + 1}}} = \bigO{k_n^{-\frac{4\alpha^2 - \alpha + 1}{2\alpha + 1}}} = \smallO{1},
\]
from which we deduce that
\[
	\lim_{n \to \infty} k_n^{6\alpha - 3} \int_{\Kcal_{C}(k_n)} \mathcal{I}_n^{(2)}(y)
	e^{-\alpha y} \dd y = 0.
\]
For $\alpha \ge 3/4$ we have that both $n^{-\alpha} \log(n) k_n^{-1}$ and $n^{-2} \log(n)^2 k_n^{-1}$ converge to zero as $n \to \infty$ and hence in this case
\[
	\lim_{n \to \infty} k_n^{2\alpha} \int_{\Kcal_{C}(k_n)} \mathcal{I}_n^{(2)}(y)
	e^{-\alpha y} \dd y = 0.
\]

We will now consider the term in~\eqref{eq:term3}. 
Recall that $\dom{3}$ consists of all pairs $(p_1, p_2 ) \in \dom{}$ such that $R/2 < y_1 \leq (1-\eps)R \wedge (R-y)$ and $y_1 \leq y_\omega$ with the property that 
$p_1 \in \BallHyp{y}$ and $p_2 \in \BallSym{p_1} \cap \BallHyp{y}$.    
So, in particular, $p_2 \in (\BallHyp{p_1} \cup \BallPo{p_1}) \cap \BallHyp{y}$.

We will consider this intersection more closely. We use Lemma~\ref{lem:asymptotics_Omega_hyperbolic} to define a ball around $p_1$ that contains both 
$\BallHyp{p_1}$ and $\BallPo{p_1}$.
For $K > 0$, we define, for any point $p_1=(x_1,y_1) \in \R \times \R_+$,
\begin{equation}\label{eq:def_fatball}
	\FatBallHyp{p_1} : = \{ (x',y') \ : \  y' < R - y_1, \ | x_1 - x'| < (1+K) e^{\frac{1}{2} (y_1 + y')}  \}.
\end{equation}
It is an implication of Lemma~\ref{lem:asymptotics_Omega_hyperbolic}  that 
\begin{equation*} 
(\BallHyp{p_1} \cup \BallPo{p_1}) \cap \Rcal([0, R - y_1]) \subseteq \FatBallHyp{p_1}
\end{equation*}
Therefore, any point $p_2 = (x_2,y_2) \in \BallSym{p_1} \cap \BallHyp{y}$ with 
$y_2 \leq R-y_1$ must belong to $\FatBallHyp{p_1} \cap \FatBallHyp{y}$.

We will use this in order to derive a lower bound on $y_2$ as a function of $x_1, y_1$. 
Let us suppose without loss of generality that $x_1 < 0$. 
The left boundary of $\FatBallHyp{(0,y)}$ is given by the equation 
$x^\prime = (1-K)e^{\frac{1}{2} (y + y^\prime)}$ whereas the right boundary of $\FatBallHyp{p_1}$ is given by the curve having equation $x^\prime = x_1 + (1+ K)e^{\frac{1}{2} (y_1 + y^\prime)}.$
The equation that determines the intersection point $(\hat{x},\hat{y})$ of these curves  is
\[
	x_1 + (1+K)e^{(y_1 + \hat{y})/2}= (1-K) e^{(y + \hat{y})/2}.
\]
We can solve the above for $\hat{y}$  
\begin{equation*} 
\begin{split}
|x_1| &=(1+K) e^{\hat{y}/2} \left( e^{y_1/2} + e^{y/2} \right).
\end{split}
\end{equation*}
But $y_1 > R/2$ and since $y \in \Kcal_{C}(k_n)$, it follows that for sufficiently large $n$, $y \le (1+\eps) R /(2\alpha +1)$. So if $\eps$ is small enough depending on $\alpha$, we have 
$$ |x_1| =(1+K) e^{\hat{y}/2} \left( e^{y_1/2} + e^{y/2} \right) = (1+K+o(1))e^{\hat{y}/2 + y_1/2}. $$
Let $c_K^2$ denote the multiplicative term $1+ K+o(1)$, which appears in the above.
The above yields
\begin{equation} \label{eq:to_use_I1}
\hat{y}= \left(2 \log(|x_1|e^{-y_1/2}) - \log c_K \right) \vee 0 := \hat{y}(x_1,y_1). 
\end{equation}
In particular, note that $\hat{y} = 0$ if and only if $|x_1| \leq c_K e^{y_1/2}$.  
Moreover, since $p_1 \in \BallHyp{y}$ and $x_1 \leq R - y$, we also have that 
$|x_1| \leq e^{(y+y_1)/2} (1+o(1))$. This upper bound on $|x_1|$ together with~\eqref{eq:to_use_I1}, imply that for $n$ sufficiently large, we have $\hat{y} \leq y$. This observation will be used below, where 
we integrate over $y_2$, thus ensuring that the integrals are non-zero. 

We conclude that 
\begin{equation*}
	p^\prime \in \FatBallHyp{y}\cap \FatBallHyp{(x_1,y_1)} \Rightarrow y^\prime \ge \hat{y}(x_1,y_1),
\end{equation*}
which implies
\begin{equation} \label{eq:symdiff_loc}
\begin{split} 
 \ind{p_2 \in \BallSym{p_1}\cap \BallHyp{y}} \leq \ind{y_2 \geq \hat{y}(x_1,y_1), p_2 \in \FatBallHyp{(0,y)}}.
\end{split}
\end{equation}
If we integrate this over $x_2, y_2$ we get 
\begin{align*}
\int_{-I_n}^{I_n} \int_{0}^{y_1}  \ind{p_2 \in \BallSym{p_1}\cap \BallHyp{y}}  
e^{-\alpha y_2} dy_2 dx_2
&  \leq 
\int_{-I_n}^{I_n} \int_{0}^{y_1}  \ind{y_2 \geq \hat{y}(x_1,y_1), p_2 \in \FatBallHyp{y}}  
e^{-\alpha y_2} dy_2 dx_2 \\
&\leq (1+K) \cdot e^{y/2} \int_{\hat{y}(x_1,y_1)}^{y_1} e^{y_2/2 - \alpha y_2} dy_2 \\
&=O(1) \cdot e^{y/2 + (1/2 -\alpha) \hat{y}(x_1,y_1)}.
\end{align*}
Note also that 
\[
	\Exp{D_{\Po} (y,k_n-2;\Pcal \setminus \{p_1, p_2\})} = \rho_{\Po}(y,k_n-2),
\]
uniformly over all $(p_1, p_2) \in \dom{3}$. Hence the Campbell-Mecke formula yields that $\mathcal{I}_n^{(3)}(y)$ equals: 
\begin{align*}
 	&O(1) \rho_{\Po}(y,k_n-2) \, e^{y/2} \int_{-I_n}^{I_n} \int_{R/2}^{(R - y) \wedge (1-\eps)R}  
		\ind{p_1 \in \BallHyp{y}} e^{(1/2 -\alpha) \hat{y}(x_1,y_1) - \alpha y_1} dy_1dx_1 \\
	&= O(1) \rho_{\Po}(y,k_n-2) \, e^{y/2} \int_{-I_n}^{I_n} \int_{R/2}^{(R - y) \wedge (1-\eps)R}  
		\ind{p_1 \in \FatBallHyp{y}} e^{(1/2 -\alpha) \hat{y}(x_1,y_1) - \alpha y_1} dy_1dx_1.
\end{align*}
Due to the symmetry of $\FatBallHyp{y}$, the integration over $x_1$ is: 
\[
	O(1) \cdot e^{y/2} \cdot \int_0^{(1+K)e^{y/2 + y_1/2}} e^{\hat{y}(x_1,y_1) (1/2 -\alpha)} dx_1
\]
We will split this integral into two parts according to the value of $\hat{y}(x_1,y_1)$:
\[
\int_0^{(1+K) e^{y/2 + y_1/2}} e^{\hat{y}(x_1,y_1) (1/2 -\alpha)} dx_1 = 
\int_{c_K e^{y_1/2}}^{(1+K)e^{y/2 + y_1/2}} e^{\hat{y}(x_1,y_1) (1/2 -\alpha)} dx_1 + \int_0^{c_K e^{y_1/2}} dx_1.
\]
The first integral becomes: 
\begin{align*}
&\int_{c_K e^{y_1/2}}^{(1+K)e^{y/2 + y_1/2}} e^{\hat{y}(x_1,y_1) (1/2 -\alpha)} dx_1  = 
\int_{c_K e^{y_1/2}}^{(1+K)e^{y/2 + y_1/2}} e^{\hat{y}(x_1,y_1)/2 (1 -2\alpha)} dx_1  \\
&= O(1)\cdot \int_{c_K e^{y_1/2}}^{(1+K)e^{y/2 + y_1/2}} x_1^{1 -2\alpha} 
e^{-\frac{y_1}{2} (1-2\alpha)} dx_1 \\
&= O(1) \cdot e^{-y_1/2 + \alpha y_1} \cdot e^{\frac{(y+y_1)}{2} 2(1-\alpha)} \\
&=O(1) \cdot e^{y_1/2 +y(1-\alpha)}.  
\end{align*}
The second integral trivially gives: 
\[
	\int_0^{c_K e^{y_1/2}} dx_1 = O(1) \cdot e^{y_1/2} = O(1) \cdot e^{y_1/2 +y(1-\alpha)}.
\]
We conclude that 
\[
	e^{y/2} \cdot \int_0^{(1+K)e^{y/2 + y_1/2}} e^{\hat{y}(x_1,y_1) (1/2 -\alpha)} dx_1 = 
	O(1) \cdot e^{y_1/2 +y(3/2-\alpha)}.
\]
Now, we integrate this with respect to $y_1$ and get 
\[
	e^{y(3/2 -\alpha)} \int_{R/2}^{R-y} e^{(1/2-\alpha)y_1} dy_1 =  O(1) \cdot e^{y(3/2 -\alpha)} 
	e^{(1/2 -\alpha) R/2} = O(1) \cdot n^{1/2 -\alpha} \cdot e^{y(3/2 - \alpha)},
\]
from which we deduce
\begin{equation} \label{eq:term3_intermediate}
	\mathcal{I}_n^{(3)}(y) = 
	O(1) \cdot n^{1/2 -\alpha} e^{y(3/2 - \alpha)} \, \rho_{\Po}(y,k_n-2). 
\end{equation}
We now apply Lemma~\ref{lem:gamma_approx} with $t = \frac{3}{2} - \alpha$ and get 
\begin{align*}
	\int_{\Kcal_{C}(k_n)} \mathcal{I}_n^{(3)}(y) e^{-\alpha y} \dd y
	&= \bigO{1} n^{-(\alpha - \frac{1}{2})} \int_{\Kcal_{C}(k_n)} e^{(3/2 - \alpha)y} \rho_{\Po}(y,k_n-2) e^{-\alpha y} 
		\dd y\\
	&= \bigO{n^{-(\alpha - \frac{1}{2})} k_n^{2 -4\alpha}}.
\end{align*}

Since for $\alpha > 1/2$, $k_n = \bigO{n^{\frac{1}{2\alpha +1}}} = \smallO{n^{1/2}}$ we have that $k_n^{6\alpha - 3} k_n^{2 - 4\alpha} n^{-(\alpha - 1/2)} = \smallO{1}$ and hence for $1/2 < \alpha < 3/4$.
\[
	\lim_{n \to \infty} k_n^{6\alpha - 3} \int_{\Kcal_{C}(k_n)} \mathcal{I}_n^{(3)}(y) e^{-\alpha y} \dd x \dd y = 0,
\]
For $\alpha \ge 3/4$ we observe that $2\alpha^2 + 2\alpha - 5/2 > 0$. Hence,
\[
	k_n^{2\alpha} n^{-(\alpha - \frac{1}{2})} k_n^{2 -4\alpha} 
	= \bigO{n^{-(\alpha - 1/2)} n^{\frac{2 - 2\alpha}{2\alpha + 1}}}
 	= \bigO{n^{- \frac{2\alpha^2 + 2\alpha - 5/2}{2\alpha + 1}}} = \smallO{1}.
\]
and we get for $\alpha \ge 3/4$
\[
	\lim_{n \to \infty} k_n^{2\alpha} \int_{\Kcal_{C}(k_n)} \mathcal{I}_n^{(3)}(y) e^{-\alpha y} \dd x \dd y = 0.
\]

\paragraph{The sums~\eqref{eq:sum_3} and~\eqref{eq:sum_4}}
Again, we will only consider~\eqref{eq:sum_3} since the analysis for the other term is similar. Recall that in this case, we consider pairs $(p_1,p_2)$, with $p_1 = (x_1,y_1)$ satisfying 
$y_1 \geq (R - y) \wedge (1-\eps)R$, and $p_1 \in \BallHyp{y}$, $p_2 \in \BallSym{p_1} \cap \BallHyp{y}$. 
We split this into three sub-domains:  i) $y_2 \geq R - y$; ii) $R -y_1 \leq y_2 \leq R -y$ and iii) $y_2 < R - y_1$. Similar to the analysis above we define
\begin{align*}
	\dom{1} &:= \{(p_1, p_2) \ : \  p_1,p_2 \ \in \Pcal \setminus \{(0,y) \}, \ y_1 \geq (1-\eps) R \wedge (R-y), 
		\ R -y \leq y_2 \leq R \}\\
	\dom{2} &:= \{(p_1, p_2) \ : \  p_1,p_2 \ \in \Pcal \setminus \{(0,y) \}, \ y_1 \geq (1-\eps) R \wedge (R-y), 
		\ R -y_1 \leq y_2 \leq R - y \}\\
	\dom{3} &:= \{(p_1, p_2) \ : \  p_1,p_2 \ \in \Pcal \setminus \{(0,y) \}, \ y_1 \geq (1-\eps) R \wedge (R-y), 
		\ y_2 \leq R - y_1 \}
\end{align*}
and write, for $i = 1,2,3$,
\[
	\mathcal{I}_n^{(i)}(y) := \Exp { \sum_{(p_1,p_2)  \in \dom{i}} 
	\ind{p_1 \in \BallHyp{y}} \cdot \ind{p_2 \in \BallSym{p_1} \cap \BallHyp{y}}
	\cdot D_{\Po} (y,k_n-2;\Pcal \setminus \{ p_1, p_2\})}.
\]

In the first case, note that for $y \in \Kcal_{C}(k_n)$ we have, for small enough $\varepsilon$ and sufficiently large $n$, $2y \le 2(1+\eps)
\frac{R}{2\alpha +1} = \smallO{R}$. Thus $y_1 + y_2 \geq 2(R - y) = \Omega(R)$ and thus $p_2 \in \BallHyp{p_1}$ for large enough $n$. Furthermore, 
$y_2 > R - y_1 + 2\ln (\pi/2)$, which implies that $p_2 \in \BallPo{p_1}$ too. 
Hence, the contribution from these pairs is zero.   

The Campbell-Mecke formula yields that: 
\begin{equation*}
\begin{split} 
\mathcal{I}_n^{(1)}(y) 
&=O(1) \int_{-I_n}^{I_n} \int_{(1-\eps) R \wedge (R-y)}^{R} \ind{p_1 \in \BallHyp{y}} \times\\
& \hspace{15pt}\int_{-I_n}^{I_n} \int_{R - y}^{R} 
\ind{p_2 \in \BallSym{p_1} \cap \BallHyp{y}}
  \rho_{\Po}(y,k_n-2) \cdot
e^{-\alpha( y_2 + y_1)} \dd y_2 \dd x_2 \dd y_1 \dd x_1.
\end{split}
\end{equation*}

We proceed to bound the integral: 
\begin{align*}
	&\int_{-I_n}^{I_n} \int_{(1-\eps) R \wedge (R-y)}^{R} \hspace{-5pt} \ind{p_1 \in \BallHyp{y}}
		\int_{-I_n}^{I_n} \int_{R - y}^{R } \ind{p_2 \in \BallSym{p_1} \cap \BallHyp{y}} 
		e^{-\alpha(y_1 + y_2)} \dd y_2 \dd x_2 \dd y_1 \dd x_1 \\
	&\leq \int_{-I_n}^{I_n} \int_{(1-\eps) R \wedge (R-y)}^{R} 
		\int_{-I_n}^{I_n} \int_{R - y}^{R}  e^{-\alpha(y_1 + y_2)} \dd y_2 \dd x_2 \dd y_1 \dd x_1\\
	&= \left( \int_{-I_n}^{I_n} \int_{(1-\eps) R \wedge (R-y)}^{R} e^{- \alpha y_1} \dd y_1 \dd x_1 \right) 
	\left(\int_{-I_n}^{I_n} \int_{R - y}^{R}  e^{-\alpha y_2} \dd y_2 \dd x_2 \right).
\end{align*}
We evaluate
$$  \int_{-I_n}^{I_n} \int_{(1-\eps) R \wedge (R-y)}^{R} 
e^{- \alpha y_1} dy_1 dx_1= O(1) \cdot n \cdot e^{-\alpha R + ((\eps R) \vee y))\alpha}
=O(1) \cdot n \cdot e^{-\alpha R + \alpha y + \alpha \eps R }
$$
and 
$$\int_{-I_n}^{I_n} \int_{R - y}^{R}  e^{-\alpha y_2} dy_2 dx_2 
=O(1) \cdot n \cdot e^{-\alpha R +\alpha y}.
$$
Also, $n \cdot e^{-\alpha R} = O(1) \cdot e^{(1/2 -\alpha) R}$, whereby we deduce that 
\begin{equation*}
\begin{split}
&\hspace{-30pt}\int_{\dom{1}} \ind{p_1 \in \BallHyp{y}} \ind{p_2 \in \BallSym{p_1} \cap \BallHyp{y}} 
 e^{-\alpha(y_1 + y_2)} dy_2 dx_2 dy_1 dx_1 \\
&=O(1) \cdot e^{(1-2\alpha)R + 2\alpha y + \alpha \eps R} =O(1) \cdot n^{2(1-2\alpha) + 2\alpha \eps} \cdot e^{2\alpha y}.
\end{split}.
\end{equation*}

With these computations we obtain
\begin{align*} 
	\int_{\Kcal_{C}(k_n)} \mathcal{I}_n^{(1)}(y) e^{-\alpha y} \dd x \dd y
	&= O(1) n^{2(1-2\alpha) + 2\alpha \eps}
		\int_{\Kcal_{C}(k_n)} e^{2\alpha y} \rho_{\Po}(y,k_n-2) e^{-\alpha y} \dd y \dd x \\ 
	&=O(1) n^{2(1-2\alpha) + 2\alpha \eps} \, k_n^{2\alpha - 1}.
\end{align*}
Thus, for $1/2 < \alpha < 3/4$, we have 
\begin{equation*}
k_n^{6 \alpha -3} \,  n^{2(1-2\alpha) + 2\alpha \eps} \, k_n^{2\alpha -1} = 
n^{2\alpha \eps} \, \left( \frac{k_n^2}{n} \right)^{2(2\alpha -1)} = o(1), 
\end{equation*}
provided that $\eps = \eps (\alpha)>0$ is small enough, and hence for such $\eps$
\[
	\lim_{n \to \infty} k_n^{6\alpha - 3} \int_{\Kcal_{C}(k_n)} \mathcal{I}_n^{(1)}(y) e^{-\alpha y} \dd x \dd y = 0.
\]

When $\alpha \ge 3/4$ we have $2(1-2\alpha) < 1/2(4\alpha - 1)$ and we get
\begin{equation*} 
k_n^{2\alpha} \, n^{2(1-2\alpha) + 2\alpha \eps} \cdot k_n^{2\alpha -1}
\le  k_n^{4\alpha - 1} \, n^{2(1-2\alpha)} n^{2\alpha \eps} = o(1),
\end{equation*}
provided that $\eps$ is small enough, depending on $\alpha$, so that
\[
	\lim_{n \to \infty} k_n^{2\alpha} \int_{\Kcal_{C}(k_n)} \mathcal{I}_n^{(1)}(y) e^{-\alpha y} \dd x \dd y = 0.
\]

We now consider the second sub-domain $\dom{2}$. The Campbell-Mecke formula yields that: 
\begin{align*}
	\mathcal{I}_n^{(2)}(y) 
	&= \Exp { \sum_{(p_1,p_2)  \in \dom{2}} \ind{p_1 \in \BallHyp{y}} \ind{p_2 \in \BallSym{p_1} \cap \BallHyp{y}}
		D_{\Po} (y,k_n-2;\Pcal \setminus \{ p_1\})} \\
	&= O(1) \rho_{\Po}(y,k_n-2) \cdot \int_{-I_n}^{I_n} \int_{(1-\eps) R \wedge (R-y)}^{R} 
		\ind{p_1 \in \BallHyp{y}} \times \\
	&\hspace{15pt} \int_{-I_n}^{I_n} \int_{R - y_1}^{R-y} \ind{p_2 \in \BallSym{p_1} \cap \BallHyp{y}}
		e^{-\alpha(y_1 + y_2)} \dd y_2 \dd x_2 \dd y_1 \dd x_1.
\end{align*}

We bound the integral as follows: 
\begin{align*}
	&\int_{-I_n}^{I_n} \int_{(1-\eps) R \wedge (R-y)}^{R} \hspace{-5pt} \ind{p_1 \in \BallHyp{y}}
		\int_{-I_n}^{I_n} \int_{R - y_1}^{R - y} \ind{p_2 \in \BallSym{p_1} \cap \BallHyp{y}}
		e^{-\alpha(y_1 + y_2)} \dd y_2 \dd x_2 \dd y_1 \dd x_1 \\
	&\leq \int_{-I_n}^{I_n} \int_{(1-\eps) R \wedge (R-y)}^{R} \ind{p_1 \in \BallHyp{y}}
		\int_{-I_n}^{I_n} \int_{R - y_1}^{R - y} 
		\ind{p_2 \in  \BallHyp{y}} e^{-\alpha(y_1 + y_2)} \dd y_2 \dd x_2 \dd y_1 \dd x_1.
\end{align*}
Now, by Lemma~\ref{lem:asymptotics_Omega_hyperbolic},
\begin{align*}
&\int_{-I_n}^{I_n} \int_{R - y_1}^{R - y} \ind{p_2 \in  \BallHyp{y}} \cdot  
e^{-\alpha y_2} dy_2 dx_2 = O(1) \cdot e^{y/2} \int_{R - y_1}^{R - y} e^{(1/2 - \alpha) y_2} dy_2 \\
&= O(1) \cdot e^{y/2 + (1/2 - \alpha) (R - y_1)}.
\end{align*}
We then integrate with respect to $y_1$:
\begin{align*}
	O(1) \cdot e^{y/2} \cdot 
	&\int_{-I_n}^{I_n} \int_{(1-\eps) R \wedge (R-y)}^{R} \ind{p_1 \in \BallHyp{y}} 
		e^{(1/2 - \alpha) (R - y_1)} e^{-\alpha y_1} dy_1 dx_1 \\
	&\le O(1) \cdot e^{y/2 + (1/2 -\alpha) R} \cdot \int_{-I_n}^{I_n} \int_{(1-\eps) R \wedge (R-y)}^{R} 
		e^{(\alpha -1/2) y_1} e^{-\alpha y_1} dy_1dx_1 \\
	&=O(1) \cdot e^{y/2 + (1 -\alpha) R -((1-\eps) R \wedge (R-y))/2} \\
	&=O(1) \cdot e^{y/2 + (1/2 -\alpha) R + ((\eps R) \vee y)/2}\\
	&=O(1) \cdot e^{y + (1/2 -\alpha) R + \eps R}
		= O(1) \cdot n^{1- 2\alpha+ \eps} \cdot e^{y}. 
\end{align*}

Therefore we get
\begin{align*} 
	&\int_{\Kcal_{C}(k_n)}\mathcal{I}_n^{(2)}(y) e^{-\alpha y} \dd x \dd y\\
	&= O\left( n^{1-2\alpha + \eps} \right)
		\int_{\Kcal_{C}(k_n)} \rho_{\Po}(y,k_n-2) e^{y} e^{-\alpha y} \dd x \dd y\\
	&= \bigO{1} n^{1-2\alpha + \eps} k_n^{-2\alpha + 1},
\end{align*}
where we used Lemma~\ref{lem:gamma_approx} with $t = 1$.

For $1/2 < \alpha < 3/4$, we have
\[
	k_n^{4 \alpha -2} \cdot  n^{1-2\alpha + \eps} = n^{\eps} \left( \frac{k_n^2}{n}\right)^{2\alpha -1} = o(1),
\]
provided that $\eps = \eps (\alpha)>0$ is small enough, yielding
\[
	\lim_{n \to \infty} k_n^{6\alpha - 3} 
			\int_{\Kcal_{C}(k_n)}\mathcal{I}_n^{(2)}(y) e^{-\alpha y} \dd x \dd y = 0.
\]
Similarly, for $\alpha > 3/4$ we have $2\alpha -1 > 1/2$ and we get
\[
	k_n \cdot n^{1-2\alpha + \eps} \ll n^{-1/2 + \eps}  \cdot k_n  = o(1),
\]
provided that $\eps$ is small enough, so that
\[
	\lim_{n \to \infty} k_n^{2\alpha}
			\int_{\Kcal_{C}(k_n)} \mathcal{I}_n^{(2)}(y) e^{-\alpha y} \dd x \dd y = 0.
\]

For the third sub-domain $\dom{3}$ we shall use~\eqref{eq:symdiff_loc} which states that if 
$p_2=(x_2,y_2) \in \BallSym{p_1}\cap \BallHyp{y}$ and $y_2\leq R - y_1$, then 
$y_2 \geq \hat{y}(x_1,y_1)$, where $\hat{y}(x_1,y_1) = \left(2 \log(|x_1|e^{-y_1/2}) - \log c_K \right) \vee 0$ (cf.~\eqref{eq:to_use_I}). Moreover, $p_2 \in \FatBallHyp{p_1}$.

Again, we will use the Campbell-Mecke formula: 
\begin{align*}
	\mathcal{I}_n^{(3)}(y) &= \Exp { \sum_{(p_1,p_2)  \in \dom{3}} 
		\ind{p_1 \in \BallHyp{y}} \cdot \ind{p_2 \in \BallSym{p_1} \cap \BallHyp{y}}
		\cdot D_{\Po} (y,k_n-2;\Pcal \setminus \{ p_1, p_2\})} \\
	&= O(1) \rho_{\Po}(y,k_n-2) \int_{-I_n}^{I_n} \int_{(1-\eps) R \wedge (R-y)}^{R} \ind{p_1 \in \BallHyp{y}}
		\times\\
	&\hspace{25pt} \int_{-I_n}^{I_n} \int_{0}^{R-y_1} 
		\ind{p_2 \in \BallSym{p_1} \cap \BallHyp{y}} 
		e^{-\alpha(y_1 + y_2)} dy_2 dx_2 dy_1 dx_1 \\
\end{align*}

The inner integral with respect to $p_2 := (x_2,y_2)$ is 
\begin{align*}
	&\hspace{-30pt} \int_{-I_n}^{I_n} \int_{0}^{R - y_1}  \ind{p_2 \in \BallSym{p_1}\cap \BallHyp{y}}  
		e^{-\alpha y_2} dy_2 dx_2\\
	&\leq \int_{-I_n}^{I_n} \int_{0}^{R-y_1}  \ind{y_2 \geq \hat{y}(x_1,y_1), p_2 \in \FatBallHyp{(0,y)}}  
		e^{-\alpha y_2} dy_2 dx_2 \\
	&= O(1) e^{y/2} \int_{\hat{y}(x_1,y_1)}^{R - y_1} e^{y_2/2 - \alpha y_2} dy_2 \\
	&= O(1) e^{y/2 + (1/2 -\alpha) \hat{y}(x_1,y_1)}.
\end{align*}

Thus, we get
\begin{align*}
	&\int_{-I_n}^{I_n} \int_{(1-\eps) R \wedge (R-y)}^{R} \ind{p_1 \in \BallHyp{y}}
		\int_{-I_n}^{I_n} \int_{0}^{R-y_1} 
		\ind{p_2 \in \BallSym{p_1} \cap \BallHyp{y}} \times \\ 
	& \hspace{2cm}  e^{-\alpha(y_1 + y_2)} dy_2 dx_2 dy_1 dx_1 \\
	&\leq O(1) \int_{-I_n}^{I_n} \int_{(1-\eps) R \wedge (R-y)}^{R} 
		e^{y/2 + (1/2 -\alpha) \hat{y}(x_1,y_1)} e^{-\alpha y_1} dy_1 dx_1. 
\end{align*}
Due to symmetry, to bound the integral it is enough to integrate this with respect to $x_1$ from 0 to $I_n$.
We will split this integral into two parts according to the value of $c(x_1,y_1)$:
\[
	\int_0^{I_n} e^{\hat{y}(x_1,y_1) (1/2 -\alpha)} dx_1 = 
	\int_{c_K e^{y_1/2}}^{I_n} e^{c(x_1,y_1) (1/2 -\alpha)} dx_1 + \int_0^{c_K e^{y_1/2}} dx_1.
\]
The first integral becomes: 
\begin{align*}
	&\hspace{-30pt}\int_{c_K e^{y_1/2}}^{I_n} e^{\hat{y}(x_1,y_1) (1/2 -\alpha)} dx_1  = 
 		O(1)\cdot \int_{c_K e^{y_1/2}}^{I_n} x_1^{1 -2\alpha} 
		e^{-\frac{y_1}{2} (1-2\alpha)} dx_1 \\
	&= \begin{cases}
		O(R) \cdot e^{-y_1/2 + \alpha y_1} \cdot e^{\frac{R}{2} 2(1-\alpha)} & \ \mbox{if $\alpha \leq 1$}\\
		O(1) \cdot e^{-y_1/2 + \alpha y_1 + 2(1-\alpha)y_1/2} & \ \mbox{if $\alpha > 1$}
		\end{cases}\\
	&=\begin{cases}
		O(R) \cdot e^{(\alpha -1/2) y_1} \cdot n^{2(1-\alpha)} & \ \mbox{if $\alpha \leq 1$}\\
		O(1) \cdot e^{y_1/2} &\ \mbox{if $\alpha > 1$}
	\end{cases}.  
\end{align*}
The second integral trivially gives: 
\[
	\int_0^{c_K e^{y_1/2}} dx_1 = O(1) \cdot e^{y_1/2}.
\]
Putting these two together we conclude that 
\[
	e^{y/2} \cdot \int_0^{I_n} e^{\hat{y}(x_1,y_1) (1/2 -\alpha)} dx_1
	= O(1) \cdot e^{y_1/2 +y(3/2-\alpha)}.
\]

Now, we integrate these with respect to $y_1$:
\begin{align*}
	&n^{2(1-\alpha)} \cdot \int_{(1-\eps) R \wedge (R-y)}^{R} e^{(\alpha -1/2)y_1 - \alpha y_1} dy_1 
		=O(1) \cdot n^{2(1-\alpha)} \cdot e^{-R/2 + \eps R/2  + y/2}  \\
	&= O(1) \cdot n^{1-2\alpha + \eps} \cdot e^{y/2}.
\end{align*}
Therefore, we conclude that
\[
	\mathcal{I}_n^{(3)}(y) = \bigO{R} n^{1-2\alpha +\eps(2\alpha -1)} \, e^{y/2} \rho_{\Po}(y,k_n-2)
\]
and hence, using again Lemma~\ref{lem:gamma_approx},
\begin{align*}
	\int_{\Kcal_{C}(k_n)} \mathcal{I}_n^{(3)}(y) e^{-\alpha y} \dd x \dd y
	&= \bigO{R} n^{1-2\alpha +\eps(2\alpha -1)} \int_{\Kcal_{C}(k_n)} e^{y/2} 
		\rho_{\Po}(y,k_n-2) e^{-\alpha y} \dd x \dd y\\
	&= \bigO{R} n^{1-2\alpha +\eps(2\alpha -1)} k_n^{-2\alpha + 1}.
\end{align*}

It follows that for $\eps = \eps(\alpha)$ small enough
\[
	k_n^{6\alpha - 3} R n^{1-2\alpha +\eps(2\alpha -1)} k_n^{-2\alpha + 1}
	= R n^{\eps(2\alpha -1)} \left(\frac{k_n^2}{n}\right)^{2\alpha - 1} = \smallO{1}
\]
and hence for $\alpha > 1/2$,
\[
	\lim_{n \to \infty} k_n^{6\alpha - 3} \int_{\Kcal_{C}(k_n)} \mathcal{I}_n^{(3)}(y) e^{-\alpha y} \dd x \dd y = 0.
\]
Since $2\alpha - 1 \ge 1/2$ when $\alpha \ge 3/4$ it immediately follows that
\[
	\lim_{n \to \infty} k_n^{2\alpha} \int_{\Kcal_{C}(k_n)} \mathcal{I}_n^{(3)}(y) e^{-\alpha y} \dd x \dd y = 0.
\]

\paragraph{The sums~\eqref{eq:sum_5} and~\eqref{eq:sum_5b}}

Again, the analysis for both terms are similar and we shall analyze~\eqref{eq:sum_5}. Let us set $p=(0,y)$. Recall that $\BallSym{y}\cap \Rcal([R - y + 2 \log\left(\frac{\pi}{2}\right),R]) = \emptyset$. Thus, the summand in~\eqref{eq:sum_5} is equal to 0, when $y_1 > R - y + 2 \log (\pi/2)$. 

Recall the definition of the extended ball $\FatBallHyp{p}$ around $p$~\eqref{eq:def_fatball} that contains both $\BallHyp{p}$ and $\BallPo{p}$ 
\[
	\FatBallHyp{y} : = \{ p^\prime : y^\prime < R - y, \ |x^\prime| < (1+K) e^{\frac{1}{2} (y + y^\prime)}  \},
\]
and that we have $\Exp{D_{\Po} (y,k_n-2;\Pcal \setminus \{ p_1,p_2\})} = \rho_{\Po}(y,k_n-2)$.

Further, observe that,
\begin{equation*} 
\BallHyp{y} \cap \Rcal ([0,R-y)) \subseteq \FatBallHyp{y}
\end{equation*}
and 
\begin{equation*} 
\BallHyp{y} \cap \Rcal ([R-y,R]) = \Rcal ([R-y,R]).
\end{equation*}
We thus conclude that 
\begin{equation} \label{eq:ball_inclusions}
\BallHyp{y} \subseteq \FatBallHyp{y} \cup \Rcal([R-y,R]).
\end{equation}
Hence, if we set 
\[
h_y(p_1) := \ind{p_1 \in \BallHyp{p}\setminus \BallPo{y}} \cdot    
\left( \mu \left( \FatBallHyp{p_1} \cap \FatBallHyp{y} \right)
+ \mu \left( \Rcal([R-y,R]) \right) \right),
\]
then 
\begin{align*}
&\ind{p_1\in \BallHyp{p}\setminus \BallPo{y}} \cdot \Exp{ \left(\sum_{p_2 \in \Pcal \setminus 
\{p,p_1\}} \ind{p_2 \in \BallHyp{y} \cap \BallPo{p_1}}\right) \cdot 
D_{\Po} (y,k_n-2;\Pcal \setminus \{ p_1, p_2\})
} \\
&=O(1)\cdot
\ind{p_1\in \BallHyp{y}\setminus \BallPo{y}} \cdot \mu (\BallHyp{y}\cap \BallHyp{p_1}) \rho_{\Po}(y,k_n-2) \\
& \leq O(1) \cdot  h_y (p_1) \rho_{\Po}(y,k_n-2) . 
\end{align*}
To calculate the expectation of the above function we need to approximate the 
intersection of the two balls $\FatBallHyp{y}$ and $\FatBallHyp{p_1}$, 
where $p_1= (x_1,y_1)$. 
Let us assume without loss of generality that $x_1 > 0$. 
The right boundary of $\FatBallHyp{y}$ is given by the equation 
$x = x(y^\prime) = (1+K)e^{\frac{1}{2} (y + y^\prime)}$ whereas the left boundary of $\FatBallHyp{p_1}$ is given by the curve $x = x(y^\prime)= x_1 - (1+ K)e^{\frac{1}{2} (y_1 + y^\prime)}.$ 

The equation that determines the intersecting point of the two curves is
\[
	x_1 - (1+K)e^{(\hat{y} + y_1)/2}= (1+K) e^{(\hat{y} + y)/2},
\]
where $\hat{y}$ is the $y$-coordinate of the intersecting point. 
We can solve the above for $\hat{y}$  
\begin{equation*} 
\begin{split}
x_1 &=(1+K) e^{\hat{y}/2} \left( e^{y/2} + e^{y_1/2} \right).
\end{split}
\end{equation*}
But since $p_1=(x_1,y_1)  \in \BallSym{p}$, we also have $x_1 > e^{\frac{y + y_1}{2}}$. Therefore, 
\begin{equation}\label{eq:bounds_fat_ball_points}
\begin{split}
 e^{\hat{y}/2}& > \frac{1}{1+K}~\frac{e^{\frac{y + y_1}{2}}}{ e^{y/2}+ e^{y_1/2}} \geq 
\frac{1}{2(1+K)}~\frac{e^{\frac{y_1 + y}{2}}}{ e^{(y \vee y_1) /2}} 
> \frac{1}{2(1 + K)} ~ e^{(y \wedge y_1)/2}. 
 \end{split}
\end{equation}
The above yields
\begin{equation} \label{eq:to_use_I}
\hat{y} > (y \wedge y_1) - 2\log(2(1+K)) := \hat{y}(y_1, y). 
\end{equation}
which, in turn, implies the following 
\begin{equation}\label{eq:intersex_approx}
	p \in \FatBallHyp{(0,y)}\cap \FatBallHyp{p_1} \Rightarrow y(p) \ge \hat{y}(y_1,y).
\end{equation}
We thus conclude that 
\[ 
	\BallHyp{p_1} \cap \BallHyp{p} \subseteq \left(\FatBallHyp{p} \cap \Rcal([\hat{y}(y_1,y), R])\right)
	\, \cup \, \Rcal ([R - y,R]),
\]
which in turn implies that
\[
	\mu \left( \FatBallHyp{p_1} \cap \BallHyp{p} \right) \leq 
	\mu\left( \FatBallHyp{p} \cap  \Rcal([\hat{y}(y_1,y), R]\right) + 
	\mu (\Rcal ([R - y, R]) ).
\]
Therefore, 
\begin{align*} 
	h_y(p_1, \Pcal) &\leq \ind{p_1 \in \BallHyp{p}\setminus \BallPo{p}} 
    	\mu  \left( \FatBallHyp{p} \cap  \Rcal([\hat{y}(y_1,y), R])\right)
        \\
	&\hspace{10pt}+ \ind{p_1 \in \BallHyp{p}\setminus \BallPo{p}}
    	\mu  \left( \Rcal ([R - y, R]) \right).
\end{align*}

Now, the Campbell-Mecke formula gives
\begin{align*}
	&\Exp{ \sum_{p_1,p_2 \in\Pcal \setminus \{(0,y)\} \atop y(p_1) \geq K} 
		\hspace{-10pt} \ind{p_1\in \BallHyp{y}\setminus \BallPo{y}} \ind{p_2\in \BallHyp{y}\cap \BallPo{y}} 
		\, D_{\Po} (y,k_n-2;\Pcal \setminus \{ p_1,p_2\})}\\
	&\le \Exp{\left( \sum_{p_1 \in \Pcal} 
		h_y (p_1, \Pcal \setminus \{ p_1 \})\right)} \\
	&=\frac{\nu \alpha}{\pi} \int_{\Rcal} \Exp{h_y(p_1, \Pcal \setminus \{p_1\})}
		e^{-\alpha y_1} \dd x_1 \dd y_1\\
	&\le \frac{\nu \alpha}{\pi} \int_{\Rcal} \ind{p_1 \in \BallHyp{p}\setminus \BallPo{p}} 
	    	\mu  \left( \FatBallHyp{p} \cap  \Rcal([\hat{y}(y_1,y), R])\right)
	    	e^{-\alpha y_1} \dd x_1 \dd y_1 \numberthis \label{eq:h_upper_bound_1}\\
	&\hspace{10pt}+ \frac{\nu \alpha}{\pi} \int_{\Rcal} \ind{p_1 \in \BallHyp{p}\setminus \BallPo{p}}
	    	\mu  \left( \Rcal ([R - y, R]) \right) e^{-\alpha y_1} \dd x_1 \dd y_1.
	    	\numberthis \label{eq:h_upper_bound_2}
\end{align*}

Recall that $(\BallSym{(0,y)})\cap \Rcal([R - y + 2 \log\left(\frac{\pi}{2}\right),R]) = \emptyset$. 
We will first calculate the measures $\mu$ appearing in~\eqref{eq:h_upper_bound_1} and~\eqref{eq:h_upper_bound_2}. The first one is:
\begin{align*}
	\mu\left( \FatBallHyp{y} \cap  \Rcal([c(y_1,y), R])\right) 
	&\leq (1+ K) \frac{\nu \alpha}{\pi} \cdot e^{y/2}  \int_{\hat{y}(y_1,y)}^{R} e^{-(\alpha - \frac{1}{2}) y'} \, dy' \\
	&=  \bigO{e^{\frac{y}{2} - (\alpha-\frac{1}{2}) (y \wedge y_1)}}.
\end{align*}

The second term is: 
\begin{align*}
	\mu \left( \Rcal([R - y,R]) \right) 
    &= \frac{\nu \alpha}{\pi} \int_{R - y}^{R} \pi e^{\frac{R}{2}} e^{-\alpha y'} \, dy' 
    	= \bigO{e^{\frac{R}{2}} e^{-\alpha (R-y)}} = \bigO{e^{\alpha y - (\alpha - \frac{1}{2})R}}. 
\end{align*}

Using these, we get
\begin{align} 
	&\int_{\Rcal ([0, R - y_n + 2 \ln \frac{\pi}{2}])} \Exp{h_y(p_1, \Pcal \setminus \{p_1\})} 
    e^{-\alpha y_1} \, dx_1 \, dy_1 \notag \\
	&= \bigO{1} \int_{\Rcal ([0, R - y+ 2 \ln \frac{\pi}{2}])} \ind{p_1 \in \BallSym{p}} 
		e^{\frac{y}{2} - (\alpha - \frac{1}{2}) (y \wedge y_1) - \alpha y_1} \, dx_1 \, dy_1
		\label{eq:Mecke_sum_1}\\ 
	&\hspace{10pt}+ \bigO{1} \int_{\Rcal ([0, R - y + 2 \ln \frac{\pi}{2}])} 
    	\ind{p_1 \in \BallHyp{(0,y)}} 
    	e^{\alpha y - (\alpha - \frac{1}{2})R - \alpha y_1} \, dx_1 \, dy_1.\label{eq:Mecke_sum_2}
\end{align}
Now, Lemma~\ref{lem:asymptotics_Omega_hyperbolic} implies that 
for any $y_1 \in [0, R - y + 2 \ln \frac{\pi}{2}]$, we have 
\[ 
	\int_{-I_n}^{I_n} \ind{p_1 \in \BallSym{y}} \, dx_1 \leq 2 K e^{\frac{3}{2} (y_1 + y) - R}.
\]

Therefore, \eqref{eq:Mecke_sum_1} is 
\begin{align*}
	&\hspace{-20pt}O(1) \cdot e^{2 y - R} \int_{0}^{R - y + 2 \ln \frac{\pi}{2}} 
    	e^{\frac{3y_1}{2} - (\alpha - \frac{1}{2})(y_1 \wedge y) - \alpha y_1} \, dy_1 \\
 	&=  O(1) \cdot e^{2 y - R} \left( \int_{0}^{y} e^{\frac{3y_1}{2} - (2\alpha - \frac{1}{2})y_1} \, dy_1 
 		+ e^{-(\alpha-\frac{1}{2}) y}\int_{y}^{R - y + 2 \ln \frac{\pi}{2}} e^{(\frac{3}{2} - \alpha) y_1} \, dy_1 \right)\\
  	&= O(1) \left( 
	\begin{cases}
	e^{(4-2\alpha) y - R}, & \mbox{if $\alpha < 1$} \\
	R \cdot e^{2y - R}, & \mbox{if $\alpha \geq 1$}
	\end{cases}
	+
	\begin{cases}
	e^{-(\alpha - \frac{1}{2})R +y}, & \mbox{if $\alpha < 3/2$} \\
	R \cdot  e^{2(2-\alpha)y - R}, & \mbox{if $\alpha \geq 3/2$}
	\end{cases}\right).
\end{align*}

Similarly, for~\eqref{eq:Mecke_sum_2} we have
\begin{align*}
	&\hspace{-30pt}\int_{\Rcal ([0, R - y + 2 \ln \frac{\pi}{2}])} \ind{p_1 \in \BallSym{(0,y)}} e^{\alpha y - (\alpha - \frac{1}{2})R - \alpha y_1} \, dx_1 \, dy_1\\
	&= e^{\frac{3y}{2} - R + \alpha y - (\alpha - \frac{1}{2})R} 
    	\cdot \int_{0}^{R - y + 2 \ln \frac{\pi}{2}} e^{\frac{3y_1}{2}-\alpha y_1} \, dy_1\\
	&= O(1)\cdot 
	\begin{cases} 
	e^{\frac{3y}{2} - R + \alpha y - (\alpha - \frac{1}{2})R + (\frac{3}{2} - \alpha)(R-y)} 
	, & \mbox{if $\alpha < 3/2$} \\ 
	R \cdot e^{(\frac{3}{2} +\alpha)y -  (\alpha + \frac{1}{2})R}, & \mbox{if 
	$\alpha \geq 3/2$}
	\end{cases}	\\
	&= O(1) \cdot 
	\begin{cases}
	  e^{-(2\alpha-1) R + 2 \alpha y}, & \mbox{if $\alpha < 3/2$} \\
	  R \cdot e^{(\frac{3}{2} +\alpha)y -  (\alpha + \frac{1}{2})R}, & \mbox{if 
	$\alpha \geq 3/2$}
	\end{cases}.
\end{align*}

We thus conclude, using $2(2 - \alpha)y \le y$ for $\alpha > 3/2$, that 
\begin{equation} \label{eq:upper_bound_faulty_edges} 
\Exp{\left( \sum_{p_1 \in \Pcal \setminus\{p\}} 
		h_y (p_1)\right)} \le \bigO{1} \cdot 
\left( \mathcal{I}_n^{(1)}(y) + \mathcal{I}_n^{(2)}(y) + \mathcal{I}_n^{(3)}(y) \right),
\end{equation}
where 
\begin{align*}
 \mathcal{I}_n^{(1)}(y) &= \begin{cases}
	e^{(4-2\alpha) y - R}, & \mbox{if $\alpha < 1$} \\
	R \cdot e^{2y - R}, & \mbox{if $\alpha \geq 1$}
	\end{cases},  \\
	\mathcal{I}_n^{(2)}(y) &= 
	\begin{cases}
	e^{-(\alpha - \frac{1}{2})R +y}, & \mbox{if $\alpha < 3/2$} \\
	R \cdot  e^{y - R}, & \mbox{if $\alpha \geq 3/2$}
	\end{cases}\\
	\mathcal{I}_n^{(3)}(y) &= 
\begin{cases}
	  e^{-(2\alpha-1) R + 2 \alpha y}, & \mbox{if $\alpha < 3/2$} \\
	  R \cdot e^{(\frac{3}{2} +\alpha)y -  (\alpha + \frac{1}{2})R}, & \mbox{if 
	$\alpha \geq 3/2$}
	\end{cases}.
\end{align*}

We proceed to calculate:
\begin{align*}
&\int_{\Kcal_{C}(k_n)} \Exp{\left( \sum_{p_1 \in \Pcal} h_y (p_1, \Pcal \setminus \{ p_1 \})\right)} \cdot 
 \rho_{\Po}(y,k_n-2) e^{-\alpha y} \dd y.
\end{align*}
For this we define
\[
	M_i = \int_{\Kcal_{C}(k_n)} \mathcal{I}_n^{(i)}(y) \rho_{\Po}(y,k_n-1) e^{-\alpha y} \dd y
\]
so that
\begin{align*}
\int_{\Kcal_{C}(k_n)} \Exp{\left( \sum_{p_1 \in \Pcal \setminus \{(0,y)\}} h_y (p_1)\right)} \rho_{\Po}(y,k_n-1) e^{-\alpha y} \dd y
&= \bigO{M_1 + M_2 + M_3}.
\end{align*}

Computing each of the integral separately we obtain, using Lemma~\ref{lem:gamma_approx} and the fact that $n = \nu e^{R/2}$,
\begin{align*} 
M_1:= \int_{\Kcal_{C}(k_n)} \mathcal{I}_n^{(1)}(y) \rho_{\Po}(y,k_n-1) e^{-\alpha y} \dd y
&= O(1) \cdot 
\begin{cases} 
	\frac{k_n^{7-6\alpha}}{n^2}, & \mbox{if $\alpha <1$} \\
	R \frac{k_n^{3-2\alpha}}{n^2}, & \mbox{if $\alpha \geq 1$}
\end{cases}. 
\end{align*}
\begin{align*} 
M_2:= \int_{\Kcal_{C}(k_n)} \mathcal{I}_n^{(2)}(y) \rho_{\Po}(y,k_n-1) e^{-\alpha y} dy
&= O(1) \cdot 
\begin{cases}
\frac{k_n^{1-2\alpha}}{n^{2\alpha-1}}, & \mbox{if $\alpha < 3/2$} \\
R   \frac{k_n^{1-2\alpha}}{n^{2}}, & \mbox{if $\alpha \geq 3/2$}
\end{cases}
\end{align*}
and finally 
\begin{align*} 
M_3:= \int_{\Kcal_{C}(k_n)} \mathcal{I}_n^{(3)}(y) \rho_{\Po}(y,k_n-1) e^{-\alpha y} dy
&= O(1) \cdot 
\begin{cases} 
\frac{k_n^{2\alpha - 1}}{n^{4\alpha -2}}, & \mbox{if $\alpha <3/2$} \\ 
R \cdot \frac{k_n^2}{n^{2\alpha+1}}, &\mbox{if $\alpha \geq 3/2$}
\end{cases}  .
\end{align*}

Now, we will consider the two cases according to the value of $\alpha$. First we note that $R = \bigO{\log(n)}$ and since
$k_n = O(n^{\frac{1}{2\alpha +1}})$ and $\alpha > 1/2$ we have that $R k_n^2 n^{-1} = \smallO{1}$.
Assume first that $1/2 < \alpha < 3/4$. In this case, we want to show that 
\begin{equation} \label{eq:int3_to_prove_I}
\lim_{n \to \infty} k_n^{6\alpha -3} (M_1 + M_2 + M_3) = 0. 
\end{equation}
Using the above expression for $M_i$, we have 
\begin{align*} 
 k_n^{6\alpha -3} (M_1 + M_2 + M_3) &= O(1) \cdot  
 k_n^{6\alpha -3} 
\left( 
\frac{k_n^{7-6\alpha}}{n^2} + \frac{k_n^{1-2\alpha}}{n^{2\alpha-1}} 
+\frac{k_n^{2\alpha-1}}{n^{4\alpha - 3}}.
\right) 
\end{align*}
We wish to show that each one of the above three terms is $o(1)$ for $k_n = O(n^{\frac{1}{2\alpha +1}})$. 
For the first one we have 
\[ 
	k_n^{6\alpha -3} \frac{k_n^{7-6\alpha}}{n^2} = \left(\frac{k_n^{2}}{n}\right)^2 = \smallO{1}. 
\]
The second term yields: 
\[
	k_n^{6\alpha -3}  \frac{k_n^{-2\alpha+1}}{n^{2\alpha-1}} = \left(\frac{k_n^2}{n}\right)^{2\alpha - 1} = \smallO{1}.
\]
Finally, the third one yields: 
\[
	k_n^{6\alpha -3} \cdot \frac{k_n^{2\alpha -1}}{n^{4\alpha - 2}}  
	= \left(\frac{k_n^2}{n}\right)^{4\alpha - 2} = \smallO{1}.
\]
 
For $\alpha \ge 3/4$, we would like to show that 
\begin{equation} \label{eq:int3_to_prove_II}
\lim_{n \to \infty} k_n^{2\alpha} \cdot (M_1 + M_2 + M_3) = 0. 
\end{equation}
Firstly, we note that each $M_i$ is as above if $3/4 < \alpha < 1$. Therefore, since for this range $2 \alpha <6\alpha - 3$ the result follows from the above analysis. Next we consider the case $1 \le \alpha < 3/2$. Here, only the value of $M_1$ changes and we compute that
\[
	k_n^{6\alpha - 3} M_1 = \bigO{1} \log(n) n^{-2} k_n^{4\alpha} \le \bigO{\log(n)} \left(\frac{k_n^2}{n}\right)^2 = \smallO{1},
\]
so that~\eqref{eq:int3_to_prove_II} holds for $3/4 < \alpha < 1$.

Proceeding with the case $\alpha \ge 3/2$, it is only $M_2$ and $M_3$ that change values. In particular, for any $\alpha \geq 3/2$ we have 
\[
	\frac{k_n}{n}  M_2 =O(1) R \frac{k_n}{n^2} = o(1).
\]
Also,
\[
	k_n^{2\alpha }  M_3 = O(1) R \frac{k_n^{2\alpha + 2}}{n^{2\alpha + 1}}
	= R \smallO{\frac{n^{\alpha + 1}}{n^{2\alpha +1}}} = o(1),
\]
since $k_n = o(n^{1/2})$ and hence~\eqref{eq:int3_to_prove_II} holds. This finished the proof for~\eqref{eq:sum_5}.

\paragraph{The sum of~\eqref{eq:sum_6}}
Using the Campbell-Mecke formula, we write 
\begin{align*}
	&\hspace{-30pt}\Exp { \sum_{p_1,p_2 \in\Pcal \setminus \{(0,y)\}, \ y_1 < K, \atop \text{distinct}}
		\ind{p_1\in \BallSym{y}} \ind{p_2\in \BallHyp{y}\cap \BallPo{y}}}\\
	&\leq  \int_0^K \int_{-I_n}^{I_n}  \int_0^{R} \int_{-I_n}^{I_n}
		\ind{p_1 \in \BallSym{y}} 
 		\ind{p_2 \in \BallHyp{y}\cap \BallPo{y}} e^{-\alpha y_2} e^{-\alpha y_1} \dd x_2 \dd y_2 \dd x_1 \dd y_1  \\
 	&\leq  \mu (\BallHyp{y})\cdot \int_{-I_n}^{I_n} \int_0^K \ind{p_1 \in \BallSym{y}} 
		e^{-\alpha y_1}  dx_1 dy_1.
\end{align*}
Recall that $\Mu{\BallHyp{y}} =O(1) e^{y/2}$. We bound the integral using Lemma~\ref{lem:asymptotics_Omega_hyperbolic}. In particular,~\eqref{eq:asymp1} implies that if $p_1 = (x_1,y_1) \in \BallSym{y}$, then because $y_1 < K$
\[
	|x_1 - e^{(y+y_1)/2} |\leq e^{(y+y_1)/2} \cdot K e^{y+y_1- R} = O(1) 
	e^{(y+y_1)/2} \cdot e^{y- R}.
\]
Therefore, 
\[
	\int_{-I_n}^{I_n} \int_0^K \ind{(x_1,y_1) \in \BallSym{(0,y)}} 
	e^{-\alpha y_1}  dx_1 dy_1 = O(1) \cdot e^{y-R} 
	\cdot \int_0^K e^{(y+y_1)/2} 
	e^{-\alpha y_1}   dy_1 = O(1)\cdot e^{3y/2 - R}, 
\]
and hence
\begin{align*}
	\Exp { \sum_{{p_1,p_2 \in\Pcal \setminus \{(0,y)\} \atop y_1 < K,} \atop \text{distinct}}
	\ind{p_1\in \BallSym{y}} \ind{p_2\in \BallHyp{y}\cap \BallPo{y}}} 
	= O(1) \cdot e^{2y - R}.
\end{align*}
Now, we integrate this over $y$ to obtain that
\begin{align*}
	&\hspace{-30pt}\int_{\Kcal_{C}(k_n)} \Exp { \sum_{{p_1,p_2 \in\Pcal \setminus \{(0,y)\} \atop y_1 < K,} \atop \text{distinct}}
		\ind{p_1\in \BallSym{y}} \ind{p_2\in \BallHyp{y}\cap \BallPo{y}}} e^{-\alpha y} \dd y\\
	&= O(1) e^{-R} \int_{\Kcal_{C}(k_n)} e^{2y -\alpha y} \dd y
		= O(1) n^{-2}
			\begin{cases}
			k_n^{4-2\alpha}, & \mbox{if $\alpha < 2$} \\
			\log k_n, & \mbox{if $\alpha =2$} \\
			1, & \mbox{if $\alpha >2$}
			\end{cases}.
\end{align*}

To finish the argument assume first that $1/2 <\alpha \leq 3/4$. In this case,
\[
	k_n^{6\alpha -3}  n^{-2} k_n^{4 - 2\alpha} = 
	n^{-2} \cdot k_n^{4\alpha +1} = \smallO{1}.
\]
For $3/4 \le \alpha < 2$ we use that $2\alpha < 6\alpha - 3$, so that $k_n^{2\alpha} n^{-2} k_n^{4 - 2\alpha} = \smallO{1}$. Finally, when $\alpha \geq 2$, we have that
\[
	k_n^{2\alpha} (\log(k_n) \wedge 1) n^{-2} \le k_n^{2\alpha + 1} n^{-2} = \bigO{n^{-1}} = \smallO{1}.
\] 
which completes the proof for~\eqref{eq:sum_6} and thus the proof of Proposition~\ref{prop:couling_c_H_P}.
\end{proofof}

\subsection{Coupling \texorpdfstring{$G_n$}{Gn} to \texorpdfstring{$\GPo$}{G Po}}\label{ssec:coupling_H_HP}

Now that we have established the equivalence of the clustering function between the Poissonized KPKVB graph $\GPo$ and the finite box graph $\Gbox$ the final step is to relate the clustering function in $\GPo$ to the KPKVB graph $G_n$. As mentioned in Section~\ref{ssec:KPKVB_to_GPo_infinite_k}, this is done by moving from $c(k_n; G_n)$ to the adjusted clustering function $c^\ast(k_n; G_n)$ (Lemma~\ref{lem:clustering_ast_H}) and then to $c^\ast(k_n; \GPo)$ (Proposition~\ref{prop:clustering_ast_H_Pois}). For this we will use the coupling result (Lemma~\ref{lem:diff_Nk_hyperbolic_binomial_poisson}) from Section~\ref{ssec:coupling_Gn_GPo}. We first give the proof of Proposition~\ref{prop:clustering_ast_H_Pois} and after that we prove Lemma~\ref{lem:clustering_ast_H}. Recall that Proposition~\ref{prop:clustering_ast_H_Pois} states
\[
	\lim_{n \to \infty} s(k_n)\Exp{\left|c^\ast(k_n; G_n) - c^\ast(k_n; \GPo)\right|} = 0.
\]

\begin{proofof}{Proposition~\ref{prop:clustering_ast_H_Pois}}
First we note that Proposition~\ref{prop:couling_c_H_P}, \ref{prop:concentration_local_clustering_P_n} and~\ref{prop:convergence_average_clustering_P_n} together imply that
\[
	\Exp{c^\ast(k_n; \GPo)} = (1+\smallO{1})s(k_n)
\]
Therefore it suffices to show that
\[
	\Exp{\left|c^\ast(k_n; G_n) - c^\ast(k_n; \GPo)\right|} = \smallO{\Exp{c^\ast(k_n; \GPo)}}.
\]
For this we observe that we are looking at the modified clustering coefficient, where we divide by the expected number of degree $k_n$ vertices. As the expected numbers of degree $k_n$ vertices in $\GPo$ and $G_n$ are asymptotically equivalent (see Lemma~\ref{lem:diff_Nk_hyperbolic_binomial_poisson}), it is therefore sufficient to consider the sum of the clustering coefficients of all vertices of degree $k_n$.
Given again the standard coupling between the binomial and Poisson process (as used in the proof of Lemma~\ref{lem:diff_Nk_hyperbolic_binomial_poisson}), we again denote by $V_n(k_n)$ the set of degree $k_n$ vertices in $G_n$ and by $\VPo(k_n)$ the set of degree $k_n$ vertices in $\GPo$. If a vertex is contained in both sets, it must have the same degree in both the Poisson and KPKVB graph, and given the nature of the coupling, the neighbourhoods are therefore the same and hence also their clustering coefficients agree.

The difference of the sum of the clustering coefficients therefore comes from all the clustering coefficients of the symmetric difference $V_n(k_n) \Delta \VPo(k_n)$. By Lemma~\ref{lem:diff_Nk_hyperbolic_binomial_poisson} the expected number vertices in this set is $\Exp{\left|N_{n}(k_n) - N_{\Po}(k_n)\right|} = \smallO{\Exp{N_{\Po}(k_n)}}$. Therefore we have that
\begin{align*}
	\Exp{\left|c^\ast(k_n; G_n) - c^\ast(k_n; \GPo)\right|}
	&\le \frac{\Exp{\left|N_{n}(k_n)-N_{\Po}(k_n)\right|}}{(1+\smallO{1})\Exp{N_{\Po}(k_n)}} \Exp{c^\ast(k_n; \GPo)}
	= \smallO{1}\Exp{c^\ast(k_n; \GPo)},
\end{align*}
which finishes the proof.
\end{proofof}

Finally we prove Lemma~\ref{lem:clustering_ast_H}, whose statement is
\[
\left|c^\ast(k_n; G_n) - c(k_n; G_n)\right| = \smallOp{s(k_n)}.
\]

\begin{proofof}{Lemma~\ref{lem:clustering_ast_H}}
	Since Propositions~\ref{prop:clustering_ast_H_Pois}-\ref{prop:convergence_average_clustering_P_n} imply that
	\[
	\Exp{c^\ast(k_n; G_n)} = \bigO{s(k_n)},
	\]  
	and since
	\begin{align*}
	\frac{|N_n(k_n)-\Exp{N_{n}(k_n)}|}{N_{n}(k_n)} = \smallOp{1},
	\end{align*}
	we immediately infer that
	\begin{align*}
	\left|c^\ast(k_n; G_n) - c(k_n; G_n)\right|
	&= c^\ast(k_n; G_n)\left|\frac{\Exp{N_{n}(k_n)}}{N_{n}(k_n)} - 1\right| = \smallOp{s(k_n)}.
	\end{align*}
\end{proofof}

\paragraph{Acknowledgments} We are grateful to Dmitri Krioukov for pointing us to the problem of local clustering in this model and his helpful insights during discussions of the topic. We thank Remco van der Hofstad for pointing out a mistake in an earlier version of the proof of Theorem~\ref{thm:mainktoinfty}. We also thank an anonymous referee for his/her suggestions and comments on improving the manuscript. 

\bibliographystyle{plain}
\bibliography{references}

\begin{appendices}

\section{Meijer's G-function}\label{sec:Meijer_G_functions}

Recall that $\Gamma(z)$ denotes the Gamma function. Let $p, q, m, \ell$ be four integers satisfying $0 \le m \le q$ and $0 \le \ell \le p$ and consider two sequences ${\bf a}_p = \{a_1, \dots, a_p\}$ and ${\bf b}_q = \{b_1, \dots, b_q\}$ of reals such that $a_i - b_j$ is not a positive integer for all $1 \le i \le p$ and $1 \le j \le q$ and $a_i - a_j$ is not an integer for all distinct indices $1 \le i, j \le p$. Then, with $\iota$ denoting the complex unit, Meijer's G-Function~\cite{meijer1946gfunction} is defined as
\begin{equation}\label{eq:def_Meijer_G_function}
	\MeijerGnew{m}{\ell}{p}{q}{{\bf a}}{{\bf b}}{z} 
	= \frac{1}{2 \pi \iota} \int_L 
	\frac{\prod_{j = 1}^{m}\Gamma(b_j - t) \prod_{j = 1}^\ell \Gamma(1 - a_j + t)}
	{\prod_{j = m + 1}^{q}\Gamma(1 - b_j + t) \prod_{j = \ell + 1}^p\Gamma(a_j-t)} \, z^t \dd t,
\end{equation}
where the path $L$ is an upward oriented loop contour which separates the poles of the function $\prod_{j = 1}^{m}\Gamma(b_j - t)$ from those of $\prod_{j = 1}^n \Gamma(1 - a_j + t)$ and begins and ends at $+\infty$ or $-\infty$.

The Meijer's G-Function is of very general nature and has relation to many known special functions such as the Gamma function and the generalized hypergeometric function. For more details, such as many identities for $\MeijerGnew{m}{\ell}{p}{q}{{\bf a}}{{\bf b}}{z}$ see \cite{gradshteyn2015table,luke2014mathematical}.

For our purpose we need the following identity which follows from an Mellin transform operation.

\begin{lemma}\label{lem:gamma_meijer_G}
For any $a\in \R$ and $\xi, s>0$,
$$\Gamma^+(-a-1,\xi/s) = \MeijerGnew{2}{0}{1}{2}{1}{-a-1,0}{\frac{\xi}{s}}$$
\end{lemma}
\begin{proof}
Let $x>0$ and $q\in\R$ and note that as the $\Gamma$-function is the Mellin transform of $e^{-x}$, by the inverse Mellin transform formula, we have $e^{-x}=\frac{1}{2\pi \iota}\int_{c-\iota\infty}^{c+\iota\infty} \Gamma(p)x^{-p}dp$ for $c>0$ (see \cite[p.196]{davies2012integral}). Applying the change of variable $p(r)=q-r$ yields $e^{-x}=\frac{1}{2\pi \iota}\int_{c+q-\iota\infty}^{c+q+\iota\infty} \Gamma(q-r) x^{r-q}dr$, then multiplying both sides with $-x^{q-1}$ gives $-x^{q-1}e^{-x} = -\frac{1}{2\pi \iota}\int_{c+q-\iota\infty}^{c+q+\iota\infty} \Gamma(q-r) x^{r-1}dr$. Now, integrating both sides gives $\int_x^\infty t^{q-1}e^{-t}dt = \frac{1}{2\pi \iota}\int_{c+q-\iota\infty}^{c+q+\iota\infty}\frac{\Gamma(q-r)}{-r}x^r dr$. On the left-hand side is the incomplete gamma function and on the right-hand side with using $-r= \frac{\Gamma(1-r)}{\Gamma(-r)}$ is the Meijer $G$-function, i.e. $\Gamma^+(q,x)=\MeijerGnew{2}{0}{1}{2}{1}{q,0}{x}$. The claim follows by plugging in $q=-a-1$ and $x=\frac{\xi}{s}$.
\end{proof}

\section{Incomplete Beta function}\label{sec:beta_function}

Here we derive the asymptotic behavior for the function $B^-(1-z; 2\alpha, 3-4\alpha )$ as $z \to 0$, which is used to analyze the asymptotic behavior of $P(y)$, see Section~\ref{ssec:asymptotics_local_clustering_P}.

\begin{lemma}\label{lem:asymptotics_incomplete_beta}
We have the following asymptotic results for $B^-(1-z; 2\alpha, 3-4\alpha )$
\begin{enumerate}
\item For $1/2 < \alpha < 3/4$
\[
	\lim_{z \to 0} B^-(1-z, 2\alpha, 3-4\alpha ) = B(2\alpha, 3 - 4\alpha).
\]
\item When $\alpha = 3/4$,
\[
	\lim_{z \to 0} \frac{B^-(1-z, 2\alpha, 3-4\alpha)}{\log(z)} = -1.
\]
\item For $\alpha > 3/4$,
\[
	\lim_{z \to 0} z^{4\alpha - 3} B^-(1-z, 2\alpha, 3-4\alpha ) = \frac{1}{4\alpha - 3}.
\]
\end{enumerate}
\end{lemma}

\begin{proof}
We use the hypergeometric representation of the incomplete Beta function,
\[
	B^-(x, a, b) = \frac{x^a}{2a}F(a, 1-b,a+1,x),
\]
where $F$ denote the hypergeometric function~\cite{temme2011special} (or see~\cite[Section 8.17 (ii)]{dlmf2019digital}). In particular we have that
\[
	B^-(1-z; 2\alpha, 3-4\alpha ) = \frac{(1-z)^{2\alpha}}{2\alpha} F(2\alpha,4\alpha-2,2\alpha+1,1-z).
\]

The behavior of $F(a,b,c,1-z)$ as $z \to 0$ depend on the real part of the sum of $c - a - b$ and whether $c = a + b$~\cite{andrews2000special} (or see~\cite[Section 15.4(ii)]{dlmf2019digital}). Since in our case $a,b,c$ will be real it only depends on the sum of $c - a - b$. For $c - a - b > 0$ we have
\begin{equation}\label{eq:F_cab_less_than_zero}
	\lim_{z \to 0} F(a,b,c,1-z) = \frac{\Gamma(c)\Gamma(c-a-b)}{\Gamma(c-a)\Gamma(c-b)},
\end{equation}
if $c = a + b$ then
\begin{equation}\label{eq:F_cab_zero}
	\lim_{z \to 0} \frac{F(a,b,c,1-z)}{\log(z)} = -\frac{\Gamma(a+b)}{\Gamma(a)\Gamma(b)},
\end{equation}
and finally, when $c - a - b < 0$
\begin{equation}\label{eq:F_cab_greater_than_zero}
	\lim_{z \to 0} \frac{F(a,b,c,1-z)}{z^{c - a - b}} = \frac{\Gamma(c)\Gamma(a + b -c)}{\Gamma(a)\Gamma(b)}.
\end{equation}

In our case we have,
\[
	B^-(1-z; 2\alpha, 3-4\alpha ) = \frac{(1-z)^{2\alpha}}{2\alpha} F(a,b,c,1-z),
\]
with $a := 2\alpha$, $b: = 4\alpha-2$ and $c := 2\alpha + 1$. Therefore,
\[
	c - a - b = 2\alpha + 1 - 2\alpha -(4\alpha - 2) = 3 - 4\alpha.
\]

Now if $\alpha < 3/4$ then $c - a - b > 0$ and hence
\[
	\lim_{z \to 0} B^-(1-z; 2\alpha, 3-4\alpha ) = \frac{1}{2\alpha} \frac{\Gamma(2\alpha + 1)\Gamma(3 - 4\alpha)}{\Gamma(1)\Gamma(3-2\alpha)} = \frac{\Gamma(2\alpha)\Gamma(3 - 4\alpha)}{\Gamma(3-2\alpha)}
	= B(2\alpha, 3 - 4\alpha),
\]
where we used that $\Gamma(2\alpha + 1) = 2\alpha \Gamma(2\alpha)$.

When $\alpha =3/4$ then $c - a - b = 0$ and therefore~\eqref{eq:F_cab_zero}, together with the fact that $(1-z)^{3/2} \sim 1$ as $z \to 0$, implies that
\begin{align*}
	\lim_{z \to 0} \frac{B^-(1-z; 2\alpha, 3-4\alpha )}{\log(z)} 
	= -\frac{1}{2\alpha} \frac{\Gamma(6\alpha - 2)}{\Gamma(2\alpha)\Gamma(4\alpha - 2)}
	= - \frac{\Gamma(5/2)}{\frac{3}{2}\Gamma(3/2)} = -1.
\end{align*}

Finally, when $\alpha > 3/4$, $c - a - b = 3 - 4\alpha < 0$ and using~\eqref{eq:F_cab_greater_than_zero} we get
\begin{align*}
	\lim_{z \to 0} z^{4\alpha - 3} B^-(1-z, 2\alpha, 3-4\alpha ) 
	= \frac{1}{2\alpha}\frac{\Gamma(2\alpha + 1)\Gamma(4\alpha - 3)}{\Gamma(2\alpha)\Gamma(4\alpha - 2)}
	= \frac{\Gamma(4\alpha - 3)}{\Gamma(4\alpha - 2)} = \frac{1}{4\alpha - 3}.
\end{align*}
\end{proof}

\section{Some results on functions}

\begin{lemma}\label{lem:arccos_approx}
For any $0 < \lambda < 1$ there exists a $K > 0$, such that for all $0 < x \le (1 - \lambda)2$
\[
	\frac{1}{2}\arccos(1-x)
	\le \frac{x}{\sqrt{1-(1-x)^2}} 
	\le \frac{1}{2}\arccos(1-x)\left(1 + x\right).
\]
In particular, as $x \to 0$,
\[
	\frac{x}{\sqrt{1-(1-x)^2}} \sim \frac{1}{2}\arccos(1-x).
\]
\end{lemma}

\begin{proof}
First we observe that for all $0 < x < 2$
\[
	0 < \sqrt{2x}\left(1 - \frac{x}{\sqrt{8}}\right) \le \arccos(1-x) \le 
	\sqrt{2x}\left(1 + \frac{x}{\sqrt{8}}\right)
\]
while for every $0 < \lambda < 1$, there exists a $K > 0$ such that for all $0 < x \le (1-\lambda) 2$,
\[
	0 < \frac{1}{\sqrt{2x}}\left(1 - \frac{x}{2}\right) \le \frac{1}{\sqrt{1 - (1 - x)^2}} \le
	\frac{1}{\sqrt{2x}}\left(1 + K x\right).
\]
It then follows that for all $0 < x \le (1-\lambda) 2$,
\begin{align*}
	\frac{x}{\sqrt{1 - (1-x^2)}} &\le \frac{1}{2} \sqrt{2x}\left(1 + K\frac{x}{\sqrt{2}}\right)
		\le \frac{1}{2} \arccos(1-x) \frac{1 + Kx}{1 - \frac{x}{\sqrt{8}}}
		\le \frac{1}{2} \arccos(1-x)\left(1 + \frac{(K + 1)x}{1 - x}\right),
\end{align*}
and
\begin{align*}
	\frac{x}{\sqrt{1 - (1-x^2)}} &\ge \frac{1}{2} \sqrt{2x}\left(1 - \frac{x}{2}\right)
		\ge \frac{1}{2} \arccos(1-x) \frac{1 - \frac{x}{2}}{1 + \frac{x}{\sqrt{8}}}
		\ge \frac{1}{2} \arccos(1-x)\left(1 - \frac{(1+\sqrt{2})x}{1 + x}\right),
\end{align*}
which finishes the proof.
\end{proof}

\section{Some results for random variables}

Here we summarize several known results for random variables and provide one technical lemma for Binomial random variables.



First of all, we recall two versions of the Chernoff bound for Poisson and Binomial random variables. They can be found in \cite[Lemma 1.2]{penrose2003random}; note that the Chernoff bound exists in many different versions, the original idea was developed by Chernoff in the context of efficiency of statistical hypothesis testing in \cite{Chernoff1952}):.

\begin{lemma}
Let $\Po(\lambda)$ denote a Poisson random variable with mean $\lambda$ and let $H(x) = x\log(x) - x + 1$. Then
\begin{align*}
	&\Prob{\Po(\lambda) \ge k} \le e^{-\lambda H(k/\lambda)} \quad \text{for all } k \ge \lambda\\
	&\Prob{\Po(\lambda) \le k} \le e^{-\lambda H(k/\lambda)} \quad \text{for all } k \le \lambda.
\end{align*}
\end{lemma}

It follows from the above lemma that
\begin{equation*}
	\Prob{\left|\mathrm{Po}(\lambda) - \lambda\right| \ge x} \le 2e^{-\frac{x^2}{2(\lambda + x)}}.
\end{equation*}
In particular, if $\lambda_n \to \infty$, then, for any $C>0$,
\begin{equation*}
	\Prob{\left|\mathrm{Po}(\lambda_n) - \lambda_n\right| \ge C \sqrt{\lambda_n\log(\lambda_n)}} \le 2e^{-\frac{C^2 \lambda_n \log(\lambda_n}{2\left(\lambda_n + C\sqrt{\lambda_n\log(\lambda_n)}\right)}}
	= \bigO{\lambda_n^{-\frac{C^2}{2}}}.
\end{equation*}

Note that these are equations~\eqref{eq:def_chernoff_bound_poisson} and~\eqref{eq:def_chernoff_bound_poisson_C} from the main text.

Let $\mathrm{Bin}(n,p)$ denote a Binomial random variable with $n$ trials and success probability $p$, and $0 < \delta < 1$. Then we have the following well-known Chernoff bound.
\begin{equation}\label{eq:def_chernoff_bound_binomial}
	\Prob{|\mathrm{Bin}(n,p) - np| > \delta np} \le e^{-\frac{\delta^2 np}{3}}.
\end{equation}

The following lemma gives an upper bound on the Binomial distribution for $p = \lambda/n$ in terms a Poisson distribution with mean $\lambda$.
The following lemma  gives a standard comparison between Binomial and Poisson distribution. We provide a short proof  for completeness.

\begin{lemma}\label{lem:binomial_poisson_bound}
Let $n \ge 1$, $0 < \lambda < n$. Then, for any integer $0 \le k \le n - 1$,
\[
	\Prob{\mathrm{Bin}(n,\lambda/n) = k} \le \frac{e}{\sqrt{2\pi}} \sqrt{\frac{n}{n-k}} \Prob{\Po(\lambda) = k}.
\]
\end{lemma}

\begin{proof}
Using Stirling's bounds (see e.g.~\cite{Dutkay2013},~\cite{Nanjundiah1959})
	\[
	\sqrt{2\pi s} \left(\frac{s}{e}\right)^{-s} \le s! \le e \sqrt{s} \left(\frac{s}{e}\right)^{-s},
	\]
we have
\begin{align*}
	\Prob{\mathrm{Bin}(n,\lambda/n) = k}
	&= \binom{n}{k} \left(\frac{\lambda}{n}\right)^k \left(1 - \frac{\lambda}{k}\right)^{n - k}\\
	&\le \frac{e}{\sqrt{2\pi}} \sqrt{\frac{n}{n - k}} \frac{n^n}{k!} (n-k)^{-(n-k)} e^{-k}
		\left(\frac{\lambda}{n}\right)^k \left(1 - \frac{\lambda}{n}\right)^{n - k}\\
	&= \frac{e}{\sqrt{2\pi}} \sqrt{\frac{n}{n - k}} \frac{\lambda^k e^{-\lambda}}{k!}
		\left(\frac{n-\lambda}{n - k}\right)^{n - k} e^{\lambda - k}\\
	&= \frac{e}{\sqrt{2\pi}} \sqrt{\frac{n}{n-k}} \Prob{\Po(\lambda) = k}
		\left(\frac{n-\lambda}{n - k}\right)^{n - k} e^{\lambda - k}.
\end{align*}
The result then follows by observing that $\left(\frac{n-\lambda}{n - k}\right)^{n - k} e^{\lambda - k} \le 1$ for all $0 < \lambda < n$ and $0 \le k \le n - 1$.
\end{proof}

\section{Concentration of heights for vertices with degree \texorpdfstring{$k$}{k}}\label{sec:concentration_argument}

Here we will prove Proposition~\ref{prop:concentration_height_general}. We start by considering integration with respect to the function $\rho(y,k_n) = \Prob{\Po(\mu(y)) = k_n}$ (the degree distribution of a typical point in $\Ginf$). Here we show that we may restrict integration with respect to the \emph{height} $y$ to the interval $\Kcal_C(k_n) = [y_{k_n,C}^-, y_{k_n,C}^+]$ on which $\mu(y) = \bigT{k_n}$. Next we show that if we consider any other measure $\hat{\mu}_n(y)$ that is sufficiently equivalent to $\mu(y)$ on this interval (which will be made precise later), then we may replace $\hat{\rho}_n(y,k_n) := \Prob{\Po(\hat{\mu}_n(y)) = k_n}$ in integrals with $\rho(y,k_n)$. This then implies that we can also restrict integration to the interval $\Kcal_C(k_n)$. We will refer to such results as a \emph{concentration of heights} result.

We start with a concentration of heights result for the infinite model $\Ginf$ (Lemma~\ref{lem:concentration_argument}). We then present a generalization of this result (Lemma~\ref{lem:concentration_argument}) and use this to establish concentration of heights results for the Poissonized KPKVB $\GPo$ and finite box model $\Gbox$. 

Finally we provide a general result that allow to substitute $\hat{\rho}_n(y,k_n)$ in the integrand with $\rho(y,k_n)$ and show that this holds in particular for the degree distributions in $\GPo$ and $\Gbox$, given by, respectively $\rho_{\Po}(y,k_n) := \Prob{\Po(\mu_{\Po}(y)) = k_n}$ and $\rho_{\text{box}}(y,k_n) := \Prob{\Po(\mu_{\mathrm{box}}(y)) = k_n}$.

\subsection{Concentration of heights argument for the infinite model}

The next lemma states that for a large class of functions $h(y)$ and $k_n \to \infty$, to compute the integral 
\[
	\int_{0}^\infty \rho(y,k_n) h(y) e^{-\alpha y} \dd y
\]
it is enough to consider integration over a small interval on which $e^{y/2} \approx k_n$, instead of $\R_+$. 

\begin{lemma}\label{lem:concentration_argument}
Let $\alpha > \frac{1}{2}$, $\nu > 0$, $(k_n)_{n \ge 1}$ be any positive sequence such that $k_n \to \infty$ and $k_n = \smallO{n}$. Then the following holds.

For any continuous function $h : \R_+ \rightarrow  \R$, such that $h(y) = \bigO{e^{\beta y}}$ as $y \to \infty$ for some $\beta < \alpha$, 
\begin{equation}\label{eq:error_bound_int_rho_not_K}
	\int_{\R_+ \setminus \Kcal_C(k_n)} \rho(y,k_n) h(y) \alpha e^{-\alpha y} \dd y
	= \bigO{k_n^{-C^2/2}},
\end{equation}
as $n \to \infty$.
\end{lemma}

\begin{proof}
Since $\mu^\prime(y) = \mu(y)/2$, we get that
\[
	\frac{\partial \rho(y,k)}{\partial y} = \frac{1}{2}\left(k - \mu(y)\right)\rho(y,k),
\]
which implies that $\rho(y,k)$ attains its maximum at $\mu(y) = k$. Moreover we see that the derivative is strictly positive when $\mu(y) < k$ and strictly negative when $\mu(y) > k$. Since $\mu(y_{k,C}^-) < k$ and $\mu(y_{k,C}^+) > k$, we conclude that $\rho(y,k)$, as a function of $y$, is strictly increasing on $[0,y_{k,C}^-]$ and strictly decreasing on $[y_{k,C}^+,\infty)$.

Therefore, by our assumption on $h(y)$,
\begin{align*}
	&\hspace{-20pt}\int_{\R_+ \setminus \Kcal_C(k_n)} h(y) \rho(y,k_n) 
		\alpha e^{-\alpha y} \dd y\\
    &= \bigO{1} \int_0^{y_{k_n,C}^-} e^{\beta y} \rho(y,k_n) \alpha e^{-\alpha y} \dd y 
    	+ \bigO{1}\int_{y_{k_n,C}^+}^{\infty} e^{\beta y} \rho(y,k_n) \alpha e^{-\alpha y} \dd y \\
    &= \bigO{1} \int_0^{y_{k_n,C}^-} \rho(y,k_n) e^{-(\alpha-\beta) y} \dd y 
   		+ \bigO{1} \int_{y_{k_n,C}^+}^{\infty} \rho(y,k_n) e^{-(\alpha-\beta) y} \dd y\\
   	&\le \bigO{1}\rho(y_{k_n,C}^-,k_n)\int_0^{y_{k_n,C}^-} e^{-(\alpha-\beta) y} \, \dd y
   		+ \bigO{1} \rho(y_{k_n,C}^+,k_n) \int_{y_{k_n,C}^+}^{\infty} e^{-(\alpha-\beta) y} \, \dd y.
\end{align*}
Since $\alpha - \beta > 0$, we conclude that
\begin{equation}\label{eq:concentration_lemma_integral_bound}
	\int_{\R_+ \setminus \Kcal_C(k_n)} h(y) \rho(y,k_n) \alpha e^{-\alpha y} \dd y
	= \bigO{1} \left(\rho(y_{k_n,C}^-,k_n) + \rho(y_{k_n,C}^+,k_n)\right).
\end{equation}

We shall now bound the terms $\rho(y_{k_n,C}^\pm,k_n)$. We explicitly show the bound for $\rho(y_{k_n,C}^+,k_n)$, the computation for $\rho(y_{k_n,C}^-,k_n)$ is similar. First note that $\mu(y_{k_n,C}^+) = k_n + C \sqrt{\frac{\log(k_n)}{k_n}}$. Hence we can write
\begin{align*}
	\rho(y_{k_n,C}^+,k_n) &= \Prob{\Po(\mu(y_{k_n,C}^+)) = k_n} \le \Prob{\Po(\mu(y_{k_n,C}^+)) \ge k_n}\\
	&\le \Prob{\left|\Po(\mu(y_{k_n,C}^+)) - \mu(y_{k_n,C}^+)\right| \ge C \sqrt{\frac{\log(k_n)}{k_n}}}.
\end{align*}
Apply the Chernoff bound~\eqref{eq:def_chernoff_bound_poisson_C} then yields
\begin{equation}\label{eq:concentration_lemma_bound_an+}
	\rho(y_{k_n,C}^+,k_n) = \bigO{k_n^{-C^2/2}}.
\end{equation}
A similar analysis yields
\begin{equation}\label{eq:concentration_lemma_bound_an-}
	\rho(y_{k_n,C}^-,k_n) \le \bigO{k_n^{-C^2/2}}.
\end{equation} 
Plugging~\eqref{eq:concentration_lemma_bound_an-} and~\eqref{eq:concentration_lemma_bound_an+}  into~\eqref{eq:concentration_lemma_integral_bound} yields the result.

\end{proof}

Note that we can tune the error in~\eqref{eq:error_bound_int_rho_not_K} by selecting an appropriately large $C > 0$, i.e. by restricting the function $h(y)$ inside the integral to an appropriate interval around $2\log(k_n/\xi)$. This makes Lemma~\ref{lem:concentration_argument} very powerful. As an example we give the following corollary, which allows us to bound integrals of functions $h_n(y)$ by considering their maximum of $\Kcal_{C}(k_n)$.

%

\begin{corollary}\label{cor:concentration_heights_bounds_n}
Let $h_n : \R_+ \to \R_+$ be a sequence of continuous functions which such that for some $s \in \R$ and $\beta < \alpha$, as $n \to \infty$, $h_n(y) = \bigO{k_n^{s} e^{\beta y}}$ and $h_n(y) = \Omega(1)$, uniformly on $0 \le y \le (1-\varepsilon)R$ for some $0 < \varepsilon < 1$. Then for large enough $C > 0$, as $n \to \infty$,
\[
	\int_{\R_+} h_n(y) \rho(y,k_n) e^{-\alpha y} \dd y 
	= (1 + \smallO{1}) \int_{\Kcal_{C}(k_n)} h_n(y) \rho(y,k_n) \alpha e^{-\alpha y} \dd y.
\]
In particular,
\[
	\int_{\R_+} h_n(y) \rho(y,k_n) e^{-\alpha y} \dd y = \bigO{1} k_n^{-(2\alpha + 1)} \max_{y \in \Kcal_{C}(k_n)} h_n(y),
\]
as $n \to \infty$.
\end{corollary}

\begin{proof}
The second result follows immediately from the first. For the first result we note that by Lemma~\ref{lem:concentration_argument}
\begin{align*}
	\int_{\R_+ \setminus \Kcal_C(k_n)} h_n(y) \rho(y,k_n) e^{-\alpha y} \dd y
	&\le \bigO{k_n^s} \int_{\R_+ \setminus \Kcal_C(k_n)} e^{\beta y } \rho(y,k_n) e^{-\alpha y} \dd y\\
	&= \bigO{k_n^{s - C^2/2}}.
\end{align*}
By assumption on $h_n(y)$,
\[
	\int_{\Kcal_{C}(k_n)} h_n(y) \rho(y,k_n) e^{-\alpha y} \dd y 
	= \bigO{k_n^{s+2\beta}} \int_{\Kcal_{C}(k_n)} \rho(y,k_n) e^{-\alpha y} \dd y
	= \bigO{k_n^{s+2\beta - (2\alpha + 1)}},
\]
and
\[
	\int_{\Kcal_{C}(k_n)} h_n(y) \rho(y,k_n) e^{-\alpha y} \dd y 
	= \Omega(1) \int_{\Kcal_{C}(k_n)} \rho(y,k_n) e^{-\alpha y} \dd y
	= \Omega(k_n^{-(2\alpha + 1)}).
\]
Hence, by taking $C > 0$ such that $C^2/2 > \max\{2\alpha + 1 + s, 2\alpha +1 - \beta\}$ we get that
\[
	\int_{\R_+ \setminus \Kcal_C(k_n)} h_n(y) \rho(y,k_n) e^{-\alpha y} \dd y
	= \smallO{1} \int_{\Kcal_{C}(k_n)} h_n(y) \rho(y,k_n) e^{-\alpha y} \dd y.
\]
\end{proof}

\subsection{Concentration of heights for the KPKVB and finite box model}\label{ssec:general_concentration_lemma}

Although powerful, the current versions of the concentration of heights argument is only valid for the function $\rho(y,k_n) := \Prob{\Po\left(\Mu{\BallPo{y}}\right) = k_n}$. We want to extend this to the Poissonized KPKVB model $\GPo$ and the finite box model $\Gbox$. To be more precise, recall that $\mu_{Po}(y) = \Mu{\BallHyp{y}}$ and $\mu_{\mathrm{box}}(y) = \Mu{\BallPon{y}}$ and let us define
\[
	\rho_{\Po}(y,k) = \Prob{\Po(\mu_{\Po}(y)) = k}
\] 
and
\[
	\rho_{\mathrm{box}}(y,k) = \Prob{\Po(\mu_{\mathrm{box}}(y) = k}.
\] 
Then we want when Lemma~\ref{lem:concentration_argument} to remain true if we replace $\rho(y,k_n)$ with either the function $\rho_{\Po}(y,k_n)$ or $\rho_{\text{box}}(y,k_n)$. To establish this result we first prove the following technical lemma.

\begin{lemma}\label{lem:concentration_heights_mu_approx}
Let $0 < \delta < 1$ and $k_n \to \infty$ be such that $k_n = \bigO{n^{1-\delta}}$. Let $\hat{\mu}_n(y)$ be a monotone increasing differentiable function such for some $0 < \varepsilon < 1$, $\hat{\mu}_n(y) = (1+\smallO{1})\mu(y)$ holds uniformly for $0 \le y \le (1-\varepsilon)R$. Furthermore, let $h : \R_+ \rightarrow  \R$, be a continuous function such that $h(y) = \bigO{e^{\beta y}}$ as $y \to \infty$ for some $\beta < \alpha$. Then, for $C > 0$ large enough
\[
	\int_0^\infty h(y) \hat{\rho}_n(y,k_n) e^{-\alpha y} \dd y \sim  
		\int_{\Kcal_C(k_n)} h(y) \rho(y,k_n) e^{-\alpha y} \dd y,
\]
as $n \to \infty$.
\end{lemma}

\begin{proof}
Take any $0 < \eta < \min\{\delta, \varepsilon\}$. We first show that we can restrict to integration to the interval $[0, (1-\eta)R)$. By construction $\eta < \varepsilon$, and hence by the assumption on $\hat{\mu}$ we have that $\hat{\mu}_n((1-\eta)R) = \bigT{\mu((1-\eta)R} = \bigT{n^{(1-\delta)}}$. Therefore, since $\eta<\delta$ and $k_n = \bigO{n^{1-\delta}}$, it follows that $\hat{\mu}_n((1-\eta)R)/k_n = \omega\left(n^{\delta - \eta}\right) = \omega(1)$ as $n \to \infty$. Hence $\hat{\rho}_n(y,k_n) \le \hat{\rho}((1-\eta)R,k_n)$ for all $y \ge (1-\eta)R$. It now follow that
\begin{align*}
	\int_{(1-\eta)R}^{R} h(y) \hat{\rho}_n(y,k_n) e^{-\alpha y} \dd y
	&= \bigO{1} \hat{\rho}_n((1-\eta)R,k_n) e^{-(\alpha-\beta)(1-\eta)R} \\
	&= \bigO{\hat{\rho}_n((1-\eta)R,k_n) n^{-2(\alpha-\beta)(1-\eta)}},
\end{align*}
where we used that that $h(y) = \bigO{e^{\beta y}}$. Next we use Stirling's bound $k! \ge \sqrt{2\pi} k^{k+\frac{1}{2}} e^{-k}$, to bound $\hat{\rho}_n((1-\eta)R,k_n)$,
\begin{align*}
	\hat{\rho}_n((1-\eta)R,k_n) &= \Prob{\Po(\hat{\mu}_n((1-\eta)R)) = k_n} \\
	&= \frac{\hat{\mu}_n((1-\eta)R)^{k_n}}{k_n!} e^{-\hat{\mu}_n((1-\eta)R)}\\
	&= \bigO{1} k_n^{-1/2} \left(\frac{\hat{\mu}_n((1-\eta)R)}{k_n}\right)^{k_n} e^{k_n - \hat{\mu}_n((1-\eta)R)}\\
	&= \bigO{1} k_n^{-1/2} e^{k_n\left(1 - \frac{\hat{\mu}_n((1-\eta)R)}{k_n}+ \log\left(\frac{\hat{\mu}_n((1-\eta)R)}{k_n}\right)\right)}\\
	&\le \bigO{1} k_n^{-1/2} e^{-\hat{\mu}_n((1-\eta)R)/2},
\end{align*}
where the last line follows since $\hat{\mu}_n((1-\eta)R)/k_n \to \infty$ and $1 - x + \log(x) \le -x/2$ for large enough $x$. Because $\hat{\mu}_n((1-\eta)R) = \bigT{n^{(1-\delta)}}$ we conclude that for any $C > 0$
\[
	\int_{(1-\delta)R}^{R} h(y) \hat{\rho}_n(y,k_n) e^{-\alpha y} \dd y
	= \bigO{k_n^{-1/2} n^{-2(\beta-\alpha)(1-\delta)} e^{-n^{(1-\delta)}/2}}
	= \bigO{k_n^{-C^2/2}}.
\]

It thus remains to prove that 
\[
	\int_{y_{k_n,C}^+}^{(1-\eta)R} h(y) \hat{\rho}_n(y,k_n) e^{-\alpha y} \dd y
	\le \bigO{k_n^{-C^2/8}}
\]
and
\[
	\int_0^{y_{k_n,C}^-} h(y) \hat{\rho}_n(y,k_n) e^{-\alpha y} \dd y
	\le \bigO{k_n^{-C^2/8}}.
\]

Define $\hat{y}_n^\pm$ to be such that $\hat{\mu}_n(y_n^\pm) = k_n \pm C \sqrt{k_n \log(k_n)}$. Then by assumption on $\hat{\mu}_n$ we have that
\[
	k_n \pm C \sqrt{k_n \log(k_n)} = \hat{\mu}_n(\hat{y}_n^\pm) = (1+\smallO{1}) \mu(\hat{y}_n^\pm)
	= (1+\smallO{1}) \xi e^{\hat{y}_n^\pm/2}.
\]
and hence
\[
	\hat{y}_n^\pm = 2\log\left(\frac{k_n \pm C \sqrt{k_n \log(k_n)}}{\xi}\right) - 2\log(1+\smallO{1})
	= y_{k_n,C}^\pm - 2\log(1+\smallO{1}) := y_{k_n,C}^\pm - \epsilon_n,
\]
with $\epsilon_n \to 0$. Recall that $\hat{\mu}_n(y)$ is monotonic increasing. Now let $n$ be large enough such that  $\hat{\mu}_n(\hat{y}_n^+ - \epsilon_n) > k_n + \frac{C}{2}\sqrt{k_n \log(n)}$. Then 
\begin{align*}
	\int_{y_{k_n,C}^+}^{(1-\eta)R} h(y) \hat{\rho}_n(y,k_n) e^{-\alpha y} \dd y
	&\le \int_{\hat{y}_n^+ - \epsilon_n}^{(1-\eta)R} h(y) \hat{\rho}_n(y,k_n) e^{-\alpha y} \dd y\\
	&\le \hat{\rho}_n(\hat{y}_n^+ - \epsilon_n,k_n) \int_{\hat{y}_n^+ + \eta}^{(1-\eta)R} h(y) e^{-\alpha y} \dd y\\
	&\le \bigO{1} \hat{\rho}_n(\hat{y}_n^+ - \epsilon_n,k_n),
\end{align*}
where we used that $\hat{\mu}_n(y)$ is monotonic increasing and $\hat{\rho}_n(y,k_n)$ is decreasing for all $y \ge \hat{y}_n^+$. Write $\lambda_n = \hat{\mu}_n(\hat{y}_n^+ - \epsilon_n)$. Then, similar to the proof of Lemma~\ref{lem:concentration_argument}, we have that
\[
	\hat{\rho}_n(\hat{y}_n^+ - \epsilon_n,k_n)
	\le \Prob{\left|\Po(\lambda_n) - \lambda_n\right| \ge \frac{C}{2} \sqrt{k_n \log(k_n)}}
	\le \bigO{k_n^{-C^2/8}},
\]
where the last step follows from the Chernoff bound~\eqref{eq:def_chernoff_bound_poisson_C} with $C = C/2$.

In a similar fashion we can let $n$ be large enough such that
$\hat{\mu}_n(\hat{y}_n^- + \epsilon_n) < k_n - \frac{C}{2}\sqrt{k_n \log(n)}$ can show that
\begin{align*}
	\int_0^{y_{k_n,C}^-} h(y) \hat{\rho}_n(y,k_n) e^{-\alpha y} \dd y
	&\le \int_0^{\hat{y}_n^- + \epsilon_n} h(y) \hat{\rho}_n(y,k_n) e^{-\alpha y} \dd y\\
	&\le \bigO{1} \hat{\rho}_n(\hat{y}_n^- + \epsilon_n,k_n) = \bigO{k_n^{-C^2/8}}.
\end{align*}
\end{proof}

The conclusion of Lemma~\ref{lem:concentration_heights_mu_approx} is that as long as $\mu_{Po}(y)$ and $\mu_{\mathrm{box}}(y)$ are $(1+\smallO{1})\mu(y)$, uniformly on $[0,(1-\varepsilon)R]$, then indeed the concentration of height result (Lemma~\ref{lem:concentration_argument}) also holds in both $\GPo$ and $\Gbox$. This was proven in Lemma~\ref{lem:average_degree_P_n} and Lemma~\ref{lem:average_degree_G_box}, respectively. For completeness we give the proof of Proposition~\ref{prop:concentration_height_general}.

\begin{proofof}{Proposition~\ref{prop:concentration_height_general}}
The proof for $\hat{\mu}(y) = \mu(y)$ follows directly from Lemma~\ref{lem:concentration_argument}.

Now consider the case $\hat{\mu}(y) = \mu_{Po}(y)$. Then by Lemma~\ref{lem:average_degree_P_n} $\hat{\mu}(y) = (1+\smallO{1}) \mu(y)$, uniformly on $[0,(1-\varepsilon)R]$ and thus in particular on $\Kcal_C(k_n)$. Finally we note that by Lemma~3.3. in \cite{gugelmann2012random} $\hat{\mu}(y)$ is monotonic increasing. The statement then follows by applying Lemma~\ref{lem:concentration_heights_mu_approx}.

Finally, for $\hat{\mu}(y) = \mu_{\mathrm{box}}(y)$ we recall that by Lemma~\ref{lem:average_degree_G_box} $\hat{\mu}(y) = (1+\smallO{1})\mu(y)$. More precisely,
\[
	\hat{\mu}(y) = \mu(y)(1-\phi_n(y)) = \begin{cases}
		\mu(y)\left(1 - e^{-(\alpha-\frac{1}{2})R}\right) &\mbox{if } 0 \le y \le 2\log(\pi/2) \\
		\mu(y)\left(1 - \phi_n(y)\right) &\mbox{if } 2\log(\pi/2) < y \le (1-\varepsilon)R,
	\end{cases}
\]
where
\[
	\phi_n(y) :=  \left(\frac{\pi}{2}\right)^{-(2\alpha - 1)} \hspace{-3pt} e^{-(\alpha - \frac{1}{2})(R - y)}
				+ \frac{\nu}{\xi}e^{-(\alpha - \frac{1}{2})R - \frac{y}{2}} - \frac{\nu}{\xi}\left(\frac{\pi}{2}\right)^{-2\alpha} e^{-(\alpha-\frac{1}{2})(R - y)}.
\]
Note that $\phi_n(2\log(\pi/2)) = e^{-(\alpha-\frac{1}{2})R}$. In addition, since $|\phi_n(y)| \le \bigO{e^{-(\alpha-\frac{1}{2})\varepsilon R}}$ for $0 \le y \le (1-\varepsilon)R$ we have that $\hat{\mu}(y)$ is monotonic increasing for large enough $n$. The statement now follows by applying Lemma~\ref{lem:concentration_heights_mu_approx}.
\end{proofof}

\section{Derivative of \texorpdfstring{$\mu_{Po}(y)$}{mu Po (y)}}

Recall that $\mu_{Po}(y) = \Mu{\BallHyp{y}}$ denote the measure of the ball at height $y$ in the KPKVB model and $\mu(y) = \xi e^{\frac{y}{2}}$ denotes the measure of a ball at height $h$ in the infinite model $\Ginf$. In this section we will show that $\mu_{Po}^\prime(y) = (1+\smallO{1})\mu^\prime(y)$, uniformly on $[0,(1-\varepsilon)R]$, for some $0 < \varepsilon < 1$. This is a technical result that is needed in the proof of Lemma~\ref{lem:degree_integral} in Section~\ref{ssec:expected_degrees_GPo}.

First we note that it follows from Lemma~\ref{lem:average_degree_P_n} that $\mu_{Po}(y) = \mu(y)(1 + \phi_n(y))$, where $\phi_n(y) := \mu_{Po}(y)/\mu(y) - 1$. Taking the derivative we have
\[
	\mu_{Po}^\prime(y) = \mu^\prime(y)(1 + \phi_n(y)) + \mu(y)\phi_n^\prime(y)
	= \mu^\prime(y)(1 + \phi_n(y) + 2 \phi_n^\prime(y)),
\] 
where we used that $\frac{\partial}{\partial y} \mu(y) = \frac{1}{2}\mu(y)$. Hence, to show the result for we thus need to show that $\phi_n^\prime(y)) = \smallO{1}$, uniformly on $[0,(1-\varepsilon)R]$. 

Writing out the derivative we have
\[
	\phi_n^\prime(y) = 
	\mu_{Po}(y) ^{-1} \frac{\partial}{\partial y} \mu_{Po}(y) -  \frac{1}{2} \frac{\mu_{Po}(y) }{\mu(y)},
\]
where we used again that $\frac{\partial}{\partial y} \mu(y) = \frac{1}{2}\mu(y)$. For the second term Lemma~\ref{lem:average_degree_P_n} implies that $\frac{1}{2} \frac{\mu_{Po}(y) }{\mu(y)} = (1+\smallO{1})\frac{1}{2}$ uniformly on $[0,(1-\varepsilon)R]$. The following lemma shows that the same holds for the first term from which we conclude that $\phi_n^\prime(y)) = \smallO{1}$ and hence $\mu_{Po}^\prime(y) = (1+\smallO{1})\mu^\prime(y)$, uniformly on $[0,(1-\varepsilon)R]$.

\begin{lemma}\label{lem:derivative_mu_Po}
For any $0 < \varepsilon < 1$,
\[
	\lim_{n \to \infty} \sup_{0 \le y \le (1-\varepsilon)R} \left|\mu(y)^{-1}
	\frac{\partial}{\partial y} \mu_{Po}(y) - \frac{1}{2}\right| = 0.
\]
\end{lemma}

\begin{proof}
We again split $\mu_{\Po}(y)$ over the top and bottom part,
\[
	\mu_{\Po}(y) 
	= \Mu{\BallHyp{y} \cap \Rcal ([0,R - y))} + \Mu{\BallHyp{y} \cap \Rcal ([R - y, R])},
\]
where
\[
	\Mu{\BallHyp{y} \cap \Rcal ([0,R - y))} = \frac{2\alpha \nu}{\pi}\int_0^{R - y} \Phi(y,y^\prime) 
		e^{-\alpha y^\prime} \dd y^\prime,
\]
with $\Phi(y,y^\prime)$ defined as in~\eqref{eq:def_Omega_hyperbolic}. For the second term we have
\[
	\Mu{\BallHyp{y} \cap \Rcal ([R - y, R])}
	= \int_{R-y}^R \int_{I_n} f(x^\prime,y^\prime) \dd x^\prime \dd y^\prime
	= \xi e^{y/2}\frac{2\alpha - 1}{4\pi} \left( e^{-(\alpha - \frac{1}{2})(R - y)}
	- e^{-(\alpha - \frac{1}{2})R - y/2}\right).
\]
Taking the derivative of the last expression gives
\begin{align*}
	&\frac{\partial}{\partial y} \Mu{\BallHyp{y} \cap \Rcal ([R - y, R])}\\
	&= \frac{1}{2}\Mu{\BallHyp{y} \cap \Rcal ([R - y, R])}
		+ \xi e^{y/2}\frac{2\alpha - 1}{4\pi}\left(
		\left(\alpha - \frac{1}{2}\right)e^{-(\alpha - \frac{1}{2})(R - y)} 
		+ \frac{1}{2}e^{-(\alpha - \frac{1}{2})R - y/2}\right)\\
	&= \frac{1}{2}\Mu{\BallHyp{y} \cap \Rcal ([R - y, R])}\left(1 +
		\frac{(2\alpha - 1)e^{-(\alpha - \frac{1}{2})(R - y)} + e^{-(\alpha - \frac{1}{2})R - y/2}}
		{e^{-(\alpha - \frac{1}{2})(R - y)} - e^{-(\alpha - \frac{1}{2})R - y/2}}\right).
\end{align*}
Since, $\lim_{n \to \infty} \sup_{0 < y \le (1-\varepsilon)R} \Mu{\BallPo{y}}^{-1} \Mu{\BallHyp{y} \cap \Rcal ([R - y, R])} = 0$, we are left to show that
\begin{equation}\label{eq:derivative_mu_hyp_ball_main}
	\lim_{n \to \infty} \sup_{0 < y \le (1-\varepsilon)R} \left|\Mu{\BallPo{y}}^{-1} \frac{2\alpha \nu}{\pi} \frac{\partial}{\partial y} \int_0^{R - y} \Phi(y,y^\prime) e^{-\alpha y^\prime} \dd y^\prime
	- \frac{1}{2}\right| = 0.
\end{equation}

We start with some preliminary computations. For convenience we define
\[
	\Xi(y,y^\prime) = 1 - \frac{\cosh(R- y)\cosh(R-y^\prime) - \cosh(R)}{\sinh(R - y) \sinh(R - y^\prime)},
\]
so that
\[
	\Phi(y, y^\prime) = \frac{1}{2}e^{R/2} \arccos\left(1 - \Xi(y,y^\prime)\right).
\]
Next, following the same calculation as in the proof of~\cite[Lemma 28]{fountoulakis2018law}, we write
\begin{align*}
	\Xi(y,y^\prime)
	&= 2 e^{-(R - y - y^\prime)} \frac{\left(1 - e^{y^\prime - y - R}\right)\left(1 - e^{y - y^\prime - R}\right)}
		{\left(1 - e^{-2(R - y^\prime)}\right)\left(1 - e^{-2(R- y)}\right)}\\
	&:= 2 e^{-(R - y - y^\prime)} \frac{h_1(y) h_2(y)}{h_3(y^\prime) h_3(y)},
\end{align*}
with
\[
	h_1(y) = 1 - e^{y^\prime - y - R}, \quad h_2(y) = 1 - e^{y - y^\prime - R}
	\quad \text{and} \quad h_3(y) = 1 - e^{-2(R- y)}.
\]
We suppressed the dependence on $n$ and, in some cases, on $y^\prime$ for notation convenience.

We make two important observations. First, $\Xi(y,y^\prime)$ is an increasing function in both arguments, for $y, y^\prime < R$ and $y + y^\prime < R$. Second, for all $y + y^\prime < R$, $h_1(y) \le h_3(y^\prime)$ and $h_2(y) \le h_3(y)$, while $h_3(y), h_3(y^\prime) < 1$, so that
\begin{equation}\label{eq:derivative_hyp_ball_Xi_bounds}
	2 e^{-(R - y - y^\prime)}h_1(y) h_2(y) \le \Xi(y,y^\prime) \le 2 e^{-(R - y - y^\prime)}.
\end{equation}
In particular, since $R-y$ is an increasing function of $n$ uniformly on $0 < y < (1-\varepsilon)R$, there exists a $0 < \delta < 1$ such that $1/2 \le \Xi(y,y^\prime) < 2$ for all $y + y^\prime < R$ and $(1-\delta)(R-y) < y^\prime < R$ and $n$ large enough.

Next, taking the derivative of $\Xi(y,y^\prime)$ yields,
\begin{align*}
	\frac{\partial}{\partial y} \Xi(y,y^\prime) &= \Xi(y,y^\prime) + 2 e^{-(R - y - y^\prime)}
		\left(\frac{h_1^\prime(y) h_2(y)}{h_3(y^\prime) h_3(y)} + \frac{h_1(y)h_2^\prime(y)}{h_3(y^\prime) h_3(y)}
		- \frac{h_1(y) h_2(y) h_3^\prime(y)}{h_3(y^\prime) h_3(y)^2}\right)\\
	&= \Xi(y,y^\prime)\left(1 + \frac{h_1^\prime(y)}{h_1(y)} + \frac{h_2^\prime(y)}{h_2(y)} 
		- \frac{h_3^\prime(y)}{h_3(y)}\right)\\
	&:= \Xi(y,y^\prime)\left(1 + \varphi_n(y,y^\prime)\right),
\end{align*}
with
\[
	\varphi_n(y,y^\prime) = \frac{e^{y^\prime - y- R}}{1 - e^{y^\prime - y - R}} 
	- \frac{e^{y - y^\prime - R}}{1 - e^{y - y^\prime - R}} - \frac{2e^{-2(R - y)}}{1 - e^{-2(R-y)}}. 
\]
Therefore, by the chain rule,
\begin{align*}
	\frac{\partial}{\partial y} \Phi(y, y^\prime)
	&= \frac{1}{2}e^{R/2} \frac{1}{\sqrt{1 - \left(1 - \Xi(y,y^\prime)\right)^2}} 
		\frac{\partial}{\partial y} \Xi(y,y^\prime)\\
	&=  \frac{ \frac{1}{2}e^{R/2} \Xi(y,y^\prime)}{\sqrt{1 - \left(1 - \Xi(y,y^\prime)\right)^2}}\left(1 + 
		\varphi_n(y,y^\prime)\right). \numberthis \label{eq:derivative_Phi}
\end{align*}
Applying the Leibniz's rule we then get
\begin{align*}
	&\hspace{-30pt}\frac{\partial}{\partial y} \int_0^{R - y} \Phi(y,y^\prime) e^{-\alpha y^\prime} \dd y^\prime\\
	&= - \Phi(y,R - y)e^{-\alpha(R-y)} + \int_0^{R - y} \frac{\partial}{\partial y}  \Phi(y,y^\prime) 
		 e^{-\alpha y^\prime} \dd y^\prime\\
	&= -\frac{1}{2}e^{-(\alpha-\frac{1}{2})R + \alpha y} + \int_0^{R- y} \frac{ \frac{1}{2}e^{R/2} \Xi(y,y^\prime)}{\sqrt{1 - \left(1 - \Xi(y,y^\prime)\right)^2}}
		\left(1 + \varphi_n(y,y^\prime)\right) e^{-\alpha y^\prime} \dd y^\prime\\
	&= -\frac{1}{2}e^{-(\alpha-\frac{1}{2})R + \alpha y}  + \int_0^{(1-\delta)(R- y)} \frac{ \frac{1}{2}e^{R/2} \Xi(y,y^\prime)}
		{\sqrt{1 - \left(1 - \Xi(y,y^\prime)\right)^2}}
		\left(1 + \varphi_n(y,y^\prime)\right) e^{-\alpha y^\prime} \dd y^\prime\\
	&\hspace{10pt}+ \int_{(1-\delta)(R- y)}^{R - y} \frac{ \frac{1}{2}e^{R/2} \Xi(y,y^\prime)}
			{\sqrt{1 - \left(1 - \Xi(y,y^\prime)\right)^2}}
			\left(1 + \varphi_n(y,y^\prime)\right) e^{-\alpha y^\prime} \dd y^\prime\\
	&:= -I_1(y) + I_2(y) + I_3(y),
\end{align*}
with $0 \le \delta < 1$ such that $0 < \Xi(y,y^\prime) < 2$ for all $0 < y < R$ and $(1-\delta)(R-y) < y^\prime < R$.

We proceed by showing that
\begin{equation}\label{eq:derivative_hyp_ball_error_1}
	\lim_{n \to \infty} \sup_{0 < y \le (1-\varepsilon)R} \left|\frac{I_t(y)}{\Mu{\BallPo{y}}}\right| 
	= 0, \quad \text{for } t = 1,3
\end{equation}
while
\begin{equation}\label{eq:derivative_hyp_ball_main_part}
	\lim_{n \to \infty} \sup_{0 \le y \le (1-\varepsilon)R} \left|\frac{2\nu \alpha}{\pi \Mu{\BallPo{y}}} I_2(y) - \frac{1}{2}\right| = 0.
\end{equation}
This then implies~\eqref{eq:derivative_mu_hyp_ball_main} and finishes the proof.

For $I_1(y)$ we have 
\[
	\lim_{n \to \infty} \sup_{0 < y \le (1-\varepsilon)R} \Mu{\BallPo{y}}^{-1} I_1(y) 
	\le \lim_{n \to \infty} \sup_{0 < y \le (1-\varepsilon)R} \frac{1}{2\xi} e^{-(\alpha - \frac{1}{2})(R - y)} = 0.
\]

For $I_3(y)$ we first use that $y^\prime < R - y$ to bound $\varphi(y,y^\prime)$ as follows,
\[
	\varphi_n(y,y^\prime) \le \frac{e^{y^\prime - y - R}}{1 - e^{y^\prime - y - R}} \le \frac{e^{-2y}}{1 - e^{-2y}}.
\]
This then yields that
\[
	I_3(y) \le \frac{1}{2}\left(1 + \frac{e^{-2y}}{1 - e^{-2y}}\right)e^{R/2}
	\int_{(1-\delta)(R- y)}^{R - y} \frac{ \Xi(y,y^\prime)}{\sqrt{1 - \left(1 - \Xi(y,y^\prime)\right)^2}}
	e^{-\alpha y^\prime} \dd y^\prime.
\]
To bound the integral we recall that $0 < \Xi(y,y^\prime) \le 2e^{-(R-y-y^\prime)} < 2$ and for all $1/2 \le x < 2$,
\[
	\frac{1}{\sqrt{1- (1-x)^2}} \le \frac{2}{\sqrt{2-x}},
\]a
where the right hand side is a monotonic increasing function.
Therefore
\begin{align*}
	\int_{(1-\delta)(R- y)}^{R - y} \frac{ \Xi(y,y^\prime)}{\sqrt{1 - \left(1 - \Xi(y,y^\prime)\right)^2}}
		e^{-\alpha y^\prime} \dd y^\prime
	&\le 2 \int_{(1-\delta)(R- y)}^{R - y} \frac{\Xi(y,y^\prime)}{\sqrt{(2-\Xi(y,y^\prime))}} e^{-\alpha y^\prime} 
		\dd y^\prime \\
	&\le \sqrt{2} e^{-\alpha(R - y)} \int_{(1-\delta)(R- y)}^{R - y} 
		\frac{e^{-(R - y - y^\prime)}}{\sqrt{1 - e^{-(R - y - y^\prime)}}} e^{\alpha (R - y - y^\prime)} \dd y^\prime,
\end{align*}
Making the change of variables $z = e^{-(R - y - y^\prime)}$ ($\dd y^\prime = z^{-1} \dd z$) we get that
\begin{align*}
	&\hspace{-20pt}\sqrt{2} e^{-\alpha(R - y)} \int_{(1-\delta)(R- y)}^{R - y} 
			\frac{e^{-(R - y - y^\prime)}}{\sqrt{1 - e^{-(R - y - y^\prime)}}} e^{\alpha (R - y - y^\prime)} \dd y^\prime\\
	&= \sqrt{2} e^{-\alpha(R - y)} \int_{e^{-\delta (R - y)}}^{1} \frac{z^{-\alpha}}{\sqrt{1 - z}} \dd z
		\le \sqrt{2} e^{-\alpha(R - y)} \sqrt{1 - e^{-\delta (R - y)}}
		\le \sqrt{2} e^{-\alpha (R -y)}.
\end{align*}
We therefore conclude that
\[
	I_3(y) \le \frac{1}{\sqrt{2}} \left(1 + \frac{e^{-2y}}{1 - e^{-2y}}\right)e^{-(\alpha -\frac{1}{2})R + \alpha y}.
\]
which implies~\eqref{eq:derivative_hyp_ball_error_1} for $t=3$.

Finally, to show~\eqref{eq:derivative_hyp_ball_main_part} we first write
\begin{align*}
	\left|\frac{2\alpha \nu}{\pi \Mu{\BallPo{y}}} I_2(y) - \frac{1}{2}\right|
	&\le \left|\frac{2\alpha \nu}{\pi \Mu{\BallPo{y}}}  \int_0^{(1-\delta)(R- y)} \frac{\Phi(y,y^\prime)}{2} 
		e^{-\alpha y^\prime} \dd y^\prime - \frac{1}{2}\right| \\
	&\hspace{10pt}+ \frac{2\alpha \nu}{\pi \Mu{\BallPo{y}}}\left|\int_0^{(1-\delta)(R- y)} 
		\frac{\Phi(y,y^\prime)}{2} e^{-\alpha y^\prime} \dd y^\prime - I_2(y)\right|.
\end{align*}

Note that by Lemma~\ref{lem:average_degree_P_n}
\[
	\frac{2\nu \alpha}{\pi} \int_{0}^{(1-\delta)(R-y)} \Phi(y,y^\prime) e^{-\alpha y^\prime} \dd y^\prime
	= (1+\smallO{1})\Mu{\BallPo{y}},
\]
uniformly for all $0 \leq y \leq (1-\varepsilon)R$. Therefore
\[
	\lim_{n \to \infty} \sup_{0 < y \le (1-\varepsilon)R} \left|\frac{2\nu \alpha}{\pi \Mu{\BallPo{y}}} 
	\int_{0}^{(1-\delta)(R-y)} \frac{\Phi(y,y^\prime)}{2} e^{-\alpha y^\prime} \dd y^\prime
	- \frac{1}{2}\right| = 0,
\] 
and thus it suffices to show that the $\lim_{n \to \infty} \sup_{0 < y \le (1-\varepsilon)R}$ of the second term goes to zero.

Recalling the definition of $I_2(y)$ we have
\begin{align*}
	&\hspace{-30pt}\left|\int_0^{(1-\delta)(R- y)} \frac{\Phi(y,y^\prime)}{2} e^{-\alpha y^\prime} \dd y^\prime 
		- I_2(y)\right|\\
	&\le \int_0^{(1-\delta)(R- y)} \left|\frac{\Phi(y,y^\prime)}{2}- \frac{\frac{1}{2}e^{R/2} \Xi(y,y^\prime)}
		{\sqrt{1 - \left(1 - \Xi(y,y^\prime)\right)^2}}(1+\varphi_n(y,y^\prime)\right|e^{-\alpha y^\prime} \dd y^\prime.
		\numberthis \label{eq:derivative_hyp_ball_final_error}
\end{align*}
We will proceed to bound the term inside the integral. For this we first note that for $0 \le y^\prime \le (1-\delta)(R-y)$,
\[
	\varphi_n(y,y^\prime) \le \frac{e^{-\delta(R-y)}}{1 - e^{-\delta(R-y)}}.
\]
and recall that $\Xi(y,y^\prime) \le 2 e^{-(R-y-y^\prime)}$. Moreover, since $x/\sqrt{1-(1-x)^2} = x/\sqrt{2x-x^2}$ is an increasing function and $e^{-(R - y - y^\prime)} \le e^{-\delta(R- y)}$ for $0 < y^\prime < (1-\delta)(R-y)$,
\[
	\frac{\frac{1}{2}e^{R/2} \Xi(y,y^\prime)}{\sqrt{1 - \left(1 - \Xi(y,y^\prime)\right)^2}}
	\le e^{R/2} \frac{e^{-\delta(R-y)}}{\sqrt{2e^{-\delta(R-y)} - e^{-2\delta(R-y)}}}.
\]
Next, recall that $\Phi(y,y^\prime) = \frac{1}{2}e^{R/2}\arccos(1-\Xi(y,y^\prime))$. Then, since $\Xi(y,y^\prime) < 1$ for all $y^\prime < (1-\delta)(R-y)$, $y < R$ and $n$ large enough, we have (see Lemma~\ref{lem:arccos_approx}),
\[
	\left|\frac{1}{2}\Phi(y,y^\prime) - \frac{\frac{1}{2}e^{R/2} \Xi(y,y^\prime)}{\sqrt{1 - \left(1 - \Xi(y,y^\prime)\right)^2}}\right| \le \frac{1}{2}\Phi(y,y^\prime) \Xi(y,y^\prime).
\]
for all $y^\prime < (1-\delta)(R - y)$ and $y < R$. Together these facts imply that for $n$ large enough
\begin{align*}
	&\hspace{-30pt}\left|\frac{\Phi(y,y)}{2} - \frac{\frac{1}{2}e^{R/2} \Xi(y,y^\prime)}{\sqrt{1 - \left(1 - 	
		\Xi(y,y^\prime)\right)^2}}(1+\varphi_n(y,y^\prime))\right|\\
	&\le \frac{\Phi(y,y^\prime)\Xi(y,y^\prime)}{2} 
		+ \frac{\frac{1}{2}e^{R/2} \Xi(y,y^\prime) \varphi_n(y,y^\prime)}
		{\sqrt{1 - \left(1 - \Xi(y,y^\prime)\right)^2}} \\
	&\le e^{-\delta(R-y)} \Phi(y,y^\prime) + \frac{e^{-\delta(R-y)}}{1-e^{-\delta(R-y)}}
		\frac{e^{R/2} e^{-\delta(R-y)}}{\sqrt{2e^{-\delta(R-y)} - e^{-2\delta(R-y)}}}\\
	&\le e^{-\delta(R-y)} \Phi(y,y^\prime) + \frac{e^{\frac{R}{2}}e^{-\frac{3}{2}\delta(R-y)}}
		{\left(1 - e^{-\delta(R-y)}\right)^{3/2}}
\end{align*}

Plugging this into~\eqref{eq:derivative_hyp_ball_final_error} yields
\begin{align*}
	&\hspace{-30pt}\left|\int_0^{(1-\delta)(R- y)} \frac{\Phi(y,y^\prime)}{2} e^{-\alpha y^\prime} \dd y^\prime 
			- I_2(y)\right|\\
	&\le \int_0^{(1-\delta)(R- y)} \left(e^{-\delta(R-y)} \Phi(y,y^\prime) + 
		\frac{e^{\frac{R}{2}}e^{-\frac{3}{2}\delta(R-y)}}{\left(1 - e^{-\delta(R-y)}\right)^{3/2}}\right)
		e^{-\alpha y^\prime} \dd y^\prime \\
	&\le e^{-\delta(R-y)} \Mu{\BallPo{y}} + e^{\frac{y}{2}} \frac{e^{-(\alpha - \frac{1}{2} - (\alpha - \frac{3}{2})\delta)(R-y)}}
		{\alpha \left(1 - e^{-\delta(R-y)}\right)^{3/2}}
\end{align*}
To finish the argument we note that $R-y > 0$ for all $0 < y \le (1-\varepsilon)R$ and observe that $\delta < 1$ implies that $(\alpha - \frac{1}{2} - (\alpha - \frac{3}{2})\delta > 1$. Since $\Mu{\BallPo{y}} = \bigT{e^{\frac{y}{2}}}$ it the
then follows that
\[
	\lim_{n \to \infty} \sup_{0 < y \le (1-\varepsilon)R} \frac{2\alpha \nu}{\pi \Mu{\BallPo{y}}}
	\left|\int_0^{(1-\delta)(R- y)} \frac{\Phi(y,y^\prime)}{2} e^{-\alpha y^\prime} \dd y^\prime - I_2(y)\right| = 0,
\]
which completes the proof.
\end{proof}

\section{Code for the simulations}

The simulations of the clustering coefficient and function in the KPKVB model were done using Wolfram Mathematica 11.1.
The simulation dots for the clustering coefficient in Figure~\ref{fig:gamma} were generated by the following code (where in the second line, the entire script was also run for the values \verb|nu=1| and \verb|nu=0.5|):
\begin{lstlisting}[language=Mathematica,breaklines]
n=10000;
nu=2;
R=2*Log[n/nu];
plotpoints=20;
reps=100;
Plotingdataalpha = ConstantArray[0,{plotpoints,2}];
SeedRandom[1];
For[z=1,z<=plotpoints,z++,a=0.4+z (4.6/plotpoints); sum=0;
	For[r=1,r<=reps,r++,V = ConstantArray[0,{n,2}];
		For[i=1,i<=n,i++,
			V[[i,1]]=RandomReal[{-Pi ,Pi }];
			V[[i,2]]=ArcCosh[RandomReal[{0,1}](Cosh[a*R]-1)+1]/a];
		A= ConstantArray[0,{n,n}];
		For[i=1,i<=n,i++,
			For[j=1,j<=n,j++,
				If[Cosh[V[[i,2]]]Cosh[V[[j,2]]]-Sinh[V[[i,2]]]Sinh[V[[j,2]]]Cos[Abs[V[[i,1]]-V[[j,1]]]] <= Cosh[R] && i != j,A[[i,j]]=1,A[[i,j]]=0]]];
		g = AdjacencyGraph[A];
		sum=sum+MeanClusteringCoefficient[g]];
	Plotingdataalpha[[z,1]]=a;
	Plotingdataalpha[[z,2]]=1.0*sum/reps;]
Print[Plotingdataalpha]
\end{lstlisting}
The simulation dots for the clustering function in Figure~\ref{fig:gammak} were generated by the following code (where in the third line, the entire script was also run for the values \verb|nu=1| and \verb|nu=0.5|):
\begin{lstlisting}[language=Mathematica,breaklines]
n=10000;
a=0.8;
nu=2;
R=2*Log[n/nu];
plotpoints=24;
reps=100;
Plotingdatak = ConstantArray[0,{reps,plotpoints,2}];
SeedRandom[1];
For[r=1,r<=reps,r++,V = ConstantArray[0,{n,2}];
	For[i=1,i<=n,i++,
		V[[i,1]]=RandomReal[{-Pi ,Pi }];
		V[[i,2]]=ArcCosh[RandomReal[{0,1}](Cosh[a*R]-1)+1]/a];
	A= ConstantArray[0,{n,n}];
	For[i=1,i<=n,i++,
		For[j=1,j<=n,j++,
			If[Cosh[V[[i,2]]]Cosh[V[[j,2]]]-Sinh[V[[i,2]]]Sinh[V[[j,2]]]Cos[Abs[V[[i,1]]-V[[j,1]]]] <= Cosh[R] && i != j,A[[i,j]]=1,A[[i,j]]=0]]];
	g = AdjacencyGraph[A];
	For[k=1,k<=plotpoints,k++,
		sum=0;
		result=0;
		nrdegk=0;
		For[v =1,v<=n,v++; 
			If[VertexDegree[g,v]==k+1,
				result=result+LocalClusteringCoefficient[g,v];nrdegk++]];
		Plotingdatak[[r,k,1]]=k+1;
		If[nrdegk>0,Plotingdatak[[r,k,2]]=1.0*result/nrdegk]];]
Print[Mean[Plotingdatak]];
\end{lstlisting}

\section{Explicit expressions for \texorpdfstring{$\gamma, \gamma(k)$}{gamma, gamma(k)} when \texorpdfstring{$\alpha=1$}{alpha equals 1}.}\label{ssec:alphais1}

We've already established that $\gamma, \gamma(k)$ can be obtained at $\alpha=1$ by
taking the $\alpha\to 1$ limit of the expression obtained for $\alpha=1$. Here we derive an alternative explicit expression for completeness. Since the rest of our 
proofs do not the depend on it the reader could decide to skip this section on a first reading.

Recall that $\Gamma^\ast(q,z) = \Gamma^+(q+1,z) + \Gamma^+(q,z)$. We will prove the following.

\begin{proposition}\label{prop:gammaais1}
If $\alpha=1$ then 
\begin{align*} 
\gamma &= \frac{575 - 12 \pi^2}{576} + \frac{\eta^4(7 + \pi^2)\Gamma^\ast(-4, \eta)}{4}\\
 	&\hspace{10pt}- \frac{1}{2} \int_0^1 (1 - 4z + 3z^3)\log(1-z)(z + \eta)e^{-\eta/z} \dd z\\
 	&\hspace{10pt}- \int_0^1 \Li_2(z)(z^3 + \eta z^2) e^{-\eta/z} \dd z,		
 \end{align*}
and
\begin{align*}
 \gamma(k) &= \frac{9 \eta^3}{2 k!} \Gamma^+(k-3,\eta)-\frac{\xi^4}{k!}\frac{7+\pi^2}{4}\Gamma^+(k-4,\eta)\\
 	&\hspace{10pt}+ \frac{\eta^k}{2k!}\int_0^1 (1-4z+3z^2)\ln(1-z)z^{1-k}e^{-\eta/z}\dd z\\ 
 	&\hspace{10pt}+ \frac{\eta^k}{k!}\int_0^1 z^{3-k} \Li_2(z) e^{-\eta/z} \dd z,
 \end{align*}
with $\eta = 4\nu/\pi$ and $\Li_2(z) = \sum_{t = 1}^\infty z^t/t^2$, the dilogarithm 
function.
\end{proposition}

Naturally, the proof proceeds by proving the analogue of Lemma~\ref{lem:Paneq1}:

\begin{lemma}\label{lem:Pais1}
If $\alpha = 1$, then for all $y>0$:
	\begin{align*}
	&P(y) =\frac{9}{4} e^{-\frac{1}{2}y} + \frac{1 - 4 e^{-\frac{1}{2}y} + 3 e^{-y}}{4}\ln(1 - e^{-\frac{1}{2}y}) - \frac{7+\pi^2}{8}
e^{-y}  + 
\frac{1}{2}e^{-y}\Li_2(e^{-y})
	\end{align*}
	where $\Li_2(z)=-\int_0^z \frac{\ln(1-t)}{t}dt$ is the dilogarithm function.
\end{lemma}

\begin{proof}
We want to compute the limit $\lim_{\alpha \rightarrow 1} P_\alpha(y_0(z_0))$. For $\alpha \neq 1$, we label the terms as follows:
\begin{align*}
&P_\alpha(y_0(z_0)) \\
&= \frac{1}{\alpha-1}\left(s_1(\alpha,z_0)+s_2(\alpha,z_0)+\frac{1}{\alpha-1}(s_3(\alpha,z_0)+s_4(\alpha,z_0)) 
+s_5(\alpha,z_0)+s_6(\alpha,z_0)+s_7(\alpha,z_0)\right)
\end{align*}
where
\begin{align*}
s_1(\alpha,z_0) &= -\frac{1}{8 \alpha}\\
s_2(\alpha,z_0) &= (\alpha-1/2)z_0 \\
s_3(\alpha,z_0) &= - \frac{(\alpha - 1/2)^2 z_0^2}{4}\\
s_4(\alpha,z_0) &= z_0^{-2 + 4 \alpha} \frac{2^{-4 \alpha-1} (3 \alpha - 1)}{\alpha}\\
s_5(\alpha,z_0) &= z_0^{-2 + 4 \alpha} \frac{(\alpha - 1/2 ) B^-(1/2; 1 + 2 \alpha, -2 + 2 \alpha)}{2\alpha}\\
s_6(\alpha,z_0) &= \frac{(1 - z_0)^{2 \alpha}}{8 \alpha} \\
s_7(\alpha,z_0) &= - \frac{z_0^{4 \alpha - 2} B^-(1 - z_0; 2 \alpha, 3 - 4 \alpha)}{4}
\end{align*}
Now, we consider the functions $s_i(\alpha) = s_i(\alpha,z_0)$ as functions of $\alpha$ only and compute their Taylor 
expansion at $\alpha=1$, for $i\in \{1,2,5,6,7\}$ up to linear and for $i\in \{3,4\}$ up to quadratic order, i.e. we 
write $s_i(\alpha) = s_i(1)+s_i'(1)(\alpha-1)+o(\alpha-1)$ for $i\in \{1,2,5,6,7\}$ and 
$s_i(\alpha)=s_i(1)+s_i'(1)(\alpha-1)+\frac{s_i''(1)}{2}(\alpha-1)^2+o((\alpha-1)^2)$ for $i\in \{3,4\}$. Using these expansions, we can rewrite
\begin{align*}
&P(y_0(z_0)) = \frac{1}{\alpha-1}\left( \sum_{i\in \{1,2,5,6,7\}} s_i(1)+\sum_{i\in\{1,2,5,6,7\}} 
s_i'(1)(\alpha-1)+o(\alpha-1) \right. \\
&\left.+\frac{s_3(1)+s_4(1)}{\alpha-1} + s_3'(1)+s_4'(1)+\frac{1}{2}(s_3''(1)+s_4''(1))(\alpha-1) +o((\alpha-1)) \right)
\end{align*}
In order to continue, we compute:
\begin{align*}
s_1(\alpha) &=-\frac{1}{8}+\frac{1}{8}(\alpha-1)+o(\alpha-1) \\
s_2(\alpha) &=\frac{1}{2}z_0 +z_0 (\alpha-1)+o(\alpha-1) \\
s_3(\alpha) &= -\frac{1}{16}z_0^2 -\frac{1}{4}z_0^2(\alpha-1)-\frac{1}{2}z_0^2(\alpha-1)^2+o( (\alpha-1)^2) \\
s_4(\alpha) &= \frac{1}{16}z_0^2 +\frac{z_0^2}{4}\left(\frac{1}{8}+\ln\frac{z_0}{2}\right)(\alpha-1)\\
&\qquad+\frac{z_0^2}{8}\left(8\left(\ln\frac{z_0}{2}\right)^2+2\ln\frac{z_0}{2} - \frac{1}{2}\right)(\alpha-1)^2+o( (\alpha-1)^2)\\
s_5(\alpha) &= \frac{z_0^2}{4} B^-(1/2;3,0) +o(\alpha-1) \\
	&\hspace{10pt}+ z_0^2\left(\left(\ln(z_0)+\frac{1}{4}\right) B^-(1/2;3,0) +1/2\int_0^{\frac{1}{2}} \ln(t(1-t))t^2(1-t)^{-1} \dd t \right)  (\alpha-1)+o (\alpha-1)\\
s_6(\alpha) &= \frac{(1-z_0)^2}{8}+\frac{(1-z_0)^2}{4} (\ln(1-z_0)-1/2 )(\alpha-1 + \smallO{a-1})\\
s_7(\alpha) &= -\frac{z_0^2}{4}B^-(1-z_0;2,-1) +o(\alpha-1) \\
	&\hspace{10pt}- z_0^2 \left(\ln(z_0)B^-(1-z_0;2,-1) 
		+\int_0^{1-z_0} t(1-t)^{-2}\ln\left(\frac{\sqrt{t}}{1-t}\right)t(1-t)^{-2} \dd t \right) (\alpha-1).
\end{align*}
Based on this we see that
\begin{align*}
s_3(1)+s_4(1)=-\frac{1}{16}z_0^2+\frac{1}{16}z_0^2 =0
\end{align*}
and
\begin{align*}
&\sum_{i\in \{1,2,5,6,7\}} s_i(1) + s_3'(1)+s_4'(1) \\
&= -\frac{1}{8}+\frac{1}{2}z_0-\frac{1}{4}z_0^2+\frac{z_0^2}{32}+\frac{z_0^2}{4}\ln(\frac{z_0}{2}) 
+\frac{z_0^2}{4}B^-(1/2;3,0)+\frac{(1-z_0)^2}{8}-\frac{z_0^2}{4}B^-(1-z_0;2,-1) \\
&=-\frac{1}{8}+\frac{1}{2}z_0-\frac{1}{4}z_0^2+\frac{z_0^2}{32}+\left(\frac{z_0^2}{4}\ln(z_0)-\frac{z_0^2}{4}\ln 2\right) 
+\left(-\frac{5z_0^2}{32}+\frac{z_0^2}{4}\ln2\right) \\ 
&\hspace{50pt}+\left(\frac{1}{8}-\frac{z_0}{4}+\frac{z_0^2}{8}\right)+\left(\frac{z_0^2}{4}-\frac{z_0}{4}-\frac{z_0^2}{4}\ln z_0\right) &=0,
\end{align*}
using that 
\begin{align*}
B^-(\frac{1}{2};3,0)&=\int_0^\frac{1}{2}t^2(1-t)^-1dt=\int_\frac{1}{2}^1 (1-s)^2 s^{-1}ds\\
&=\int_\frac{1}{2}^1 s^{-1}-2+sds = -2+\frac{1}{2}-\ln\frac{1}{2}+1-\frac{1}{8}=-\frac{5}{8}+\ln 2
\end{align*}
and
\begin{align*}
B^-(1-z_0;2,-1)&=\int_0^{1-z_0} t(1-t)^{-2}dt = \int_{z_0}^1 (1-s)s^{-2}ds \\
&=\int_{z_0}^1 s^{-2}-s^{-1}ds = -1+z_0^{-1}+\ln z_0.
\end{align*}

Finally, it follows that as $\alpha \to 1$,
\begin{align*}
P(y_0(z_0)) = \sum_{i \in\{1,2,5,6,7\}} s_i'(1)+\frac{1}{2}(s_3''(1)+s_4''(1)) + o(1)
\end{align*}
Therefore, the desired value of $\lim_{\alpha \rightarrow 1} P(y_0(z_0))$ is given by
\begin{align*}
&\sum_{i \in\{1,2,5,6,7\}} s_i'(1)+\frac{1}{2}(s_3''(1)+s_4''(1))\\
	&=\frac{1}{8}+z_0-\frac{z_0^2}{4}+\frac{z_0^2}{8}(4(\ln\frac{z_0}{2})^2+\ln\frac{z_0}{2} - \frac{1}{4}) 
		+\frac{(1-z_0)^2}{4} (\ln(1-z_0)-1/2 )\\
	&\hspace{10pt}+z_0^2\left(\left(\ln(z_0)+\frac{1}{4}\right) B^-(1/2;3,0)+ \frac{1}{2}\int_0^{\frac{1}{2}} 
		\ln(t(1-t)) \, t^2(1-t)^{-1} \dd t \right) \\
	&\hspace{10pt}- z_0^{2}\left(\ln(z_0)B^-(1-z_0;2,-1)+\int_0^{1-z_0}\ln\left(\frac{\sqrt{t}}{1-t}\right)t(1-t)^{-2} 
		\dd t \right) \\
	&=\frac{1}{8}+z_0-\frac{z_0^2}{4}+\frac{z_0^2}{2}(\ln\frac{z_0}{2})^2
		+\frac{z_0^2}{8}\ln\frac{z_0}{2} - \frac{z_0^2}{32} \\
	&\hspace{10pt}-\frac{5}{8}z_0^2\ln(z_0)+z_0^2\ln(z_0)\ln 2-\frac{5z_0^2}{32} +\frac{z_0^2 \ln2}{4}\\
	&\hspace{10pt}+z_0^2/2\int_0^{\frac{1}{2}} \ln(t(1-t)) \, t^2(1-t)^{-1}\dd t \\
	&\hspace{10pt}+\frac{(1-z_0)^2}{4}\ln(1-z_0) -\frac{1}{8}+\frac{z_0}{4}-\frac{z_0^2}{8}\\
	&\hspace{10pt}+ z_0^2\ln(z_0)-z_0 \ln z_0-z_0^2(\ln z_0)^2
		-z_0^2\int_0^{1-z_0} \ln\left(\frac{\sqrt{t}}{1-t}\right)t(1-t)^{-2}  \dd t \\
	&=\frac{5}{4}z_0-\frac{9}{16}z_0^2 +\frac{z_0^2}{2}(\ln\frac{z_0}{2})^2+\frac{z_0^2}{8}\ln\frac{z_0}{2} 
		+\frac{(1-z_0)^2}{4}\ln(1-z_0) \\
	&\hspace{10pt}+\frac{3}{8}z_0^2\ln(z_0)+z_0^2\ln(z_0)\ln 2+\frac{z_0^2 \ln2}{4}
		+z_0^2/2\int_0^{\frac{1}{2}} \ln(t(1-t)) \, t^2(1-t)^{-1} \dd t \\
	&\hspace{10pt}-z_0 \ln z_0-z_0^2(\ln z_0)^2-z_0^2\int_0^{1-z_0} \ln\left(\frac{\sqrt{t}}{1-t}\right)t(1-t)^{-2} \dd t \\
	&=\frac{5}{4}z_0-\frac{9}{16}z_0^2 +\frac{z_0^2}{2}(\ln\frac{z_0}{2})^2+\frac{z_0^2}{8}\ln\frac{z_0}{2} 
		+\frac{(1-z_0)^2}{4}\ln(1-z_0) \\
	&\hspace{10pt}+\frac{3}{8}z_0^2\ln(z_0)+z_0^2\ln(z_0)\ln 2+\frac{z_0^2 \ln2}{4}
		+z_0^2/2(11/8 -1/4 \ln 2 -3/2\ln(2)^2 -  \Li_2(1/2)) \\
	&\hspace{10pt}-z_0 \ln z_0-z_0^2(\ln z_0)^2+z_0(1 + \frac{1}{2}(2-z_0) \ln(z_0) 
		+ \frac{1}{2}z_0 \ln(z_0)^2 - \frac{1}{2} (1-z_0) \ln(1-z_0) \\
	&\hspace{10pt}+\frac{1}{2}z_0 \Li_2(z_0)) - z_0^2   -\frac{1}{2}  z_0^2\Li_2(1) \\
	&=\frac{9}{4}z_0-\frac{25}{16}z_0^2 +\frac{z_0^2}{2}(\ln\frac{z_0}{2})^2+\frac{z_0^2}{8}\ln\frac{z_0}{2} 
		+\frac{(1-z_0)^2}{4}\ln(1-z_0) \\
	&\hspace{10pt}-\frac{1}{8}z_0^2\ln(z_0)+z_0^2\ln(z_0)\ln 2+\frac{z_0^2 \ln2}{4}
		+z_0^2/2(11/8 -1/4 \ln 2 -3/2\ln(2)^2 \\
	&\hspace{10pt}- \Li_2(1/2)-\Li_2(1)+\Li_2(z_0)) -\frac{1}{2}z_0^2(\ln z_0)^2 - \frac{1}{2} z_0(1-z_0) \ln(1-z_0)
\end{align*}
where we used that
\begin{align*}
	z_0^2/2\int_0^{\frac{1}{2}} \ln(t)t^2(1-t)^{-1}+\ln(1-t)t^2(1-t)^{-1} \dd t
	&= 11/8 -1/4 \ln 2 -3/2\ln(2)^2 -  \Li_2(1/2),
\end{align*}
and
\begin{align*}
	&z_0^2\int_0^{1-z_0} 1/2\ln(t)t(1-t)^{-2}-t\ln(1-t)(1-t)^{-2} \dd t\\
	&= -\frac{1}{z_0} (1 + \frac{1}{2}(2-z_0) \ln(z_0) + \frac{1}{2}z_0 \ln(z_0)^2 - \frac{1}{2} (1-z_0) \ln(1-z_0) +\frac{1}{2}z_0 \Li_2(z_0))+1   +\frac{1}{2}  \Li_2(1).
\end{align*}
By expanding the squares and collecting terms, the last expression can be simplified to
\begin{align*}
&\frac{9}{4} z_0 + \frac{1 - 4 z_0 + 3 z_0^2}{4}\ln(1 - z_0) + 
z_0^2 \left(-7/8 - \frac{\ln(2)^2+2\Li_2(1/2) + 2\Li_2(1)}{4} \right) + 
\frac{1}{2}z_0^2 \Li_2(z) \\
=&\frac{9}{4} z_0 + \frac{1 - 4 z_0 + 3 z_0^2}{4}\ln(1 - z_0) - \frac{7+\pi^2}{8}
z_0^2  + 
\frac{1}{2}z_0^2 \Li_2(z)
\end{align*}
which finishes the computation.
\end{proof}

\begin{proofof}{Proposition~\ref{prop:gammaais1}}
It suffices to find the value of $J$ and $I^{(k)}$ 
at $\alpha = 1$. We can do this by computing the integrals with the expression for $P(y)$ that we found for $\alpha=1$, i.e.
\begin{align*}
	J &= 2\alpha\int_0^1 \left(\frac{9}{4} z + \frac{1 - 4 z + 3 z^2}{4}\ln(1 - z) - \frac{7+\pi^2}{8}
		z^2  + \frac{1}{2}z^2 \Li_2(z)\right) z^{2\alpha -1} \dd z \\
	&=\frac{575-12\pi^2}{576}
\end{align*}
and 
\begin{align*}
	I^{(k)} &= \frac{2\alpha \xi^k}{k!}\int_0^1 \left(\frac{9}{4} z + \frac{1 - 4 z + 3 z^2}{4}\ln(1 - z) 
		- \frac{7+\pi^2}{8}	z^2  + \frac{1}{2}z^2 \Li_2(z) \right)z^{2\alpha-k-1} e^{-\xi/z} \dd z \\
	&= \frac{2 \eta^k}{k!}\int_0^1 \left(\frac{9}{4} z + \frac{1 - 4 z + 3 z^2}{4}\ln(1 - z) 
		- \frac{7+\pi^2}{8} z^2  + \frac{1}{2}z^2 \Li_2(z) \right)z^{1-k} e^{-\eta/z} \dd z \\
	&= \frac{9 \eta^k}{2 k!}\eta^{3-k}\Gamma^+(k-3,\eta)-\frac{\eta^k}{k!}\frac{7+\pi^2}{4}\eta^{4-k}\Gamma^+(k-4,\eta)\\
	&\hspace{10pt}+ \frac{\eta^k}{2k!}\int_0^1 (1-4z+3z^2)\ln(1-z)z^{1-k}e^{-\eta/z}dz+\frac{\eta^k}{k!}\int_0^1 z^{3-k} 
		\Li_2(z) e^{-\eta/z} \dd z \\
	&=\frac{9 \eta^3}{2 k!}\Gamma^+(k-3,\eta)-\frac{\eta^4}{k!}\frac{7+\pi^2}{4}\Gamma^+(k-4,\eta)\\
	&\hspace{10pt}+ \frac{\eta^k}{2k!}\int_0^1 (1-4z+3z^2)\ln(1-z)z^{1-k}e^{-\eta/z}dz+\frac{\eta^k}{k!}\int_0^1 z^{3-k} 
		\Li_2(z) e^{-\eta/z} \dd z
\end{align*}
where $\eta = \frac{4\nu}{\pi}$ and and $\Li_2(z) = \sum_{t = 1}^\infty z^t/t^2$, the dilogarithm function. 
Plugging this into ~\eqref{eq:gammaint} and~\eqref{eq:gammakint} yields the expressions in the statement of the proposition.
\end{proofof}

\end{appendices}

\end{document}